\documentclass[12pt,letterpaper]{amsart}
\usepackage[margin=1.1in]{geometry}
\usepackage[english]{babel}
\usepackage[alphabetic]{amsrefs}

\usepackage{amsmath}
\usepackage{amssymb}
\usepackage{amsthm}
\usepackage{amscd}
\usepackage{bbm}
\usepackage{esvect}
\usepackage{mathrsfs}
\usepackage{hyperref}
\usepackage{cleveref}
\usepackage{pgf,tikz}
\usetikzlibrary{arrows}
\usepackage[T1]{fontenc}

\DeclareMathOperator{\E}{\mathbb{E}}

\newcommand{\norm}[1]{\left\lVert#1\right\rVert}
\newcommand{\pref}[1]{(\ref{#1})}

\newcommand{\product}[1]{\left\langle#1\right\rangle}

\newcommand{\ma}{\mathcal{M}^A}
\newcommand{\mbc}{\mathcal{M}^{BC}}
\newcommand{\md}{\mathcal{M}^D}
\newcommand{\Wa}{\mathcal{W}^A}
\newcommand{\Wbc}{\mathcal{W}^{BC}}

\newcommand{\even}{\Gamma_{\text{even}}}
\newcommand{\evenN}{\Gamma_{\text{even};\,N}}

\newcommand{\NCeven}{NC^{\text{even}}}

\theoremstyle{plain}
\newtheorem{theorem}{Theorem}[section]
\newtheorem{lemma}[theorem]{Lemma}
\newtheorem{proposition}[theorem]{Proposition}
\newtheorem{corollary}[theorem]{Corollary}
\newtheorem{conjecture}[theorem]{Conjecture}

\theoremstyle{remark}
\newtheorem{remark}[theorem]{Remark}
\newtheorem{example}[theorem]{Example}

\theoremstyle{definition}
\newtheorem{definition}[theorem]{Definition}

\crefname{subsection}{subsection}{subsections}
\Crefname{subsection}{Subsection}{Subsections}

\crefname{definition}{definition}{definitions}
\Crefname{definition}{Definition}{Definitions}

\crefname{lemma}{lemma}{lemmas}
\Crefname{lemma}{Lemma}{Lemmas}

\crefname{proposition}{proposition}{propositions}
\Crefname{proposition}{Proposition}{Propositions}

\crefname{corollary}{corollary}{corollaries}
\Crefname{corollary}{Corollary}{Corollaries}

\crefname{conjecture}{conjecture}{conjectures}
\Crefname{conjecture}{Conjecture}{Conjectures}

\crefname{remark}{remark}{remarks}
\Crefname{remark}{Remark}{Remarks}

\crefname{example}{example}{examples}
\Crefname{example}{Example}{Examples}

\allowdisplaybreaks

\title{Approximating the coefficients of the Bessel functions}
\author{Andrew Yao}

\address{Massachusetts Institute of Technology, Cambridge, Massachusetts}
\email{ajyao@mit.edu}

\begin{document}

\begin{abstract} 
We determine equivalent conditions between the asymptotic coefficients of the Bessel generating functions of a sequence of probability measures and the asymptotic expected values of power sums when their inputs are sampled from these measures. We establish these conditions over the $|\theta N| \rightarrow\infty$ regime for the type A and D root systems and over the $|\theta_0 N|\rightarrow \infty, \frac{\theta_1}{\theta_0 N}\rightarrow c\in\mathbb{C}$ regime for the type BC root system. We also establish equivalent conditions over the $\theta N \rightarrow c\in\mathbb{C}$ regime for the type A and D root systems and over the $\theta_0 N\rightarrow c_0\in\mathbb{C}, \frac{\theta_1}{\theta_0 N}\rightarrow c_1\in\mathbb{C}$ regime for the type BC root system that generalize existing results. Furthermore, we determine the asymptotics of the coefficients of the Bessel functions over the regimes that we have mentioned.
\end{abstract}

\maketitle

\section{Introduction}

For positive integers $N\geq 2$, we study the asymptotics of the coefficients of the Bessel functions associated to the irreducible root systems $A^{N-1}$, $B^N$, $C^N$, and $D^N$. We denote the Bessel function associated to a finite root system $\mathcal{R}$ and multiplicity function $\theta: \mathcal{R}\rightarrow\mathbb{C}$ by $J^{\mathcal{R}(\theta)}_a(x)$, where $a,x\in\mathbb{C}^N$, provided that the function exists. It is an eigenfunction of the Dunkl operators associated to $\mathcal{R}$ and $\theta$, which are introduced in \cite{dunkloperators}. Furthermore, the function is symmetric with respect to the reflection group generated by $\mathcal{R}$ in the variables $a$ and $x$.

For certain choices of $\theta$, the paper \cite{dunklbessel} shows that $J^{\mathcal{R}(\theta)}_a(x)$ exists, is unique, and is holomorphic over $a,x\in\mathbb{C}^N$, see \Cref{thm:dunkl_eigenfunction_symmetry}. The paper also discusses a nonsymmetric eigenfunction $E_a^{\mathcal{R}(\theta)}(x)$, see \Cref{thm:dunkl_eigenfunction}; we have that $J^{\mathcal{R}(\theta)}_a(x)$ is the average of $E_{ga}^{\mathcal{R}(\theta)}(x)$ when $g$ is selected uniformly at random from the finite reflection group generated by $\mathcal{R}$.

A goal of this paper is to analyze the coefficients of $J_a^{\mathcal{R}(\theta)}(x)$ when $a$ is fixed and $\mathcal{R}$ is one of $A^{N-1}$, $B^N$, $C^N$, and $D^N$. The multiplicity function $\theta$ varies with $N$ and should be viewed as a function of $N$, although we denote it as $\theta$ rather than $\theta(N)$ for brevity.

To compute the coefficients of $J_a^{\mathcal{R}(\theta)}(x)$, it is equivalent to compute the values of 
\[
f(\mathcal{D}_1(\mathcal{R}(\theta)),\ldots,\mathcal{D}_N(\mathcal{R}(\theta)))g(x_1,\ldots,x_N)
\]
for all $f,g\in\mathbb{C}[x_1,\ldots,x_N]$ that are homogeneous, have the same degree, and are symmetric with respect to the reflection group generated by $\mathcal{R}$; for an explanation of why this is the case, see \Cref{lemma:eigenvectorformula}. We compute the asymptotics of these values in \Cref{thm:main}.

\subsection{Notation}

We introduce some notation required for stating the main result. For additional definitions and notation, see \Cref{sec:def}. 

For $N\geq 1$, $\Gamma_N$ consists of partitions with length at most $N$, $\Gamma\triangleq \bigcup_{N\geq 1}\Gamma_N$, and $\even$ and $\evenN$ are the elements of $\Gamma$ and $\Gamma_N$, respectively, with all parts of even size; for $\lambda\in\Gamma$, $p_\lambda(x_1,\ldots,x_N)\triangleq\prod_{i=1}^{\ell(\lambda)} \sum_{j=1}^N x_j^{\lambda_i}$. Additionally, $e(x_1,\ldots,x_N)\triangleq x_1\cdots x_N$. Furthermore, for $k\geq 1$, $NC(k)$ denotes the set of noncrossing partitions of $[k]$ and $\NCeven(2k)$ denotes the set of noncrossing partitions of $[2k]$ with all even block sizes.

We outline the notation we use that is relevant to the $A^{N-1}$, $B^N$, $C^N$, and $D^N$ root systems and the corresponding multiplicity functions. Note that $e_i\triangleq[\mathbf{1}\{i=j\}]_{j\in[N]}^T$ for $i\in [N]$.
\begin{itemize}
\item For $\theta\in\mathbb{C}$, we let $A^{N-1}(\theta)$ denote the root system $A^{N-1}$ with multiplicity function assigning $\theta$ to the roots $e_i-e_j$ for distinct $i,j\in [N]$.
\item For $\theta_0,\theta_1\in\mathbb{C}$, we let $BC^N(\theta_0,\theta_1)$ denote the root system $B^N$ or $C^N$ with multiplicity function assigning $\theta_0$ to the roots $e_i-e_j$, $e_j-e_i$, $e_i+e_j$, and $-e_i-e_j$ for $i,j\in [N]$ such that $i<j$ and $\theta_1$ to the roots that are scalar multiples of $e_i$ for $i\in [N]$.
\item For $\theta\in\mathbb{C}$, we let $D^N(\theta)$ denote the root system $D^N$ with multiplicity function assigning $\theta$ to the roots $e_i-e_j$, $e_j-e_i$, $e_i+e_j$, and $-e_i-e_j$ for $i,j\in [N]$ such that $i<j$.
\end{itemize}
Furthermore, we define the Dunkl operators associated with these root systems:
\begin{align*}
& \mathcal{D}_i(A^{N-1}(\theta)) \triangleq \partial_i + \theta\sum_{j\in[N]\backslash\{i\}} \frac{1-s_{ij}}{x_i-x_j}, \\
& \mathcal{D}_i(BC^N(\theta_0,\theta_1)) \triangleq \partial_i + \theta_1 \frac{1-\tau_i}{x_i} + \theta_0\sum_{j\in[N]\backslash\{i\}} \left(\frac{1-s_{ij}}{x_i-x_j} + \frac{1-\tau_i\tau_j s_{ij}}{x_i+x_j}\right), \\
& \mathcal{D}_i(D^N(\theta)) \triangleq \partial_i + \theta\sum_{j\in[N]\backslash\{i\}} \left(\frac{1-s_{ij}}{x_i-x_j} + \frac{1-\tau_i\tau_j s_{ij}}{x_i+x_j}\right)
\end{align*}
for $i\in [N]$, where $s_{ij}$ switches the $i$th and $j$th entries of an element of $\mathbb{C}^N$ for distinct $i,j\in [N]$ and $\tau_i$ flips the sign of the $i$th entry of an element of $\mathbb{C}^N$ for $i\in [N]$.

\subsection{Bessel generating functions and related works}
\label{subsec:bgf}
In addition to considering the Bessel function $J_a^{\mathcal{R}(\theta)}(x)$ for a single value of $a$, we consider the average of $J_a^{\mathcal{R}(\theta)}(x)$ when $a$ is randomly sampled. The resulting function is referred to as the Bessel generating function and has been studied previously in \cites{gaussianfluctuation,matrix,limits_prob_measures,airy_beta,rectangularmatrix}. This notion is also related to the Dunkl transform introduced in \cite{dunklbound}. We discuss results related to Bessel generating functions in \Cref{subsec:expdecay}.

\begin{definition}
\label{def:bgf}
Suppose $\mu$ is a Borel probability measure over $\mathbb{C}^N$. Then, the \textit{Bessel generating function} associated to $\mu$ and $\mathcal{R}(\theta)$ is
\[
G_\mu^{\mathcal{R}(\theta)}(x) \triangleq \mathbb{E}_{a\sim\mu}[J_a^{\mathcal{R}(\theta)}(x)],
\]
provided that the expected value converges.
\end{definition}

For $k\geq 1$, the type A Bessel function satisfies 
\[
\sum_{i=1}^N \mathcal{D}_i(A^{N-1}(\theta))^k J_a^{A^{N-1}(\theta)}(x) = \sum_{i=1}^N a_i^k J_a^{A^{N-1}(\theta)}(x).
\]
Therefore, after assuming regularity conditions on $\mu$, we have that 
\[
[1]\sum_{i=1}^N \mathcal{D}_i(A^{N-1}(\theta))^k G_\mu^{A^{N-1}(\theta)}(x) = \mathbb{E}_{a\sim\mu}\left[\sum_{i=1}^N a_i^k\right], 
\]
see \Cref{lemma:exponent_decay} for more details. As $N\rightarrow\infty$, in order to analyze the moments of $\mu$, it suffices to analyze $[1]\sum_{i=1}^N \mathcal{D}_i(A^{N-1}(\theta))^k G_\mu^{\mathcal{R}(\theta)}(x)$. For the type A root system, the paper \cite{gaussianfluctuation} addresses this question for $\theta=1$ and \cite{matrix} does so for the $\theta N\rightarrow c \geq 0$ regime. For the type BC root system, \cite{rectangularmatrix} does so for the $\theta_0 N \rightarrow c_0\geq 0$ and $\theta_1 \rightarrow c_1\geq 0$ regime. We address this question for both regimes and the type A, BC, and D root systems, see \Cref{thm:main}.

The paper \cite{gaussianfluctuation} defines the type A Bessel generating function to be the average of $\frac{J_a^{A^{N-1}(\theta)}(x)}{J_a^{A^{N-1}(\theta)}(\rho)}$ when $a$ is sampled from a probability ensemble; we consider when $\rho=0$. In order to determine the moments of the probability ensemble, the paper computes the value of the Bessel generating function after applying Dunkl operators at $x=\rho$, where $\rho$ converges to a fixed distribution after rescaling as $N\rightarrow\infty$.

For the $|\theta N|\rightarrow\infty$ regime, the paper \cite{limits_prob_measures} introduces a framework for analyzing the asymptotic moments when the coefficients of the logarithms of the type A Bessel generating functions with two or more variables can have nonzero $N\rightarrow\infty$ limits after scaling by $(\theta N)^{-1}$. The paper \cite{domino_tiling} studies Schur generating functions and considers a similar setting.

The converse statement is that we can determine the asymptotics of the coefficients of $G_\mu^{\mathcal{R}(\theta)}$ as $N\rightarrow\infty$ from the asymptotic values of $\mathbb{E}_{a\sim\nu}[p_\nu(a)]$ for $\nu\in\Gamma$. The paper \cite{matrix} establishes this statement for the type A root system and the $\theta N\rightarrow c\geq 0$ regime and conjectures that it is true for the $|\theta N|\rightarrow\infty$ regime. We establish the converse statement for both regimes and the type A, BC, and D root systems.

The papers \cites{representations, fluctuations, unitary,llnclt,qmk,jackgenfunc,perelomov-popov,domino_tiling} similarly analyze the coefficients of Schur and Jack generating functions to establish a law of large numbers. The papers \cites{unitary,llnclt,jackgenfunc,rectangularmatrix} also prove converse statements and \cites{fluctuations,unitary,llnclt,gaussianfluctuation,domino_tiling} consider the central limit theorem in addition to the law of large numbers.

\subsection{Main result}

We state the main result of this paper, which is over the regime $|\theta N| \rightarrow \infty$ for the $A^{N-1}$ and $D^N$ root systems and the regime $|\theta_0 N| \rightarrow\infty$, $\frac{\theta_1}{\theta_0 N} \rightarrow c\in\mathbb{C}$ for the $BC^N$ root system. Observe that the first regime includes the case where $\theta$ is a fixed element of $\mathbb{C}^\times$ and the second regime includes the case where $\theta_0$ is a fixed element of $\mathbb{C}^\times$ and $\frac{\theta_1}{\theta_0 N} \rightarrow c\in\mathbb{C}$. The main application of this result is when $F_N^A=G_\mu^{A^{N-1}(\theta)}$, $F_N^{BC}=G_\mu^{BC^N(\theta_0, \theta_1)}$, and $F_N^D=G_\mu^{D^N(\theta)}$ are Bessel generating functions.

\begin{theorem}
\label{thm:main}
Assume that $\lim_{N\rightarrow\infty} |\theta N|=\infty$, $\lim_{N\rightarrow\infty} |\theta_0 N|=\infty$, and $\lim_{N\rightarrow\infty} \frac{\theta_1}{\theta_0 N} = c\in\mathbb{C}$.
\\ (A): Suppose $F^A_N(x_1,\ldots,x_N) = \exp\left(\sum_{\lambda\in\Gamma_N} c_\lambda(N) p_\lambda\right)\in\mathbb{C}[[x_1,\ldots,x_N]]$ for $N\geq 2$. Then, the following are equivalent.
    \begin{enumerate}
        \item[(a)] For all $\lambda\in\Gamma$, $\lim_{N\rightarrow\infty} \frac{c_\lambda(N)}{\theta N} = c_\lambda \in \mathbb{C}$ if $\ell(\lambda)=1$ and $\lim_{N\rightarrow\infty} \frac{c_\lambda(N)}{(\theta N)^{\ell(\lambda)}}=0$ if $\ell(\lambda)\geq 2$. 
        \item[(b)] For all $\nu\in\Gamma$,
\[
\lim_{N\rightarrow\infty}\frac{[1]\prod_{i=1}^{\ell(\nu)} \sum_{j=1}^{N} \mathcal{D}_j(A^{N-1}(\theta))^{\nu_i} F_N^A}{(\theta N)^{|\nu|} N^{\ell(\nu)}}  = \prod_{i=1}^{\ell(\nu)} \sum_{\pi\in NC(\nu_i)} \prod_{B\in \pi} |B|c_{(|B|)}.
\]
    \end{enumerate}
(B): Suppose $F^{BC}_N(x_1,\ldots,x_N) = \exp\left(\sum_{\lambda\in\evenN} c_{\lambda}(N)p_\lambda\right)\in\mathbb{C}[[x_1^2,\ldots,x_N^2]]$ for $N\geq 2$. Consider the following statements.
\begin{enumerate}
\item[(c)] For all $\lambda\in\even$, $\lim_{N\rightarrow\infty} \frac{c_\lambda(N)}{\theta_0 N}=c_{\lambda} \in\mathbb{C}$ if $\ell(\lambda)=1$ and $\lim_{N\rightarrow\infty} \frac{c_\lambda(N)}{(\theta_0 N)^{\ell(\lambda)}}=0$ if $\ell(\lambda)\geq 2$.
\item[(d)] For all $\nu\in\even$,
\begin{align*}
& \lim_{N\rightarrow\infty}\frac{[1]\prod_{i=1}^{\ell(\nu)} \sum_{j=1}^{N} \mathcal{D}_j(BC^N(\theta_0,\theta_1))^{\nu_i} F_N^{BC}}{(\theta_0 N)^{|\nu|} N^{\ell(\nu)}} \\
& = \prod_{i=1}^{\ell(\nu)} \sum_{\pi\in \NCeven(\nu_i)} (1+c)^{o(\pi)}\prod_{B\in \pi} 2^{|B|-1}|B|c_{(|B|)}.
\end{align*}
\end{enumerate}
Then, (c) implies (d) and if $c\not=-1$, then (d) implies (c).
\\ (C): Suppose 
\begin{align*}
F_N^D(x_1,\ldots,x_N)& =\exp\left(\sum_{\lambda\in\evenN} c_\lambda(N)p_\lambda\right) + ed_0(N) + e\exp\left(\sum_{\lambda\in\evenN}d_\lambda(N)p_\lambda\right) \\
& \in\mathbb{C}[[x_1^2,\ldots,x_N^2]] + e\mathbb{C}[[x_1^2,\ldots,x_N^2]]
\end{align*}
for $N\geq 2$. Consider the following statements.
\begin{enumerate}
\item[(e)] For all $\lambda\in\even$, $\lim_{N\rightarrow\infty} \frac{c_\lambda(N)}{\theta N}=c_\lambda\in\mathbb{C}$ if $\ell(\lambda)=1$ and $\lim_{N\rightarrow\infty} \frac{c_\lambda(N)}{(\theta N)^{\ell(\lambda)}}=0$ if $\ell(\lambda)>1$. 
\item[(f)] It is the case that $\lim_{N\rightarrow\infty}d_0(N)=d_0\in\mathbb{C}$ and for all $\lambda\in\even$, $\lim_{N\rightarrow\infty} \frac{d_\lambda(N)}{\theta N}=d_\lambda\in\mathbb{C}$ if $\ell(\lambda)=1$ and $\lim_{N\rightarrow\infty} \frac{d_\lambda(N)}{(\theta N)^{\ell(\lambda)}}=0$ if $\ell(\lambda)>1$.
\item[(g)] For all $\nu\in\even$,
\[
\lim_{N\rightarrow\infty}\frac{[1]\prod_{i=1}^{\ell(\nu)} \sum_{j=1}^{N} \mathcal{D}_j(D^N(\theta))^{\nu_i} F_N^D}{N^{\ell(\nu)}(\theta N)^{|\nu|} } = \prod_{i=1}^{\ell(\nu)} \sum_{\pi\in \NCeven(\nu_i)} \prod_{B\in \pi} 2^{|B|-1}|B|c_{(|B|)}.
\]
\item[(h)] It is the case that
\[
\lim_{N\rightarrow\infty}\frac{[1]\prod_{j=1}^N\mathcal{D}_j(D^N(\theta))F_N^D}{\prod_{j=1}^N (1+2(j-1)\theta)} = d_0+1
\]
and for all $\nu\in\even$,
\begin{align*}
& \lim_{N\rightarrow\infty}\frac{[1]\prod_{j=1}^N\mathcal{D}_j(D^N(\theta))\prod_{i=1}^{\ell(\nu)} \sum_{j=1}^{N} \mathcal{D}_j(D^N(\theta))^{\nu_i} F_N^D}{N^{\ell(\nu)}(\theta N)^{|\nu|}\prod_{j=1}^N (1+2(j-1)\theta)} \\ 
& = \prod_{i=1}^{\ell(\nu)} \sum_{\pi\in \NCeven(\nu_i)} \prod_{B\in \pi} 2^{|B|-1}|B|d_{(|B|)}.
\end{align*}
\end{enumerate}
Then, (e) and (g) are equivalent.

Assume that if $N$ is sufficiently large, then $\prod_{j=1}^N (1+2(j-1)\theta)\not=0$. Then, (f) and (h) are equivalent.
\end{theorem}

\begin{proof}
See \Cref{cor:equivalence_a1,cor:equivalence_bc1,cor:equivalence_odd1} for the proofs of (A), (B), and (C), respectively.
\end{proof}

First, we note that we actually prove the generalizations where the $N\rightarrow\infty$ limits of $\frac{c_\lambda(N)}{(\theta N)^{\ell(\lambda)}}$, $\frac{c_\lambda(N)}{(\theta_0N)^{\ell(\lambda)}}$, and $\frac{d_\lambda(N)}{(\theta N)^{\ell(\lambda)}}$ can be nonzero when $\ell(\lambda)>1$, see \Cref{thm:equivalence_a1,thm:equivalence_bc1,thm:equivalence_odd1}. We present this corollary because it is more applicable to the setting where the functions $F_N^A$, $F_N^{BC}$, and $F_N^D$ are set as Bessel generating functions over some sufficiently small neighborhood of the origin where the functions converge and are nonzero so that their logarithms are holomorphic.

For example, in the case where $F_N^A\triangleq J_{a(N)}^{A^{N-1}(\theta)}$, condition (a) implies that for $\nu\in\Gamma$,
\[
\lim_{N\rightarrow\infty} \frac{\prod_{i=1}^{\ell(\nu)} \sum_{j=1}^N \left(\frac{a(N)_j}{\theta N}\right)^{\nu_i}}{N^{\ell(\nu)}} = \prod_{i=1}^{\ell(\nu)} \sum_{\pi\in NC(\nu_i)}\prod_{B\in \pi} |B|c_{(|B|)}.
\]
Then, we can deduce the convergence of the sequence $\{\frac{a(N)}{\theta N}\}_{N\geq 2}$ in terms of moments. Conversely, if the sequence $\{\frac{a(N)}{\theta N}\}_{N\geq 2}$ converges in terms of moments, then we can determine the asymptotics of the coefficients of the Bessel functions $J_{a(N)}^{A^{N-1}(\theta)}$. 

As we mentioned previously, \Cref{thm:main} resolves the conjecture posed in the appendix of the paper \cite{matrix}. The idea of scaling the coefficients by varying powers of $N$ while $\theta$ is fixed that we use is mentioned in the paper. In particular, we generalize the results of \cite{limits_prob_measures}, where the coefficients are only scaled by $(\theta N)^{-1}$. 

As we mention in \Cref{sec:a}, a method to prove part (A) of \Cref{thm:main} is to prove \Cref{thm:leadingorder_a} using the results of \cite{limits_prob_measures}. However, this method is not applicable to parts (B) and (C), so we develop a new approach which can be used in these settings in \Cref{subsec:proofa}. This approach is applicable to the $\theta N\rightarrow c\in\mathbb{C}$ regime as well.

We discuss additional implications of \Cref{thm:main} for when $F_N^A$, $F_N^{BC}$, and $F_N^D$ are the Bessel generating functions of measures that are not point masses. First, we define the notion of convergence in terms of power sums, which implies convergence in terms of moments.

\begin{definition}
\label{def:converge_powersum}
For $N\geq 1$, suppose $\eta_N$ is a Borel probability measure over $\mathbb{C}^N$. Then, $\{\eta_N\}_{N\geq 1}$ converges in terms of power sums if there exists $m_k\in \mathbb{C}$ for $k\geq 1$ such that for all $\nu\in\Gamma$,
\[
\lim_{N\rightarrow\infty}\frac{1}{N^{\ell(\nu)}}\mathbb{E}_{a\sim\eta_N}\left[p_\nu(a)\right] = \prod_{i=1}^{\ell(\nu)} m_{\nu_i}.
\]
If $\eta$ is a Borel probability measure over $\mathbb{C}$ with $k$th moment equal to $m_k$ for $k\geq 1$, then $\{\eta_N\}_{N\geq 1}$ converges to $\eta$ in terms of power sums.
\end{definition}

It is clear that if $\eta_N$ is given by a delta mass at $a_N\in\mathbb{C}^N$ such that $a_N$ converges to $\eta$ in terms of moments, then $\{\eta_N\}_{N\geq 1}$ converges to $\eta$ in terms of power sums. On the other hand, we have that if $\eta_N=\otimes^N \eta$ and $\eta$ has finite moments, then $\eta_N$ converges to $\eta$ in terms of power sums as $N\rightarrow\infty$. This is because for all $\nu\in\Gamma$,
\[
\mathbb{E}_{a\sim\eta_N}\left[p_\nu(a)\right] = \frac{N!\prod_{i=1}^{\ell(\nu)}\mathbb{E}_{a\sim\mu}[a^{\nu_i}]}{(N-\ell(\nu))!} + O(N^{\ell(\nu)-1}),
\]
since $\eta$ has finite moments. Moreover, observe that when $\eta_N$ is distributed as a point mass for all $N\geq 1$, convergence in terms of power sums is equivalent to convergence in terms of moments.

Converging in terms of moments corresponds to requiring that the condition of \Cref{def:converge_powersum} is satisfied when $\nu\in \Gamma$ has length one. Another weaker condition that is often studied is to require that the condition is satisfied when $\nu\in\Gamma$ has length one or consists of two equal parts. This would imply that $\frac{1}{N^{\ell(\nu)}}p_\nu(\eta_N)$ converges to $\prod_{i=1}^{\ell(\nu)} \mathbb{E}_{a\sim \eta} [a^{\nu_i}]$ in probability for all $\nu\in\Gamma$.

Furthermore, if $\eta$ is determined by its moments, then after applying \cite{billingsley}*{Theorem 30.2}, convergence in terms of power sums implies that $\mathbb{E}_{a\sim \eta_N}\left[\frac{1}{N}\sum_{i=1}^N \delta(a_i)\right]$ converges to $\eta$ in distribution as $N\rightarrow\infty$. In particular, Carleman's condition implies that $\eta$ is determined by its moments when it is compactly supported.

Observe that we also consider requiring that the condition of \Cref{def:converge_powersum2} is satisfied when $\nu\in\even$. This is equivalent to requiring that $\{\eta_N(x\mapsto x^2)\}_{N\geq 1}$ converges in terms of power sums. If this is the case, then $\{\frac{1}{2}\eta_N + \frac{1}{2}\eta_N(x\mapsto -x)\}_{N\geq 1}$ converges in terms of power sums, see \Cref{lemma:powersum_even}. We use this implication to prove \Cref{cor:freeconv,cor:freeconv_rectangular,cor:lr_bc,cor:lr_d}.

For a probability measure $\mu_N$ that is exponentially decaying, see \Cref{def:exponent_decay}, we can apply part (A) of \Cref{thm:main} to the Bessel generating function $F_N^A\triangleq G_{\mu_N}^{A^{N-1}(\theta)}$. Afterwards, we deduce that if $\theta$ has nonnegative real part and $|\theta N|\rightarrow \infty$, then $\{\mu_N(x\mapsto \frac{x}{\theta N}\}_{N\geq 2}$ converges in terms of power sums if and only if the normalized coefficients of $\ln(F_N^A)$ converge, see \Cref{cor:lln_satisfaction_a}. Note that $\mu_N\left(x\mapsto \frac{x}{\theta N}\right)$ is the pushforward of $\mu_N$ with respesct to the map $x\mapsto \frac{x}{\theta N}$, see \Cref{subsec:pushforward}, and we previously discuss when $\mu_N=\delta_{a(N)}$. Moreover, we can state similar results for the high temperature regime as well as the type BC and D root systems, see \Cref{cor:lln_satisfaction_bc,cor:lln_satisfaction_d}. For applications of this framework to analyzing the convergence of various probability ensembles, see \Cref{subsec:prob_application}.

\subsection{The high temperature regime}
\label{subsec:regime}
Observe that in \Cref{thm:main}, we consider the following regimes:
\begin{enumerate}
\item[1.]The $N\rightarrow\infty$ limit of $|\theta N|$ is $\infty$ for $A^{N-1}(\theta)$ and $D^N(\theta)$.
\item[2.]The $N\rightarrow\infty$ limit of $|\theta_0 N|$ is $\infty$ and the $N\rightarrow\infty$ limit of $\frac{\theta_1}{\theta_0 N}$ is $c\in\mathbb{C}$ for $BC^N(\theta_0,\theta_1)$.
\end{enumerate}
We consider the following regimes in addition to those listed above:
\begin{enumerate}
\item[3.] The $N\rightarrow\infty$ limit of $\theta N$ is $c\in\mathbb{C}$ for $A^{N-1}(\theta)$ and $D^N(\theta)$.
\item[4.] The $N\rightarrow\infty$ limit of $\theta_0 N$ is $c_0\in\mathbb{C}$ and the $N\rightarrow\infty$ limit of $\theta_1$ is $c_1\in\mathbb{C}$ for $BC^N(\theta_0,\theta_1)$.
\end{enumerate}
The approach that we use to deduce \Cref{thm:main} over regimes 1 and 2 is adaptable to deducing similar results over regimes 3 and 4. Results that we obtain over regimes 3 and 4 include \Cref{thm:equivalence_a2,thm:equivalence_bc2}, respectively, which generalize the results of \cites{matrix,rectangularmatrix}. We also justify \Cref{thm:equivalence_odd2} for regime 3 and the type D root system.

\subsection{Proof method}

We summarize the ideas of the proof of \Cref{thm:main}, which build upon the methods of \cites{matrix,limits_prob_measures}. The most significant difference is that we utilize the power sum basis rather than the monomial basis.

For the proof for the type A root system, the main idea is that for both the low and high temperature regimes, when $\prod_{i=1}^{\ell(\lambda)}\sum_{j=1}^N\mathcal{D}_j(A^{N-1}(\theta))^{\lambda_i}$
is applied to $p_\nu$ for some $\lambda,\nu\in\Gamma$ such that $|\lambda|=|\nu|$, only certain types of sequences of operators contribute to the leading order term of the resulting constant. 

First, observe that the input is in $\mathbb{C}^{H(A^{N-1})}[x_1,\ldots,x_N]$ and that this ring is closed under the operation of $\sum_{j=1}^N \mathcal{D}_j(A^{N-1}(\theta))^k$ for $k\geq 1$. Therefore, by \Cref{lemma:first_op}, the first operator of $\mathcal{D}_j(A^{N-1}(\theta))^{\lambda_i}$ must be $\partial_j$ for $1\leq i\leq \ell(\lambda)$ and $1\leq j\leq N$. In other words, we may replace $\prod_{i=1}^{\ell(\lambda)}\sum_{j=1}^N\mathcal{D}_j(A^{N-1}(\theta))^{\lambda_i}$ with $\prod_{i=1}^{\ell(\lambda)}\sum_{j=1}^N\mathcal{D}_j(A^{N-1}(\theta))^{\lambda_i-1}\partial_j$.

Next, to contribute to the leading order term, the sequence of operators must be of the form $\prod_{i=1}^{\ell(\lambda)} \mathcal{D}_{j_i}(A^{N-1}(\theta))^{\lambda_i-1}\partial_{j_i}$ such that the $j_i$ are distinct for $1\leq i\leq \ell(\lambda)$. If the $j_i$ are not distinct, then the sequence does not contribute to the leading order term. Without loss of generality, assume that $j_i=i$ for $1\leq i\leq \ell(\lambda)$.

Initialize $S$ to be $[\ell(\nu)]$. We first apply the subsequence $\mathcal{D}_1(A^{N-1}(\theta))^{\lambda_1-1}\partial_1$ of operators. The operator $\partial_1$ is applied to $p_{(\nu_j)}$ for some $j\in S=[\ell(\lambda)]$, since $p_\nu$ is symmetric. Afterwards, we remove $j$ from $S$ and replace $p_\nu$ with $\nu_j x_1^{\nu_j-1} \frac{p_{\nu}}{p_{(\nu_j)}}$. Next, we decide how we apply the subsequence $\mathcal{D}_1(A^{N-1}(\theta))^{\lambda_1-1}$ of operators. For each of the operators $\mathcal{D}_1(A^{N-1}(\theta))$, we either diminish the degree of $x_1$ by one or apply $\partial_1$ to $p_{(\nu_j)}$ for some $j\in S$ and remove $j$ from $S$.

The first step is to select which of the derivatives $\partial_1$ we apply to $p_{(\nu_j)}$ for some remaining element $j\in S$ and to assign the values of $j$ to these derivatives. Next, we analyze the operators we apply to diminish the degree of $x_1$.

Recall that 
\[
\mathcal{D}_1(A^{N-1}(\theta))=\partial_1 + \theta\sum_{j\in[N]\backslash\{i\}} \frac{1-s_{1j}}{x_1-x_j}.
\]
To diminish the degree of $x_1$ by one, we either apply $\partial_1$ or the switch $\theta \frac{1-s_{1j}}{x_1-x_j}$ for some $j\in [N]\backslash\{1\}$. In the $|\theta N|\rightarrow\infty$ regime, the derivatives are assigned lower weights than the switches, so they do not contribute to the leading order term. However, in the $\theta N\rightarrow c$ regime, the derivatives contribute to the leading order term.

After we apply the operators $\mathcal{D}_1(A^{N-1}(\theta))^{\lambda_1-1}\partial_1$, we apply the operators $\mathcal{D}_i(A^{N-1}(\theta)\\)^{\lambda_i-1}\partial_i$ in the same way from $i=2$ to $i=\ell(\lambda)$. To contribute to the leading order term, we require the indices $j$ of the switches $\theta \frac{1-s_{ij}}{x_i-x_j}$ that we apply, where $i\in [\ell(\lambda)]$, to be distinct and greater than $\ell(\lambda)$. As a result of this, after applying $\mathcal{D}_1(A^{N-1}(\theta))^{\lambda_1-1}\partial_1$, we require the degree of $x_1$ is zero. Similarly, after applying $\mathcal{D}_i(A^{N-1}(\theta))^{\lambda_i-1}\partial_i$, we require the degree of $x_i$ to be zero for $2\leq i\leq \ell(\lambda)$.

For $1\leq i\leq \ell(\lambda)$, we let $S_i$ denote the set of $j\in [\ell(\nu)]$ such that a derivative $\partial_i$ is assigned to $p_{(\nu_j)}$, so that $S_1\sqcup \cdots \sqcup S_{\ell(\lambda)}= [\ell(\nu)]$. It should be the case that $\sum_{j\in S_i} |\nu_j|=\lambda_i$ for all $i\in [\ell(\lambda)]$, so that the degree of $x_i$ is zero after applying $\mathcal{D}_i(A^{N-1}(\theta))^{\lambda_i-1}\partial_i$. For an example, see \Cref{example:lambda_nu}.

If we condition on the values of the $S_i$, then the contribution to the leading order term associated to each value of $i\in [\ell(\lambda)]$ is independent. In other words, we can deduce that the contribution to the leading order term for each of the choices of $S_i$ is the leading order term of
\[
\prod_{i=1}^{\ell(\lambda)} \left[p_{\lambda_i}, \prod_{j\in S_i} p_{(\nu_j)}\right]_{A^{N-1}(\theta)}.
\]
Afterwards, we compute the leading order term of $\left[p_{\lambda_i}, \prod_{j\in S_i} p_{(\nu_j)}\right]_{A^{N-1}(\theta)}$ by formulating a mapping from sequences of operators to noncrossing partitions for $i\in [\ell(\lambda)]$. Summing over the choices for the $S_i$ gives the leading order term of $[p_\lambda, p_\nu]_{A^{N-1}(\theta)}$. 

For a description of the mapping for the low temperature regime and type A root system, see \Cref{lemma:allocation}. Also, see \Cref{fig:mapping} for an illustration.

After formulating a mapping with noncrossing partitions, we assign each noncrossing partition a weight and sum over the weights to compute the leading order term of $[p_{\lambda_i}, \prod_{j\in S_i} p_{(\nu_j)}]_{A^{N-1}(\theta)}$. The difference between the low and high temperature regimes arises in the weights. As mentioned previously, the switches have a higher weight than the derivatives in the low temperature regime while they have the same weight in the high temperature regime. However, it is not challenging to convert from one weighting of the partitions to the other.

The proof for the type BC root system over the low and high temperature regimes is similar. The most significant differences are that the noncrossing partitions have all even block sizes and we use different weightings of the partitions. For the type D root system, by the results of \cites{heckman,dunkl_singular_poly,demni_type_d}, it suffices to consider the type BC root system.

\begin{example}
\label{example:lambda_nu}

We describe an example for $\lambda=(6,5,3)$ and $\nu=(5,3,2,2,1,1)$, which is displayed in \Cref{fig:lambda_nu}.

\begin{figure}[!ht]
\begin{center}
\begin{tikzpicture}
    \fill[red!30] (0,0) rectangle (6,1);
    \fill[blue!30] (6,0) rectangle (11, 1);
    \fill[green!30] (11, 0) rectangle (14,1);
    \fill[blue!30] (0,-4) rectangle (5, -3);
    \fill[red!30] (5,-4) rectangle (8, -3);
    \fill[green!30] (8, -4) rectangle (10, -3);
    \fill[red!30] (10,-4) rectangle (12, -3);
    \fill[red!30] (12,-4) rectangle (13,-3);
    \fill[green!30] (13,-4) rectangle (14,-3);
    \draw (0,0) grid (14,1);
    \draw (0, -4) grid (14, -3);
    \node at (0.5, 0.5) {$\partial$};
    \node at (3.5, 0.5) {$\partial$};
    \node at (5.5, 0.5) {$\partial$};
    \node at (6.5, 0.5) {$\partial$};
    \node at (11.5, 0.5) {$\partial$};
    \node at (12.5, 0.5) {$\partial$};

    \node at (0.5, -3.5) {$1$};
    \node at (5.5, -3.5) {$2$};
    \node at (8.5, -3.5) {$3$};
    \node at (10.5, -3.5) {$4$};
    \node at (12.5, -3.5) {$5$};
    \node at (13.5, -3.5) {$6$};

    \draw (0.5,0) -- (10.5,-3);
    \draw (3.5,0) -- (5.5,-3);
    \draw (5.5, 0) -- (12.5, -3);
    \draw (6.5, 0) -- (0.5, -3);
    \draw (11.5, 0) -- (8.5, -3);
    \draw (12.5, 0) -- (13.5, -3);
\end{tikzpicture}
\end{center}

\caption{In the first row, the squares that are marked with $\partial$ indicate a derivative while the unmarked squares indicate a switch. Moreover, each color in the first row corresponds to a part of $\lambda=(6,5,3)$. The squares that are colored red, blue, and green correspond to a choice of the operators in $\mathcal{D}_i(A^{N-1}(\theta))^{\lambda_i-1}\partial_i$ from $i=1$ to $i=3$, respectively. In the second row, the square marked with the number $i$ indicates the start of part $i$ of $\nu=(5,3,2,2,1,1)$ so that the first part consists of the first $\nu_1$ squares, the second part consists of the next $\nu_2$ squares, and so forth. A line between a square marked with $\partial$ and a square marked with a number indicates that the derivative that corresponds to the former is assigned to the part of $\nu$ that corresponds to the latter. Moreover, the squares contained in the $i$th part of $\nu$ are colored with the color of the part of $\lambda$ that contains the derivative that is assigned to $i$th part of $\nu$.}
\label{fig:lambda_nu}
\end{figure}

The sequence 
\[
\partial_1 S_{1,4} S_{1,5} \partial_1 S_{1,6} \partial_1 \partial_2 S_{2,7} S_{2,8} S_{2,9} S_{2,10} \partial_3 \partial_3 S_{3,11}
\]
of operators corresponds to the first row, where $S_{i,j}$ indicates the operator $\theta \frac{1-s_{ij}}{x_i-x_j}$ for $i\not=j$. While counting the contribution to the leading order term, recall that we require the indices $j$ of the switches $\theta \frac{1-s_{ij}}{x_i-x_j}$ that we apply, where $i\in [\ell(\lambda)]$, to be distinct and greater than $\ell(\lambda)$. Given this condition, there are multiple choices for the switches $S_{i,j}$ in the sequence of operators that we list previously, although all of these choices are equivalent.

Moreover, if we assign the derivatives with the numbers $1$ through $6$ from left to right, derivatives $1$ through $6$ are assigned to parts $4$, $2$, $5$, $1$, $3$, and $6$ of $\nu$, respectively. We have that $S_1=\{2, 4, 5\}$, $S_2=\{1\}$, and $S_3=\{3, 6\}$ in this case, since the first three derivatives are associated to $\lambda_1$, the next derivative is associated to $\lambda_2$, and the last two are associated to $\lambda_3$. Observe that $\sum_{j\in S_i} \nu_j = \lambda_i$ for all $i \in [3]$. Equivalently, the numbers of squares that are colored with each color are the same in the first and second rows.
\end{example}

\subsection{Probabilistic applications}
\label{subsec:prob_application}

As an example of an application, we prove that if $\theta N\rightarrow \infty$ and the normalized initial distribution of the Dyson Brownian motion (DBM) converges in terms of power sums to $\eta$, then the normalized observation of the DBM at timestamp $\alpha \theta N$ converges in terms of power sums to the free convolution of $\eta$ and the semicircle law multiplied by $\sqrt{\alpha}$, see \Cref{cor:lln_dbm1}. We prove a similar result for the $\theta N\rightarrow c\geq 0$ regime, although we do not specify the limiting measure, see \Cref{cor:lln_dbm2}. For this regime, the time stamp of the observation is fixed as $N\rightarrow\infty$. For the $\beta$-Laguerre ensemble, we also discuss the law of large numbers as $\theta N\rightarrow \infty$ and reprove a result of \cite{betalaguerre}. 

We can deduce similar results for the type BC and type D DBMs using the same method. For the definitions of the DBM associated to any root system and multiplicity function, see \cites{gallardo_yor,skew-product,demni_sde}. We state the relevant results for the type D DBM in the low and high temperature regimes in \Cref{cor:lln_dbm_d1,cor:lln_dbm_d2}, respectively. To state these results, we generalize \Cref{def:converge_powersum} to signed Borel measures, see \Cref{def:converge_powersum2}. 

The paper \cite{airy_beta} also uses Dunkl operators to analyze the type A DBM. The paper shows that multiple observations of the largest particles of the DBM converge to the Airy$_\beta$ line ensemble.

In the following subsection, we describe another probabilistic application of the main results to integral representations of the products of two Bessel functions. The Littlewood-Richardson coefficients for Schur polynomials induce probability measures over the partitions and we can similarly consider integral representations for the products of two Bessel functions, although it is a conjecture that the measure we integrate over is nonnegative for all choices of the multiplicity function, see \Cref{conjecture}. We also discuss representations for the Bessel functions as integrals of exponential functions over the intertwining measures.

\subsection{Integral representations of the Bessel functions}

First, we introduce the intertwining measure, which induces a representation of the Bessel function as an integral of exponential functions. For a finite root system $\mathcal{R}$, $H(\mathcal{R})$ denotes the finite reflection group generated by $r_\alpha$ for $\alpha\in\mathcal{R}$ and for $a\in\mathbb{R}^N$, $H(\mathcal{R})a\triangleq \{ha: h\in H(\mathcal{R})\}$. Moreover, $\theta(\mathcal{R})$ is the set of multiplicity functions over $\mathcal{R}$.

\begin{theorem}[\cite{rosler_positivity}]
\label{thm:positivity}
Suppose $N\geq 1$, $\mathcal{R}\subset\mathbb{R}^N$ is a finite root system, and $\theta\in\theta(\mathcal{R})$ is nonnegative. Suppose $a\in\mathbb{R}^N$. There exists a unique Borel probability measure $\nu^{\mathcal{R}(\theta)}_a$ whose support is contained in the convex hull of $H(\mathcal{R})a $ such that 
\[
E_a^{\mathcal{R}}(x) = \int_{\mathbb{R}^N} e^{\sum_{i=1}^N x_i\epsilon_i} d\nu^{\mathcal{R}(\theta)}_a(\epsilon)
\]
for all $x\in\mathbb{C}^N$. Furthermore, $\text{supp}(\nu_a^{\mathcal{R}(\theta)})\cap H(\mathcal{R})a$ is nonempty and the Borel probability measure 
\[
\nu_a^{\text{sym; }\mathcal{R}(\theta)} \triangleq \frac{1}{|H(\mathcal{R})|}\sum_{h\in H(\mathcal{R})} \nu_{ha}^{\mathcal{R}(\theta)}
\]
is invariant with respect to the action of $H(\mathcal{R})$ and satisfies
\[
J^{\mathcal{R}(\theta)}_a(x) = \int_{\mathbb{R}^N} e^{\sum_{i=1}^N x_i\epsilon_i}d\nu_a^{\text{sym; } \mathcal{R}(\theta)}(\epsilon)
\]
for all $x\in\mathbb{C}^N$. 
\end{theorem}

\begin{remark}
\label{remark:invertible}
From \cites{dunkloperators,dunklbessel}, if $\theta\in\theta(\mathcal{R})$ such that $\theta(r)$ has nonnegative real part for all $r\in\mathcal{R}$, then $\theta\in\Theta(\mathcal{R})$, or equivalently, $\mathcal{D}(\mathcal{R}(\theta))$ is invertible. See \Cref{subsec:dunkl} for the definitions of these notions. The invertibility of the Dunkl operator implies the existence of it associated eigenfunctions, see \Cref{thm:dunkl_eigenfunction,thm:dunkl_eigenfunction_symmetry}. In particular, it implies that the eigenfunctions exist in the context of \Cref{thm:positivity}.
\end{remark}

The intertwining measure $\nu_a^{\mathcal{R}(\theta)}$ is equivalent to the intertwining operator between the Dunkl operator associated to $\mathcal{R}$ and $\theta$ and $\partial$, the derivative operator which is equivalent to setting $\theta=0$. For the definition of an intertwining operator, see \Cref{def:intertwine} and for the explanation of this equivalence, see \Cref{example:dunkl}. We study intertwiners in a more general setting in \Cref{thm:equivalences}.

Another probabilistic application is derived from the interpretation of the Littlewood-Richardson coefficients for the Bessel functions. The paper \cite{representations} shows that if $\lambda(N),\\\mu(N)\in\Gamma_N\cup\{0\}$ and $(\lambda(N)_i+N-i, i\in[N])$ and $(\mu(N)_i+N-i, i\in[N])$ converge in terms of moments after normalization, then the measure over $\Gamma_N$ that is induced by the Littlewood-Richardson coefficients associated to $\lambda(N)$ and $\mu(N)$ also converges in terms of moments after normalization. Since $\lambda(N)$ and $\mu(N)$ are point masses, converging in terms of moments is equivalent to converging in terms of power sums.

We can similarly analyze the product of two Bessel functions, although it is unknown whether the analogues of the Littlewood-Richardson coefficients are nonnegative. This is the content of the following conjecture about the existence of an integral representation of such products; it can easily be extended to the product of an arbitrary number of Bessel functions. The conjecture similarly assumes that the multiplicity function is nonnegative.

\begin{conjecture}
\label{conjecture}
Suppose $N\geq 1$, $\mathcal{R}\subset\mathbb{R}^N$ is a finite root system, and $\theta\in\theta(\mathcal{R})$ is nonnegative. Suppose $a_1,a_2\in\mathbb{R}^N$. There exists a nonnegative Borel probability measure $\mu^{\mathcal{R}(\theta)}_{a_1,a_2}$ over $\mathbb{R}^N$ such that 
\[
J^{\mathcal{R}(\theta)}_{a_1}(x)J^{\mathcal{R}(\theta)}_{a_2}(x) = \int_{\mathbb{R}^N} J^{\mathcal{R}(\theta)}_a(x) d\mu^{\mathcal{R}(\theta)}_{a_1,a_2}(a)
\]
for all $x\in\mathbb{C}^N$.
\end{conjecture}

The measure $\mu_{a_1,a_2}^{\mathcal{R}(\theta)}$ is clearly not unique, by the symmetry of $J_a^{\mathcal{R}(\theta)}(x)$. For simplicity, we may assume that it is symmetric. By \cite{trimeche}, there exists a distribution $\mu^{\mathcal{R}(\theta)}_{a_1,a_2}$ supported over $B(0, \norm{a_1}_2+\norm{a_2}_2)$ such that the equation in \Cref{conjecture} is satisfied. It remains to determine whether the symmetric version of this distribution is nonnegative.

\Cref{conjecture} is known to be true for $A^{N-1}(\theta)$ when $\theta\in\{\frac{1}{2}, 1, 2\}$. The paper \cite{rosler_product} shows that the conjecture holds for additional root systems and multiplicity function as well as in the context of radially symmetric Bessel functions. Moreover, \cite{bessel_convolution} establishes the conjecture for $BC^N(\theta_0,\theta_1)$ when $\theta_0\in \{\frac{1}{2}, 1, 2\}$ and either $\theta_1=\theta_0(p-N+1)$ for an integer $p\geq N$ or $\theta_1 \geq \theta_0 N - \frac{1}{2}$.

In addition to the Schur polynomials, analogues of \Cref{conjecture} have been established for the Hall-Littlewood polynomials, zonal polynomials, quaternionic zonal polynomials, and Schur's $Q$-functions, see \cite{macdonald}. However, the conjecture is false in the context of the nonsymmetric eigenfunctions of the Dunkl operators, see \cite{convolution_maximal}. 

As corollaries of the main results of this paper, we prove the following results about the measures mentioned in \Cref{conjecture} converging in terms of power sums. For the proofs of the corollaries, see \Cref{subsec:conjecture_applications}.

\begin{corollary}
\label{cor:freeconv}
Assume that $\theta\geq 0$ for all $N\geq 2$ and $\lim_{N\rightarrow\infty} \theta N = \infty$. Let $\mu_a$ and $\mu_b$ be Borel probability measures over $\mathbb{R}$ with finite moments. Suppose $a(N),b(N)\in \mathbb{R}^N$ for $N\geq 2$ such that $\delta_{a(N)}(x\mapsto \frac{x}{\theta N})\rightarrow\mu_a$ and $\delta_{b(N)}(x\mapsto \frac{x}{\theta N})\rightarrow\mu_b$ in terms of moments as $N\rightarrow\infty$. Let $\mu$ be the free convolution of $\mu_a$ and $\mu_b$. 

Furthermore, define $\mu_a^-$ and $\mu_b^-$ by $\mu_a^-(B)\triangleq \mu_a(-B)$ and $\mu_b^-(B)\triangleq\mu_b(-B)$, respectively, for all open subsets $B$ of $\mathbb{R}$. Let $\tilde{\mu}$ be the free convolution of $\frac{1}{2}(\mu_a+\mu_a^-)$ and $\frac{1}{2}(\mu_b+\mu_b^-)$.

\begin{enumerate}
\item[(A)] Assume that for all $N\geq 2$, the measure $\mu_{a(N),\,b(N)}^{A^{N-1}(\theta)}$ satisfies the conditions of \Cref{conjecture}. Then, $\mu_{a(N),\,b(N)}^{A^{N-1}(\theta)}(x\mapsto \frac{x}{\theta N})$ converges to $\mu$ as $N\rightarrow\infty$ in terms of power sums. 
\item[(B)] Assume that for all $N\geq 2$, the measure $\mu_{a(N),\,b(N)}^{D^N(\theta)}$ satisfies the conditions of \Cref{conjecture}. Then, $\frac{1}{2}\mu_{a(N),\,b(N)}^{D^N(\theta)}(x\mapsto \frac{x}{\theta N}) + \frac{1}{2}\mu_{a(N),\,b(N)}^{D^N(\theta)}(x\mapsto -\frac{x}{\theta N})$ converges to $\tilde{\mu}$ as $N\rightarrow\infty$ in terms of power sums. 
\end{enumerate}
\end{corollary}

The free convolution is introduced in \cite{voiculescu1} and is defined for pairs of probability measures with finite variance in \cite{maassen} as well as for all pairs of probability measures in \cite{free_convolution_general}. However, the results of this paper require $\mu_a$ and $\mu_b$ to have finite moments. For a discussion of the concepts of free probability, see \cite{nica_speicher_2006}.

We can prove a weaker version of part (A) of \Cref{cor:freeconv} when $\theta\in\{\frac{1}{2},1,2\}$ using the following argument. By \cite{crystalrandommat}*{Proposition 2.3}, the measure $\mu_{a(N),\,b(N)}^{A^{N-1}(\theta)}$ is given by the spectrum of the addition of two random matrices, one with spectrum given by $a(N)$ and the other with spectrum given by $b(N)$. Moreover, when $\theta=\frac{1}{2}$, $1$, and $2$, the matrix is uniformly random and self-adjoint with real, complex, and quaternion entries, respectively. Then, by \cite{voiculescu2}, see also \cite{matrix}*{Theorem 1.1}, we can deduce that the measure weakly converges to the free convolution of $\mu_a$ and $\mu_b$. However, this does not prove convergence in terms of power sums.

For the type BC root system, we deduce convergence to the rectangular free convolution introduced in \cite{rectangular_free_convolution}. 

\begin{corollary}
\label{cor:freeconv_rectangular}
Assume that $\theta_0, \theta_1\geq 0$ for all $N\geq 2$, $\lim_{N\rightarrow\infty} \theta_0 N = \infty$, and $\lim_{N\rightarrow\infty}\frac{\theta_1}{\theta_0 N}\\ = c\in\mathbb{R}_{\geq 0}$. Let $\mu_a$ and $\mu_b$ be Borel probability measures over $\mathbb{R}$ with finite moments. Suppose $a(N),b(N)\in \mathbb{R}^N$ for $N\geq 2$ such that $\delta_{a(N)}(x\mapsto\frac{x}{\theta_0 N})\rightarrow\mu_a$ and $\delta_{b(N)}(x\mapsto\frac{x}{\theta_0N})\rightarrow\mu_b$ in terms of moments as $N\rightarrow\infty$. 

Furthermore, define $\mu_a^-$ and $\mu_b^-$ by $\mu_a^-(B)\triangleq \mu_a(-B)$ and $\mu_b^-(B)\triangleq\mu_b(-B)$, respectively, for all open subsets $B$ of $\mathbb{R}$. Let $\mu$ be the rectangular free convolution with $\lambda$ set as $\frac{1}{1+c}$ of $\frac{1}{2}(\mu_a+\mu_a^-)$ and $\frac{1}{2}(\mu_b+\mu_b^-)$ as defined in \cite{rectangular_free_convolution}*{Proposition-Definition 2.1}. 

Assume that for all $N\geq 2$, the measure $\mu_{a(N),\,b(N)}^{BC^N(\theta_0,\theta_1)}$ satisfies the conditions of \Cref{conjecture}. Then, $\frac{1}{2}\mu_{a(N),\,b(N)}^{BC^N(\theta_0,\theta_1)}(x\mapsto \frac{x}{\theta_0 N}) + \frac{1}{2}\mu_{a(N),\,b(N)}^{BC^N(\theta_0,\theta_1)}(x\mapsto -\frac{x}{\theta_0 N})$ converges to $\mu$ as $N\rightarrow\infty$ in terms of power sums.
\end{corollary}

We also note that analogues of part (A) of \Cref{cor:freeconv} and \Cref{cor:freeconv_rectangular} have been established in the $\theta N \rightarrow c\geq 0$ and $\theta_0 N\rightarrow c_0\geq 0, \theta_1\rightarrow c_1\geq 0$ regimes, respectively, see \cites{matrix,rectangularmatrix}. We can similarly establish the analogue of part (B) of \Cref{cor:freeconv} in the $\theta N \rightarrow c\geq 0$ regime; its statement is almost the same as that of the analogue of part (A) of the corollary. As expected, we can recover the free convolution from the convolution for the $\theta N \rightarrow c\geq 0$ regime by taking $c\rightarrow\infty$, see \cite{matrix}*{Theorem 8.7}.

Furthermore, we discuss generalizations of \Cref{cor:freeconv,cor:freeconv_rectangular} in \Cref{subsec:expdecay} when $a(N)$ and $b(N)$ are sampled from exponentially decaying measures rather than chosen deterministically. In particular, we use this framework to analyze the DBM with random initial values in \Cref{example:dbm}. In this case, $a(N)$ would be sampled from the Hermite $\theta$-ensemble while $b(N)$ would be sampled from the initial distribution of the DBM.

A difference between the statements of these corollaries and those of the results of \cite{representations} arises from the fact that it is unknown whether the measure $\mu_{a,b}^{\mathcal{\mathcal{R}(\theta)}}$ from \Cref{conjecture} is nonnegative. Due to this, we must assume that such a measure exists in the corollaries.

In addition to \cite{representations}, the papers \cites{llnclt,jackgenfunc} discuss the convergence of the measure formed by the Littlewood-Richardson coefficients for the product of two Jack generating functions in the low and high temperature regimes. Similarly to the setting of Bessel functions, it is a conjecture that these coefficients are nonnegative for all values of the Jack parameter $\alpha\triangleq \theta^{-1}$, where Schur polynomials correspond to $\alpha=1$.

\subsection{Uniform convergence of the Bessel functions}

For a fixed value of $\lambda\in\Gamma$, we can compute the asymptotics of the coefficients of $p_\lambda(x)$ in the Taylor series expansions of $J_a^{A^{N-1}(\theta)}(x)$, $J_a^{BC^N(\theta_0,\theta_1)}(x)$, and $J_a^{D^N(\theta)}(x)$, which are homogeneous polynomials of degree $|\lambda|$ in $a\in\mathbb{C}^N$. These computations are included in \Cref{sec:bessel_coeff}. 

Assume that the sequence $\{a(N)\}_{N\geq 2}$ satisfies the property that $\lim_{N\rightarrow\infty} \frac{\sum_{i=1}^N a(N)_i^k}{N^k}$ exists for all $k\in\mathbb{N}$, where $a(N)\in\mathbb{C}^N$ for all $N\geq 2$. Such sequences are also referred to as Vershik-Kerov sequences and have been studied in \cites{assiotis,rank_infinity}. In the case where $\theta\in\mathbb{C}^\times$ is fixed or $\theta_0\in\mathbb{C}^\times$ is fixed and $\lim_{N\rightarrow\infty} \frac{\theta_1}{\theta_0 N}=c\in\mathbb{C}\backslash\{-1\}$, we can compute the asymptotic coefficients of $J_{a(N)}^{A^{N-1}(\theta)}(x)$ and $J_{a(N)}^{BC^N(\theta_0,\theta_1)}(x)$ and the asymptotic coefficients of the terms of $J_{a(N)}^{D^N(\theta)}(x)$ with all even degrees as $N\rightarrow\infty$. For the type BC and D Bessel functions, we only require $\lim_{N\rightarrow\infty} \frac{\sum_{i=1}^N a(N)_i^k}{N^k}$ to converge for all even $k\in\mathbb{N}$. See Theorems \ref{thm:vk_a}, \ref{thm:vk_bc}, and \ref{thm:vk_d} for these computations.

These computations recover the results of \cite{assiotis} for the type A Bessel function and \cite{rank_infinity} for the type A and BC Bessel functions. A similar setting is considered for Jack symmetric polynomials in \cite{okounkov_olshanki}.

If we are given that $\lim_{N\rightarrow\infty} \frac{e(a(N))}{\prod_{i=1}^N (1+2(i-1)\theta)}$ converges and $\lim_{N\rightarrow\infty} \frac{\sum_{i=1}^N a(N)_i^{2k}}{ N^{2k}}$ converges for all $k\in\mathbb{N}$, then we can compute the asymptotic coefficients of the terms of $J_{a(N)}^{D^N(\theta)}(x)$ with all odd degrees as $N\rightarrow\infty$. See \Cref{thm:vk_d} for these computations.

Furthermore, in the case where $\lim_{N\rightarrow\infty} \theta N = c\in\mathbb{C}$ or $\lim_{N\rightarrow\infty} \theta_0 N = c_0\in\mathbb{C}$ and $\lim_{N\rightarrow\infty} \theta_1 = c_1\in\mathbb{C}$, we can compute the asymptotic coefficients when $\lim_{N\rightarrow\infty} \frac{\sum_{i=1}^N a(N)_i^k}{N}$ exists for all $k\in\mathbb{N}$. This setting is more general than that of the Vershik-Kerov sequences although $\theta$ is not fixed. Additionally, we require that $c_0$ is not a negative integer and $2c_0+2c_1$ is not a negative odd integer.

When $|\theta N|\rightarrow\infty$ and $\lim_{N\rightarrow\infty} \frac{\sum_{i=1}^N a(N)_i^k}{N(\theta N)^k}$ exists for all $k\in\mathbb{N}$, the coefficients will no longer converge after normalizing by a single factor; for example, in part (A) of \Cref{thm:main}, $c_\lambda(N)$ has order $(\theta N)^{\ell(\lambda)}$ for $\lambda\in\Gamma$. In this case, we can still approximate the coefficients.

However, even if the coefficients of the sequence of Bessel functions converge, we have not yet determined that the sequence of functions uniformly converges. In order to prove uniform convergence over compact subsets of an open and simply connected domain, we follow the argument of \cite{rank_infinity} and first prove that the Bessel functions are uniformly bounded and then apply Montel's theorem, see \Cref{subsec:uniformconverge,subsec:vk}.

To prove that the Bessel functions are bounded, we assume that $\theta,\theta_0,\theta_1\in\mathbb{R}_{\geq 0}$ so that we can apply \Cref{thm:positivity}. An interesting direction for future research is to generalize these arguments to whenever $\theta,\theta_0,\theta_1\in\mathbb{C}$ such that the corresponding Bessel functions exist.

\subsection[A graded ring of operators]{Applying a graded ring of operators to a graded vector field}

In Sections \ref{sec:intro_graded}, \ref{sec:invertible}, and \ref{sec:bijection}, we discuss the applications of a graded ring of operators to a graded vector field. The paper \cite{dunkl_singular_poly} discusses the applications of operators to a graded vector field. We extend this notion by considering a graded ring of operators. The results that we obtain are relevant to Dunkl operators and in particular the Dunkl bilinear form, see \Cref{example:dunkl}; for the definition of the Dunkl bilinear form, see \Cref{subsec:dunkl}. 

We also study the notion of invertible operators, which is the focus of \cite{dunkl_singular_poly}; see \Cref{def:invertible} for the definition of invertibility. In particular, we prove \Cref{thm:equivalences}, which establishes many equivalent conditions for the invertibility of an operator. By the equivalence of conditions (a) and (i) of the theorem, for a fixed multiplicity function, the Bessel function exists if and only if the Dunkl operator is invertible. Due to this, the content of \Cref{sec:invertible} is useful for analyzing the Bessel functions. In particular, in \Cref{sec:bessel_coeff}, we obtain the asymptotic coefficients of the Bessel functions by computing the asymptotics of the inverses of the matrices that store the values of the Dunkl bilinear form. Moreover, in \Cref{sec:bijection}, we justify the existence of invertible operators after certain conditions are satisfied by presenting a bijection between them and sequences of invertible matrices.

\subsection[Non-exponential holomorphic functions]{Equivalence conditions for non-exponential holomorphic functions}

An exponential formal power series is a formal power series $\exp(F_N)$ for some $F_N\in\mathbb{C}[[x_1,\ldots,x_N\\]]$. Analyzing exponential formal power series is useful since the Bessel function $J_a^{\mathcal{R}(\theta)}(x)$ equals one when evaluated at $x=0$. Due to this, the Bessel generating function $\E_{a\sim\mu}[J_a^{\mathcal{R}(\theta)}(x)]$ can be expressed as an exponential formal power series in a neighborhood of the origin that it converges in, assuming that the neighborhood exists.

We have established equivalence conditions for exponential formal power series. However, it is also straightforward to deduce analogous results for non-exponential formal power series, which we describe in the following list.
\begin{itemize}
\item In \Cref{thm:equivalence_a1,thm:equivalence_a2}, we set $F_N(x_1,\ldots,x_N)=\sum_{\lambda\in\Gamma_N} c_\lambda(N)p_\lambda$ and replace $\exp\left(\sum_{\gamma\in\Gamma}c_\gamma p_\gamma\right)$ with $\sum_{\gamma\in\Gamma}c_\gamma p_\gamma$.
\item In \Cref{thm:equivalence_bc1,thm:equivalence_bc2}, we set $F_N(x_1,\ldots,x_N)=\sum_{\lambda\in\evenN} c_\lambda(N)p_\lambda$ and replace $\exp\left(\sum_{\gamma\in\even}c_\gamma p_\gamma\right)$ with $\sum_{\gamma\in\even}c_\gamma p_\gamma$.
\item In \Cref{thm:equivalence_odd1,thm:equivalence_odd2}, we set $F_N(x_1,\ldots,x_N)=\sum_{\lambda\in\evenN} (c_\lambda(N)+ed_\lambda(N))p_\lambda$, replace $\exp\left(\sum_{\gamma\in\even}c_\gamma p_\gamma\right)$ with $\sum_{\gamma\in\even}c_\gamma p_\gamma$, and replace $\exp\left(\sum_{\gamma\in\even}d_\gamma p_\gamma\right)$ with $\sum_{\gamma\in\even}d_\gamma p_\gamma$.
\end{itemize}

\subsection{Paper organization}
In \Cref{sec:def}, we define the Dunkl operators and notation regarding partitions and noncrossing partitions. In \Cref{sec:intro_graded}, we introduce the setting of applying a graded ring of operators to a graded vector field and in \Cref{sec:invertible}, we define the notion of an invertible graded ring of operators. Afterwards, in \Cref{sec:bijection}, we discuss a representation of invertible graded rings of operators as sequences of invertible matrices and connect the framework introduced in \Cref{sec:intro_graded} to the Dunkl operators. In Sections \ref{sec:a}, \ref{sec:bc}, and \ref{sec:d}, we discuss the leading order terms of the Dunkl bilinear form and prove \Cref{thm:main} for the $A^{N-1}$, $BC^N$, and $D^N$ root systems, respectively. Following this, in \Cref{sec:bessel_coeff}, we determine the asymptotics of the coefficients of the terms of the Bessel functions that are homogeneous with a fixed degree. In \Cref{sec:applications}, we discuss applications of the results of this paper. Following this, in \Cref{sec:combinatorics}, we present combinatorial expressions for the Dunkl bilinear form.

\subsection{Acknowledgments}

The author would like to thank Vadim Gorin for giving comments on the paper.

\section{Basic definitions and notation}
\label{sec:def}

\subsection{Dunkl operators}
\label{subsec:dunkl}

Suppose $N\geq 1$ and that $\mathcal{R}\subset\mathbb{R}^N$ is a finite root system. For $\alpha\in\mathcal{R}$, we let $r_\alpha$ denote the reflection $r_\alpha: x\mapsto x-2\product{x,\alpha}\norm{\alpha}_2^{-2}\alpha$. Let $H(\mathcal{R})$ be the finite reflection group generated by $r_\alpha$ for $\alpha\in\mathcal{R}$. For a function $f$ over $\mathbb{C}^N$, we define the action of $h\in H(\mathcal{R})$ over $f$ by $hf(x)\triangleq f(hx)$.

Let $\mathcal{R}^+$ be a set of positive roots in $\mathcal{R}$. Furthermore, let $\theta(\mathcal{R})$ be the set of multiplicity functions $\theta:\mathcal{R}\rightarrow \mathbb{C}$ such that $\theta(\alpha_1)=\theta(\alpha_2)$ for all $\alpha_1,\alpha_2\in\mathcal{R}$ such that $r_{\alpha_1}$ and $r_{\alpha_2}$ are conjugates in $H$. When we write $\mathcal{R}$, we assume that $\mathcal{R}$ is a finite root system and when we write $\mathcal{R}(\theta)$ to denote a root system and a multiplicity function, it is implicit that $\theta\in \theta(\mathcal{R})$. Unless stated otherwise, $N$ is a positive integer and $\mathcal{R}\subset\mathbb{R}^N$.

Next, we define the \textit{Dunkl operator} introduced in \cite{dunkloperators}. For $u\in\mathbb{R}^N$, define the operator $\mathcal{D}_u(\mathcal{R}(\theta))$ over the ring $\mathbb{C}[[x_1,\ldots,x_N]]$ of complex formal power series with variables $x_1,\ldots,x_N$ by $\mathcal{D}_u(\mathcal{R}(\theta)): f\mapsto \product{\nabla_\theta f, u}$, where 
\[
\nabla_\theta f (x)\triangleq \nabla f(x) + \sum_{\alpha\in\mathcal{R}^+} \theta(\alpha)\frac{f(x)-f(r_\alpha x)}{\product{x,\alpha}}\alpha.
\]
The definition of the Dunkl operator does not depend on the choice of $\mathcal{R}^+$, see \cite{dunklbound}*{Remark 2.4}. Additionally, it is well known that the Dunkl operators are commutative, which is stated in the following lemma.

\begin{lemma}[\cite{dunkloperators}]
\label{lemma:commutative}
For $u_1,u_2\in\mathbb{R}^N$, $\mathcal{D}(\mathcal{R}(\theta))_{u_1}\mathcal{D}(\mathcal{R}(\theta))_{u_2}=\mathcal{D}(\mathcal{R}(\theta))_{u_2}\mathcal{D}(\mathcal{R}(\theta))_{u_1}$ . 
\end{lemma}

For $1\leq i\leq N$, define $\mathcal{D}_i(\mathcal{R}(\theta))\triangleq \mathcal{D}_{[\mathbf{1}\{i=j\}]_{1\leq j\leq N}^T}(\mathcal{R}(\theta))$. Furthermore, for $f\in\mathbb{C}[x_1,\ldots,x_N]$, we define $\mathcal{D}(\mathcal{R}(\theta))(f)$ to be the operator $f(\mathcal{D}_1,\ldots,\mathcal{D}_N)$; note that this operator is well defined by \Cref{lemma:commutative}. The following lemma is also well known.

\begin{lemma}[\cite{dunkloperators}]
\label{lemma:equivariance}
Suppose $h\in H(\mathcal{R})$. Then, for all $f\in\mathbb{C}[x_1,\ldots,x_N]$, $h\mathcal{D}(\mathcal{R}(\theta))\\(f)h^{-1}=\mathcal{D}(\mathcal{R}(\theta))(hf)$.
\end{lemma}

By the previous lemma, we have that for all $h\in H(\mathcal{R})$, $f\in\mathbb{C}[x_1,\ldots,x_N]$, and $g\in\mathbb{C}[[x_1,\ldots,x_N]]$, $h\mathcal{D}(\mathcal{R}(\theta))(f)g = \mathcal{D}(\mathcal{R}(\theta))(hf)hg$. We use this result later in the paper, for example to prove that a function exhibits symmetries after applications of Dunkl operators.

Furthermore, we have that $\mathcal{D}$ defines a bilinear form which is introduced in \cite{dunkl_integralkernel}. For $f,g\in\mathbb{C}[x_1,\ldots,x_N]$, the \textit{Dunkl bilinear form} is defined as
\[
[f,g]_{\mathcal{R}(\theta)} \triangleq [1]\mathcal{D}(\mathcal{R}(\theta))(f)g.
\]

\begin{theorem}[\cite{dunkl_integralkernel}]
For all $f,g\in\mathbb{C}[x_1,\ldots,x_N]$, $[f,g]_{\mathcal{R}(\theta)} = [g,f]_{\mathcal{R}(\theta)}$. 
\end{theorem}

Due to the symmetry of $[\cdot,\cdot]_{\mathcal{R}(\theta)}$, it is straightforward to define and compute the values of $[f,g]_{\mathcal{R}(\theta)}$ and $[g,f]_{\mathcal{R}(\theta)}$ when $f\in\mathbb{C}[x_1,\ldots,x_N]$ and $g\in\mathbb{C}[[x_1,\ldots,x_N]]$. When $f,g\in\mathbb{C}[[x_1,\ldots,x_N]]$, the value of $[f,g]_{\mathcal{R}(\theta)}$ does not necessarily converge.

\begin{definition}
\label{def:dunkl_invertible}
The function $\mathcal{D}(\mathcal{R}(\theta))$ is \textit{invertible} if for all $k\geq 1$, there does not exist $f\in\mathbb{C}[x_1,\ldots,x_N]$ such that $f$ is homogeneous of degree $k$ and $[f,g]_{\mathcal{R}(\theta)}=0$ for all $g\in\mathbb{C}[x_1,\ldots,x_N]$ that is homogeneous of degree $k$. If $\mathcal{D}(\mathcal{R}(\theta))$ is not invertible, then it is \textit{singular}. Let $\Theta(\mathcal{R})$ be the set of $\theta\in\theta(\mathcal{R})$ such that $\mathcal{D}(\mathcal{R}(\theta))$ is invertible.
\end{definition}

Given $\mathcal{R}$, the paper \cite{dunkl_singular_poly} computes all $\theta\in\theta(\mathcal{R})$ such that $\mathcal{D}(\mathcal{R}(\theta))$ is invertible. Note that the statements of the definition of invertibility in this paper and \cite{dunkl_singular_poly} are not the same. However, the definitions are equivalent, see \Cref{thm:equivalences} where we state equivalent conditions for invertibility in the more general setting of applying a graded ring of operators to a graded vector field. 

Moreover, in \Cref{sec:bijection}, we state a bijection between invertible operators and sequences of invertible matrices over the base field. In particular, in \Cref{cor:bijection2}, we show that an invertible operator exists given that the base field commutes with elements of the ring and a sequence of invertible matrices exists.

In this paper, we focus on the asymptotics of $[\cdot,\cdot]_{\mathcal{R}(\theta)}$ for the root systems $A_{N-1}$, $B_N$, $C_N$, and $D_N$ as the number of variables $N$ increases to infinity. We use these asymptotics to determine the asymptotics of the eigenfunctions of $\mathcal{D}(\mathcal{R}(\theta))$, which we define next.

\begin{theorem}[\cite{dunklbessel}]
\label{thm:dunkl_eigenfunction}
Suppose $\theta\in\Theta(\mathcal{R})$. Then, there exists a unique function $E_a^{\mathcal{R}(\theta)}(x)$ that is holomorphic over the domain $\mathbb{C}^N\times\mathbb{C}^N$ for $(a,x)$ and satisfies
\[
\begin{cases}
\mathcal{D}(\mathcal{R}(\theta))(f) E_a^{\mathcal{R}(\theta)}(x) = f(a) E_a^{\mathcal{R}(\theta)}(x) & \forall f\in\mathbb{C}[x_1,\ldots,x_N], \\
E_a^{\mathcal{R}(\theta)}(0)=1.
\end{cases}
\]
Furthermore, $E_a^{\mathcal{R}(\theta)}(x)$ is holomorphic over the domain $\mathbb{C}^N\times \mathbb{C}^N\times \Theta(\mathcal{R})$ for $(a,x,\theta)$.
\end{theorem}

The following lemma contains some properties about the eigenfunction $E_a^{\mathcal{R}(\theta)}(x)$. See \Cref{lemma:besselupper_1} for an upper bound on the symmetric eigenfunction that is similar to part (e) of the lemma.

\begin{lemma}[\cite{dunklbound}]
\label{lemma:dunkl_properties}
Suppose $\theta\in\Theta(\mathcal{R})$.
\begin{enumerate}
\item[(a)] For all $h\in H(\mathcal{R})$, $E_{ha}^{\mathcal{R}(\theta)}(hx)=E_a^{\mathcal{R}(\theta)}(x)$.
\item[(b)] $E_a^{\mathcal{R}(\theta)}(x)=E_x^{\mathcal{R}(\theta)}(a)$.
\item[(c)] Suppose $c\in\mathbb{C}$. Then, $E_a^{\mathcal{R}(\theta)}(cx)=E_{ca}^{\mathcal{R}(\theta)}(x)$.
\item[(d)] $\overline{E_a^{\mathcal{R}(\theta)}(x)} = E_{\overline{a}}^{\mathcal{R}(\overline{\theta})}(\overline{x})$.
\item[(e)] If $\text{Re}(\theta(r))\geq 0$ for all $r\in\mathcal{R}$, then $|E_a^{\mathcal{R}(\theta)}(x)| \leq \sqrt{|H(\mathcal{R})|}\exp(\max_{h\in H(\mathcal{R})} \\\text{Re}(\product{ha, x}))$.
\end{enumerate}
\end{lemma}

We can also consider the symmetric analogue of \Cref{thm:dunkl_eigenfunction} after averaging over $H(\mathcal{R})$. First, we define $\mathbb{C}^{H(\mathcal{R})}[x_1,\ldots,x_N]$ to be the set of $f\in\mathbb{C}[x_1,\ldots,x_N]$ that are fixed under the action of $H(\mathcal{R})$. The following lemma is basic.

\begin{lemma}
\label{lemma:first_op}
Suppose $f\in\mathbb{C}^{H(\mathcal{R})}[x_1,\ldots,x_N]$. Then, for $i\in [N]$, $\mathcal{D}_i(\mathcal{R}(\theta))f=\partial_i f$.
\end{lemma}

The next result is the symmetric analogue of \Cref{thm:dunkl_eigenfunction}. See \cite{dunkl_singular_poly} for elaboration on the proof of the result.

\begin{theorem}[\cite{dunklbessel}]
\label{thm:dunkl_eigenfunction_symmetry}
Suppose $\theta\in\Theta(\mathcal{R})$. Then, there exists a unique function $J_a^{\mathcal{R}(\theta)}(x)$ that is holomorphic over the domain $\mathbb{C}^N\times\mathbb{C}^N$ for $(a,x)$ and satisfies
\[
\begin{cases}
\mathcal{D}(\mathcal{R}(\theta))(f) J_a^{\mathcal{R}(\theta)}(x) = f(a) J_a^{\mathcal{R}(\theta)}(x) & \forall f\in\mathbb{C}^{H(\mathcal{R})}[x_1,\ldots,x_N], \\
J_a^{\mathcal{R}(\theta)}(0)=1.
\end{cases}
\]
Furthermore, $J_a^{\mathcal{R}(\theta)}(x)$ is holomorphic over the domain $\mathbb{C}^N\times \mathbb{C}^N\times \Theta(\mathcal{R})$ for $(a,x,\theta)$ and 
\[
J_a^{\mathcal{R}(\theta)}(x) = \frac{1}{|H(\mathcal{R})|}\sum_{h\in H(\mathcal{R})} h E_a^{\mathcal{R}(\theta)}(x).
\]
\end{theorem}

A generalization of these two results is included in \Cref{thm:equivalences}, which concerns applying a graded ring of operators to a graded vector space.

Furthermore, we let $\mathcal{D}_H(\mathcal{R}(\theta))$ denote the function such that for $f\in \mathbb{C}^{H(\mathcal{R})}[x_1,\ldots,x_N]$, $\mathcal{D}_H(\mathcal{R}(\theta))(f)$ is the restriction of $\mathcal{D}(\mathcal{R}(\theta))(f)$ to $\mathbb{C}^{H(\mathcal{R})}[[x_1,\ldots,x_N]]$, which is the set of $f\in\mathbb{C}[[x_1,\ldots,x_N]]$ that are fixed under the action of $H(\mathcal{R})$. Then, $\mathcal{D}_H(\mathcal{R}(\theta))$ is a symmetric version of $\mathcal{D}(\mathcal{R}(\theta))$.

Furthermore, for $k\geq 1$, we let $E_a^{\mathcal{R}(\theta)}[k](x)$ and $J_a^{\mathcal{R}(\theta)}[k](x)$ denote the sums of the terms of the Taylor expansions of $E_a^{\mathcal{R}(\theta)}(x)$ and $J_a^{\mathcal{R}(\theta)}(x)$, respectively, that are homogeneous of degree $k$ in $x$. This notation is used in \Cref{sec:bessel_coeff}.

We mention that if it is clear what root system and multiplicity function we are considering, then we often do not include $\mathcal{R}(\theta)$ in the notation. For example, if this is the case then we would denote $\mathcal{D}_i(\mathcal{R}(\theta))$ by $\mathcal{D}_i$ and $\mathcal{D}(\mathcal{R}(\theta))$ by $\mathcal{D}$.

For $i\in [N]$, we let $d_i$ denote the operator that lowers the degree in $x_i$ by one. That is, $d_i$ maps $x_i^k$ to $x_i^{k-1}$ for $k\geq 1$ and $1$ to zero.

As discussed earlier, we focus on the irreducible root systems $A^{N-1}$, $B^N$, $C^N$, and $D^N$, which are subsets of $\mathbb{R}^N$ for $N\geq 2$. We define these root systems.

\textbf{The definition of $A^{N-1}$.} For $i\in [N]$, we define $e_i\triangleq [\mathbf{1}\{i=j\}]_{j\in [N]}^T\in\mathbb{R}^N$. Let $A^{N-1}\triangleq \{e_i-e_j: i,j\in[N],\,i\not=j\}$. Each $\theta\in \theta(A^{N-1})$ is constant over the root system, so we let $A^{N-1}(\theta)$ for $\theta\in\mathbb{C}$ denote the choice of $A^{N-1}$ as the root system and $\theta$ as the multiplicity function. 

The reflection group $H(A^{N-1})$ permutes the entries of $\mathbb{C}^N$. For $i\in [N]$, the associated Dunkl operator is 
\[
\mathcal{D}_i(A^{N-1}(\theta)) \triangleq \partial_i + \theta\sum_{j\in[N]\backslash\{i\}} \frac{1-s_{ij}}{x_i-x_j},
\]
where $s_{ij}$ switches the $i$th and $j$th entries of an element of $\mathbb{C}^N$ for distinct $i,j\in[N]$.

Furthermore, $\mathbb{C}^{H(A^{N-1})}[x_1,\ldots,x_N]$ is the set of symmetric functions in $\mathbb{C}[x_1,\ldots,x_N]$. Equivalently, it is the span of $\{1\}\cup \{p_\lambda: \lambda\in \Gamma_N\}$ and $\{1\}\cup \{M_\epsilon: \epsilon\in\Gamma_N\}$.

\textbf{The definitions of $B^N$ and $C^N$.} Let $B^N\triangleq \bigcup_{i,j\in[N],\,i<j} \{e_i-e_j,e_j-e_i,e_i+e_j,-e_i-e_j\}\bigcup_{i\in[N]}\{e_i,-e_i\}$ and $C^N\triangleq \bigcup_{i,j\in[N],\,i<j} \{e_i-e_j,e_j-e_i,e_i+e_j,-e_i-e_j\}\bigcup_{i\in[N]}\{2e_i,-2e_i\}$. A multiplicity function $\theta\in \theta(B^N)$ is constant over the roots of length $\sqrt{2}$ and over the roots of length $1$. Similarly, a multiplicity function $\theta\in \theta(C^N)$ is constant over the roots of length $\sqrt{2}$ and over the roots of length $2$.

We have that $B^N$ and $C^N$ are dual root systems such that $H(B^N)=H(C^N)$. Furthermore, we always have that $\mathcal{D}(B^N(\theta))=\mathcal{D}(C^N(\theta))$. Since we do not need to differentiate between these two root systems in this paper, we let $BC^N$ denote the root system $B^N$ or $C^N$. Furthermore, for $\theta_0,\theta_1\in\mathbb{C}$, we let $BC^N(\theta_0,\theta_1)$ denote the choice of $B^N$ or $C^N$ as the root system and the function that assigns $\theta_1$ to the scalar multiplies of $e_i$ and $\theta_0$ to the remaining roots as the multiplicity function.

We have that $H(BC^N)$ permutes the entries of $\mathbb{C}^N$ and applies sign flips to any number of entries. Furthermore, for $i\in[N]$, the associated Dunkl operator is 
\[
\mathcal{D}_i(BC^N(\theta_0,\theta_1)) \triangleq \partial_i + \theta_1 \frac{1-\tau_i}{x_i} + \theta_0\sum_{j\in[N]\backslash\{i\}} \left(\frac{1-s_{ij}}{x_i-x_j} + \frac{1-\tau_i\tau_j s_{ij}}{x_i+x_j}\right),
\]
where $\tau_i$ flips the sign of the $i$th entry of an element of $\mathbb{C}^N$ for $i\in [N]$. When we are working in the context of type BC root systems, we refer to $\theta_0\left(\frac{1-s_{ij}}{x_i-x_j} + \frac{1-\tau_i\tau_j s_{ij}}{x_i+x_j}\right)$ for distinct $i,j\in [N]$ as a type 0 switch and $\theta_1\frac{1-\tau_i}{x_i}$ for $i\in [N]$ as a type 1 switch. 

Furthermore, $\mathbb{C}^{H(BC^N)}[x_1,\ldots,x_N]$ is the set of symmetric functions in $\mathbb{C}[x_1,\ldots,x_N]$ that have all even degrees. Equivalently, it is the span of $\{1\}\cup \{p_\lambda: \lambda\in \evenN\}$.

\textbf{The definition of $D^N$.} Let $D^N\triangleq \bigcup_{i,j\in[N],i<j} \{e_i-e_j,e_j-e_i,e_i+e_j,-e_i-e_j\}$. A multiplicity function $\theta\in \theta(D^N)$ is constant over the root system, so we let $D^N(\theta)$ for $\theta\in\mathbb{C}$ denote the choice of $D^N$ as the root system and $\theta$ as the multiplicity function.

We have that $H(D^N)$ permutes the entries of $\mathbb{C}^N$ and applies sign flips to an even number of entries. Furthermore, for $i\in[N]$, the associated Dunkl operator is 
\[
\mathcal{D}_i(D^N(\theta)) \triangleq \partial_i + \theta\sum_{j\in[N]\backslash\{i\}} \left(\frac{1-s_{ij}}{x_i-x_j} + \frac{1-\tau_i\tau_j s_{ij}}{x_i+x_j}\right).
\]

Additionally, $\mathbb{C}^{H(D^N)}[x_1,\ldots,x_N]$ is the set of symmetric functions in $\mathbb{C}[x_1,\ldots,x_N]$ that are sums of monomials that have all degrees of the same parity. Equivalently, it is the span of $\{1\}\cup \{p_\lambda: \lambda\in \evenN\}\cup \{ep_\lambda:\lambda\in\evenN\}$.

\subsection{Partitions}
\label{subsec:partitions}

Suppose $N\geq 1$. Let $\Gamma_N$ denote the set of nonempty partitions with at most $N$ parts. Furthermore, define $\Gamma\triangleq\bigcup_{N\geq 1}\Gamma_N$. Note that we do not assume that $\Gamma$ contains the empty partition.

For $\lambda=(\lambda_1\geq \cdots\geq\lambda_m)\in\Gamma$, let $|\lambda|\triangleq \sum_{i=1}^m \lambda_i$ and $\ell(\lambda)\triangleq m$. Also, we define $\evenN$ (resp. $\even$) to be the set of $\lambda\in\Gamma_N$ (resp. $\Gamma$) such that $\lambda_i$ is even for all $i\in [\ell(\lambda)]$.

Suppose $k\geq 1$. Define $\Gamma_N[k]\triangleq \{\lambda\in\Gamma_N:|\lambda|=k\}$ and define $\evenN[k]$, $\Gamma[k]$, and $\even[k]$ analogously. Furthermore, for a set $S$ and $M\in S^{\Gamma\times\Gamma}$ (resp. $S^{\even\times\even}$), we define $M[k] \in S^{\Gamma[k]\times\Gamma[k]}$ (resp. $S^{\even[k]\times\even[k]}$) to be $M$ with rows and columns restricted to $\Gamma[k]$ (resp. $\even[k]$).

For a positive integer $m$ and a sequence $s=(a_1,\ldots,a_m)$ of nonnegative integers, define $\gamma(s)$ to be the element of $\Gamma$ that is a permutation of the sequence formed from $s$ after deleting the entries that equal zero. Furthermore, for $\nu\in \Gamma$, define $\pi(\nu)$ to be $\ell(\nu)!$ divided by the number of permutations of $\nu$ and $\nu!\triangleq\prod_{i=1}^{\ell(\nu)}\nu_i!$; if $\nu$ contains $n_i$ copies of $i$ for all $i\geq 1$, then $\pi(\nu)=\prod_{i\geq 1} n_i!$. We similarly define $\pi(s)\triangleq \pi(\gamma(s))$ and $s!\triangleq \gamma(s)!$. For $x\in\mathbb{C}^m$, we define $x^s\triangleq \prod_{i=1}^m x_i^{a_i}$.

We also consider the sums of partitions. For $\lambda_1,\ldots,\lambda_k\in\Gamma$, we define $\lambda_1+\cdots+\lambda_k\triangleq\gamma((\lambda_1,\ldots,\lambda_k))$ for all $k\geq 2$, where $(\lambda_1,\ldots,\lambda_k)$ denotes the tuple formed by combining the entries of $\lambda_1,\ldots,\lambda_k$.

For $k\geq 1$, define $p_{(k)}(x_1,\ldots,x_N)\triangleq x_1^k+\cdots +x_N^k$ and for $\lambda\in\Gamma$, define $p_\lambda(x_1,\ldots,x_N)\triangleq \prod_{i=1}^{\ell(\lambda)}p_{(\lambda_i)}(x_1,\ldots,x_N)$. Furthermore, for $\epsilon\in\Gamma_N$, define
\[
M_\epsilon(x_1,\ldots,x_N)\triangleq \sum_{\substack{(a_1,\ldots,a_N)\in\mathbb{Z}_{\geq 0}^N, \\ \gamma((a_1,\ldots,a_N))=\epsilon}} \prod_{i=1}^N x_i^{a_i} = \sum_{\substack{(a_1,\ldots,a_N)\in\mathbb{Z}_{\geq 0}^N, \\ \gamma((a_1,\ldots,a_N))=\epsilon}} (x_1,\ldots,x_N)^{(a_1,\ldots,a_N)}.
\]
Also, define $e(x_1,\ldots,x_N)\triangleq x_1\cdots x_N$.

The following two lemmas are straightforward to deduce, but are essential components of this paper.
\begin{lemma}
Suppose $N\geq 1$ and $c_{(k)}\in\mathbb{C}$ for all $k\geq 1$. Then,
\[
\exp\left(\sum_{k\geq 1} c_{(k)}p_{(k)}(x_1,\ldots,x_N)\right) = 1+\sum_{\lambda\in\Gamma} \pi(\lambda)^{-1} \prod_{i=1}^{\ell(\lambda)} c_{(\lambda_i)} p_\lambda(x_1,\ldots,x_N)
\]
over $\mathbb{C}[[x_1,\ldots,x_N]]$.
\end{lemma}

\begin{lemma}
Suppose $N\geq 1$ and $a,b\in\mathbb{C}^N$. Then,
\[
\exp(\product{a,b}) = 1+ \sum_{\epsilon\in\mathbb{Z}^N_{\geq 0},\, \epsilon\not=0} \frac{a^\epsilon b^\epsilon}{\epsilon!}.
\]
\end{lemma}

For $k\geq 1$, it is clear that a basis for the set of elements of $\mathbb{C}[x_1,\ldots,x_N]$ that are symmetric and homogeneous of degree $k$ is $\{M_\epsilon(x_1,\ldots,x_N): \epsilon\in\Gamma_N[k])$. The following lemma is also well-known.

\begin{lemma}
Suppose $k\geq 1$. Then, $\mathbb{C}$-linear bases for the set of elements of $\mathbb{C}[x_1,\ldots,\\x_N]$ that are:
\begin{enumerate}
\item symmetric and homogeneous of degree $k$;
\item symmetric, homogeneous of degree $2k$, and have all even degrees;
\item symmetric, homogeneous of degree $2k+N$, and have all odd degrees
\end{enumerate}
are $\{p_\lambda(x_1,\ldots,x_N): \lambda\in\Gamma_N[k]\}$, $\{p_\lambda(x_1,\ldots,x_n): \lambda\in\evenN[2k]\}$, and $\{e(x_1,\ldots,x_N)\times\\p_\lambda(x_1,\ldots,x_N): \lambda\in\evenN[2k]\}$, respectively.
\end{lemma}

Additionally, define the unique coefficients $c_{\epsilon\lambda}\in\mathbb{Q}$ for $\epsilon,\lambda\in\Gamma$ such that 
\[
M_\epsilon = \sum_{\lambda\in\Gamma} c_{\epsilon\lambda}p_\lambda.
\]
We define $c_{\lambda\epsilon}^{\text{inv}}$ for $\epsilon,\lambda\in\Gamma$ such that 
\[
p_\lambda = \sum_{\epsilon\in\Gamma} c_{\lambda\epsilon}^{\text{inv}}M_\epsilon.
\]
\begin{remark}
When we write $M_\epsilon = \sum_{\lambda\in\Gamma} c_{\epsilon\lambda} p_\lambda$, we mean that $M_\epsilon(x) = \sum_{\lambda\in\Gamma} c_{\epsilon\lambda} p_\lambda(x)$, where $x=(x_i)_{i\geq 1}$ consists of an infinite number of variables. We could equivalently state that $M_\epsilon(x_1,\ldots,x_N) = \sum_{\lambda\in\Gamma} c_{\epsilon\lambda} p_\lambda(x_1,\ldots,x_N)$ for all $N\geq 1$. This is similarly the case when we write $p_\lambda = \sum_{\epsilon\in\Gamma} c_{\lambda\epsilon}^{\text{inv}}M_\epsilon$. It is clear that $c_{\lambda \lambda} = \frac{1}{\pi(\lambda)}$ and $c_{\lambda\lambda}^{\text{inv}} = \pi(\lambda)$.
\end{remark}

\subsection{Non-crossing partitions}
\label{subsec:noncrossing}

Suppose $k\geq 1$. Let $NC(k)$ denote the set of noncrossing partitions of $[k]$. Recall that a partition $B_1\sqcup \cdots \sqcup B_m$ of $[k]$ is noncrossing if there does not exist distinct $i,j\in [m]$, $a,c\in B_i$, and $b,d\in B_j$ such that $a<b<c<d$. Furthermore, for $k\geq 1$, define $\NCeven(2k)$ to be the set of $\pi\in NC(2k)$ such that each block of $\pi$ has even size.

Suppose $\pi=B_1\sqcup\cdots\sqcup B_m\in NC(k)$ such that the minimal element of $B_q$ is less than the minimal element of $B_{q+1}$ for $1\leq q\leq m-1$. For $q\in [m]$ and $i\in B_q$, define $b(i;\,\pi)\triangleq \mathbf{1}\{i=\min(B_q)\}$ and 
\[
d(i;\,\pi)\triangleq \left|\left(\bigcup_{r=1}^q B_r\right) \bigcap \{i,i+1,\ldots,k\}\right|.
\]
Furthermore, define $o(\pi)$ to be the number of $i\in [k]$ such that $b(i;\,\pi)=0$ and $d(i;\,\pi)$ is odd. Also, define $\gamma(\pi)\triangleq \gamma((|B_1|,\ldots,|B_m|))$. For an illustration of a noncrossing partition and the statistics that we have defined, see \Cref{fig:noncrossing}.

\begin{figure}[!ht]
\begin{center}
\begin{tikzpicture}
    \fill[red!30] (0,0) rectangle (3,1);
    \fill[blue!30] (3,0) rectangle (4, 1);
    \fill[green!30] (4, 0) rectangle (6,1);
    \fill[blue!30] (6,0) rectangle (8,1);
    \fill[yellow!30] (8, 0) rectangle (9,1);
    \fill[cyan!30] (9, 0) rectangle (11,1);
    \fill[red!30] (11,0) rectangle (13,1);
    \draw (0,0) grid (13,1);
    \node at (0.5, 0.5) {5};
    \node at (1.5, 0.5) {4};
    \node at (2.5, 0.5) {3};
    \node at (3.5, 0.5) {5};
    \node at (4.5, 0.5) {6};
    \node at (5.5, 0.5) {5};
    \node at (6.5, 0.5) {4};
    \node at (7.5, 0.5) {3};
    \node at (8.5, 0.5) {3};
    \node at (9.5, 0.5) {4};
    \node at (10.5, 0.5) {3};
    \node at (11.5, 0.5) {2};
    \node at (12.5, 0.5) {1};
\end{tikzpicture}
\end{center}

\caption{An illustration of the element 
\[
\pi = \{1, 2, 3, 12, 13\} \sqcup \{4, 7, 8\} \sqcup \{5, 6\} \sqcup \{9\} \sqcup \{10, 11\}
\]
of $NC(13)$. Each color corresponds to a block of $\pi$ and square $i$ displays the value of $d(i; \pi)$. Observe that $o(\pi)=5$.}
\label{fig:noncrossing}
\end{figure}

\subsection{Additional notation}
\label{subsec:pushforward}
For a measure $\mu$ over $\mathbb{R}^N$ and measurable $f: \mathbb{R}^N \rightarrow \mathbb{R}^N$, we let $\mu(f)$ denote the pushforward of $\mu$ with respect to the map $f$. For example, for a point mass $\delta_a$ located at $a\in\mathbb{R}^N$, $\delta_a(f)$ denotes the point mass $\delta_{f(a)}$. If $f:\mathbb{R}\rightarrow\mathbb{R}$ is measurable, then $\mu(f)$ denotes the pushforward of $\mu$ with respect to the map $(a_1,\ldots,a_N)\mapsto (f(a_1),\ldots,f(a_N))$.

\section[An introduction to graded rings of operators]{An introduction to applying a graded ring of operators to a graded vector field}
\label{sec:intro_graded}

The paper \cite{dunkl_singular_poly} considers the applications of operators to a graded vector field $V\triangleq \oplus_{i\geq 0} V_i$, where the application of an operator to an element of $V_i$ outputs an element of $V_{i-1}$ for $i\geq 1$. When $V=\mathbb{C}[x_1,\ldots,x_N]$ and $V_i$ is the set of elements of $V$ that are homogeneous of degree $i$ for all $i\geq 0$, we have that $\partial_i$, $\mathcal{D}(\mathcal{R}(\theta))_i$, and $d_i$ for $i\in [N]$ are examples of such an operator. In this section, we introduce a similar setting that is motivated by the Dunkl bilinear form $[\cdot,\cdot]_{\mathcal{R}(\theta)}$ and we continue to discuss this setting in \Cref{sec:invertible,sec:bijection}. See \Cref{example:dunkl} for the application of the framework that we introduce to the context of Dunkl operators.

Let $K$ be a field. Assume that $V\triangleq\oplus_{i\geq 0} V_i$ is a graded vector space such that $V_0=K$ and $V_i$ is a finite dimensional $K$-vector space for $i\geq 1$. Furthermore, assume that $R\triangleq\oplus_{i\geq 0} R_i$ is a graded ring such that $R_i$ is a finite dimensional $K$-vector space for $i\geq 0$. From the definition of a graded ring, recall that $R_i$ is an additive abelian group for $i\geq 0$ and $R_{i_1}R_{i_2}\subset R_{i_1+i_2}$ for all $i_1,i_2\geq 0$.

For $k\in K$, let $k^*$ denote the element $\{v\mapsto kv\}$ of $\text{End}_K(V)$. Let $\mathcal{L}$ be the set of $K$-linear ring homomorphisms $L: R\rightarrow \text{End}_K(V)$ such that for $i_1,i_2\geq 0$, $f\in V_{i_2}$, and $g\in R_{i_1}$,
\[
\begin{cases}
L(g)f \in V_{i_2-i_1} & \text{if } i_1\leq i_2,
\\ L(g)f = 0 & \text{if } i_2 < i_1.
\end{cases}
\]

\begin{remark}
By the definition of a ring homomorphism, $L(gh)=L(g)L(h)$ and $L(g+h)=L(g)+L(h)$ for all $g,h\in R$. Furthermore, $L(0)=0^*$, $L(1)=1^*$, and the $K$-linearity condition implies that $L(kr)=kL(r)$ for all $r\in R$ and $k\in K$.
\end{remark}

\begin{lemma}
\label{lemma:center}
Suppose $L\in \mathcal{L}$.
\begin{enumerate}
\item[(A)] For all $k\in K$, $L(k)=k^*$.
\item[(B)] For all $r\in R$ and $k\in K$, $L(rk-kr)=0^*$.
\end{enumerate}
\end{lemma}

\begin{proof}
For (A), by the $K$-linearity condition, $L(k)=kL(1)=k^*$ for all $k\in K$. For (B), observe that 
\[
L(rk) = L(r)k^*= k^*L(r) = L(kr) 
\]
for all $r\in R$ and $k\in K$, because $L(r)$ is a $K$-linear endomorphism of $V$.
\end{proof}

\begin{remark}
We do not assume that $K\subset\text{center}(R)$ despite part (B) of \Cref{lemma:center}.
\end{remark}

For $i\geq 0$, let $A_i$ and $B_i$ be bases of $V_i$ and $R_i$, respectively, as $K$-vector spaces. Define the isomorphisms $a_i: V_i\rightarrow K^{A_i}$ by $a_i(r)=[\mathbf{1}\{s=r\}]_{s\in A_i}^T$ for $r\in A_i$ and $b_i: R_i\rightarrow K^{B_i}$ by $b_i(r)=[\mathbf{1}\{s=r\}]_{s\in B_i}^T$ for $r\in B_i$. Furthermore, for $A\in \text{End}_K(V_i)$, define $a_i(A)\in K^{A_i\times A_i}$ to be the matrix with column $r$ equal to $a_i(Ar)$ for $r\in A_i$. Then, we have that
\[
a_i(Af) = a_i(A)a_i(f)
\]
for all $f\in A_i$. Note that we denote $a_i$ and $b_i$ by $a$ and $b$ if it is clear that the input is in $V_i$ and $R_i$, respectively.

Furthermore, for $L\in\mathcal{L}$ and $i\geq 0$, let $M^{i; L} \in K^{B_i\times A_i}$ denote the matrix such that
\[
M^{i; L}_{rs} = L(r)s
\]
for $r\in B_i$ and $s\in A_i$. 

\begin{lemma}
\label{lemma:commutative2}
Suppose $R$ is commutative. Then, the operators $L(g)$ for $g\in R$ are commutative.
\end{lemma}

\begin{proof}
Suppose $g,h\in R$. Then, $L(g)L(h)=L(gh)=L(hg)=L(h)L(g)$.
\end{proof}

\begin{lemma}
\label{lemma:matrixeq}
Suppose $i\geq 0$, $f\in V_i$, and $g\in R_i$. Then,
\[
L(g)f =  b(g)^TM^{i;L}a(f).
\]
\end{lemma}
\begin{proof}
First, observe that
\[
f = \sum_{r\in A_i} a(f)_rr \text{ and } g = \sum_{s\in B_i} b(g)_ss.
\]
We therefore have that
\begin{align*}
L(g)f = \sum_{r\in A_i,\,s\in B_i} b(g)_sa(f)_rL(s)r = \sum_{r\in A_i,\,s\in B_i} b(g)_s M^{i;L}_{sr}a(f)_r = b(g)^TM^{i;L}a(f).
\end{align*}
\end{proof}

\begin{definition}
\label{def:degreepreserve}
Suppose $\mathcal{V}\in \text{End}_K(V)$. Then, $\mathcal{V}$ is \textit{degree-preserving} if $\mathcal{V}V_i\subset V_i$ for all $i\geq 0$.
\end{definition}

For $\mathcal{V}\in\text{End}_K(V)$ that is degree-preserving, let $\mathcal{V}_i$ denote its restriction to $V_i$ for all $i\geq 0$. Then, it is evident that $(a(\mathcal{V}_i))_{i\geq 0}$ is a representation for the action of $\mathcal{V}$ on $V$. A degree-preserving operator we consider is an intertwining operator, which is well studied in the context of Dunkl operators.

\begin{definition}
\label{def:intertwine}
Suppose $L_1,L_2\in\mathcal{L}$. Then, $(L_1, L_2)$ \textit{intertwines} with $\mathcal{V} \in \text{End}_K(V)$ if:
\begin{enumerate}
\item The operator $\mathcal{V}$ is degree-preserving.
\item For $k\in V_0\triangleq K$, $\mathcal{V}k = k$.
\item For all $f\in V$ and $g\in R$, $L_1(g)\mathcal{V} f = \mathcal{V}L_2(g)f$ 
\end{enumerate}
Furthermore, if these conditions are satisfied, then $(L_1, L_2)$ is \textit{intertwining}.
\end{definition}

\begin{remark}
If $(L_1,L_2)$ is intertwining, we do not necessarily have that $(L_2,L_1)$ is intertwining, although this is the case if $L_2$ is invertible, see \Cref{lemma:intertwine_reorder}. We discuss invertibility in the next section.
\end{remark}

Let $\mathcal{F}(V)$ denote the abelian group of formal power series $(E^i)_{i\geq 0}$ such that $E^i\in V_i$ for all $i\geq 0$. For the addition operation, we have that $(E_1^i)_{i\geq0}+(E_2^i)_{i\geq0}\triangleq(E_1^i+E_2^i)_{i\geq0}$.

Suppose $\Psi\in \text{Hom}_K(R,K)$, which is the set of $K$-linear ring homomorphisms from $R$ to $K$. The formal power series $E\triangleq (E^i)_{i\geq 0} \in \mathcal{F}(V)$ is a \textit{$\Psi$-eigenvector} of $L\in\mathcal{L}$ if for all $f\in R$, we have that
\[
L(f)E = \Psi(f) E.
\]
More specifically, $E$ is a \textit{$\Psi$-eigenvector} of $L$ if for all $j\geq i\geq 0$ and $f\in R_i$,
\[
L(f) E^j = \Psi(f) E^{j-i}.
\]

\begin{definition}
\label{def:kernel}
Suppose $\mathcal{S}\subset \text{Hom}_K(R,K)$. Define the \textit{kernel} of $\mathcal{S}$, which we denote as $\text{ker}(\mathcal{S})$, to be the set of $r\in R$ such that $\Psi(r)=0$ for all $\Psi\in\mathcal{S}$.
\end{definition}

In the next section, we introduce the notion of invertible elements of $\mathcal{L}$, which is related to intertwining pairs and eigenvectors, see \Cref{thm:equivalences}.

\section{Invertible graded rings of operators}
\label{sec:invertible}

In this section, we introduce and characterize the invertible elements of $\mathcal{L}$. The notion of invertibility generalizes the notion of singular Dunkl operators arising from certain multiplicity functions that has been studied in \cites{dunklbessel,dunkl_singular_poly}, with a singular Dunkl operator corresponding to a non-invertible element of $\mathcal{L}$. First, we define when an element of $\mathcal{L}$ is invertible; the definition is analogous to \Cref{def:dunkl_invertible}.

\begin{definition} 
\label{def:invertible}
The homomorphism $L\in\mathcal{L}$ is \textit{invertible} if $M^{i;L}$ is an invertible square matrix with nonzero dimensions for all $i\geq 0$. 
\end{definition}

\begin{corollary}
If $L\in\mathcal{L}$ is invertible, then $\text{dim}(R_i)=\text{dim}(V_i) \geq 1\Leftrightarrow |A_i|=|B_i| \geq 1$ for all $i\geq 0$.
\end{corollary}

\begin{corollary}
\label{cor:zero}
Suppose $L\in\mathcal{L}$, $A_0=\{a_0\}$, and $B_0=\{b_0\}$ for $a_0,b_0\in K^{\times}$. Then, $M^{0;L}=[a_0b_0]$.
\end{corollary}

\begin{proof}
This follows from part (A) of \Cref{lemma:center}.
\end{proof}

We define the notion of $R$ being left- and right-spanned by $S\subset R\backslash R_0$. The case of $R$ being right-spanned by $S=R_1$ is mentioned as a condition in \Cref{thm:equivalences}, the main result of this section, and later the case of $R$ being left-spanned by $S=R_1$ is mentioned as a condition in \Cref{thm:notinvertible}.

\begin{definition}
Suppose $S\subset R\backslash R_0$. Then, $R$ is \textit{left-spanned} by $S$ if $R\subset R_0\cup SR$ and is \textit{right-spanned} by $S$ if $R\subset R_0\cup RS$.
\end{definition}

The following result generalizes ideas that have appeared previously in the study of Dunkl operators in \cites{dunklbessel,dunkl_singular_poly,dunkl_complex}. In particular, \cite{dunkl_singular_poly}*{Section 2} discusses a similar setting regarding operators over a graded vector space and the existence of intertwiners for these operators. We extend this idea by considering a graded ring of operators that acts on a graded vector space and connecting the invertibility of an operator with the existence of an intertwiner. Furthermore, we consider the existence of eigenvectors, which leads to additional applications such as the existence of unique eigenfunctions for the complex Dunkl operators introduced in \cite{dunkl_complex}, although we do not study this direction further.

The key contribution that we discuss in this section while proving \Cref{thm:equivalences} is the analysis of the matrices $\{M^{i;L}\}_{i\geq 0}$. This method allows for straightforward proofs that are applicable in a general setting.

\begin{theorem}
\label{thm:equivalences}
Assume that $\text{dim}(R_i)=\text{dim}(V_i) \geq 1$, $A_i=\{a_{ij}\}_{1\leq j\leq |A_i|}$, and $B_i=\{b_{ij}\}_{1\leq j\leq |A_i|}$ for all $i\geq 0$. The following are equivalent.
\begin{enumerate}
\item[(a)] The homomorphism $L\in\mathcal{L}$ is invertible, which by definition is equivalent to the matrix $M^{i;L}$ being invertible for all $i\geq 0$.
\item[(b)] For all $i\geq 0$, there does not exist nonzero $f\in V_i$ such that $L(g)f=0$ for all $g\in R_i$.
\item[(c)] For all $i\geq 0$, there does not exist nonzero $g\in R_i$ such that $L(g)f=0$ for all $f\in V_i$.
\item[(d)] The equation $\sum_{j=1}^{|A_i|} (L(b_{ij})f)a_{ij} =0$ has no nonzero solutions $f\in V_i$ for all $i\geq 0$.
\item[(e)] For all $i\geq 0$, there exist unique $\text{inv-row}_{ij}\in R_i$ for $1\leq j\leq |A_i|$ such that $L(\text{inv-row}_{ij_1})a_{ij_2} = \mathbf{1}\{j_1=j_2\}$ for $j_1,j_2\in [|A_i|]$.
\item[(f)] For all $i\geq 0$, there exist unique $\text{inv-col}_{ij}\in V_i$ for $1\leq j\leq |A_i|$ such that $L(b_{ij_1}) \text{inv-col}_{ij_2} = \mathbf{1}\{j_1=j_2\}$ for $j_1,j_2\in [|A_i|]$. 
\item[(g)] For some invertible homomorphism $L'\in\mathcal{L}$, $(L,L')$ is intertwining.
\end{enumerate}
Assume that $K\subset\text{center}(R)$. Then, (h) is equivalent to (a).
\begin{enumerate}
\item[(h)] For all homomorphisms $L'\in\mathcal{L}$, $(L,L')$ intertwines with a unique intertwiner.
\end{enumerate}
Assume that $\text{ker}(\text{Hom}_K(R,K))=\{0\}$. Then, (i) and (j) are equivalent to (a).
\begin{enumerate}
\item[(i)] For some $\mathcal{S}\subset \text{Hom}_K(R,K)$ such that $\text{ker}(\mathcal{S})=\{0\}$, there exists a $\Psi$-eigenvector $E$ such that $E^0=1$ for all $\Psi\in\mathcal{S}$.
\item[(j)] For all $\Psi\in \text{Hom}_K(R,K)$, there exists a unique $\Psi$-eigenvector $E$ such that $E^0=1$.
\end{enumerate}
Assume that $R$ is right-spanned by $R_1$. Then, (k) is equivalent to (a).
\begin{enumerate}
\item[(k)] There does not exist $f\in V\backslash{V_0}$ such that $L(g)f=0$ for all $g\in R_1$.
\end{enumerate}
\end{theorem}

In the remaining portions of this section, we assume that $\text{dim}(R_i)=\text{dim}(V_i) \geq 1$ for all $i\geq 0$. Observe that $\text{dim}(R_0)=\text{dim}(V_0)=1$, so $R_0=K$. The goal of this section is to prove \Cref{thm:equivalences}. We also prove some additional results relating the singular value decompositions of the matrices $\{M^{i;L}\}_{i\geq 0}$ to eigenvectors, see \Cref{prop:svd}, and the structure of non-invertible operators, see \Cref{thm:notinvertible}.

\begin{lemma}
The following is true: $(a)\Leftrightarrow (b) \Leftrightarrow (c)$.
\end{lemma}

\begin{corollary}
\label{cor:invertible_center}
If $L\in\mathcal{L}$ is invertible, then $K\subset\text{center}(R)$.
\begin{proof}
See part (B) of \Cref{lemma:center} and condition (c).
\end{proof}
\end{corollary}

In fact, $K\subset\text{center}(R)$ is a necessary and sufficient condition for the existence of invertible $L\in \mathcal{L}$, see \Cref{thm:bijection} in addition to the previous corollary.

\subsection{Statements (d), (e), and (f)}

We consider (d), (e), and (f). Recall that $A_i=\{a_{ij}\}_{1\leq j\leq |A_i|}$ and $B_i=\{b_{ij}\}_{1\leq j\leq |A_i|}$ for $i\geq 0$.

\begin{lemma}
\label{lemma:eigenvalue}
Suppose $i\geq 0$ and $v\in K^{A_i}$. Define $\varphi: K^{A_i}\rightarrow K^{B_i}$ by $\varphi(a(a_{ij}))= b(b_{ij})$ for $1\leq j\leq |A_i|$. Let $f=a^{-1}(v)\triangleq \sum_{s\in A_i} v_ss$. Then, for $L\in\mathcal{L}$ and $\lambda \in K$, $M^{i;L}v=\lambda \varphi(v) $ if and only if $\sum_{j=1}^{|A_i|} (L(b_{ij}) f)a_{ij} = \lambda f$.
\end{lemma}
\begin{proof}
We have that $M^{i;L}v=\lambda\varphi(v)$ if and only if
\[
L(b_{ij})f = (M^{i;L}v)_{b_{ij}} = \lambda v_{a_{ij}} \,\,\, \forall 1\leq j\leq |A_i|.
\]
This is equivalent to $\sum_{j=1}^{|A_i|} (L(b_{ij}) f) a_{ij} = \lambda \sum_{j=1}^{|A_i|}v_{a_{ij}}a_{ij}  = \lambda f$.
\end{proof}

\begin{corollary}
The conditions (a) and (d) are equivalent.
\end{corollary}

\begin{proof}
This follows from (b) or checking that $M^{i;L}$ has no zero eigenvalues for all $i\geq 0$ using \Cref{lemma:eigenvalue}.
\end{proof}

\begin{lemma}
\label{lemma:invcol}
The conditions (a) and (e) are equivalent. If either is satisfied, $b(\text{inv-row}_{ij})$ is the transpose of row $b_{ij}$ of $(M^{i;L})^{-1}$ for $i\geq 0$ and $1\leq j\leq |A_i|$.
\end{lemma}
\begin{proof}
Assume that $L$ is invertible. Suppose $i\geq 0$ and $1\leq j_1\leq |A_i|$. We have that $L(\text{inv-row}_{ij_1})a_{ij_2}=\mathbf{1}\{j_1=j_2\}$ for all $j_2\in [|A_i|]$ if and only if $b(\text{inv-row}_{ij_1})$ is the transpose of row $b_{ij_1}$ of $(M^{i;L})^{-1}$. If the $\text{inv-row}_{ij}$ for $1\leq j\leq |A_i|$ exist, then we can form the inverse of $M^{i;L}$ by setting row $b_{ij}$ of the inverse to be $b(\text{inv-row}_{ij})^T$ for $1\leq j\leq |A_i|$ to show that $L$ is invertible.
\end{proof}

\begin{lemma}
The conditions (a) and (f) are equivalent. If either is satisfied, $a(\text{inv-col}_{ij})$ is column $a_{ij}$ of $(M^{i;L})^{-1}$ for $i\geq 0$ and $1\leq j\leq |A_i|$.
\end{lemma}
\begin{proof}
See the proof of the previous result.
\end{proof}

The results are more straightforward in the case that $V=R$, in which case we obtain the following direct implication.

\begin{corollary}
\label{cor:vrequal}
Assume that $V=R$ and $A_i=B_i$ for all $i\geq 0$. Then, (a) is equivalent to (d'), (e'), and (f').
\begin{enumerate}
\item[(d')] For all $i\geq 0$, the equation $\sum_{s\in A_i} (L(s)f)s=0$ has no nonzero solutions $f\in R_i$.
\item[(e')] For all $i\geq 0$, there exist unique $\text{inv-row}(r; A_i)\in R_i$ for $r\in A_i$ such that $L(\text{inv-row}(r; A_i))s = \mathbf{1}\{r=s\}$ for $r,s\in A_i$.
\item[(f')] For all $i\geq 0$, there exist unique $\text{inv-col}(r; A_i)\in R_i$ for $r\in A_i$ such that $L(s)\text{inv-col}\\(r; A_i) = \mathbf{1}\{r=s\}$ for $r,s\in A_i$.
\end{enumerate}
\end{corollary}

\subsection{Statements (g) and (h)}

\begin{lemma}[$(a)\Leftrightarrow (g)$]
\label{lemma:invertible}
Assume that $L_2\in\mathcal{L}$ is invertible. Then, $(L_1, L_2)$ is intertwining if and only if $L_1\in\mathcal{L}$ is invertible.
\end{lemma}

\begin{proof}
Assume that $(L_1,L_2)$ intertwines with $\mathcal{V}$. Then, by \Cref{lemma:matrixeq}, we have that for $i\geq 0$ and for all $f\in V_i$ and $g\in R_i$,
\begin{equation}
\label{eq:product_1}
L_1(g)\mathcal{V}f = b(g)^TM^{i;L_1}a(\mathcal{V}_i)a(f);
\end{equation}
recall that $\mathcal{V}_i\triangleq\mathcal{V}|_{V_i}$. Since $\mathcal{V}$ is the identity transformation over $V_0$,
\begin{equation}
\label{eq:product_2}
\mathcal{V}L_2(g)f = b(g)^TM^{i;L_2}a(f).
\end{equation}
We therefore have that $M^{i;L_1}a(\mathcal{V}_i)=M^{i;L_2}$. Then, since $L_2$ is invertible, $L_1$ is invertible.

Next, assume that $L_1$ is invertible. For $i\geq 0$, define $C^i\triangleq(M^{i;L_1})^{-1}M^{i; L_2}$ and $\mathcal{V}_i: V_i\rightarrow V_i$ by $a(\mathcal{V}_i)\triangleq C^i$;
this uniquely defines the degree-preserving operator $\mathcal{V}$. By \Cref{cor:zero}, $\mathcal{V}$ acts as the identity over $V_0$. Then, from \pref{eq:product_1} and \pref{eq:product_2}, $L_1(g)\mathcal{V}f =\mathcal{V}L_2(g)f$ whenever $f\in V_i$ and $g\in R_i$.

Suppose $0\leq i_1<i_2$. We prove that for all $f\in V_{i_2}$ and $g\in R_{i_1}$,
\[
L_1(g)\mathcal{V}f = \mathcal{V}L_2(g)f
\]
to show that $(L_1,L_2)$ intertwines with $\mathcal{V}$. Since $\mathcal{V}$ is a linear operator, this is evident when $i_1=0$ after applying \Cref{lemma:center}, so assume that $i_1\geq 1$. Note that for all $h\in R_{i_2-i_1}$,
\[
L_1(h)L_1(g)\mathcal{V}f = \mathcal{V}L_2(hg) f = L_1(h) \mathcal{V} L_2(g) f.
\]
Since this expression is true for all $h$, by (b), we have that $L_1(g)\mathcal{V}f = \mathcal{V}L_2(g) f$.
\end{proof}

\begin{lemma}
\label{lemma:bijection}
If $L_2\in\mathcal{L}$ is invertible and $(L_1,L_2)$ intertwines with $\mathcal{V}$ for some $L_1\in\mathcal{L}$, then $\mathcal{V}$ is unique and is a bijection from $V$ to $V$.
\end{lemma}

\begin{proof}
Suppose $i\geq 0$. We have that $M^{i;L_1}a(\mathcal{V}_i) = M^{i; L_2}$. Since the matrices $M^{i;L_1}$ are invertible for $i\geq 0$ by \Cref{lemma:invertible}, $\mathcal{V}$ is unique.

To prove that $\mathcal{V}$ is a bijection, it suffices to prove that it is a bijective endomorphism of $V_i$, since it is degree-preserving. However, this is clear, since the action of $\mathcal{V}$ on $V_i$ is isomorphic to multiplying by the invertible matrix $a(\mathcal{V}_i)$.
\end{proof}

\begin{lemma}[$(a)\Leftrightarrow (h)$]
Assume that $K\subset\text{center}(R)$. The homomorphism $L_1\in\mathcal{L}$ is invertible if and only if $(L_1, L_2)$ intertwines with a unique operator for all $L_2\in\mathcal{L}$.
\end{lemma}

\begin{proof}
If $L_1$ is invertible, then from the second part of the proof of \Cref{lemma:invertible}, $(L_1,L_2)$ is intertwining for all $L_2$. The uniqueness of the intertwiner follows from $M^{i;L_1}a(\mathcal{V}_i)=M^{i;L_2}$ for $i\geq 0$. For the reverse direction, we can select $L_2$ that is invertible and apply \Cref{lemma:invertible}; for the existence of an invertible $L_2$, see \Cref{cor:bijection2}.
\end{proof}

\begin{lemma}
\label{lemma:intertwine_reorder}
Suppose $L_2\in\mathcal{L}$ is invertible. If $L_1\in\mathcal{L}$ and $(L_1,L_2)$ intertwines with $\mathcal{V}$, then $\mathcal{V}$ is invertible and $(L_2,L_1)$ intertwines with $\mathcal{V}^{-1}$.
\end{lemma}

\begin{proof}
Since $\mathcal{V}$ is invertible by \Cref{lemma:bijection},
\[
L_2(g)\mathcal{V}^{-1}f=\mathcal{V}^{-1}L_1(g)\mathcal{V}\circ \mathcal{V}^{-1}f = \mathcal{V}^{-1}L_1(g)f
\]
for all $f,g\in R$.
\end{proof}

\begin{remark}
We do not have $(L_2, L_1)$ intertwining with $\mathcal{V}$ implying that $(L_1,L_2)$ intertwines with $\mathcal{V}^{-1}$, since the invertibility of $L_2$ does not imply the invertibility of $L_1$ and $\mathcal{V}$ in this case.
\end{remark}

\subsection{Statements (i) and (j)}

\begin{lemma}[$(j)\Rightarrow (i)$]
Assume that $\text{ker}(\text{Hom}_K(R,K))=\{0\}$. Then, $(j)\Rightarrow (i)$.
\end{lemma}

\begin{lemma}[$(a)\Rightarrow (j)$]
\label{lemma:eigenvectorformula}
If the homomorphism $L\in\mathcal{L}$ is invertible and $\Psi\in \text{Hom}_K(R, \\K)$, then there exists a unique $\Psi$-eigenvector $E$ such that $E^0=1$; in particular, after using the notation from (f), for all $i\geq 0$,
\[
E^i = \sum_{j=1}^{|A_i|} \Psi(b_{ij})\text{inv-col}_{ij}.
\]
\end{lemma}

\begin{proof}Suppose $i\geq 0$ and $f\in R_i$. We require that $L(f) E^i = \Psi(f)E^0 = \Psi(f)$, which would imply that
\begin{equation}
\label{eq:eigenval}
M^{i;L}a(E^i) = [\Psi(r)]_{r\in B_i}^T.
\end{equation}
If $M^{i;L}$ is invertible, $a(E^i)$ is unique, so $E$ is unique. Afterwards, we can use \pref{eq:eigenval} and \Cref{lemma:invcol} to derive the given expression for $E^i$.

Next, we must prove that $E$ satisfies the eigenvalue condition. Suppose $a(E^i)=(M^{i;L})^{-1}[\Psi(r)]_{r\in B_i}^T$ for all $i\geq 0$. Then, for all $f\in R_i$, we have that $L(f)E^i = \Psi(f)$ because $\Psi$ is additive and $K$-linear. To show that $E^0=1$, suppose $A_0=\{a_0\}$ and $B_0=\{b_0\}$ so that
\[
a(E_0)=(M^{0;L})^{-1}\Psi(b_0) = \frac{b_0}{a_0b_0} = \frac{1}{a_0}
\]
after applying \Cref{cor:zero}. We prove that for $1\leq i_1<i_2$ and $f\in R_{i_1}$, $L(f)E^{i_2}=\Psi(f)E^{i_2-i_1}$.

Suppose $g\in R_{i_2-i_1}$. Then, since $\Psi$ is multiplicative and $L(g)$ is $K$-linear,
\[
L(g)L(f)E^{i_2} =\Psi(g)\Psi(f) = L(g) \Psi(f)E^{i_2-i_1}.
\]
We therefore have that $L(g)L(f)E^{i_2}=L(g)\Psi(f)E^{i_2-i_1}$ for all $g\in R_{i_2-i_1}$. By (b), we have that $L(f)E^{i_2}=\Psi(f)E^{i_2-i_1}$.
\end{proof}

\begin{lemma}
\label{lemma:kernel}
Suppose $\mathcal{S}\subset\text{Hom}_K(R,K)$ and that there exists a $\Psi$-eigenvector $E_\Psi$ such that $E_\Psi^0=1$ for all $\Psi\in\mathcal{S}$. For $i\geq 0$, any $r\in R_i$ such that $L(r)=0$ as an operator over $V_i$ must be an element of $\text{ker}(\mathcal{S})$.
\end{lemma}

\begin{proof}
We have that $\Psi(r)E_\Psi^0=L(r)E_\Psi^i=0$ for all $\Psi\in \mathcal{S}$, so $r\in\ker(\mathcal{S})$.
\end{proof}

\begin{corollary}[$(i)\Rightarrow (a)$]
Suppose $\mathcal{S}\subset\text{Hom}_K(R,K)$ and $\text{ker}(\mathcal{S})=\{0\}$. If there exists a $\Psi$-eigenvector $E$ such that $E^0=1$ for all $\Psi\in\mathcal{S}$, then $L$ is invertible.
\end{corollary}

\begin{proof}
For the sake of contradiction, assume that $L$ is not invertible. Then, for some $i\geq 0$, there exists nonzero $r\in R_i$ such that $L(r)f=0$ for all $f\in R_i$ by (b). This is a contradiction to \Cref{lemma:kernel}.
\end{proof}

\begin{proposition}[Singular value decomposition]
\label{prop:svd}
Suppose $L\in\mathcal{L}$ is invertible. Suppose $i\geq 0$ and $M^{i;L}$ satisfies
\[
M^{i;L} = U\Sigma V^T
\]
for $U\in K^{B_i\times B_i}$, $\Sigma\in K^{B_i\times A_i}$, and $V\in K^{A_i\times A_i}$ such that:
\begin{enumerate}
\item The matrices $U$ and $V$ are orthogonal.
\item The matrix $\Sigma$ satisfies $\Sigma_{b_{ij_1}a_{ij_2}}=0$ if $j_1,j_2\in [|A_i|]$ are not equal.
\end{enumerate}
Then, $\Sigma_{b_{ij}a_{ij}}\in K^\times$ for $j\in [|A|]$ and $(M^{i;L})^{-1}=V\Sigma^{-1}U^T$. 

For $r\in B_i$, let $U_r$ denote column $r$ of $U$ and for $s\in A_i$, let $V_s$ denote column $s$ of $V$. Then, for $\Psi\in \text{Hom}_K(R,K)$, the unique $\Psi$-eigenvector $E\in\mathcal{F}(V)$ with $E^0=1$ satisfies
\[
E^i = \sum_{j=1}^{|A_i|} \Sigma_{b_{ij}a_{ij}}^{-1} \Psi(b^{-1}(U_{b_{ij}}))a^{-1}(V_{a_{ij}}).
\]
\end{proposition}

\begin{proof}
Since $M^{i;L}$ is invertible, $\Sigma_{b_{ij}a_{ij}}\in K^\times$ for $j\in [|A|]$ and the formula $(M^{i;L})^{-1}=V\Sigma^{-1}U^T$ is evident. It is also evident that 
\[
(M^{i;L})^{-1} = \sum_{j=1}^{|A_i|} \Sigma_{b_{ij}a_{ij}}^{-1}V_{a_{ij}}U_{b_{ij}}^T
\]
Using \pref{eq:eigenval} gives that
\[
a(E^i) = \sum_{j=1}^{|A_i|} \Sigma_{b_{ij}a_{ij}}^{-1}V_{a_{ij}}U_{b_{ij}}^T[\Psi(r)]^T_{r\in B_i}.
\]
Observe that for $1\leq j\leq |A_i|$,
\[
U_{b_{ij}}^T[\Psi(r)]^T_{r\in B_i} = \sum_{r\in B_i} U_{rb_{ij}}\Psi(r) = \Psi\left(\sum_{r\in B_i}U_{rb_{ij}}r\right) = \Psi(b^{-1}(U_{b_{ij}})),
\]
which concludes the proof.
\end{proof}

Furthermore, we may consider the question of whether a nonzero $\Psi$-eigenvector $E$ exists such that $E^0=0$. However, the existence of such an eigenvector implies that $L\in\mathcal{L}$ is not invertible.

\begin{proposition}
\label{prop:zeroconstant}
Assume that $\Psi\in \text{Hom}_K(R,K)$ and $E$ is a nonzero $\Psi$-eigenvector such that $E^0=0$. Then, (a) and (k) are false.
\end{proposition}

\begin{proof}
It is evident that setting $f=E^i$, where $i$ is the smallest positive integer such that $E^i\not=0$, is a contradiction to both (b) and (k).
\end{proof}

\begin{remark}
In fact, the negation of (a) implies the negation of (k), see part (A) of \Cref{lemma:deg1}.
\end{remark}

\subsection{Statement (k)}

\begin{lemma}
\label{lemma:deg1}
\begin{enumerate}
\item[(A)] (k) implies (a).
\item[(B)] Suppose $R$ is right-spanned by $R_1$. Then, (a) implies (k).
\end{enumerate}
\end{lemma}

\begin{proof}
We prove the converse of both statements. Assume that $R_i$ is right-spanned by $R_1$. We first prove that the negation of (k) implies the negation of (a). Assume that $f\in V$ is nonconstant and $L(g)f=0$ for all $g\in R_1$. Suppose $i\geq 1$ and the degree $i$ part of $f$, $f_i$, is nonzero. Then, we have that $L(g)f_i=0$ for all $g\in R_1$. Since $R$ is right-spanned by $R_1$, it is then clear that $L(g)f_i=0$ for all $g\in R_i$, so $M^{i; L}$ is not invertible.

Next, we prove that the negation of (a) implies the negation of (k), without assuming that $R$ is right-spanned by $R_1$. Assume that $L$ is not invertible. Then, suppose $i_{\text{min}}$ is the minimal value of $i$ such that $M^{i;L}$ is not invertible. By \Cref{cor:zero}, $i_{\text{min}}\geq 1$.

Suppose $f\in R_{i_{\text{min}}}$ is nonzero and $L(g)f=0$ for all $g\in V_{i_{\text{min}}}$; $f$ exists since $M^{i_{\text{min}};L}$ is not invertible. Suppose $g\in R_1$. Then,
\[
L(h)L(g)f=0 \,\,\,\forall h\in R_{i_{\text{min}}-1}.
\]
By the minimality of $i_{\text{min}}$, $M^{i_{\text{min}}-1;L}$ is invertible, so $L(g)f=0$.
\end{proof}

In the following result, we analyze the structure of non-invertible $L\in\mathcal{L}$ by inspecting the invertibility of each of the matrices $M^{i;L}$ for $i\geq 0$. First, recall the following definition of the left annihilator.

\begin{definition}
Suppose $S\subset R$. Then, the \textit{left annihilator} of $S$ is $\text{Ann}_R(S)\triangleq\{r\in R: rs=0\,\,\,\forall s\in S\}$.
\end{definition}

\begin{theorem}
\label{thm:notinvertible}
Assume that $L\in\mathcal{L}$ is not invertible.
\begin{enumerate}
\item[(A)] Suppose $i_{\text{min}}$ is the minimal value of $i\geq 0$ such that $M^{i;L}$ is not invertible. Then, $i_{\text{min}}\geq 1$ and there exists $f\in R_{i_{\text{min}}}$ such that $L(g)f=0$ for all $g\in R_1$. 
\item[(B)] Assume that $R$ is left-spanned by $R_1$ and that $\text{Ann}_R(R\backslash R_0) = \{0\}$. Then, $M^{i;L}$ is not invertible for all $i\geq i_{\text{min}}$.
\item[(C)] Assume that $\text{Ann}_R(R\backslash R_0) = \{0\}$. Then, the number of $i\geq 1$ such that $M^{i;L}$ is not invertible is infinite.
\end{enumerate}
\end{theorem}

\begin{proof}
(A): By \Cref{cor:zero}, $i_{\text{min}}\geq 1$. For the existence of $f\in R_{i_{\text{min}}}$ such that $L(g)f=0$ for all $g\in R_1$, see the proof of part (A) of \Cref{lemma:deg1}.

(B): For the sake of contradiction, assume that $\ell$ is a positive integer such that $M^{\ell; L}$ is not invertible and $M^{\ell+1;L}$ is invertible. Suppose $g\in R_\ell$ is nonzero and satisfies $L(g)f=0$ for all $f\in V_\ell$. For all $r\in R_1$, we have that $L(gr)f=0$ for all $f\in V_{\ell+1}$. Since $M^{\ell+1}$ is invertible, we have that $gr=0$ for all $r\in R_1$. As $R\backslash R_0\subset R_1R$, it then follows that $gr=0$ for all $r\in R\backslash R_0$, which is a contradiction.

(C): For the sake of contradiction, assume that $i_{\text{max}}$ is the maximum positive integer $i$ such that $M^{i;L}$ is not invertible. Similarly, suppose $g\in R_{i_{\text{max}}}$ is nonzero and satisfies $L(g)f=0$ for all $f\in V_{i_{\text{max}}}$. Suppose $d\geq 1$. Then, for all $r\in R_d$ and $f\in V_{i_{\text{max}}+d}$, $L(gr)f=0$. Since $M^{i_{\text{max}}+d}$ is invertible, we have that $gr=0$ for all $r\in R_d$. Thus, $gr=0$ for all $r\in R\backslash R_0$, which is a contradiction.
\end{proof}

\section{Representations of invertible graded rings of operators}
\label{sec:bijection}

In this section, the goal is to obtain a bijection between equivalence classes of invertible $L\in\mathcal{L}$ and sequences of invertible matrices. The following lemma computes the value of $L(g)f$ when the degree of $g$ is at most the degree of $f$, assuming that $L\in\mathcal{L}$ is invertible.

\begin{lemma}
\label{lemma:formula}
Assume that $L\in\mathcal{L}$ is invertible. Suppose $0\leq i_1\leq i_2$, $g\in R_{i_1}$, and $f\in R_{i_2}$. Then,
\[
a_{i_2-i_1}(L(g)f) = (M^{i_2-i_1;L})^{-1} [b_{i_2}(rg)^T]_{r\in B_{i_2-i_1}} M^{i_2; L}a_{i_2}(f).
\]
\end{lemma}

\begin{proof}
We have that 
\begin{align*}
a_{i_2-i_1}(L(g)f) & =(M^{i_2-i_1;L})^{-1} [L(r)L(g)f]_{r\in B_{i_2-i_1}}^T \\
& = (M^{i_2-i_1;L})^{-1} [b_{i_2}(rg)^TM^{i_2;L}a_{i_2}(f)]_{r\in B_{d_2-d_1}}^T \\
& = (M^{i_2-i_1;L})^{-1} [b_{i_2}(rg)^T]_{r\in B_{i_2-i_1}} M^{i_2;L}a_{i_2}(f).
\end{align*}
\end{proof}

\begin{corollary}
Suppose $L\in\mathcal{L}$ is invertible. Suppose $0\leq i_1\leq i_2$ and $g\in R_{i_1}$. The homomorphism $L(g): V_{i_2}\rightarrow V_{i_2-i_1}$ is surjective if and only if $\{r\in R_{i_2-i_1}: rg=0\} = \{0\}$. 
\end{corollary}

\begin{proof}
Since $L$ is invertible, $L(g)$ is surjective if and only if
$[b_{i_2}(rg)^T]_{r\in B_{i_2-i_1}}$ has full row rank by \Cref{lemma:formula}. Afterwards, the result is straightforward to deduce.
\end{proof}

\begin{remark}
If $\{r\in R_{i_2-i_1}: rg=0\}=\{0\}$, then the previous corollary implies that $\text{dim}(R_{i_2-i_1})\leq \text{dim}(R_{i_2})$, provided that an invertible element of $\mathcal{L}$ exists. However, it is evident that an invertible element of $\mathcal{L}$ exists by \Cref{thm:bijection}, after assuming the conditions stated in the theorem.
\end{remark}

Suppose $V\subset W$ and $R\subset S$, where $W\triangleq \oplus_{i\geq 0} W_i$ and $S\triangleq\oplus_{i\geq 0} S_i$ are defined analogously to $V$ and $R$, respectively. Furthermore, define $\mathcal{T}$ analogously to $\mathcal{L}$ after replacing $(V, R)$ with $(W, S)$. In the remaining portions of this section, we still assume that $\text{dim}(V_i)=\text{dim}(R_i)\geq 1$ for all $i\geq 0$. In addition, we assume that $\text{dim}(W_i)=\text{dim}(S_i)\geq 1$ for all $i\geq 0$. 

Let $C_i$ and $D_i$ be bases of $W_i$ and $S_i$, respectively, for $i\geq 0$. Furthermore, define the isomorphisms $c_i: W_i\rightarrow K^{C_i}$ and $d_i: S_i\rightarrow K^{D_i}$ analogously to $a_i$ and $b_i$, respectively, for $i\geq 0$.

We say that $L\in\mathcal{T}$ is $(R,V)$-closed if $L(g)f\in V$ for all $g\in R$ and $f\in V$. Furthermore, we write $L_{(R,V)}: R\rightarrow \text{End}_K(V)$ to denote the $(R,V)$-restriction of an $(R,V)$-closed homomorphism $L$; it is clear that $L_{(R,V)}\in\mathcal{L}$. 

Define $\mathcal{I}$ to be the set of $L\in\mathcal{T}$ such that $L$ is $(R,V)$-closed and $L$ and $L_{(R,V)}$ are invertible. Define the equivalence relation $\sim$ over $\mathcal{I}$ such that if $A,B\in \mathcal{I}$, $A\sim B$ if and only if $A_{(R,V)}=B_{(R,V)}$. Define $\mathcal{I}'$ to be the set of $L\in\mathcal{T}$ such that $L$ is $(R,V)$-closed and $L_{(R,V)}$ is invertible and define the equivalence relation $\sim$ over $\mathcal{I}'$ similarly.

\subsection{Main result}
\label{subsec:bijection}

We state and prove the main result of this section. One of its direct consequences is the existence of an invertible element of $\mathcal{L}$ given that $\text{dim}(V_i)=\text{dim}(R_i)\geq 1$ for all $i\geq 1$, see \Cref{cor:bijection2}. However, for this to be true, we must assume that $K\subset\text{center}(R)$ based on \Cref{cor:invertible_center}, otherwise no elements of $\mathcal{L}$ will be invertible. We must also assume that a sequence of invertible matrices exists.

Let $\mathcal{M}$ be the set of sequences $\{M^i\}_{i\geq 0}$ of invertible matrices $M^i\in K^{B_i\times A_i}$ for $i\geq 0$ such that $M^0=[a_0b_0]$, where $A_0=\{a_0\}$ and $B_0=\{b_0\}$. 

\begin{theorem}
\label{thm:bijection}
Assume that there exists a subring $T_i$ of $S_i$ for $i\geq 0$ such that:
\begin{enumerate}
\item For all $i\geq 0$, $S_i=R_i\oplus T_i$.
\item For all $i_1,i_2\geq 0$, $r\in R_{i_1}$, and $t\in T_{i_2}$, $tr\in T_{i_1+i_2}$.
\end{enumerate}
Also, assume that $K\subset\text{center}(R)$. The function $\Phi: L \mapsto \{M^{i;L_{(R,V)}}\}_{i\geq 0}$ is a bijection from $\mathcal{I}/\sim$ to $\mathcal{M}$.
\end{theorem}

\begin{proof}
First, observe that if $L\in\mathcal{I}$, then $\Phi(L)\in\mathcal{M}$ after applying \Cref{cor:zero}. Suppose $m\in \mathcal{M}$ and that $\Phi(L)=m$ for $L\in \mathcal{I}$. From \Cref{lemma:formula}, we can obtain the value of $L(g)f$ for $f\in V$ and $g\in R$, so $L_{(R,V)}$ is fixed. Therefore, $\Phi$ is injective.

To finish the proof, we show that $\Phi$ is surjective. Suppose $m=\{M^i\}_{i\geq 0}\in\mathcal{M}$. We construct $L\in\mathcal{I}$ such that $M^{i;L_{(R,V)}}=M^i$ for all $i\geq 0$. Without loss of generality, assume that $A_i\subset C_i$ and $B_i\subset D_i$ for $i\geq 0$.

Suppose $i\geq 0$. Assume that $C_i\backslash A_i = \{c_{ij}\}_{1\leq j\leq |C_i|-|A_i|}$ and $D_i\backslash B_i=  \{d_{ij}\}_{1\leq j\leq |C_i|-|A_i|}$. Suppose $N^i\in \mathbb{C}^{D_i\times C_i}$ satisfies $N^i_{rv}=M^i_{rv}$ for $r\in B_i$ and $v\in A_i$, $N^i_{d_{ij}c_{ij}}=1$ for $j\in [|C_i|-|A_i|]$, and $N^i$ equals zero elsewhere. For $i\geq 0$, $N^i$ is invertible because $M^i$ is invertible.

Next, we define $L$ given the matrices $N^i$. For $0\leq i_1\leq i_2$, $f\in W_{i_2}$, and $g\in S_{i_1}$, we define
\begin{equation}
\label{eq:formula}
c_{i_2-i_1}(L(g)f)\triangleq (N^{i_2-i_1})^{-1} [d_{i_2}(rg)^T]_{r\in D_{i_2-i_1}} N^{i_2}c_{i_2}(f),
\end{equation}
which extends to a definition of $L(g)f$ for all $f\in W$ and $g\in S$. We must prove that $L\in \mathcal{I}$ and $\Phi(L)=m$.

\textbf{Step 1.} For the first step, we show that $L\in\mathcal{T}$. It suffices to prove that $L$ is a $K$-linear ring homomorphism from $S$ to $\text{End}_K(W)$, so it suffices to prove that $L(0)=0^*_W$, $L(1)=1^*_W$, $L(ks)=kL(s)$ for all $s\in S$ and $k\in K$, and $L(g+h)=L(g)+L(h)$ and $L(gh)=L(g)L(h)$ for $g,h\in S$; note that for $k\in K$, $k^*_W$ is the element $\{w\mapsto kw\}$ of $\text{End}_K(W)$ and is analogous to $k^*$.

It is straightforward to deduce that $L(0)=0^*_W$ and $L(g+h)=L(g)+L(h)$. To show that $L(1)=1^*_W$, set $i_1=0$ and $g=1$ to get that
\[
c_{i_2}(L(1)f) = (N^{i_2})^{-1}[d_{i_2}(r)^T]_{r\in D_{i_2}} N^{i_2} c_{i_2}(f) = c_{i_2}(f).
\]
Furthermore, to show that $L(ks)=kL(s)$ for all $k\in K$ and $s\in S$, we have that
\begin{align*}
c_{i_2-i_1}(L(ks)f) & = (N^{i_2-i_1})^{-1} [d_{i_2}(rks)^T]_{r\in D_{i_2-i_1}} N^{i_2}c_{i_2}(f) \\
& = (N^{i_2-i_1})^{-1} [d_{i_2}(krs)^T]_{r\in D_{i_2-i_1}} N^{i_2}c_{i_2}(f) \\
& = c_{i_2-i_1}(kL(s)f),
\end{align*}
where we have used $K\subset\text{center}(R)$. Next, we prove that $L(gh)=L(g)L(h)$.

Suppose $0\leq i_1,i_2\leq i_3$ and $i_1+i_2\leq i_3$. We prove that for $f\in W_{i_3}$, $g_1\in S_{i_1}$, and $g_2\in S_{i_2}$,
\[
L(g_2g_1)f = L(g_2)L(g_1)f.
\]
Observe that, based on \pref{eq:formula},
\begin{align*}
c_{i_3-i_1-i_2}(L(g_2)L(g_1)f) =& (N^{i_3-i_1-i_2})^{-1} [d_{i_3-i_1}(r_2g_2)^T]_{r_2\in D_{i_3-i_1-i_2}} N^{i_3-i_1}\\ 
& (N^{i_3-i_1})^{-1} [d_{i_3}(r_1g_1)^T]_{r_1\in D_{i_3-i_1}} N^{i_3}c_{i_3}(f) \\
=& (N^{i_3-i_1-i_2})^{-1} [d_{i_3-i_1}(r_2g_2)^T]_{r_2\in D_{i_3-i_1-i_2}}
[d_{i_3}(r_1g_1)^T]_{r_1\in D_{i_3-i_1}}N^{i_3} \\ & c_{i_3}(f)
\end{align*}
and
\[
c_{i_3-i_1-i_2}(L(g_2g_1)f) = (N^{i_3-i_1-i_2})^{-1} [d_{i_3}(rg_2g_1)^T]_{r\in D_{i_3-i_1-i_2}} N^{i_3}c_{i_3}(f).
\]
Hence, it suffices to prove that 
\[
[d_{i_3-i_1}(r_2g_2)^T]_{r_2\in D_{i_3-i_1-i_2}}
[d_{i_3}(r_1g_1)^T]_{r_1\in D_{i_3-i_1}} = [d_{i_3}(rg_2g_1)^T]_{r\in D_{i_3-i_1-i_2}}.
\]

Note that for $x\in D_{i_3-i_1-i_2}$ and $y\in D_{i_3}$, entry $(x,y)$ of the left hand side is
\[
\sum_{z\in D_{i_3-i_1}} d_{i_3-i_1}(xg_2)_z d_{i_3}(zg_1)_y.
\]
We have that 
\[
\sum_{y\in D_{i_3}}\sum_{z\in D_{i_3-i_1}} d_{i_3-i_1}(xg_2)_z d_{i_3}(zg_1)_yy = \sum_{z\in D_{i_3-i_1}} d_{i_3-i_1}(xg_2)_z zg_1 = xg_2g_1,
\]
so row $x$ of the left hand side matches row $x$ of the right hand side.

\textbf{Step 2.} To finish the proof, we must verify that $L\in\mathcal{I}$. First, it is clear that $L$ is invertible and from inspecting the matrices $N^i$ for $i\geq 0$, if $L$ is $(R,V)$-closed, then $L_{(R,V)}$ is invertible. Hence, it suffices to prove that $L$ is $(R,V)$-closed to complete the proof. 

Assume that $0\leq i_1\leq i_2$, $f\in V_{i_2}$ and $g\in R_{i_1}$. We must prove that $c_{i_2-i_1}(L(g)f)_{c_{i_2-i_1,j}} = 0$ for $1\leq j\leq |C_{i_2-i_1}|-|A_{i_2-i_1}|$. Using \pref{eq:formula}, we have that
\[
c_{i_2-i_1}(L(g)f)_{c_{i_2-i_1,j}} =\sum_{\substack{x\in D_{i_2-i_1},\,y\in D_{i_2}, \\ z\in C_{i_2}}} (N^{i_2-i_1})^{-1}_{c_{i_2-i_1,j}x}d_{i_2}(xg)_y N^{i_2}_{yz}c_{i_2}(f)_z.
\]
Since $f\in V_{i_2}$, $c_{i_2}(f)_z=0$ for $z\in C_{i_2}\backslash A_{i_2}$, so we may assume that $z\in A_{i_2}$. We may also assume that $y\in B_{i_2}$ so that $N_{yz}^{i_2}$ is nonzero. Since $xg\in T_{i_2}$ for all $x\in D_{i_2-i_1}\backslash B_{i_2-i_1}$, in order for $d_{i_2}(xg)_y$ to be nonzero, we may assume that $x\in B_{i_2-i_1}$. However, it is then clear that $(N^{i_2-i_1})^{-1}_{c_{i_2-i_1,j}x}=0$. Hence, $c_{i_2-i_1}(L(g)f)_{c_{i_2-i_1,j}}=0$, which finishes the proof.
\end{proof}

\begin{example}
We give examples of tuples $(S_i, R_i, T_i)$ for $i\geq 0$ that satisfy conditions (1) and (2) of \cref{thm:bijection}. The first example is to set $S_i=R_i$ and $T_i=\{0\}$ for all $i\geq 0$. Another example is to set $S_i$ to be the ring of homogeneous polynomials in $K[x_1,\ldots,x_N]$ of degree $i$, $R_i$ to be the ring of homogeneous polynomials in $K[x_1,\ldots,x_{N-1}]$ of degree $i$, $T_0$ to be $\{0\}$, and $T_i$ to be $x_N S_{i-1}$ for all $i\geq 1$.
\end{example}

\begin{corollary}
\label{cor:bijection1}
Assume that conditions (1) and (2) from \Cref{thm:bijection} are satisfied and $K\subset\text{center}(R)$. The function $\Phi$ is a bijection from $\mathcal{I}'/\sim$ to $\mathcal{M}$.
\end{corollary}

\begin{proof}
The proof of injectivity is the same as the proof of injectivity given in the proof of \Cref{thm:bijection}. On the other hand, surjectivity follows from \Cref{thm:bijection}.
\end{proof}

\begin{corollary}
\label{cor:bijection2}
Assume that $K\subset\text{center}(R)$. The function $\alpha: L \mapsto \{M^{i;L}\}_{i\geq 0}$ is a bijection from the set of invertible elements of $\mathcal{L}$ to $\mathcal{M}$.
\end{corollary}

\begin{proof}
This follows from \Cref{thm:bijection} after setting $(S,W)=(R,V)$.
\end{proof}

\begin{corollary}
\label{cor:bijection3}
Assume that conditions (1) and (2) from \Cref{thm:bijection} are satisfied and $K\subset\text{center}(R)$. The function $\beta: L \mapsto L_{(R,V)}$ is a bijection from $\mathcal{I}/\sim$ to the set of invertible elements of $\mathcal{L}$ and from $\mathcal{I}'/\sim$ to the set of invertible elements of $\mathcal{L}$.
\end{corollary}

\begin{proof}
The result follows from \Cref{thm:bijection} and \Cref{cor:bijection1,cor:bijection2} after observing that $\beta = \alpha^{-1}\circ \Phi$. 
\end{proof}

\subsection{Equivariance with respect to a group action and examples}
\label{subsec:action}

Let $H$ be a finite group which acts on $W$ and $S$ such that it acts as the identity over $W_0=K$ and $S_0=K$. Assume that for $h\in H$ and $g_1,g_2\in S$, $hg_1g_2=(hg_1)(hg_2)$. Then, let $\mathcal{H}$ be the set of $L\in\mathcal{T}$ such that for all $f\in W$, $g\in S$, and $h\in H$,
\begin{equation}
\label{eq:equivariant}
L(hg)hf = hL(g)f.
\end{equation}
Let $V$ and $R$ be the set of $w\in W$ such that $hw=w$ and $s\in S$ such that $hs=s$, respectively, for all $h\in H$. In particular, for all $i\geq 0$, $V_i$ and $R_i$ are the sets of $w\in W_i$ and $s\in S_i$, respectively, such that $hw=w$ and $hs=s$ for all $h\in H$. We have that $R\triangleq\oplus_{i\geq 0} R_i$ forms a graded ring since $hg_1g_2=(hg_1)(hg_2)$ for all $g_1,g_2\in S$. Due to this, if $r_1\in R_{i_1}$ and $r_2\in R_{i_2}$, then $hr_1r_2=(hr_1)(hr_2)=r_1r_2$ so $r_1r_2\in R_{i_1+i_2}$.

\begin{remark}
\label{remark:closed1}
It is evident that if $v\in V_i$, then $hv = v \in V_i$, so $V_i$ is closed under the action of $H$ and $R_i$ is closed under the action as well. However, this is not necessarily the case for $W_i$ or $S_i$. Due to this, we require additional arguments to prove \Cref{lemma:restrictinv1,lemma:restrictinv2} for the cases where $W_i$ and $S_i$, respectively, are not closed under the action of $H$. 
\end{remark}

\begin{lemma}
\label{lemma:closed}
If $L\in \mathcal{H}$, then $L$ is $(R,V)$-closed.
\end{lemma}

\begin{proof}
Suppose $f\in V$, $g\in R$, and $h\in H$. It suffices to prove that $hL(g)f=L(g)f$. By \pref{eq:equivariant},
\[
hL(g)f = L(hg)hf = L(g)f,
\]
which finishes the proof.
\end{proof}

\begin{lemma}
\label{lemma:restrictinv1}
Assume that $\text{char}(K)$ does not divide $|H|$ and that $L\in\mathcal{H}$ is invertible. Then, $L_{(R,V)}$ is invertible.
\end{lemma}

\begin{proof}
For the sake of contradiction, assume that $L_{(R,V)}$ is not invertible. Suppose $i\geq 1$ and $g\in R_i$ satisfies $L(g)f=0$ for all $f\in V_i$. Suppose $f\in W_i$. Then, for all $h\in H$,
\[
L(g)hf = L(hg)hf = hL(g)f = L(g)f
\]
by \pref{eq:equivariant}, since $L(g)f\in K$ and $H$ acts as the identity over $K$. Thus, we have that
\[
K\ni L(g)f = \frac{1}{|H|}\sum_{h\in H} L(g)hf = L(g)\left(\frac{1}{|H|}\sum_{h\in H} hf\right),
\]
since $|H|$ is not divisible by $\text{char}(K)$. 

Observe that $\frac{1}{|H|}\sum_{h\in H} hf\in V$, so the projection $f'$ of $\frac{1}{|H|}\sum_{h\in H} hf$ onto $V_i$ is well defined. Then, 
because $L(g)f\in K$, we have that
\[
L(g)f = L(g)\left(\frac{1}{|H|}\sum_{h\in H} hf\right) = L(g)f' = 0,
\]
which is a contradiction.
\end{proof}

The paper \cite{dunkl_singular_poly} shows that in the setting of Dunkl operators, the invertibility of $L_{(R,V)}$ implies the invertibility of $L$. We repeat the argument from Remark 4.2 of that paper in a more general setting to prove the following result.

\begin{lemma}
\label{lemma:restrictinv2}
Assume that $S$ is an integral domain. If $L\in\mathcal{L}$ is $(R,V)$-closed and $L_{(R,V)}$ is invertible, then $L$ is invertible.
\end{lemma}

\begin{proof}
For the sake of contradiction, assume that $L$ is not invertible. Suppose $i\geq 1$ and $g\in S_i$ is nonzero and satisfies $L(g)f=0$ for all $f\in W_i$. Let
\[
g' = \prod_{h'\in H} h'g.
\]
For all $h\in H$, $hg' = \prod_{h'\in H} hh'g = g'$ by the commutativity of $S$, so $g'\in R$.

Assume that $h'\in H$ and $h'g=0$. Then, $(h')^{-1}h'g=0$, which is a contradiction to $g\not=0$. Hence, $h'g\not=0$. It follows that $g'\not=0$ because $S$ is an integral domain. 

We can express $g'$ as a unique linear combination of elements of $R_j$ for $j\geq 0$. Let $m$ be the smallest nonnegative integer $j$ such that the component of $R_j$ in this linear combination is nonzero. Then, $m\geq i\geq 1$ because we multiply by $g\in S_i$ to obtain $g'$ and $g'\not=0$. 

Suppose $f\in V_m$. Then, $L(g')f \in K$. Suppose $f'=\prod_{h'\in H\backslash\{1\}} L(h'g) f$. Then, $L(g)f'\in K$, so if $f''$ is the projection of $f'$ onto $W_0\oplus\cdots\oplus W_i$, we have that $L(g')f = L(g)f'=L(g)f''$. However, $L(g)f''=0$, which is a contradiction to the invertibility of $L_{(R,V)}$.
\end{proof}

\begin{corollary}
Assume that $S$ is an integral domain. Then, $\mathcal{I}=\mathcal{I}'$.
\end{corollary}

\begin{example}
\label{example:polyring}

We consider a more specific example. As an introduction, consider the following basic lemma, the proof of which we omit since it is straightforward.

\begin{lemma}
\label{lemma:welldefined}
Suppose $U$ is a vector space over $K$. Assume that $L: K[x_1,\ldots,x_N]\rightarrow\text{End}_K(U)$ is a $K$-linear ring homomorphism. Then, the operators $L(x_i)$ for $1\leq i\leq N$ are commutative. Conversely, for commutative operators $L_i\in\text{End}_K(U)$ for $1\leq i\leq N$, there exists a unique $K$-linear ring homomorphism $L:K[x_1,\ldots,x_N]\rightarrow \text{End}_K(U)$ such that $L(x_i)=L_i$ for $1\leq i\leq N$.
\end{lemma}

Suppose $N\geq 1$, $S=W=K[x_1,\ldots,x_N]$, and that $S_i=W_i$ is the group of polynomials over $K$ that are homogeneous of degree $i$ for $i\geq 0$. From \Cref{lemma:welldefined}, any $L\in\mathcal{T}$ is uniquely determined by $L(x_i)$ for $1\leq i\leq N$. Let $H$ be a finite group such that the action of $h$ on $K^N$ is a $K$-linear endomorphism for all $h\in H$. 

\begin{lemma}
Assume that $h\in H$. Then, for $i\geq 0$ and $f\in S_i$, $\{x\mapsto f(hx)\}\in S_i$.
\end{lemma}

\begin{proof}
Let $e_i=[\mathbf{1}\{i=j\}]_{j\in [N]}$ be an element of $K^N$ for all $i\in [N]$. Then, 
\[
hx = h\left(\sum_{i=1}^N x_i e_i\right) = \sum_{i=1}^N x_ih(e_i),
\]
so the coefficient of $e_i$ of $hx$ is an element of $S_1$ for all $i\in [N]$. Afterwards, the result is straightforward to deduce.
\end{proof}

\begin{remark}
In this example, $S_i=W_i$ are closed under the action of $H$, see \Cref{remark:closed1} for discussion about when this is not the case. 
\end{remark}

Hence, we can define the action of $H$ over $S$ to be $hf(x)\triangleq f(hx)$; indeed, the action of $H$ over $K$ is the identity and $hf_1(x)f_2(x)=(hf_1(x))(hf_2(x))$ for $f_1,f_2\in S$. Then, $R=V$ is the set of $f\in S$ such that $f(hx)=f(x)$ for all $h\in H$ and $\mathcal{H}$ is the set of $L\in\mathcal{T}$ such that
\begin{equation}
\label{eq:equivariance}
L(g(hx))f(hx) = (L(g)f)(hx).
\end{equation}
for all $f,g\in S$ and $h\in H$. In contrast to the proof of \Cref{thm:bijection}, we do not assume that $B_i\subset D_i$ for $i\geq 0$. Observe that if $k\geq 1$, $i\in [N]$, and $f:x\mapsto x_i^{2k}$, then $\sum_{h\in H} hf(x)$ is invariant under the action of $H$. If $K=\mathbb{R}$, this sum is nonzero since the identity element of $H$ acts as the identity transformation, so $R_{2k}$ is nonempty for all $k\geq 1$.

For $x\in K^N$, define $\varphi_x: S\rightarrow K, f\mapsto f(x)$. We revisit the conditions that we introduced earlier:
\begin{itemize}
    \item Both $\text{ker}(\text{Hom}_K(R,K))$ and $\text{ker}(\text{Hom}_K(S,K))$ only contain $0$, since we can set the homomorphism to be $\varphi_x$ for a fixed value of $x\in K^N$. See \Cref{thm:equivalences} for the relevant implications.
    \item The ring $S$ is left- and right-spanned by $S_1$, but the analogous condition for $R$ is not true. See \Cref{thm:equivalences,thm:notinvertible} for the relevant implications.
    \item The left annihilators $\text{Ann}_R(R\backslash R_0)$ and $\text{Ann}_S(S\backslash S_0)$ only contain $0$. See \Cref{thm:notinvertible} for the relevant implications.
\end{itemize}
\end{example}

In the following result, we discuss the $\varphi_x$-eigenvectors. The result generalizes \Cref{thm:dunkl_eigenfunction,thm:dunkl_eigenfunction_symmetry} since $L$ can be any invertible operator.

\begin{corollary}
\label{cor:eigenvectorformula2}
Assume that $L\in\mathcal{T}$ is invertible (resp. $L\in\mathcal{L}$ is invertible). Furthermore, suppose $x\in K^N$.
\begin{enumerate}
\item[(A)] There does not exist a nonzero solution $E\in \mathcal{F}(S)$ (resp. $E\in\mathcal{F}(R)$) to $L(g)E=g(x)E$ for all $g\in S$ (resp. $g\in R$) such that $E^0=0$.
\item[(B)] There exists a unique solution $E\in\mathcal{F}(S)$ (resp. $E\in\mathcal{F}(R)$) to $L(g)E=g(x)E$ for all $g\in S$ (resp. $g\in R$) such that $E^0=1$. Using the notation of \Cref{cor:vrequal}, the solution is
\[
E = \sum_{r\in D_i} \text{inv-col}(r; D_i)r(x)\,\,\,\left(resp. \sum_{r\in B_i} \text{inv-col}(r; B_i) r(x)\right).
\]
\end{enumerate}
\end{corollary}

\begin{proof}
For (A), see \Cref{prop:zeroconstant}. For (B), see condition (i) of \Cref{thm:equivalences} with $\Psi=\varphi_x$ and \Cref{lemma:eigenvectorformula}. We can apply this condition because, as we mentioned previously, $\text{ker}(\text{Hom}_K(R,K))=\text{ker}(\text{Hom}_K(S,K))=\{0\}$. 
\end{proof}

\begin{corollary}
Assume that $L\in\mathcal{H}$ is invertible and $\text{char}(K)$ does not divide $|H|$. Then, all conclusions of \Cref{cor:eigenvectorformula2} are true (this includes the conclusions over both ($R$,$V$) and ($S$, $W$)).
\end{corollary}

\begin{proof}
By applying the corollary to $L\in\mathcal{T}$, we already know that the conclusions over ($S$, $W$) are true. By \Cref{lemma:closed,lemma:restrictinv1}, $L$ is ($R$, $V$)-closed and $L_{(R,V)}$ is invertible. Then, applying the corollary to $L_{(R,V)}\in\mathcal{L}$ implies that the conclusions over ($R$, $V$) are true.\end{proof}

\begin{example}
\label{example:dunkl}
We describe the framework that we have developed in the context of Dunkl operators and prove a few well-known results. Consider the setting of \Cref{example:polyring}, but set $K=\mathbb{C}$ and $H$ to be $H(\mathcal{R})$, where $\mathcal{R}\subset\mathbb{R}^N$ is a finite root system.

Suppose $u_i\in\mathbb{R}^N$ for $1\leq i\leq N$. Then, we can define $L\in\mathcal{T}$ by $L(x_i)\triangleq \mathcal{D}_{u_i}(\mathcal{R}(\theta))$ after using \Cref{lemma:welldefined} and the commutativity of the Dunkl operators. As a special case, we have that $\mathcal{D}(\mathcal{R}(\theta))\in\mathcal{T}$, where we recall that $\mathcal{D}(\mathcal{R}(\theta))(x_i) = \mathcal{D}_i(\mathcal{R}(\theta))$ for $i\in [N]$. For $i\geq 0$, the matrix $M^{i;\,\mathcal{D}(\mathcal{R}(\theta))}$ has entry $(r,s)$ equal to $[r,s]_{\mathcal{R}(\theta)}$ for $r\in D_i$ and $s\in C_i$.

By \Cref{lemma:equivariance}, it is evident that $\mathcal{D}(\mathcal{R}(\theta))\in\mathcal{H}$ since \pref{eq:equivariance} is satisfied. Furthermore, $V=R=\mathbb{C}^{H(\mathcal{R})}[x_1,\ldots,x_N]$ and the operator $\mathcal{D}(\mathcal{R}(\theta))_{(R,V)}\in\mathcal{L}$ is equivalent to $\mathcal{D}_H(\mathcal{R}(\theta))$. By \Cref{lemma:restrictinv1,lemma:restrictinv2}, $\mathcal{D}(\mathcal{R}(\theta))$ is invertible if and only if $\mathcal{D}_H(\mathcal{R}(\theta))$ is invertible. 

Observe that $E_a^{\mathcal{R}(\theta)}(x)$ is the $\Psi$-eigenvector of $\mathcal{D}(\mathcal{R}(\theta))$ and $J_a^{\mathcal{R}(\theta)}(x)$ is the $\Psi$-eigenvector of $\mathcal{D}_H(\mathcal{R}(\theta))$ for $\Psi: r\mapsto r(a)$, where $a\in\mathbb{C}^N$ is fixed. Then, part (B) of \Cref{cor:eigenvectorformula2} implies that these functions are unique eigenvectors, which is included in the statements of \Cref{thm:dunkl_eigenfunction,thm:dunkl_eigenfunction_symmetry}.

The intertwining operator for $(\mathcal{D}(\mathcal{R}(\theta)), \partial)$ is well-studied, where $\partial\in\mathcal{T}$ satisfies $\partial(x_i)=\partial_i$ for $i\in [N]$; equivalently, $\partial \triangleq \mathcal{D}(\mathcal{R}(0))$. As an example of an application of \Cref{thm:equivalences}, the theorem implies the well-known result that $(\mathcal{D}(\mathcal{R}(\theta)),\partial)$ intertwines with a unique operator if and only if $\mathcal{D}(\mathcal{R}(\theta))$ is invertible. As explained in \cite{rosler_positivity}, when considering the setting of \Cref{thm:positivity}, the intertwiner is given by 
\[
Vf(a)\triangleq \int_{\mathbb{R}^N} f(\epsilon) d\nu^{\mathcal{R}(\theta)}_a(\epsilon)
\]
for $f\in\mathbb{C}[x_1,\ldots,x_N]$ and $a\in\mathbb{R}^N$. Since the Bessel function is 
\[
J_a^{\mathcal{R}(\theta)}(x) = \frac{1}{|H(\mathcal{R})|}V \sum_{g\in H(\mathcal{R})} e^{\product{gx,\cdot}}(a),
\]
the characteristic function of $\frac{1}{|H(\mathcal{R})|} \sum_{g\in H(\mathcal{R})} \nu_{ga}^{\mathcal{R}(\theta)}=\nu_a^{\text{sym; }\mathcal{R}(\theta)}$ is given by $J_a^{\mathcal{R}(\theta)}(\mathbf{i}x)$.

For the remainder of this example, we consider the setting of \Cref{prop:svd}. Suppose $i\geq 0$. After assuming that $A_i=B_i$, we are particularly interested about the case where a singular value decomposition with $U=V$ exits. It would follow that 
\[
E^i = \sum_{r\in A_i} \sigma(r)^{-1} \Psi(u_r)u_r,
\]
where $\sigma(r)$ is entry $(r,r)$ of $\Sigma$ and $u_r$ corresponds to column $r$ of $U$ for $r\in A_i$.

The paper \cite{okounkov_olshanki_2} discusses the singular value decomposition of the matrices $\\M^{i;\,\mathcal{D}_H(A^{N-1}(\theta))}$ for $i\geq 0$ and $\theta \geq 0$. The paper \cite{baker} generalizes this result and shows that such a decomposition exists for the matrices $M^{i;\,\mathcal{D}(A^{N-1}(\theta))}$ for $i\geq 0$ and $\theta\geq 0$. Furthermore, the paper \cite{baker2} discusses the singular value decomposition of the matrices $M^{i;\,\mathcal{D}_H(BC^N(\theta_0,\theta_1))}$ for $i\geq 0$ and $\theta_0,\theta_1\geq 0$. For each $i\geq 0$, in the first and third cases, $U=V$ correspond to symmetric Jack functions while in the second case, $U=V$ correspond to nonsymmetric Jack functions.

For $\theta\geq 0$ and even $i\geq 0$, it is straightforward to compute the decomposition of $M^{i;\,\mathcal{D}_H(D^N(\theta))}$ when we restrict the rows and columns to $p_\lambda$ for $\lambda\in\evenN[i]$ by using the decomposition of $M^{i;\,\mathcal{D}_H(BC^N(\theta,0))}$. By \pref{eq:type_d_simple}, the decomposition for the remaining entries is given by that of $M^{i;\,\mathcal{D}_H(BC^N(\theta,1))}$. Furthermore, another interesting question is, what are the decompositions of $M^{i;\,\mathcal{D}(BC^N(\theta_0,\theta_1))}$ for $i\geq 0$?
\end{example}

\section{Leading order terms for the type A Dunkl bilinear form}
\label{sec:a}

Based on \pref{eq:eigenval}, the matrix $([p_\lambda,p_\nu]_{A^{N-1}(\theta)})_{\lambda,\nu\in\Gamma}$ is essential for calculating the coefficients of $J_a^{A^{N-1}(\theta)}(x)$. For the continuation of this direction, see \Cref{sec:bessel_coeff}. The following result computes the leading order terms of the matrix in the regime $|\theta N|\rightarrow\infty$.

\begin{theorem}
\label{thm:leadingorder_a}
Suppose $k\geq 1$, $\lambda,\nu\in\Gamma[k]$, and $\ell(\lambda)\leq \ell(\nu)$. Then,
\[
[p_\lambda, p_\nu]_{A^{N-1}(\theta)} = \theta^{k-\ell(\nu)}\prod_{l=1}^{\ell(\nu)}\nu_l \pi(\nu) \left[\prod_{l=1}^{\ell(\nu)} x_{\nu_l}\right]\prod_{i=1}^{\ell(\lambda)} \sum_{\pi\in NC(\lambda_i)} \prod_{B\in \pi} x_{|B|}N^{k+\ell(\lambda)-\ell(\nu)} + R(N,\theta).
\]
where $R\in\mathbb{Q}[x,y]$ satisfies:
\begin{enumerate}
\item In each summand, the degree of $x$ is at most $\ell(\lambda)$ greater than the degree of $y$.
\item The degree $x$ is at most $k+\ell(\lambda)-\ell(\nu)-1$ and the degree of $y$ is at most $k-\ell(\nu)$.
\end{enumerate}
\end{theorem}

\begin{remark}
In the previous result, we do not require that $\lambda,\nu\in\Gamma_N$.
\end{remark}

First, we present a short proof of the theorem using the results of \cite{limits_prob_measures}. The proof lacks some details, but it includes the main ideas.

\begin{proof}[First proof]
We deduce the result using \cite{limits_prob_measures}*{Theorem 6.13}. First, replace $G_{\theta_N}(x_1,\ldots,\\x_N;\,\mu_N)$ with $\exp\left(\theta N\sum_{k\geq 1} c_kp_k\right)$, where $c_k\in\mathbb{C}$ is fixed for all $k\geq 1$. Then, the theorem implies that for all $\lambda\in\Gamma$,
\[
\lim_{N\rightarrow\infty}\frac{[1]\mathcal{D}(p_\lambda) \exp\left(\theta N\sum_{k\geq 1} c_kp_{(k)}\right)}{(\theta N)^{|\lambda|} N^{\ell(\lambda)}} = \prod_{i=1}^{\ell(\lambda)}\sum_{\pi\in NC(\lambda_i)}\prod_{B\in \pi} |B|c_{|B|}.
\]
For $\nu\in\Gamma$, the coefficient of $p_\nu$ in $\exp\left(\theta N \sum_{k\geq 1}c_kp_{(k)}\right)$ is $\frac{(\theta N)^{\ell(\nu)}}{\pi(\nu)}\prod_{i=1}^{\ell(\nu)}c_{\nu_i}$. Hence, the coefficient of $\prod_{i=1}^{\ell(\nu)}c_{\nu_i}$ in $[1]\mathcal{D}(p_\lambda) \exp\left(\theta N\sum_{k\geq 1} c_kp_{(k)}\right)$ is $\frac{(\theta N)^{\ell(\nu)}}{\pi(\nu)} [p_\lambda,p_\nu]_{A^{N-1}(\theta)}$. By finding the coefficient of $\prod_{i=1}^{\ell(\nu)}c_{\nu_i}$ on the right hand side, we obtain the formula in the statement of the theorem.
\end{proof}

Although this proof is short, we cannot use it while proving analogous results for the $BC^N(\theta_0,\theta_1)$ and $D^N(\theta)$ root systems. Therefore, we outline a new approach to justifying \Cref{thm:leadingorder_a} that can be easily applied to a general setting.

\subsection{Second proof of \texorpdfstring{\Cref{thm:leadingorder_a}}{}}
\label{subsec:proofa}

We have that 
\begin{equation}
\label{eq:expansion}
[p_\lambda, p_\nu]_{A^{N-1}(\theta)} = \mathcal{D}(p_\lambda)p_\nu = \prod_{i=1}^{\ell(\lambda)} \left(\sum_{j=1}^N \mathcal{D}_j^{\lambda_i-1}\partial_j\right) p_\nu = \sum_{j_1,\ldots,j_{\ell(\lambda)}=1}^N \prod_{i=1}^{\ell(\lambda)}(\mathcal{D}_{j_i}-\partial_{j_i}+\partial_{j_i})^{\lambda_i-1}\partial_{j_i} p_\nu. 
\end{equation}
Suppose $j_1,\ldots,j_{\ell(\lambda)}\in [N]$ and that for each $\mathcal{D}_{j_i}$ operator, we select either $\mathcal{D}_{j_i}-\partial_{j_i}$ (which only consists of switches) or $\partial_{j_i}$; the total number of such choices that we make is $k-\ell(\lambda)$. Let $d$ be the number of times we choose $\partial_{j_i}$.

Assume that $d< \ell(\nu)-\ell(\lambda)$. Then, the total number of derivatives is less than $\ell(\nu)$. Recall that when applying a derivative to a product of terms, we select one of the terms and apply the derivative to it, and we iterate over all such selections. Because the number of derivatives is less than $\ell(\nu)$, for every path, we will never apply a derivative to $p_{\nu_l}$ for some $l\in [\ell(\nu)]$. Then, we have that we will not be able to eliminate the symmetric term $p_{\nu_l}$ without eliminating the entire expression. This is straightforward to deduce, since we must eliminate the symmetric term using a switch. Thus, the total contribution in this case is zero.

Next, assume that $d\geq \ell(\nu)-\ell(\lambda)$. The total number of switches is $k-\ell(\lambda)-d$. The number of ways to select these switches is therefore $(N-1)^{k-\ell(\lambda)-d}$. Furthermore, the number of ways to select $j_1,\ldots,j_{\ell(\lambda)}$ is $N^{\ell(\lambda)}$. Note that for each selection of switches and $j_1,\ldots,j_{\ell(\lambda)}$, any term of $p_\nu$ which contains $x_i$ for some $i$ that is not included in a switch or $\{j_1,\ldots,j_{\ell(\lambda)}\}$ will have a contribution of zero. Hence, the sums of the coefficients of the terms of $p_{\nu}$ which may have nonzero contribution is at most $(k-d)^{\ell(\nu)} \leq k^{\ell(\nu)}$. Furthermore, for each sequence of switches and derivatives, the derivatives will increase the absolute value of the coefficient by at most $k^{\ell(\lambda)+d}$ while the switches will increase it by at most $(k|\theta|)^{k-\ell(\lambda)-d}$. Hence, for each selection of the switches and $j_1,\ldots,j_{\ell(\lambda)}$, the magnitude of the total contribution is at most $k^{k+\ell(\nu)}|\theta|^{k-\ell(\lambda)-d}$. Then, the magnitude of the total contribution is less than
\[
k^{k+\ell(\nu)}|\theta|^{k-\ell(\lambda)-d}N^{k-d}.
\]
To achieve the leading order term, we must have that $d=\ell(\nu)-\ell(\lambda)$. Furthermore, each of the indices $j_1,\ldots,j_{\ell(\lambda)}$ must be distinct, and each switch must be from some $j_i$ to a distinct index which does not appear in $\{j_1,\ldots,j_{\ell(\lambda)}\}$. Then, the leading order term is $N^{k-\ell(\nu)+\ell(\lambda)}$ multiplied by a constant that depends on $\theta$.

It is not challenging to determine that the remainder term $R$ is a polynomial with rational coefficients that satisfies conditions (1) and (2). To see that $R$ is a polynomial with rational coefficients, we perform case work on the choices of $j_1,\ldots,j_{\ell(\lambda)}$ and the switches. For condition (1), we note that after selecting $j_1,\ldots,j_{\ell(\lambda)}$, the only way to obtain a factor of $N$ is with a switch, which will in turn add a factor of $\theta$. For condition (2), we note that we have separated the leading order term in $x$, and that $d\geq \ell(\nu)-\ell(\lambda)$ so the number of switches is at most $k-\ell(\nu)$ to obtain the maximal degree in $y$. For more details of a similar argument, see \cite{limits_prob_measures}*{Section 6}.

The final step is to compute the coefficient of $N^{k+\ell(\lambda)-\ell(\nu)}$. Recall that the indices $j_1,\ldots,j_{\ell(\lambda)}$ are distinct and that the switches are from some $j_i$ to a distinct index not in $\{j_1,\ldots,j_{\ell(\lambda)}\}$. The total number of choices for the indices is therefore
\[
N(N-1)\cdots(N-k-\ell(\lambda)+\ell(\nu)+1) = (1+o_N(1))N^{k+\ell(\lambda)-\ell(\nu)},
\]
so the coefficient of $N^{k+\ell(\lambda)-\ell(\nu)}$ is the sum of the contributions of the sequences after we fix the indices. 

In particular, let $\mathcal{S}$ denote the set of sequences of operators $s=\{s_i\}_{1\leq i\leq |\lambda|}$ such that:
\begin{enumerate}
    \item For $1\leq i\leq \ell(\lambda)$, $s_{1+\lambda_1+\cdots+\lambda_{i-1}} = \partial_i$, and for $j\in [1+\lambda_1+\cdots+\lambda_{i-1}, \lambda_1+\cdots+\lambda_i]$, $s_j$ is a switch from $i$ to an element of $[N]\backslash\{i\}$ or is $\partial_i$ (this corresponds to $j_i=i$).
    \item For $\ell(\nu)-\ell(\lambda)$ elements $j$ of $[|\lambda|]\backslash \{1, 1+\lambda_1, \ldots, 1+\lambda_1+\cdots+\lambda_{\ell(\lambda)-1}\}$, $s_j$ is a derivative (this corresponds to the total number of derivatives being $\ell(\nu)$).
    \item The $j$th switch is from some element of $[\ell(\lambda)]$ (which is determined by condition (1)) to $\ell(\lambda)+j$ for $1\leq j\leq k - \ell(\nu)$. Note that the operators $s_i$ are ordered from $i=1$ to $|\lambda|$, meaning that $s_1$ appears first and $s_{|\lambda|}$ appears last.
\end{enumerate}
Then, the coefficient of $N^{k+\ell(\lambda)-\ell(\nu)}$ is
\begin{equation}
\label{eq:leadingorder}
\sum_{s\in \mathcal{S}} s_{|\lambda|}\circ\cdots\circ s_1 p_\nu.
\end{equation}

Recall that in order to have a nonzero contribution, we must apply at least one derivative to each of the $p_{\nu_l}$. In this case, the number of derivatives is exactly $\ell(\nu)$, so we apply exactly one derivative to each of the $p_{\nu_l}$.

Furthermore, note that when we apply the $j$th switch from $i$ to $\ell(\lambda)+j$, $x_{\ell(\lambda)+j}$ does not appear before we apply the switch, other than among the $p_{\nu_l}$ that have yet to be assigned to. Also, any terms that contain $x_{\ell(\lambda)+j}$ after applying the switch will have a contribution of zero, because none of the remaining derivatives or switches will be able to remove the variable $x_{\ell(\lambda)+j}$. Note that applying the switch from $i$ to $\ell(\lambda)+j$ to $x_i^e$ outputs $\theta(x_i^{e-1}+x_i^{e-2}x_{\ell(\lambda)+j}+\cdots+x_{\ell(\lambda)+j}^{e-1})$, and the only resulting term that does not contain $x_{\ell(\lambda)+j}$ is $\theta x_i^{e-1}$. It follows that the switch from $i$ to $\ell(\lambda)+j$ is equivalent to $\theta d_i$.

\begin{remark}
\label{remark:scale}
The derivatives multiply the leading order coefficient by $\prod_{l=1}^{\ell(\nu)}\nu_l$ and the factors of $\theta$ from the switches multiply the coefficient by $\theta^{k-\ell(\nu)}$. These expressions appear in the statement of \Cref{thm:leadingorder_a}.
\end{remark}

To compute \pref{eq:leadingorder}, we must choose the locations of the $\ell(\nu)-\ell(\lambda)$ unallocated derivatives and then assign each of the $\ell(\nu)$ derivatives to a distinct symmetric term $p_{\nu_l}$. In particular, for $s\in\mathcal{S}$, define $\mathcal{P}(s)\triangleq \{i\in[|\lambda|]: s_i \text{ is a derivative}\}$ and let $\mathcal{H}(s)$ denote the set of bijections $\zeta: \mathcal{P}(s) \rightarrow [\ell(\nu)]$. Then for $s\in\mathcal{S}$ and $\zeta\in\mathcal{H}(s)$, the pair $(s,\zeta)$ corresponds to applying the sequence $s$ of operators and for $i\in\mathcal{P}(s)$, allocating the derivative $s_i$ to $p_{\nu_{\zeta(i)}}$.

Consider the following process to compute the contribution after choosing $s$ and $\zeta$. We start with the coefficient $c=1$ and iterate over the operators of $s$ from $s_1$ to $s_{|\lambda|}$. If we apply $\partial_{j_i}=\partial_i$ to $p_{\nu_l}$ at a step, where $l$ is determined by $\zeta$, then we multiply $c$ by $\partial_i p_{\nu_l}$. Otherwise, if we apply the $j$th switch from $i$ to $\ell(\lambda)+j$, then we apply $\theta d_i$ to $c$. Based on the previous discussion, the final value of $c$ will be the contribution from $(s,\zeta)$, and we denote this value of $c$ by $c(s,\zeta)$. We then have that 
\begin{equation}
\label{eq:contribution1}
\sum_{s\in \mathcal{S}} s_{|\lambda|}\circ\cdots\circ s_1 p_\nu = \sum_{s\in\mathcal{S},\,\zeta\in\mathcal{H}(s)} c(s,\zeta).
\end{equation}

Suppose $s\in\mathcal{S}$, $\zeta\in \mathcal{H}(s)$, and $i\in[\ell(\lambda)]$. Define 
\[
S_i(s,\zeta)\triangleq \{\zeta(j): j\in [1+\lambda_1+\cdots+\lambda_{i-1}, \lambda_1+\cdots+\lambda_i]\cap \mathcal{P}(s)\};
\]
that is, $S_i$ is the set of $l$ such that $p_{\nu_l}$ is assigned to by some $\partial_i$. 

\begin{lemma}
\label{lemma:partition_a}
Suppose $s\in\mathcal{S}$, $\zeta\in\mathcal{H}(s)$, and $c(s,\zeta)\not=0$. Then
\[
\sum_{l\in S_i(s,\zeta)} |\nu_l| = \lambda_i
\]
for all $i\in [\ell(\lambda)]$.
\end{lemma}
\begin{proof}
Consider $i=1$. If $\sum_{l\in S_1(s,\zeta)} |\nu_l|<\lambda_1$, then $c(s,\zeta)=0$ will be zero. This is because among the first $\lambda_1$ operators, we apply $\lambda_1-|S_1(s,\zeta)|$ switches to $c$ and we multiply $c$ by the polynomials $\partial_1 p_{\nu_l}$ for each $l\in S_1(s,\zeta)$. Since we start at $c=1$ and the total degree of the multiplied polynomials is less than the number of switches, $c(s,\zeta)=0$. Furthermore, if $\sum_{l\in S_1(s,\zeta)} |\nu_l|>\lambda_1$, then $c(s,\zeta)=0$. This is because after the first $\lambda$ operators, $c(s,\zeta)$ will be a multiple of $x_1$. We will not be able to remove the factor of $x_1$ using later operators, so the final contribution will be zero. Therefore, for $c(s,\zeta)$ to be nonzero, $\sum_{l\in S_1(s,\zeta)}|\nu_l|=\lambda_1$. In particular, after applying the first $\lambda_1$ operators, $c$ will be a constant.

We can apply the same argument using induction to deduce that $\sum_{l\in S_i(s,\zeta)} |\nu_l| = \lambda_i$ for all $i\in [\ell(\lambda)]$. 
\end{proof}

Let $\mathcal{G}$ denote the set of $(s,\zeta)$ such that $\sum_{l\in S_i(s,\zeta)} |\nu_l| = \lambda_i$ for all $i\in [\ell(\lambda)]$. Based on the proof of \Cref{lemma:partition_a}, note that for each block of $\lambda_i$ operators, we multiply $c$ by some constant factor. For $(s,\zeta)\in\mathcal{G}$, we let $C_i(s,\zeta)$ denote the factor we multiply $c$ by when applying the block of $\lambda_i$ operators $\{s_j: j\in [1+\lambda_1+\cdots+\lambda_{i-1}, \lambda_1+\cdots+\lambda_i]\}$ associated with $j_i=i$. Then, we have that for $(s,\zeta)\in\mathcal{G}$,
\begin{equation}
\label{eq:contribution2}
c(s,\zeta)=\prod_{i=1}^{\ell(\lambda)} C_i(s,\zeta),
\end{equation}

Afterwards, it is evident that if we condition on the values $S_i(s,\zeta)$, then the blocks of $\lambda_i$ operators will be independent, since we can allocate the locations of the derivatives and assign the derivatives to the $p_{\nu_l}$ independently within each block. Suppose $[\ell(\nu)]=S_1\sqcup \cdots \sqcup S_{\ell(\lambda)}$ such that $\sum_{l\in S_i} |\nu_l|=\lambda_i$ for all $i\in [\ell(\lambda)]$. Then, conditioned on $S_i(s,\gamma)=S_i$ for all $i\in [\ell(\lambda)]$, the total contribution will be the product of the contributions from each of the blocks. In particular, the main idea for the next step is that the contribution to the leading order coefficient of $[p_\lambda,p_\nu]_{A^{N-1}(\theta)}$ when we condition on the sets $S_i$, $1\leq i\leq \ell(\lambda)$ will be the leading order coefficient of
\[
\prod_{i=1}^{\ell(\lambda)} \left[p_{\lambda_i}, \prod_{j\in S_i} p_{\nu_j}\right]_{A^{N-1}(\theta)}.
\]
So, we only need to consider the case where $\ell(\lambda)=1$ and then sum over the partitions $[\ell(\nu)]=S_1\sqcup\cdots\sqcup S_{\ell(\lambda)}$. We justify this idea rigorously in the following lemma.

\begin{lemma}
\label{lemma:contribution1}
For $\nu'\in \Gamma$, let $c_{\nu'}$ denote the quantity $\sum_{s\in\mathcal{S},\,\zeta\in\mathcal{H}(s)}c(s,\zeta)$ from \pref{eq:contribution1} when $\nu$ is set as $\nu'$ and $\lambda$ is set as $(|\nu'|)$. Suppose $[\ell(\nu)]=S_1\sqcup \cdots \sqcup S_{\ell(\lambda)}$ such that $\sum_{l\in S_i} |\nu_l|=\lambda_i$ for all $i\in [\ell(\lambda)]$. Then,
\[
\sum_{\substack{(s,\zeta)\in\mathcal{G}, \\ S_i(s,\zeta)=S_i, 1\leq i\leq \ell(\lambda)}} c(s,\zeta) = \prod_{i=1}^{\ell(\lambda)} c_{\gamma((\nu_l:\,l\in S_i))}.
\]
\end{lemma}
\begin{remark}
Note that $c_{\nu'}$ is simply the leading order coefficient of $[p_{(|\nu'|)},p_{\nu'}]_{A^{N-1}(\theta)}$.
\end{remark}
\begin{proof}[Proof of \Cref{lemma:contribution1}]
For $s\in \mathcal{S}$ and $i\in [\ell(\lambda)]$, define $s[i]\triangleq \{s_j: j\in [1+\lambda_1+\cdots+\lambda_{i-1},  \lambda_1+\cdots+\lambda_i]\}$ and $\mathcal{P}_i(s)\triangleq \mathcal{P}(s)\cap [1+\lambda_1+\cdots+\lambda_{i-1}, \lambda_1+\cdots+\lambda_i]$. 

Let $\mathcal{T}$ denote the set of $s\in\mathcal{S}$ such that $|\mathcal{P}_i(s)|=|S_i|$ for all $i\in [\ell(\lambda)]$ and for $i\in[\ell(\lambda)]$, define $\mathcal{T}_i\triangleq \{s[i]: s\in\mathcal{T}\}$. Observe that
\[
\mathcal{T}\cong \mathcal{T}_1\otimes\cdots\otimes\mathcal{T}_{\ell(\lambda)},
\]
since the number of switches in each element of $\mathcal{T}_i$ is fixed at $\lambda_i-|S_i|$; the bijection is given by $s\mapsto s[1]\times\cdots\times s[\ell(\lambda)]$.

Suppose $i\in[\ell(\lambda)]$. Define the function $\mathcal{P}_i'$ over $\mathcal{T}_i$ by $\mathcal{P}_i'(s[i])\triangleq \mathcal{P}_i(s)$ for $s\in\mathcal{T}$; this function is well defined. Furthermore, for $s\in\mathcal{T}$, define $\mathcal{H}_i(s[i])$ to be the set of bijections $\zeta:\mathcal{P}_i'(s[i])\rightarrow S_i$. Then,
\begin{align*}
& \{(s,\zeta): s\in\mathcal{T},\,\zeta\in\mathcal{H}(s),\, S_i(s,\zeta)=S_i\,\forall i\in [\ell(\lambda)]\}\cong \\
& \bigotimes_{i=1}^{\ell(\lambda)} \{(s, \zeta): s\in\mathcal{T}_i,\, \zeta\in \mathcal{H}_i(s)\},
\end{align*}
where the bijection is given by $(s,\zeta)\mapsto (s[1], \zeta|_{\mathcal{P}_1(s)})\times\cdots\times (s[\ell(\lambda)],\zeta|_{\mathcal{P}_{\ell(\lambda)}(s)})$. Furthermore, define the function $C_i'$ over $\{(s, \zeta): s\in\mathcal{T}_i,\, \zeta\in \mathcal{H}_i(s)\}$ by $C_i'(s[i],\zeta|_{\mathcal{P}_i(s)})\triangleq C_i(s,\zeta)$ for $s\in\mathcal{T}$ and $\zeta\in\mathcal{H}(s)$ such that $(s,\zeta)\in \mathcal{G}$ and $S_i(s,\zeta)=S_i$.

After using \pref{eq:contribution2}, it is evident that
\begin{align*}
\sum_{\substack{(s,\zeta)\in\mathcal{G}, \\ S_i(s,\zeta)=S_i, 1\leq i\leq \ell(\lambda)}} c(s,\zeta) 
& = \sum_{\substack{(s,\zeta)\in\mathcal{G}, \\ S_i(s,\zeta)=S_i, 1\leq i\leq \ell(\lambda)}}\prod_{i=1}^{\ell(\lambda)} C_i(s,\zeta) \\ & = \sum_{\substack{s\in\mathcal{T}, \zeta\in \mathcal{H}(s), \\S_i(s,\zeta)=S_i, 1\leq i\leq \ell(\lambda)}} \prod_{i=1}^{\ell(\lambda)} C_i(s,\zeta) \\
& = \sum_{\substack{\mathbf{s}_i\in\mathcal{T}_i,\,\zeta_i\in\mathcal{H}_i(\mathbf{s}_i),\\ 1\leq i\leq \ell(\lambda)}} \prod_{i=1}^{\ell(\lambda)} C_i'(\mathbf{s}_i,\zeta_i)  \\
& = \prod_{i=1}^{\ell(\lambda)} \sum_{\mathbf{s}_i\in\mathcal{T}_i,\,\zeta_i\in\mathcal{H}_i(\mathbf{s}_i)} C_i'(\mathbf{s}_i,\zeta_i).
\end{align*}
However, observe that for $i\in[\ell(\lambda)]$, $\sum_{\mathbf{s}_i\in\mathcal{T}_i,\,\zeta_i\in\mathcal{H}_i(\mathbf{s}_i)} C_i'(\mathbf{s}_i,\zeta_i)$ is the same as the quantity $\sum_{s\in\mathcal{S},\zeta\in\mathcal{H}(s)}c(s,\zeta)$ from \pref{eq:contribution1} for when $\lambda$ is set as $(\lambda_i)$ and $\nu$ is set as $\gamma((\nu_l: l\in S_i))$. Then, we have that 
\[
\sum_{\substack{(s,\zeta)\in\mathcal{G}, \\ S_i(s,\zeta)=S_i, 1\leq i\leq \ell(\lambda)}}\prod_{i=1}^{\ell(\lambda)} C_i(s,\zeta) = \prod_{i=1}^{\ell(\lambda)} c_{\gamma((\nu_l: l\in S_i))}.
\]
\end{proof}

Since using \pref{eq:contribution1} and \Cref{lemma:contribution1} gives that 
\begin{equation}
\label{eq:contribution3}
\begin{split}
\sum_{s\in \mathcal{S}} s_{|\lambda|}\circ\cdots\circ s_1 p_\nu  & = \sum_{\substack{[\ell(\nu)]=S_1\sqcup \cdots \sqcup S_{\ell(\lambda)}, \\ \sum_{l\in S_i} |\nu_l|=\lambda_i\,\forall i\in [\ell(\lambda)]}}\sum_{\substack{(s,\zeta)\in\mathcal{G}, \\ S_i(s,\zeta)=S_i, 1\leq i\leq \ell(\lambda)}} c(s,\zeta) \\
& = \sum_{\substack{[\ell(\nu)]=S_1\sqcup \cdots \sqcup S_{\ell(\lambda)}, \\ \sum_{l\in S_i} |\nu_l|=\lambda_i\,\forall i\in [\ell(\lambda)]}} \prod_{i=1}^{\ell(\lambda)} c_{\gamma((\nu_l: l\in S_i))},
\end{split}
\end{equation}
it suffices to compute the values of $c_{\nu'}$.

\begin{lemma}
\label{lemma:allocation}
For $\nu'\in \Gamma$, let $n_{\nu'}$ equal the number of noncrossing partitions of $[|\nu'|]$ with block size given by $\nu'$. Then, $c_{\nu'}=\theta^{|\nu'|-\ell(\nu')}\prod_{l=1}^{\ell(\nu')} \nu'_l \pi(\nu') n_{\nu'}$
\end{lemma}
\begin{remark}
The $\pi(\nu')$ factor arises from the fact that we must order the blocks of the same size.
\end{remark}

\begin{proof}[Proof of \Cref{lemma:allocation}]
In this proof, we use the same notation for $\mathcal{S}$, $\mathcal{H}(s)$ for $s\in\mathcal{S}$, and $c(s,\zeta)$ for $s\in\mathcal{S}$ and $\zeta\in\mathcal{H}(s)$ that is mentioned earlier in the subsection after we set $\lambda$ to be $(|\nu'|)$ and $\nu$ to be $\nu'$. First, recall that by definition,
\[
c_{\nu'} = \sum_{s\in\mathcal{S},\,\zeta\in\mathcal{H}(s)} c(s,\zeta).
\]
The number of derivatives $\partial_1$ in $s\in\mathcal{S}$ is $\ell(\nu')$ and the number of switches, which are each equivalent to $\theta d_1$, is $|\nu'|-\ell(\nu')$. Suppose $n'_{\nu'}$ is the number of ways to choose the locations of the $\ell(\nu')$ derivatives of $s$ and assign them to distinct $p_{\nu'_l}$ for $l\in [\ell(\nu')]$ such that the resulting contribution $c(s,\zeta)$ is nonzero. Then $c_{\nu'}=\theta^{|\nu'|-\ell(\nu')}\prod_{l=1}^{\ell(\nu')} \nu'_l n'_{\nu'}$ by the argument in \Cref{remark:scale}. Hence, it suffices to show that $n'_{\nu'}=\pi(\nu')n_{\nu'}$.

Suppose $s\in\mathcal{S}$ and $s=\{s_i\}_{1\leq i\leq |\nu'|}$; recall that $s_1=\partial_1$. Furthermore, suppose $\zeta\in\mathcal{H}(s)$. Assume that $c(s,\zeta)\not=0$, or equivalently that $c(s,\zeta)=\theta^{|\nu'|-\ell(\nu')}\prod_{l=1}^{\ell(\nu')} \nu'_l$. Then, we construct the noncrossing partition $\alpha(s,\zeta)$ of $[|\nu'|]$ using the following procedure.
\begin{enumerate}
\item Initiate $n=1$. Let $\mathcal{A}$ be a stack which is initially empty.
\item Suppose $s_n=\partial_1$ and $\zeta(n)=l$ so that $s_n$ is assigned to $\nu'_l$. Then, we add $(\{n\}, \nu'_l)$ to the top of $\mathcal{A}$.
\item Suppose $s_n$ is a switch. Then, iterate through $\mathcal{A}$ from top to bottom. For the first element $(S, L)$ such that $|S|<L$, add $n$ to $S$.
\item Increment $n$ by one.
\item If $n=|\nu'|+1$, then for the noncrossing partition, output the union of the sets $S$ for $(S,L)\in \mathcal{A}$.
\item If $n\leq |\nu'|$, then return to step (2).
\end{enumerate}
Since $c(s,\zeta)\not=0$, it is straightforward to deduce that we can always preform step (3) by counting the degree of $x_1$. Then, the output partition of $[|\nu'|]$ will have block sizes given by $\nu'$ and will be noncrossing due to the structure of $\mathcal{A}$. See \Cref{fig:mapping} for an example of an output of this mapping.

To reverse the algorithm, suppose $\sigma$ is a noncrossing partition of $[|\nu'|]$ with blocks sizes given by $\nu'$. Suppose $\sigma=B_1\sqcup \cdots \sqcup B_{\ell(\nu')}$, where $B_q$ starts before $B_{q+1}$ for $1\leq q\leq \ell(\nu')-1$. Suppose $q\in [\ell(\nu')]$ and that the first element of $B_q$ is $b$. Then, set $s_b$ to be $\partial_1$ and assign $s_b$ to $p_{\nu'_l}$ for some $l\in [\ell(\nu')]$ such that $p_{\nu'_l}$ has not been previously assigned to and $|\nu'_l|=|B_q|$; note that we then have that $\zeta(b)=l$. Afterwards, we set the remaining operators to be switches based on condition (3) in the definition of $\mathcal{S}$.

It is evident that the resulting pair $(s, \zeta)$ satisfies $\alpha(s, \zeta)=\sigma$ and that the number of choices for $(s, \zeta)$ is $\pi(\nu')$. This establishes that $n'_{\nu'}=\pi(\nu')n_{\nu'}$.
\end{proof}

\begin{figure}[!ht]
\begin{center}
\begin{tikzpicture}
    \fill[red!30] (0,0) rectangle (3,1);
    \fill[blue!30] (3,0) rectangle (4, 1);
    \fill[green!30] (4, 0) rectangle (6,1);
    \fill[blue!30] (6,0) rectangle (8,1);
    \fill[yellow!30] (8, 0) rectangle (9,1);
    \fill[cyan!30] (9, 0) rectangle (11,1);
    \fill[red!30] (11,0) rectangle (13,1);
    \draw (0,0) grid (13,1);
    \node at (0.5, 0.5) {$\partial$};
    \node at (3.5, 0.5) {$\partial$};
    \node at (4.5, 0.5) {$\partial$};
    \node at (8.5, 0.5) {$\partial$};
    \node at (9.5, 0.5) {$\partial$};
\end{tikzpicture}
\end{center}

\caption{An illustration of an output of the mapping if the input sequence of operators has the form $\partial S S \partial \partial S S S \partial S \partial S S$, where $\partial$ indicates a derivative and $S$ indicates a switch. In this case, we have that $\nu'=(5, 3, 2, 2, 1)$. If we number the derivatives with $1$ through $5$ from left to right, one of the choices for $\zeta$ is to map derivatives $1$ through $5$ to parts $1$, $2$, $3$, $5$, and $4$ of $\nu'$, respectively. Each color corresponds to a block of the noncrossing partition and a part of $\nu'$. Observe that the number of choices for $\zeta$ is $\pi(\nu')=2$, since we can switch parts $3$ and $4$ of $\nu'$.
}
\label{fig:mapping}
\end{figure}

Using \pref{eq:contribution3} and \Cref{lemma:allocation}, it then follows that the leading order coefficient is
\begin{align*}
& \sum_{[\ell(\nu)] = S_1\sqcup \cdots \sqcup S_{\ell(\lambda)}} \theta^{k-\ell(\nu)}\prod_{l=1}^{\ell(\nu)} \nu_l \times \\
& \prod_{i=1}^{\ell(\lambda)} \pi((\nu_j: j\in S_i))(\#\,\text{noncrossing partitions of } [\lambda_i] \text{ with shape } \gamma((\nu_j: j\in S_i))) \\
& = \sum_{\nu = \gamma_1 + \cdots + \gamma_{\ell(\lambda)}} \theta^{k-\ell(\nu)}\prod_{l=1}^{\ell(\nu)}\nu_l \pi(\nu) \prod_{i=1}^{\ell(\lambda)} (\#\,\text{noncrossing partitions of } [\lambda_i] \text{ with shape } \gamma_i) \\
& = \theta^{k-\ell(\nu)}\prod_{l=1}^{\ell(\nu)}\nu_l \pi(\nu) \left[\prod_{l=1}^{\ell(\nu)} x_{\nu_l}\right]\prod_{i=1}^{\ell(\lambda)} \sum_{\pi\in NC(\lambda_i)} \prod_{B\in \pi} x_{|B|}.
\end{align*}

\subsection{Proof of part (A) of \texorpdfstring{\Cref{thm:main}}{}}

Define the infinite dimensional matrix $\ma\in\mathbb{Z}_{\geq 0}^{\Gamma\times\Gamma}$ such that for $\lambda,\nu\in\Gamma$,
\[
\ma_{\lambda\nu}\triangleq \prod_{l=1}^{\ell(\nu)}\nu_l\pi(\nu)\left[\prod_{l=1}^{\ell(\nu)} x_{\nu_l}\right]\prod_{i=1}^{\ell(\lambda)} \sum_{\pi\in NC(\lambda_i)} \prod_{B\in \pi} x_{|B|}.
\]
The following lemma is straightforward to deduce.
\begin{lemma}
\label{lemma:matrix_a1} Suppose $\lambda,\nu\in\Gamma$ and $|\lambda|=|\nu|$.
\begin{enumerate}
\item[(A)] If $\lambda=\nu$, then $\ma_{\lambda\nu}=\prod_{l=1}^{\ell(\nu)}\nu_l\pi(\nu)$.
\item[(B)] If $\ell(\lambda)\geq \ell(\nu)$ and $\lambda\not=\nu$, then $\ma_{\lambda\nu}=0$.
\end{enumerate}
\end{lemma}

The next theorem generalizes part (A) of \Cref{thm:main}. As we mentioned previously, we focus on the less general version stated in \Cref{cor:equivalence_a1} because it is more relevant for analyzing the convergence of Bessel functions.

\begin{theorem} 
\label{thm:equivalence_a1}
Suppose $F_N(x_1,\ldots,x_N) = \exp\left(\sum_{\lambda\in\Gamma_N} c_{\lambda}(N)p_\lambda\right)$. Assume that $\lim_{N\rightarrow\infty} \\|\theta N| = \infty$. The following are equivalent. 
\begin{enumerate}
\item[(a)] For all $\lambda\in\Gamma$, $\lim_{N\rightarrow\infty} \frac{c_\lambda(N)}{(\theta N)^{\ell(\lambda)}}=c_\lambda\in\mathbb{C}$. 
\item[(b)] For all $\nu\in\Gamma$,
\[
\lim_{N\rightarrow\infty} \frac{1}{(\theta N)^{|\nu|}N^{\ell(\nu)}} [p_\nu, F_N]_{A^{N-1}(\theta)} = \sum_{\lambda\in\Gamma} \ma_{\nu\lambda}[p_\lambda] \exp\left(\sum_{\gamma\in\Gamma} c_\gamma p_\gamma\right).
\]
\end{enumerate}

\end{theorem}
\begin{proof}
Suppose $\nu\in\Gamma$ and $N\geq 2$ satisfies $\nu\in\Gamma_N$. We have that
\[
[p_\nu, F_N]_{A^{N-1}(\theta)} = \sum_{\lambda\in\Gamma_N} [p_\lambda]F_N\cdot [p_\nu, p_\lambda]_{A^{N-1}(\theta)}.
\]
Then, using \Cref{thm:leadingorder_a} after noting that each summand of $R$ is $o_N(|\theta N|^{|\nu|}N^{\ell(\nu)})$, we get that
\begin{equation}
\label{eq:limit}
\lim_{N\rightarrow\infty} \frac{1}{(\theta N)^{|\nu|}N^{\ell(\nu)}} [p_\nu, F_N]_{A^{N-1}(\theta)} = \lim_{N\rightarrow\infty} \sum_{\lambda\in\Gamma_N[|\nu|]} (\ma_{\nu\lambda}+o_N(1))\frac{[p_\lambda]F_N}{(\theta N)^{\ell(\lambda)}}. 
\end{equation}
We justify this equation. For $\lambda\in\Gamma_N$ such that $\ell(\lambda)<\ell(\nu)$, we have that
\[
[p_\nu, p_\lambda]_{A^{N-1}(\theta)} = O_N(|\theta N|^{|\nu|-\ell(\nu)}N^{\ell(\lambda)})=o_N(|\theta N|^{|\nu|-\ell(\lambda)}N^{\ell(\nu)}),
\]
which then corresponds to the term $\ma_{\nu\lambda}+o_N(1)$ since $\ma_{\nu\lambda}=0$. For $\lambda\in\Gamma_N$ such that $\ell(\lambda)\geq \ell(\nu)$, it is straightforward to obtain
\[
[p_\nu, p_\lambda]_{A^{N-1}(\theta)} = (\theta N)^{|\nu|-\ell(\lambda)}N^{\ell(\nu)}(\ma_{\nu\lambda}+o_N(1))
\]
using \Cref{thm:leadingorder_a}.

(a) implies (b): It is straightforward to recover the equation from (b) using \pref{eq:limit}.

(b) implies (a): We proceed with induction on $|\lambda|$.

Assume that the result is true when $|\lambda|\leq k-1$ for $k\geq 1$. We prove that the result is true when $|\lambda|=k$. First, using (b) and \pref{eq:limit} gives that for all $\nu\in\Gamma$ with $|\nu|=k$,
\[
\lim_{N\rightarrow\infty} \sum_{\lambda\in\Gamma_N[k]}(\ma_{\nu\lambda}+o_N(1)) \frac{[p_\lambda]F_N}{(\theta N)^{\ell(\lambda)}} = \sum_{\lambda\in\Gamma} \ma_{\nu\lambda}[p_\lambda] \exp\left(\sum_{\gamma\in\Gamma} c_\gamma p_\gamma\right).
\]
After deleting the terms of lower order by the inductive hypothesis, we have that 
\[
\lim_{N\rightarrow\infty} \sum_{\lambda\in\Gamma_N[k]} (\ma_{\nu\lambda}+o_N(1))\frac{c_\lambda(N)}{(\theta N)^{\ell(\lambda)}}  = \sum_{\lambda\in\Gamma} \ma_{\nu\lambda}c_\lambda.
\]
Hence, 
\[
\lim_{N\rightarrow\infty} (\ma[k]+o_N(1))\left[\frac{c_\lambda(N)}{(\theta N)^{\ell(\lambda)}} - c_\lambda\right]_{\lambda\in\Gamma[k]}^T = 0.
\]
However, observe that $\ma[k]$ is upper-triangular with a positive diagonal from \Cref{lemma:matrix_a1}. Thus, $\ma[k]$ is invertible, which proves (a).
\end{proof}

\begin{corollary}
\label{cor:equivalence_a1}
Suppose $F_N(x_1,\ldots,x_N) = \exp\left(\sum_{\lambda\in\Gamma_N} c_{\lambda}(N)p_\lambda\right)$. Assume that $\lim_{N\rightarrow\infty} \\|\theta N|=\infty$. The following are equivalent.
\begin{enumerate}
\item[(a)] For all $\lambda\in\Gamma$, $\lim_{N\rightarrow\infty} \frac{c_\lambda(N)}{\theta N} = c_\lambda \in \mathbb{C}$ if $\ell(\lambda)=1$ and $\lim_{N\rightarrow\infty} \frac{c_\lambda(N)}{(\theta N)^{\ell(\lambda)}}=0$ if $\ell(\lambda)\geq 2$.
\item[(b)] For all $\nu\in\Gamma$,
\[
\lim_{N\rightarrow\infty}\frac{1}{(\theta N)^{|\nu|} N^{\ell(\nu)}} [p_\nu, F_N]_{A^{N-1}(\theta)} = \prod_{i=1}^{\ell(\nu)} \sum_{\pi\in NC(\nu_i)} \prod_{B\in \pi} |B|c_{(|B|)}.
\]
\end{enumerate}
\end{corollary}

\begin{proof}
It suffices to prove that for all $\nu\in\Gamma$,
\[
\sum_{\lambda\in\Gamma} \ma_{\nu\lambda}[p_\lambda] \exp\left(\sum_{\gamma\in\Gamma} c_\gamma p_\gamma\right) = \prod_{i=1}^{\ell(\nu)} \sum_{\pi\in NC(\nu_i)} \prod_{B\in \pi} |B|c_{(|B|)}
\]
when $c_\gamma=0$ for all $\gamma\in\Gamma$ with $\ell(\gamma)\geq 2$. Observe that by the definition of $\ma$, the left hand side equals 
\begin{align*}
& \sum_{\lambda\in\Gamma}\prod_{l=1}^{\ell(\lambda)}\lambda_l\pi(\lambda)\frac{\prod_{l=1}^{\ell(\lambda)} c_{(\lambda_l)}}{\pi(\lambda)}\left[\prod_{l=1}^{\ell(\lambda)} x_{\lambda_l}\right]\prod_{i=1}^{\ell(\nu)} \sum_{\pi\in NC(\nu_i)} \prod_{B\in \pi} x_{|B|} \\ 
& = \sum_{\lambda\in\Gamma}\prod_{l=1}^{\ell(\lambda)}\lambda_l c_{(\lambda_l)}\left[\prod_{l=1}^{\ell(\lambda)} x_{\lambda_l}\right]\prod_{i=1}^{\ell(\nu)} \sum_{\pi\in NC(\nu_i)} \prod_{B\in \pi} x_{|B|} \\
& = \prod_{i=1}^{\ell(\nu)} \sum_{\pi\in NC(\nu_i)} \prod_{B\in \pi} |B|c_{(|B|)}.
\end{align*}
\end{proof}

From Newton's identities, it is straightforward to verify that when we write the degree $k$ elementary symmetric polynomial as a linear combination of power sums, the coefficient of $p_{(k)}$ is $\frac{(-1)^{k-1}}{k}$. Using the previous result as well as the results of \cite{limits_prob_measures}, we obtain the following generalization.

\begin{corollary}
\label{cor:newton_identity}
Suppose $\epsilon\in\Gamma$. Then, $c_{\epsilon,\,(|\epsilon|)}=(-1)^{\ell(\epsilon)-1}\frac{|\epsilon|!}{\pi(\epsilon)\ell(\epsilon)}$.
\end{corollary}

\begin{remark}
See \Cref{subsec:partitions} for the definition of $c_{\epsilon,\,(|\epsilon|)}$.
\end{remark}

\begin{proof}[Proof of \Cref{cor:newton_identity}]
Consider \cite{limits_prob_measures}*{Theorem 1.6} with $c=0$, $\theta_N=\theta>0$ for all $N\geq 2$, and $G_{\theta_N}(x_1,\ldots,x_N;\,\mu_N) = \exp(N m_\epsilon)$. Note that it is not necessarily true that there exists $\mu_N$ with this particular Bessel generating function; however, the proof of the theorem implies that 
\[
\lim_{N\rightarrow\infty} \frac{1}{(\theta N)^{|\nu|}N^{\ell(\nu)}} [p_\nu, F_N]_{A^{N-1}(\theta)} = \prod_{i=1}^{\ell(\nu)} \sum_{\pi\in NC(\nu_i)} \prod_{B\in\pi} |B|c_{(|B|)},
\]
where $c_{(k)}=0$ for $k\geq 1$ such that $k\not=|\epsilon|$ and $c_{(|\epsilon|)}=(-1)^{\ell(\epsilon)-1}\frac{\theta^{-1}|\epsilon|!}{\pi(\epsilon)\ell(\epsilon)}$.

By \Cref{cor:equivalence_a1}, we have that $\lim_{N\rightarrow\infty} \frac{c_{(|\epsilon|)}(N)}{\theta N} = (-1)^{\ell(\epsilon)-1}\frac{\theta^{-1}|\epsilon|!}{\pi(\epsilon)\ell(\epsilon)}$, where 
$G_{\theta_N}(x_1,\ldots,\\x_N;\,\mu_N) = \exp\left(\sum_{\lambda\in\Gamma_N} c_\lambda(N)p_\lambda\right)$. However, observe that 
\[
c_{(|\epsilon|)}(N) = Nc_{\epsilon,(|\epsilon|)} \Rightarrow c_{\epsilon,\,(|\epsilon|)} = (-1)^{\ell(\epsilon)-1}\frac{|\epsilon|!}{\pi(\epsilon)\ell(\epsilon)}.
\]
\end{proof}

\subsection[\texorpdfstring{$[\cdot,\cdot]_{A^{N-1}(\theta)}$}{} in the high temperature regime]{Leading order terms of the type A Dunkl bilinear form in the high temperature regime}

Next, we compute the leading order terms of the matrix $([p_\lambda,p_\nu]_{A^{N-1}(\theta)}\\)_{\lambda,\nu\in\Gamma}$ in the regime $\theta N \rightarrow c\in\mathbb{C}$.

Suppose $k\geq 1$ and $\pi\in NC(k)$. Define the polynomial $W^A(\pi)\in\mathbb{Z}[x]$ by
\[
W^A(\pi)(x)\triangleq \prod_{i\in[k],\,b(i;\,\pi)=0} (x+d(i;\,\pi)).
\]

\begin{theorem}
\label{thm:leadingorder_a2}
Suppose $k\geq 1$, $\lambda,\nu\in\Gamma[k]$, and $\ell(\lambda)\leq \ell(\nu)$. Then,
\[
[p_\lambda, p_\nu]_{A^{N-1}(\theta)} = N^{\ell(\lambda)}\prod_{l=1}^{\ell(\nu)}\nu_l \pi(\nu) \left[\prod_{l=1}^{\ell(\nu)} x_{\nu_l}\right]\prod_{i=1}^{\ell(\lambda)} \sum_{\pi\in NC(\lambda_i)} W^A(\pi)(N\theta)\prod_{B\in \pi} x_{|B|} + R(N,\theta),
\]
where $R\in\mathbb{Q}[x,y]$ satisfies the condition that in each of its summands, the degree of $x$ is at most $\ell(\lambda)-1$ greater than the degree of $y$.
\end{theorem}

Similarly to how \Cref{thm:leadingorder_a} can be proved using the results of \cite{limits_prob_measures}, this theorem can be proved using Theorems 3.8 and 3.10 of the paper \cite{matrix}, although the weight function that the paper uses is written differently than how we write $W^A(\pi)$. However, we include the proof that follows the method of the second proof of \Cref{thm:leadingorder_a}. This proof can be used while proving analogous results for the $BC^N(\theta_0,\theta_1)$ and $D^N(\theta)$ root systems.

\begin{proof}[Proof of \Cref{thm:leadingorder_a2}]
We can use the same argument as the proof of \Cref{thm:leadingorder_a}. The main difference is that we are including all terms such that the degree of $N$ is $\ell(\lambda)$ greater than the degree of $\theta$ in the leading order term. This means that we still have that $j_i$ for $1\leq i\leq \ell(\lambda)$ are distinct to obtain the factor of $N^{\ell(\lambda)}$ and that the switches are to distinct indices so that the degrees of $N$ and $\theta$ are the same, after removing the factor of $N^{\ell(\lambda)}$. However, we no longer require that $d=\ell(\nu)-\ell(\lambda)$, since we are not assuming that the order of $N\theta$ is greater than the order of $1$. In particular, $d\geq \ell(\nu)-\ell(\lambda)$ and the number of switches is $k-\ell(\lambda)-d$. This essentially means that the derivatives do not all have to be assigned to a $p_{\nu_l}$ to contribute to the leading order term.

Furthermore, for the remainder term $R$, we have that each term which is $N^{\ell(\lambda)}$ multiplied by a power of $N\theta$ is already included in the leading order term, so it is straightforward to deduce that it follows the given condition. We alter the definition of $\mathcal{S}$ as follows.

Let $\mathcal{S}$ denote the set of sequences of operators $s=\{s_i\}_{1\leq i\leq |\lambda|}$ such that:
\begin{enumerate}
    \item For $1\leq i\leq \ell(\lambda)$, $s_{1+\lambda_1+\cdots+\lambda_{i-1}} = \partial_i$, and for $j\in [1+\lambda_1+\cdots+\lambda_{i-1}, \lambda_1+\cdots+\lambda_i]$, $s_j$ is a switch from $i$ to an element of $[N]\backslash\{i\}$ or is $\partial_i$.
    \item For at least $\ell(\nu)-\ell(\lambda)$ elements $j$ of $[|\lambda|]\backslash \{1, 1+\lambda_1, \ldots, 1+\lambda_1+\cdots+\lambda_{\ell(\lambda)-1}\}$, $s_j$ is a derivative.
    \item The $j$th switch is from some element of $[\ell(\lambda)]$ (which is determined by condition (1)) to $\ell(\lambda)+j$ for $1\leq j\leq n(s)$, where $n(s)$ denotes the number of switches of $s$.
\end{enumerate}
Then, the leading order term is
\begin{equation}
\label{eq:leadingorder_a2}
N^{\ell(\lambda)}\sum_{s\in \mathcal{S}} N^{n(s)}s_{|\lambda|}\circ\cdots\circ s_1 p_\nu,
\end{equation}
where the factor of $N^{\ell(\lambda)}$ arises from the choices of the $j_i$ and the factor of $N^{n(s)}$ arises from selecting the indices of the switches. For simplicity, we compute 
\begin{equation}
\label{eq:leadingorder_a3}
\sum_{s\in \mathcal{S}}s_{|\lambda|}\circ\cdots\circ s_1 p_\nu,
\end{equation}
and then multiply by $N^{\ell(\lambda)}$ and replace $\theta$ with $\theta N$ to obtain \pref{eq:leadingorder_a2}.

We have that the analogous results in the proof of \Cref{thm:leadingorder_a} are true, other than \Cref{lemma:allocation}, which we must alter since the derivatives have the same weight as $\theta N$. In other words, as we mentioned earlier, the derivatives in $s$ do not all have to be assigned to a $p_{\nu_l}$ to contribute to the leading order term.

\begin{lemma}
\label{lemma:allocation_a2}
Suppose $\nu'\in \Gamma$. Let $c_{\nu'}$ be the value of \pref{eq:leadingorder_a3} when $\lambda$ is set as $(|\nu'|)$ and $\nu$ is set as $\nu'$ and define
\[
n_{\nu'}(x)\triangleq \sum_{\substack{\pi\in NC(|\nu'|),\,\gamma(\pi) = \nu'}}  W^A(\pi)(x).
\]
Then, $c_{\nu'}=\theta^{|\nu'|-\ell(\nu')}\prod_{l=1}^{\ell(\nu')} \nu'_l \pi(\nu') n_{\nu'}(\theta)$.
\end{lemma}
\begin{proof}
The proof is essentially the same as the proof of \Cref{lemma:allocation}. The only difference is that at each position where we previously apply a switch, we can also choose to apply a derivative. The derivative would multiply the coefficient by the current degree of $x_1$, which corresponds to the $d(i;\,\pi)$ factor in the formula for $W_\pi^A$.
\end{proof}

Using \pref{eq:contribution3}, \pref{eq:leadingorder_a3}, and \Cref{lemma:allocation_a2}, and then replacing $\theta$ with $\theta N$ and multiplying by a factor of $N^{\ell(\lambda)}$, it follows that the leading order coefficient is
\begin{align*}
& N^{\ell(\lambda)}\sum_{[\ell(\nu)] = S_1\sqcup \cdots \sqcup S_{\ell(\lambda)}} \prod_{l=1}^{\ell(\nu)} \nu_l \times \\
& \prod_{i=1}^{\ell(\lambda)} \pi((\nu_j: j\in S_i))\sum_{\substack{\pi\in NC(\lambda_i), \\ \gamma(\pi)=\gamma((\nu_j: j\in S_i))}} W^A(\pi)(N\theta) \\
& =  N^{\ell(\lambda)}\sum_{\nu = \gamma_1 + \cdots + \gamma_{\ell(\lambda)}} \prod_{l=1}^{\ell(\nu)}\nu_l \pi(\nu) \prod_{i=1}^{\ell(\lambda)} \sum_{\substack{\pi\in NC(\lambda_i), \\ \gamma(\pi)=\gamma_i}} W^A(\pi)(N\theta) \\
& =  N^{\ell(\lambda)}\prod_{l=1}^{\ell(\nu)}\nu_l \pi(\nu) \left[\prod_{l=1}^{\ell(\nu)} x_{\nu_l}\right]\prod_{i=1}^{\ell(\lambda)} \sum_{\pi\in NC(\lambda_i)} \prod_{B\in \pi} x_{|B|} W^A(\pi)(N\theta).
\end{align*}
\end{proof}

We now prove the analogue \Cref{thm:equivalence_a2} of \Cref{thm:equivalence_a1}. The method we use is the same, however we include it for completeness.

Define the infinite dimensional matrix $\Wa\in\mathbb{Z}_{\geq 0}^{\Gamma\times\Gamma}[y]$ such that for $\lambda,\nu\in\Gamma$, 
\[
\Wa_{\lambda\nu}(y)\triangleq \pi(\nu)\prod_{l=1}^{\ell(\nu)}\nu_l  \left[\prod_{l=1}^{\ell(\nu)} x_{\nu_l}\right]\prod_{i=1}^{\ell(\lambda)} \sum_{\pi\in NC(\lambda_i)} W^A(\pi)(y)\prod_{B\in \pi} x_{|B|}.
\]
The following lemma is straightforward to deduce.

\begin{lemma}
\label{lemma:matrix_a2} Suppose $\lambda,\nu\in\Gamma$ and $|\lambda|=|\nu|$.
\begin{enumerate}
\item[(A)] If $\lambda=\nu$, then $\Wa(y)_{\lambda\nu}=\pi(\nu)\prod_{l=1}^{\ell(\nu)}\nu_l W^A([\nu_l])(y)$.
\item[(B)] If $\ell(\lambda)\geq \ell(\nu)$ and $\lambda\not=\nu$, then $\Wa_{\lambda\nu}=0$.
\end{enumerate}
\end{lemma}

\begin{theorem}
\label{thm:equivalence_a2}
Suppose $F_N(x_1,\ldots,x_N) = \exp\left(\sum_{\lambda\in\Gamma_N} c_{\lambda}(N)p_\lambda\right)$. Assume that $\lim_{N\rightarrow\infty} \\\theta N = c\in\mathbb{C}$. Consider the following statements.
\begin{enumerate}
\item[(a)] For all $\lambda\in\Gamma$, $\lim_{N\rightarrow\infty} c_\lambda(N)=c_\lambda\in\mathbb{C}$. 
\item[(b)] For all $\nu\in\Gamma$,
\[
\lim_{N\rightarrow\infty} \frac{1}{N^{\ell(\nu)}} [p_\nu, F_N]_{A^{N-1}(\theta)} = \sum_{\lambda\in\Gamma} \Wa(c)_{\nu\lambda}[p_\lambda] \exp\left(\sum_{\gamma\in\Gamma} c_\gamma p_\gamma\right).
\]
\end{enumerate}
Then, (a) implies (b), and if $c$ is not a negative integer, then (b) implies (a).
\end{theorem}

\begin{proof}
We follow the method of the proof of \Cref{thm:equivalence_a1}. Suppose $\nu\in\Gamma$. Then, using \Cref{thm:leadingorder_a2} after noting that each summand of $R$ is $o_N(N^{\ell(\nu)})$, we get that
\begin{equation}
\label{eq:limit_a2}
\lim_{N\rightarrow\infty} \frac{1}{N^{\ell(\nu)}} [p_\nu, F_N]_{A^{N-1}(\theta)} = \lim_{N\rightarrow\infty} \sum_{\lambda\in\Gamma_N[|\nu|]} (\Wa(c)_{\nu\lambda}+o_N(1))[p_\lambda]F_N. 
\end{equation}
For $\lambda\in\Gamma_N$ such that $\ell(\lambda)<\ell(\nu)$, we have that
\[
[p_\nu, p_\lambda]_{A^{N-1}(\theta)} = O_N(N^{\ell(\lambda)})=o_N(N^{\ell(\nu)}),
\]
which then corresponds to the term $(\Wa(c)_{\nu\lambda}+o_N(1))$ since $\Wa(c)_{\nu\lambda}=0$. For $\lambda\in\Gamma_N$ such that $\ell(\lambda)\geq \ell(\nu)$, it is straightforward to obtain from \Cref{thm:leadingorder_a2} that
\[
[p_\nu, p_\lambda]_{A^{N-1}(\theta)} = N^{\ell(\nu)}(\Wa(c)_{\nu\lambda}+o_N(1)).
\]

(a) implies (b): It is straightforward to recover the equation from (b) using \pref{eq:limit_a2}.

(b) implies (a): Assume that $c$ is not a negative integer. We proceed with induction on $|\lambda|$.

Assume that the result is true when $|\lambda|\leq k-1$ for $k\geq 1$. We prove that the result is true when $|\lambda|=k$. First, using (b) and \pref{eq:limit_a2} gives that for all $\nu\in\Gamma[k]$,
\[
\lim_{N\rightarrow\infty} \sum_{\lambda\in\Gamma_N[k]}(\Wa(c)_{\nu\lambda}+o_N(1)) [p_\lambda]F_N = \sum_{\lambda\in\Gamma} \Wa(c)_{\nu\lambda}[p_\lambda] \exp\left(\sum_{\gamma\in\Gamma} c_\gamma p_\gamma\right).
\]
After deleting the terms of lower order by the inductive hypothesis, we have that 
\[
\lim_{N\rightarrow\infty} \sum_{\lambda\in\Gamma_N[k]} (\Wa(c)_{\nu\lambda}+o_N(1))c_\lambda(N) = \sum_{\lambda\in\Gamma} \Wa(c)_{\nu\lambda}c_\lambda.
\]
Hence, 
\[
\lim_{N\rightarrow\infty} (\Wa(c)[k]+o_N(1))\left[c_\lambda(N) - c_\lambda\right]_{\lambda\in\Gamma[k]}^T = 0.
\]
However, observe that $\Wa(c)[k]$ is upper-triangular with a nonzero diagonal from Lemma \ref{lemma:matrix_a2}, since $c$ is not a negative integer. Thus, $\Wa(c)[k]$ is invertible, which proves (a).
\end{proof}

Note that \Cref{thm:equivalence_a2} generalizes the results of the paper \cite{matrix}, since it allows $c_\lambda$ to be nonzero if $\ell(\lambda)\geq 2$, while \Cref{cor:equivalence_a2} has already been proved in the paper.

\begin{corollary}
\label{cor:equivalence_a2}
Suppose $F_N(x_1,\ldots,x_N) = \exp\\\left(\sum_{\lambda\in\Gamma_N} c_{\lambda}(N)p_\lambda\right)$. Assume that $\lim_{N\rightarrow\infty} \theta N = c\in\mathbb{C}$. Consider the following statements.
\begin{enumerate}
\item[(a)] For all $\lambda\in\Gamma$, $\lim_{N\rightarrow\infty} c_\lambda(N) =c_\lambda \in\mathbb{C}$ if $\ell(\lambda)=1$ and $\lim_{N\rightarrow\infty} c_\lambda(N)=0$ if $\ell(\lambda)>1$
\item[(b)] For all $\nu\in\Gamma$,
\[
\lim_{N\rightarrow\infty}\frac{1}{N^{\ell(\nu)}} [p_\nu, F_N]_{A^{N-1}(\theta)} = \prod_{i=1}^{\ell(\nu)} \sum_{\pi\in NC(\nu_i)} W^A(\pi)(c)\prod_{B\in \pi} |B|c_{(|B|)}.
\]
\end{enumerate}
Then, (a) implies (b), and if $c$ is not a negative integer, then (b) implies (a).
\end{corollary}

\begin{proof}
The same method as the proof of \Cref{cor:equivalence_a1} can be used.
\end{proof}

\section{Leading order terms for the type BC Dunkl bilinear form}
\label{sec:bc}

We transition to the type BC root system and first prove the analogues of the results of \Cref{sec:a} for the $|\theta_0 N|\rightarrow\infty$ regime and later prove the analogues of the results for the $\theta_0 N \rightarrow c_0\in\mathbb{C}$ regime. Note that we always have that $\frac{\theta_1}{\theta_0 N}\rightarrow c_1\in\mathbb{C}$ and that the asymptotic invertibility of $\mathcal{D}(BC^N(\theta_0,\theta_1))$ will depend on $c_0$ and $c_1$.

\subsection{Proof of part (B) of \texorpdfstring{\Cref{thm:main}}{}}

We first consider the $|\theta_0 N|\rightarrow\infty$ regime.

\begin{theorem}
\label{thm:leadingorder_bc}
Suppose $k\geq 1$, $\lambda,\nu\in\even[k]$, and $\ell(\lambda)\leq \ell(\nu)$. Then,
\begin{align*}
[p_\lambda, p_\nu]_{BC^N(\theta_0,\theta_1)} = & (2\theta_0)^{k-\ell(\nu)}\prod_{l=1}^{\ell(\nu)}\nu_l \pi(\nu) \left[\prod_{l=1}^{\ell(\nu)} x_{\nu_l}\right]\prod_{i=1}^{\ell(\lambda)} \sum_{\pi\in NC(\lambda_i)} \prod_{B\in \pi} x_{|B|}\left(1+\frac{\theta_1}{N\theta_0}\right)^{o(\pi)}\times \\
& N^{k+\ell(\lambda)-\ell(\nu)}+ R(N,\theta_0, \theta_1).
\end{align*}
where $R\in\mathbb{Q}[x,y,z]$ satisfies:
\begin{enumerate}
\item In each summand, the degree of $x$ is at most $\ell(\lambda)$ greater than the degree of $y$.
\item The degree of $x$ is at most $k+\ell(\lambda)-\ell(\nu)-1$ and the sum of the degrees of $y$ and $z$ in each summand is at most $k-\ell(\nu)$.
\item In no summand is the degree of $x$ $\ell(\lambda)$ greater than the degree of $y$ while the degrees of $y$ and $z$ add to $k-\ell(\nu)$.
\end{enumerate}
\end{theorem}

\begin{proof}

We follow the proof of \Cref{thm:leadingorder_a} given in \Cref{subsec:proofa}. The main difference is that we regard $\theta_1$ as having weight $N\theta_0$, while $N$ is considered to be the same as before and $2\theta_0$ is considered to be the same as $\theta$. However, we still regard $N\theta_0$ and $\theta_1$ as being higher order than a constant. So, we still have that the $j_i$ for $i\in[\ell(\lambda)]$ are distinct to obtain the factor of $N^{\ell(\lambda)}$. Afterwards, we have that the type $0$ switches are to distinct indices, so that the degrees of $\theta_0$ and $N$ are the same after removing the factor of $N^{\ell(\lambda)}$. In contrast with the proof of \Cref{thm:leadingorder_a2}, we still have that $d=\ell(\nu)-\ell(\lambda)$.

It is not challenging to see that the remainder term $R$ is a polynomial with rational coefficients that satisfies conditions (1), (2), and (3). For condition (1), we have that after selecting $j_1,\ldots,j_{\ell(\lambda)}$, to obtain an additional factor of $N$, we must include a type $0$ switch, which will also add a factor of $\theta_0$. For condition (2), we note that we have already identified the leading order term, so the degree of $x$ is at most $k+\ell(\lambda)-\ell(\nu)-1$. Furthermore, the number of switches is at most $k-\ell(\nu)$ since $d\geq \ell(\nu)-\ell(\lambda)$, so the degrees of $y$ and $z$ in each summand add to at most $k-\ell(\nu)$. For condition (3), we note that any term that satisfies this condition is included in the leading order term.

Let $\mathcal{S}$ denote the set of sequences of operators $s=\{s_i\}_{1\leq i\leq |\lambda|}$ such that:
\begin{enumerate}
    \item For $1\leq i\leq \ell(\lambda)$, $s_{1+\lambda_1+\cdots+\lambda_{i-1}} = \partial_i$, and for $j\in [1+\lambda_1+\cdots+\lambda_{i-1}, \lambda_1+\cdots+\lambda_i]$, $s_j$ is a type 0 switch from $i$ to an element of $[N]\backslash\{i\}$, the type 1 switch with index $i$, or $\partial_i$.
    \item For $\ell(\nu)-\ell(\lambda)$ elements $j$ of $[|\lambda|]\backslash \{1, 1+\lambda_1, \ldots, 1+\lambda_1+\cdots+\lambda_{\ell(\lambda)-1}\}$, $s_j$ is a derivative.
    \item The $j$th type 0 switch is from some element of $[\ell(\lambda)]$ (which is determined by condition (1)) to $\ell(\lambda)+j$ for $1\leq j\leq k - \ell(\nu)-n(s)$, where $n(s)$ is the number of type 1 switches of $s$.
\end{enumerate}
Then, the leading order term is
\begin{equation}
\label{eq:leadingorder2}
N^{\ell(\lambda)}\sum_{s\in \mathcal{S}} N^{k-\ell(\nu)-n(s)}s_{|\lambda|}\circ\cdots\circ s_1 p_\nu,
\end{equation}
where $N^{\ell(\lambda)}$ corresponds to the number of choices for $j_1,\ldots,j_{\ell(\lambda)}$ and $N^{k-\ell(\nu)-n(s)}$ corresponds to the number of choices for the indices of the type 0 switches. For simplicity, we compute
\begin{equation}
\label{eq:leadingorder3}
\sum_{s\in \mathcal{S}}s_{|\lambda|}\circ\cdots\circ s_1 p_\nu.
\end{equation}
Afterwards, we replace $\theta_0$ with $N\theta_0$ and multiply by a factor of $N^{\ell(\lambda)}$ to compute \pref{eq:leadingorder2}.

We have that the analogous results in the proof of \Cref{thm:leadingorder_a2} are true, other than \Cref{lemma:allocation}, which we must alter to account for the type 1 switches.

\begin{lemma}
\label{lemma:allocation2}
Suppose $\nu'\in \even$. Let $c_{\nu'}$ be the value of \pref{eq:leadingorder3} when $\lambda$ is set as $(|\nu'|)$ and $\nu$ is set as $\nu'$ and define 
\[
n_{\nu'}(x) \triangleq \sum_{\substack{\pi\in NC(|\nu'|), \\ \gamma(\pi)= \nu'}} (1+x)^{o(\pi)}.
\]
Then, $c_{\nu'}=(2\theta_0)^{|\nu'|-\ell(\nu')}\prod_{l=1}^{\ell(\nu')} \nu'_l \pi(\nu')n_{\nu'}\left(\frac{\theta_1}{\theta_0}\right)$.
\end{lemma}

\begin{proof}
We can follow the proof of \Cref{lemma:allocation} with a few differences. First, we replace $\theta$ by $2\theta_0$. The only other difference is that when the degree of $x_1$ is odd, we can apply either the type $0$ or type $1$ switch, which would be equivalent to $2\theta_0 d_1$ and $2\theta_1 d_1$, respectively. To account for these choices, for each noncrossing partition $\pi$, we must multiply by $(1+\frac{\theta_1}{\theta_0})^{o(\pi)}$, since $o(\pi)$ counts the number of locations at which there is no $\partial_1$ and the degree of $x_1$ is odd.
\end{proof}

Recall that to compute the leading order term, we must replace $\theta_0$ with $N\theta_0$ and multiply by a factor of $N^{\ell(\lambda)}$. Then, using \pref{eq:contribution3}, \pref{eq:leadingorder3}, and \Cref{lemma:allocation2}, the leading order term is
\begin{align*}
& N^{\ell(\lambda)}\sum_{[\ell(\nu)] = S_1\sqcup \cdots \sqcup S_{\ell(\lambda)}} (2N\theta_0)^{k-\ell(\nu)}\prod_{l=1}^{\ell(\nu)} \nu_l \times \\
& \prod_{i=1}^{\ell(\lambda)} \pi((\nu_j: j\in S_i))\sum_{\substack{\pi\in NC(\lambda_i), \\ \gamma(\pi)=\gamma( (\nu_j: j\in S_i))}} \left(1+\frac{\theta_1}{N\theta_0}\right)^{o(\pi)} \\
& = N^{\ell(\lambda)}\sum_{\nu = \gamma_1 + \cdots + \gamma_{\ell(\lambda)}} (2N\theta_0)^{k-\ell(\nu)}\prod_{l=1}^{\ell(\nu)}\nu_l \pi(\nu) \prod_{i=1}^{\ell(\lambda)} \sum_{\substack{\pi\in NC(\lambda_i), \\  \gamma(\pi)=\gamma_i}} \left(1+\frac{\theta_1}{N\theta_0}\right)^{o(\pi)} \\
& = N^{\ell(\lambda)}(2N\theta_0)^{k-\ell(\nu)}\prod_{l=1}^{\ell(\nu)}\nu_l \pi(\nu) \left[\prod_{l=1}^{\ell(\nu)} x_{\nu_l}\right]\prod_{i=1}^{\ell(\lambda)} \sum_{\pi\in NC(\lambda_i)} \prod_{B\in \pi} x_{|B|} \left(1+\frac{\theta_1}{N\theta_0}\right)^{o(\pi)}.
\end{align*}
\end{proof}

Define the infinite dimensional matrix $\mbc\in\mathbb{Z}_{\geq 0}^{\even\times\even}[y]$ such that for $\lambda,\nu\in\even$, 
\[
\mbc_{\lambda\nu}(y)\triangleq 2^{|\nu|-\ell(\nu)}\prod_{l=1}^{\ell(\nu)}\nu_l \pi(\nu) \left[\prod_{l=1}^{\ell(\nu)} x_{\nu_l}\right]\prod_{i=1}^{\ell(\lambda)} \sum_{\pi\in NC(\lambda_i)} \prod_{B\in \pi} x_{|B|}(1+y)^{o(\pi)}.
\]

\begin{lemma}
\label{lemma:matrix_bc1} Suppose $\lambda,\nu\in\even$ and $|\lambda|=|\nu|$.
\begin{enumerate}
\item[(A)] If $\lambda=\nu$, then $\mbc_{\lambda\nu}(y)=\pi(\nu)\prod_{l=1}^{\ell(\nu)}2^{\nu_l-1}\nu_l (1+y)^{o([\nu_l])}$.
\item[(B)] If $\ell(\lambda)\geq \ell(\nu)$ and $\lambda\not=\nu$, then $\mbc_{\lambda\nu}(y)=0$.
\end{enumerate}
\end{lemma}

\begin{theorem}
\label{thm:equivalence_bc1}
Suppose $F_N(x_1,\ldots,x_N) = \exp\left(\sum_{\lambda\in\Gamma_{N; even}} c_{\lambda}(N)p_\lambda\right)$, $\lim_{N\rightarrow\infty} |\theta_0N|=\infty$, and $\lim_{N\rightarrow\infty} \frac{\theta_1}{\theta_0 N} = c\in\mathbb{C}$. Suppose $c_\lambda\in\mathbb{C}$ for all $\lambda\in\even$. Consider the following statements:
\begin{enumerate}
\item[(a)] For all $\lambda\in\even$, $\lim_{N\rightarrow\infty} \frac{c_\lambda(N)}{(\theta_0 N)^{\ell(\lambda)}}=c_\lambda\in\mathbb{C}$. 
\item[(b)] For all $\nu\in\even$,
\[
\lim_{N\rightarrow\infty} \frac{1}{(\theta_0 N)^{|\nu|}N^{\ell(\nu)}} [p_\nu, F_N]_{BC^N(\theta_0, \theta_1)} = \sum_{\lambda\in\Gamma} \mbc_{\nu\lambda}(c)[p_\lambda] \exp\left(\sum_{\gamma\in\Gamma} c_\gamma p_\gamma\right).
\]
\end{enumerate}
Then, (a) implies (b), and if $c\not=-1$, then (b) implies (a).
\end{theorem}

\begin{proof}
The same method as the proof of \Cref{thm:equivalence_a1} can be used. For the implication of (a) from (b), we note that $\mbc(c)$ is invertible when $c\not=-1$ by \Cref{lemma:matrix_bc1}.
\end{proof}

\begin{corollary}
\label{cor:equivalence_bc1}
Suppose $F_N(x_1,\ldots,x_N) = \exp\left(\sum_{\lambda\in\evenN} c_{\lambda}(N)p_\lambda\right)$. Assume that $\\\lim_{N\rightarrow\infty} |\theta_0 N|=\infty$ and $\lim_{N\rightarrow\infty} \frac{\theta_1}{\theta_0N}=c\in\mathbb{C}$. Consider the following statements.
\begin{enumerate}
\item[(a)] For all $\lambda\in\even$, $\lim_{N\rightarrow\infty} \frac{c_{\lambda(N)}(N)}{\theta_0 N}=c_{\lambda}$ if $\ell(\lambda)=1$ and $\lim_{N\rightarrow\infty} \frac{c_\lambda(N)}{(\theta_0 N)^{\ell(\lambda)}}=0$ if $\ell(\lambda)\geq 2$.
\item[(b)] For all $\nu\in\even$,
\[
\lim_{N\rightarrow\infty}\frac{1}{(\theta_0 N)^{|\nu|} N^{\ell(\nu)}} [p_\nu, F_N]_{BC^N(\theta_0, \theta_1)} = \prod_{i=1}^{\ell(\nu)} \sum_{\pi\in \NCeven(\nu_i)} (1+c)^{o(\pi)}\prod_{B\in \pi} 2^{|B|-1}|B|c_{(|B|)}.
\]
\end{enumerate}
Then, (a) implies (b), and if $c\not=-1$, then (b) implies (a).
\end{corollary}

\begin{proof}
See the proof of \Cref{cor:equivalence_a1}.
\end{proof}

\subsection[\texorpdfstring{$[\cdot,\cdot]_{BC^N(\theta_0,\theta_1)}$}{} in the high temperature regime]{Leading order terms of the type BC Dunkl bilinear form in the high temperature regime}

Next, we consider the $\theta_0 N\rightarrow c_0\in\mathbb{C}$, $\theta_1\rightarrow c_1\in\mathbb{C}$ regime. Define the polynomial $W^{BC}(\pi)\in\mathbb{Z}[x,y]$ by
\[
W^{BC}(\pi)(x,y)\triangleq \prod_{i\in[k],\,b(i;\,\pi)=0} (2x+\mathbf{1}\{d(i;\,\pi) \text{ is odd}\} 2y + d(i;\,\pi)).
\]

\begin{theorem}
\label{thm:leadingorder_bc2}
Suppose $k\geq 1$, $\lambda,\nu\in\even[k]$, and $\ell(\lambda)\leq \ell(\nu)$. Then,
\begin{align*}
[p_\lambda, p_\nu]_{BC^N(\theta_0, \theta_1)} = & N^{\ell(\lambda)}\prod_{l=1}^{\ell(\nu)}\nu_l \pi(\nu) \left[\prod_{l=1}^{\ell(\nu)} x_{\nu_l}\right]\prod_{i=1}^{\ell(\lambda)} \sum_{\pi\in NC(\lambda_i)} W^{BC}(\pi)(N\theta_0, \theta_1)\prod_{B\in \pi} x_{|B|} \\
& + R(N,\theta_0, \theta_1),
\end{align*}
where $R\in\mathbb{Q}[x,y,z]$ satisfies the condition that in each of its summands, the degree of $x$ is at most $\ell(\lambda)-1$ greater than the degree of $y$.
\end{theorem}

Similarly to \Cref{thm:leadingorder_a,thm:leadingorder_bc}, we can prove \Cref{thm:leadingorder_bc2} using Theorems 4.8 and 5.5 of the paper \cite{rectangularmatrix}, although the weight function that the paper uses is written differently than how we write $W^{BC}(\pi)$. We include a different proof of \Cref{thm:leadingorder_bc2} as well. We similarly have that \Cref{thm:equivalence_bc2} generalizes the results of the paper while \Cref{cor:equivalence_bc2} has already been proved in the paper.

\begin{proof}[Proof of \Cref{thm:leadingorder_bc2}]
The idea is the same as the proof of \Cref{thm:leadingorder_a}, except we add the modifications from the proofs of \Cref{thm:leadingorder_a2,thm:leadingorder_bc}. The leading order term consists of $N^{\ell(\lambda)}$ multiplied by a power of $N\theta_0$ and a power of $\theta_1$. Then, we still have that the $j_i$ are distinct and the type 0 switches are to distinct indices. However, now we have that $d\geq \ell(\nu)-\ell(\lambda)$ rather than $d=\ell(\nu)-\ell(\lambda)$.

For the modification of \Cref{lemma:allocation}, at each position which is not at the start of a block, we can apply a type 0 switch, the type 1 switch, or the derivative. These correspond to the terms $x$, $\mathbf{1}\{d(i;\,\pi) \text{ is odd}\}y$, and $d(i;\,\pi)$, respectively, in the formula for $W^{BC}(x,y)$.
\end{proof}

Define the infinite dimensional matrix $\Wbc\in\mathbb{Z}_{\geq 0}^{\even\times\even}[y,z]$ such that for $\lambda,\nu\in\even$, 
\[
\Wbc_{\lambda\nu}(y,z)\triangleq \pi(\nu)\prod_{l=1}^{\ell(\nu)}\nu_l  \left[\prod_{l=1}^{\ell(\nu)} x_{\nu_l}\right]\prod_{i=1}^{\ell(\lambda)} \sum_{\pi\in NC(\lambda_i)} W^{BC}(\pi)(y,z)\prod_{B\in \pi} x_{|B|}.
\]

\begin{lemma}
\label{lemma:matrix_bc2} Suppose $\lambda,\nu\in\even$ and $|\lambda|=|\nu|$.
\begin{enumerate}
\item[(A)] If $\lambda=\nu$, then $\Wbc_{\lambda\nu}(y,z)=\pi(\nu)\prod_{l=1}^{\ell(\nu)}\nu_l W^{BC}([\nu_l])(y,z)$.
\item[(B)] If $\ell(\lambda)\geq \ell(\nu)$ and $\lambda\not=\nu$, then $\Wbc_{\lambda\nu}(y,z)=0$.
\end{enumerate}
\end{lemma}

\begin{theorem}
\label{thm:equivalence_bc2}
Suppose $F_N(x_1,\ldots,x_N) = \exp\left(\sum_{\lambda\in\evenN} c_{\lambda}(N)p_\lambda\right)$. Assume that $\\\lim_{N\rightarrow\infty} \theta_0 N = c_0\in\mathbb{C}$ and $\lim_{N\rightarrow\infty} \theta_1=c_1\in\mathbb{C}$. Consider the following statements.
\begin{enumerate}
\item[(a)] For all $\lambda\in\even$, $\lim_{N\rightarrow\infty} c_\lambda(N)=c_\lambda\in\mathbb{C}$. 
\item[(b)] For all $\nu\in\even$,
\[
\lim_{N\rightarrow\infty} \frac{1}{N^{\ell(\nu)}} [p_\nu, F_N]_{BC^N(\theta_0, \theta_1)} = \sum_{\lambda\in\even} \Wbc_{\nu\lambda}(c_0, c_1)[p_\lambda] \exp\left(\sum_{\gamma\in\even} c_\gamma p_\gamma\right).
\]
\end{enumerate}
Then, (a) implies (b), and if $c_0$ is not a negative integer and $2c_0+2c_1$ is not a negative odd integer, then (b) implies (a).
\end{theorem}

\begin{proof}
See the proof of \Cref{thm:equivalence_a2}. For the implication of (a) from (b), consider the formula for $W^{BC}([l])(c_0, c_1)$ for some $l\in 2\mathbb{N}$. If $d(i;\,\pi)$ is odd, then the term $2c_0+2c_1+d(i;\,\pi)$ is nonzero. Moreover, if $d(i;\,\pi)$ is even and at least two, then the term $2c_0+d(i;\,\pi)$ is nonzero; note that if $d(i;\,\pi)=0$, then we must have that $b(i;\,\pi)=1$. Therefore, $W^{BC}([l])(c_0, c_1)\not=0$, so $\Wbc(c_0, c_1)[k]$ is invertible for all $k\geq 1$ by \Cref{lemma:matrix_bc2}.
\end{proof}

\begin{corollary}
\label{cor:equivalence_bc2}
Suppose $F_N(x_1,\ldots,x_N) = \exp\left(\sum_{\lambda\in\evenN} c_{\lambda}(N)p_\lambda\right)$. Assume that $\\\lim_{N\rightarrow\infty} \theta_0 N = c_0\in\mathbb{C}$ and $\lim_{N\rightarrow\infty} \theta_1=c_1\in\mathbb{C}$. Consider the following statements.
\begin{enumerate}
\item[(a)] For all $\lambda\in\even$, $\lim_{N\rightarrow\infty} c_{\lambda}(N) =c_{\lambda} \in \mathbb{C}$ if $\ell(\lambda)=1$ and $\lim_{N\rightarrow\infty} c_\lambda(N)=0$ if $\ell(\lambda)\geq 2$.
\item[(b)] For all $\nu\in\even$,
\[
\lim_{N\rightarrow\infty}\frac{1}{(\theta N)^{|\nu|}} [p_\nu, F_N]_{BC^N(\theta_0, \theta_1)} = \prod_{i=1}^{\ell(\nu)} \sum_{\pi\in \NCeven(\nu_i)} W^{BC}(\pi)(c_0, c_1)\prod_{B\in \pi} |B|c_{(|B|)}.
\]
\end{enumerate}
Then, (a) implies (b), and if $c_0$ is not a negative integer and $2c_0+2c_1$ is not a negative odd integer, then (b) implies (a).
\end{corollary}

\begin{proof}
See the proof of \Cref{cor:equivalence_a2}.
\end{proof}

\section{Leading order terms of the type D Dunkl bilinear form}
\label{sec:d}

\subsection{Orthogonality results}

First, we prove some orthogonality results which are based on the presence of sign flips in $H(\mathcal{R})$. These results are relevant for separating terms with all even degrees and all odd degrees with respect to the Dunkl bilinear form for the $D^N$ root system. The following theorem expresses the Dunkl bilinear form as an integral when the multiplicity function is nonnegative.

\begin{theorem}[\cite{dunkl_integralkernel}*{Theorem 3.10}]
\label{thm:product}
Assume that $\theta\in\theta(\mathcal{R})$ and $\theta\geq 0$. Then, for $p,q\in \mathbb{C}[x_1,\ldots,x_N]$,
\[
[p, q]_{\mathcal{R}(\theta)}= c_{\mathcal{R}(\theta)}^{-1}\int_{\mathbb{R}^N} (\mathcal{D}(e^{-\frac{p_{(2)}}{2}})p)(\mathcal{D}(e^{-\frac{p_{(2)}}{2}})q) h_{\mathcal{R}(\theta)}^2wdx,
\]
where:
\begin{itemize}
\item The function $h_{\mathcal{R}(\theta)}:\mathbb{R}^N\rightarrow\mathbb{R}$ is defined as $h_{\mathcal{R}(\theta)}(x)\triangleq\prod_{r\in \mathcal{R}^+} |\product{x,r}|^{\theta(r)}$.
\item The function $w:\mathbb{R}^N\rightarrow\mathbb{R}$ is defined as $w(x)\triangleq\frac{e^{-\frac{\norm{x}_2^2}{2}}}{(2\pi)^{\frac{N}{2}}}$.
\item The constant $c_{\mathcal{R}(\theta)}$ is defined as $c_{\mathcal{R}(\theta)}\triangleq\int_{\mathbb{R}^N} h_{\mathcal{R}(\theta)}^2wdx$.
\end{itemize}
\end{theorem}

\begin{remark}
The operator $\mathcal{D}(e^{-\frac{p_{(2)}}{2}})$ is equivalent to $\sum_{k=0}^\infty \left(-\frac{1}{2}\right)^k\frac{\mathcal{D}(p_{(2)}^k)}{k!}$. Furthermore, \cite{dunkl_integralkernel} gives a formula for $c_N$.
\end{remark}

In order to analyze whether two monomials are orthogonal with respect to the Dunkl bilinear form $[\cdot,\cdot]_{\mathcal{R}(\theta)}$, we can analyze the parities of the degrees of $x_1,\ldots,x_N$ as well as the presence of sign flips in $H(\mathcal{R})$. We deduce the following result using the equivariance property of the Dunkl operators given in \Cref{lemma:equivariance}.

\begin{proposition}
\label{prop:parity}
Suppose $N\geq 1$ and $\lambda,\nu\in \mathbb{Z}_{\geq 0}^N$. Assume that for some $i\in [N]$, $\lambda_i$ and $\nu_i$ do not have the same parity and that $H(\mathcal{R})$ contains the reflection that flips the sign of $x_i$. Then, $\left[\prod_{i=1}^N x_i^{\lambda_i},\prod_{i=1}^N x_i^{\nu_i}\right]_{\mathcal{R}(\theta)}=0$.
\end{proposition}

\begin{proof}
First, assume that $\theta\geq 0$. By \Cref{thm:product},
\begin{equation}
\label{eq:integral}
\left[\prod_{i=1}^N x_i^{\lambda_i},\prod_{i=1}^N x_i^{\nu_i}\right]_{\mathcal{R}(\theta)} = c_N^{-1}\int_{\mathbb{R}^N} \left(\mathcal{D}\left(e^{-\frac{p_{(2)}}{2}}\right)\prod_{i=1}^N x_i^{\lambda_i}\right)\left(\mathcal{D}\left(e^{-\frac{p_{(2)}}{2}}\right)\prod_{i=1}^N x_i^{\nu_i}\right) h^2wdx
\end{equation}
Assume $i\in [N]$ such that $\lambda_i$ and $\nu_i$ do not have the same sign and that $\sigma\in H$ flips the sign of $i$. Without loss of generality, assume that $\lambda_i$ is odd and $\nu_i$ is even. Then, by \Cref{lemma:equivariance}, we have that 
\[
\sigma\mathcal{D}\left(e^{-\frac{p_{(2)}}{2}}\right)\prod_{i=1}^N x_i^{\lambda_i} = -\mathcal{D}\left(e^{-\frac{p_{(2)}}{2}}\right)\prod_{i=1}^N x_i^{\lambda_i}
\]
and
\[
\sigma\mathcal{D}\left(e^{-\frac{p_{(2)}}{2}}\right)\prod_{i=1}^N x_i^{\nu_i} = \mathcal{D}\left(e^{-\frac{p_{(2)}}{2}}\right)\prod_{i=1}^N x_i^{\nu_i}.
\]
Then, because $h$ is $H$-invariant and $\sigma w = w$, if we apply $\sigma$ to the integrand of \pref{eq:integral}, we will flip its sign. It is then clear that \pref{eq:integral} evaluates to zero.

Next, observe that $\left[\prod_{i=1}^N x_i^{\lambda_i},\prod_{i=1}^N x_i^{\nu_i}\right]_{\mathcal{R}(\theta)}$ is a polynomial in $\theta$ with all real coefficients. Thus, since it evaluates to zero for all choices of nonnegative real-valued $\theta$, it must evaluate to zero for all choices of complex-valued $\theta$.
\end{proof}

\begin{corollary}
\label{cor:parity_d}
Suppose $N\geq 2$ and $\lambda,\nu\in \mathbb{Z}_{\geq 0}^N$. Assume that for some $i\in [N]$, $\lambda_i$ and $\nu_i$ do not have the same parity. Then, $\left[\prod_{i=1}^N x_i^{\lambda_i},\prod_{i=1}^N x_i^{\nu_i}\right]_{D^N(\theta)}=0$ for all $\theta\in\mathbb{C}$.
\end{corollary}

\begin{proof}
We cannot directly apply \Cref{prop:parity} since $H(D^N(\theta))$ only contains reflections that flip an even number of signs. However, we have that 
\[
\left[\prod_{i=1}^N x_i^{\lambda_i},\prod_{i=1}^N x_i^{\nu_i}\right]_{D^N(\theta)} = \left[\prod_{i=1}^N x_i^{\lambda_i},\prod_{i=1}^N x_i^{\nu_i}\right]_{BC^N(\theta,0)},
\]
and we can apply \Cref{prop:parity} to deduce that $\left[\prod_{i=1}^N x_i^{\lambda_i},\prod_{i=1}^N x_i^{\nu_i}\right]_{BC^N(\theta,0)}=0$.
\end{proof}

\begin{corollary}
\label{cor:eigenfunc_parity}
Suppose $N\geq 1$ and $H(\mathcal{R})$ contains the reflection that flips the sign of $x_i$ for some $i\in[N]$. Then, $E_a^{\mathcal{R}(\theta)}(x)$ is a linear combination of $r(a)s(x)$ for monomials $r$ and $s$ such that the degree of $a_i$ in $r(a)$ and the degree of $x_i$ in $s(x)$ have the same parity.
\end{corollary}

\begin{proof}
Recall the formula \pref{eq:eigenval}, where $\Psi: r\mapsto r(a)$ after $a\in\mathbb{C}^N$ is fixed. By \Cref{prop:parity}, the matrix $M^{k;\,\mathcal{D}}$ has two orthogonal components corresponding to when the degrees of $x_i$ are even and odd for all $k\geq 1$. Then, it is straightforward to deduce the result by computing the inverse of $M^{k;\,\mathcal{D}}$, which has the same two orthogonal components.
\end{proof}

\subsection{Proof of part (C) of \texorpdfstring{\Cref{thm:main}}{}}

We obtain the leading order terms of $[ep_\lambda, ep_\nu\\]_{D^N(\theta)}$ in the following lemma by using the leading order terms of the type BC Dunkl bilinear form. The leading order terms can be expressed as a product involving gamma functions and polynomials. For example, the product $\prod_{i=1}^N(1+2(i-1)\theta)$ can also be expressed as $(2\theta)^N \frac{\Gamma(\frac{1}{2\theta} + N )}{\Gamma(\frac{1}{2\theta})}$.

\begin{lemma}[\cite{dunkl_singular_poly}*{Corollary 4.5}]
\label{lemma:leading_order_odd}
Suppose $\lambda,\nu\in\even$. Then, 
\[
[ep_\lambda, ep_\nu]_{D^N(\theta)} = \prod_{i=1}^N(1+2(i-1)\theta)[p_\lambda, p_\nu]_{BC^N(\theta,1)}.
\]
Furthermore, for $\theta_0,\theta_1\in\mathbb{C}$,
\[
[ep_\lambda, ep_\nu]_{BC^N(\theta_0,\theta_1)} = \prod_{i=1}^N(1+2(i-1)\theta_0 + 2\theta_1)[p_\lambda, p_\nu]_{BC^N(\theta_0,\theta_1+1)}.
\]
\end{lemma}

We prove part (C) of \Cref{thm:main}, see \Cref{cor:equivalence_odd1}, as well as a generalization, see \Cref{thm:equivalence_odd1}. Recall that $\md_{\lambda\nu}\triangleq 2^{|\nu|-\ell(\nu)}\ma_{\lambda\nu}$ for $\lambda,\nu\in \even$ and $\md=\mbc(0)$. 

\begin{theorem}
\label{thm:equivalence_odd1}
Suppose $F_N(x_1,\ldots,x_N)=\exp\left(\sum_{\lambda\in\evenN} c_\lambda(N)p_\lambda\right) + ed_0(N)+e\exp\\\left(\sum_{\lambda\in\evenN} d_\lambda(N)p_\lambda\right)$. Assume that $\lim_{N\rightarrow\infty}|\theta N|=\infty$. Consider the following statements.
\begin{enumerate}
\item[(a)] For all $\lambda\in\even$, $\lim_{N\rightarrow\infty} \frac{c_\lambda(N)}{(\theta N)^{\ell(\lambda)}}=c_\lambda\in\mathbb{C}$. 
\item[(b)] It is the case that $\lim_{N\rightarrow\infty} d_0(N) = d_0\in\mathbb{C}$ and for all $\lambda\in\even$, $\lim_{N\rightarrow\infty} \frac{d_\lambda(N)}{(\theta N)^{\ell(\lambda)}}=d_\lambda\in\mathbb{C}$.
\item[(c)] For all $\nu\in\even$,
\[
\lim_{N\rightarrow\infty} \frac{1}{(\theta N)^{|\nu|}N^{\ell(\nu)}} [p_\nu, F_N]_{D^N(\theta)} = \sum_{\lambda\in\even} \md_{\nu\lambda}[p_\lambda] \exp\left(\sum_{\gamma\in\even} c_\gamma p_\gamma\right).
\]
\item[(d)] It is the case that \[
\lim_{N\rightarrow\infty}\frac{[e, F_N]_{D^N(\theta)}}{\prod_{j=1}^N (1+2(j-1)\theta)} = d_0+1.
\]
Furthermore, for all $\nu\in\even$,
\[
\lim_{N\rightarrow\infty}\frac{[ep_\nu, F_N]_{D^N(\theta)}}{(\theta N)^{|\nu|}N^{\ell(\nu)}\prod_{j=1}^N (1+2(j-1)\theta)} = \sum_{\lambda\in\even} \md_{\nu\lambda}[p_\lambda]\exp\left(\sum_{\gamma\in\even} d_\gamma p_\gamma\right).
\]

\end{enumerate}
Then, (a) and (c) are equivalent.

Assume that if $N$ is sufficiently large, then $\prod_{j=1}^N (1+2(j-1)\theta)\not=0$. Then, (b) and (d) are equivalent.
\end{theorem}

\begin{proof}
The equivalence of (a) and (c) follows from \Cref{thm:equivalence_bc1} with $c=0$. For the equivalence of (b) and (d), we note that $\frac{[e,F_N]_{D^N(\theta)}}{\prod_{j=1}^N(1+2(j-1)\theta)} = d_0(N) + 1$ and for $\nu\in\even$, 
\[
\frac{[ep_\nu,F_N]_{D^N(\theta)}}{\prod_{j=1}^N(1+2(j-1)\theta)} = \left[p_\nu, \exp\left(\sum_{\lambda\in\evenN} d_\lambda(N)p_\lambda\right)\right]_{BC^N(\theta,1)}
\]
by \Cref{lemma:leading_order_odd}. Afterwards, the equivalence also follows from \Cref{thm:equivalence_bc1} with $c=0$.
\end{proof}

\begin{corollary}
\label{cor:equivalence_odd1}
Suppose $F_N(x_1,\ldots,x_N)=\exp\left(\sum_{\lambda\in\evenN} c_\lambda(N)p_\lambda\right) + ed_0(N)+e\exp\\\left(\sum_{\lambda\in\evenN}d_\lambda(N)p_\lambda\right)$. Assume that $\lim_{N\rightarrow\infty}|\theta N|=\infty$. Consider the following statements.
\begin{enumerate}
\item[(a)] For all $\lambda\in\even$, $\lim_{N\rightarrow\infty} \frac{c_\lambda(N)}{\theta N}=c_\lambda\in\mathbb{C}$ if $\ell(\lambda)=1$ and $\lim_{N\rightarrow\infty} \frac{c_\lambda(N)}{(\theta N)^{\ell(\lambda)}}=0$ if $\ell(\lambda)>1$. 
\item[(b)] It is the case that $\lim_{N\rightarrow\infty}d_0(N)=d_0\in\mathbb{C}$ and for all $\lambda\in\even$, $\lim_{N\rightarrow\infty} \frac{d_\lambda(N)}{\theta N}=d_\lambda\in\mathbb{C}$ if $\ell(\lambda)=1$ and $\lim_{N\rightarrow\infty} \frac{d_\lambda(N)}{(\theta N)^{\ell(\lambda)}}=0$ if $\ell(\lambda)>1$.
\item[(c)] For all $\nu\in\even$,
\[
\lim_{N\rightarrow\infty}\frac{1}{N^{\ell(\nu)}(\theta N)^{|\nu|} } [p_\nu, F_N]_{D^N(\theta)} = \prod_{i=1}^{\ell(\nu)} \sum_{\pi\in \NCeven(\nu_i)} \prod_{B\in \pi} 2^{|B|-1}|B|c_{(|B|)}.
\]
\item[(d)] It is the case that \[
\lim_{N\rightarrow\infty}\frac{[e, F_N]_{D^N(\theta)}}{\prod_{j=1}^N (1+2(j-1)\theta)} = d_0+1.
\]
Furthermore, for all $\nu\in\even$,
\[
\lim_{N\rightarrow\infty}\frac{[ep_\nu, F_N]_{D^N(\theta)} }{N^{\ell(\nu)}(\theta N)^{|\nu|}\prod_{j=1}^N (1+2(j-1)\theta)}= \prod_{i=1}^{\ell(\nu)}\sum_{\pi\in \NCeven(\nu_i)} \prod_{B\in \pi} 2^{|B|-1}|B|d_{(|B|)}.
\]
\end{enumerate}
Then, (a) and (c) are equivalent. 

Assume that if $N$ is sufficiently large, then $\prod_{j=1}^N (1+2(j-1)\theta)\not=0$. Then, (b) and (d) are equivalent.
\end{corollary}

\begin{proof}
See \Cref{thm:equivalence_bc1}, \Cref{cor:equivalence_bc1}, and \Cref{thm:equivalence_odd1} and use the fact that $\md=\mbc(0)$.
\end{proof}

\subsection[\texorpdfstring{$[\cdot,\cdot]_{D^N(\theta)}$}{} in the high temperature regime]{Leading order terms of the type D Dunkl bilinear form in the high temperature regime}

\begin{theorem}
\label{thm:equivalence_odd2}
Suppose $F_N(x_1,\ldots,x_N)=\exp\left(\sum_{\lambda\in\evenN} c_\lambda(N)p_\lambda\right) + ed_0(N)+e\exp\\\left(\sum_{\lambda\in\evenN}d_\lambda(N)p_\lambda\right)$. Assume that $\lim_{N\rightarrow\infty}\theta N=c\in\mathbb{C}$. Consider the following statements.
\begin{enumerate}
\item[(a)] For all $\lambda\in\even$, $\lim_{N\rightarrow\infty} c_\lambda(N)=c_\lambda\in\mathbb{C}$. 
\item[(b)] It is the case that $\lim_{N\rightarrow\infty} d_0(N)=d_0\in\mathbb{C}$ and for all $\lambda\in\even$, $\lim_{N\rightarrow\infty} d_\lambda(N)=d_\lambda\in\mathbb{C}$.
\item[(c)] For all $\nu\in\even$,
\[
\lim_{N\rightarrow\infty} \frac{1}{N^{\ell(\nu)}} [p_\nu, F_N]_{D^N(\theta)} = \sum_{\lambda\in\even} \Wbc_{\nu\lambda}(c, 0)[p_\lambda] \exp\left(\sum_{\gamma\in\even} c_\gamma p_\gamma\right).
\]
\item[(d)] It is the case that 
\[
\lim_{N\rightarrow\infty} \frac{[e,F_N]_{D^N(\theta)}}{\prod_{j=1}^N (1+2(j-1)\theta)} = d_0+1
\]
and for all $\nu\in\even$,
\begin{align*}
&\lim_{N\rightarrow\infty}\frac{[ep_\nu, F_N]_{D^N(\theta)}}{N^{\ell(\nu)}\prod_{j=1}^{N} (1+2(j-1)\theta)} \\
& = \sum_{\lambda\in\even} \Wbc_{\nu\lambda}(c, 1)[p_\lambda]\exp\left(\sum_{\gamma\in\even} d_\gamma p_\gamma\right).
\end{align*}
\end{enumerate}
Then, (a) implies (c) and if $2c$ is not a negative integer, then (c) implies (a).

Assume that if $N$ is sufficiently large, then $\prod_{j=1}^N(1+2(j-1)\theta)\not=0$. Then, (b) implies (d) and if $2c$ is not a negative integer less than $-1$, then (d) implies (b).
\end{theorem}

\begin{proof}
The implications between (a) and (c) follow from \Cref{thm:equivalence_bc2} with $c=0$. For the implications between (b) and (d), we apply \Cref{lemma:leading_order_odd} as we did to prove \Cref{thm:equivalence_odd1}. Afterwards, we apply \Cref{thm:equivalence_bc2} with $c=1$.
\end{proof}

\begin{corollary}
\label{cor:equivalence_odd2}
Suppose $F_N(x_1,\ldots,x_N)=\exp\left(\sum_{\lambda\in\evenN} c_\lambda(N)p_\lambda\right) + ed_0(N)+e\exp\\\left(\sum_{\lambda\in\evenN}d_\lambda(N)p_\lambda\right)$. Assume that $\lim_{N\rightarrow\infty}\theta N=c\in\mathbb{C}$. Consider the following statements.
\begin{enumerate}
\item[(a)] For all $\lambda\in\even$, $\lim_{N\rightarrow\infty} c_\lambda(N) =c_\lambda\in\mathbb{C}$ if $\ell(\lambda)=1$ and $\lim_{N\rightarrow\infty} c_\lambda(N)=0$ if $\ell(\lambda)>1$.
\item[(b)]
It is the case that $\lim_{N\rightarrow\infty} d_0(N)=d_0\in\mathbb{C}$ and for all $\lambda\in\even$, $\lim_{N\rightarrow\infty} d_\lambda(N)=d_\lambda\in\mathbb{C}$ if $\ell(\lambda)=1$ and $\lim_{N\rightarrow\infty} d_\lambda(N)=0$ if $\ell(\lambda)>1$.
\item[(c)] For all $\nu\in\even$,
\[
\lim_{N\rightarrow\infty}\frac{1}{N^{\ell(\nu)}} [p_\nu, F_N]_{D^N(\theta)} = \prod_{i=1}^{\ell(\nu)} \sum_{\pi\in \NCeven(\nu_i)} W^{BC}(\pi)(c, 0)\prod_{B\in \pi} |B|c_{(|B|)}.
\]
\item[(d)] It is the case that 
\[
\lim_{N\rightarrow\infty} \frac{[e,F_N]_{D^N(\theta)}}{\prod_{j=1}^N (1+2(j-1)\theta)} = d_0+1
\]
and for all $\nu\in\even$,
\begin{align*}
&\lim_{N\rightarrow\infty}\frac{[ep_\nu, F_N]_{D^N(\theta)} }{N^{\ell(\nu)}\prod_{j=1}^{N} (1+2(j-1)\theta)} \\
& = \prod_{i=1}^{\ell(\nu)}\sum_{\pi\in \NCeven(\nu_i)} W^{BC}(\pi)(c, 1)\prod_{B\in \pi} |B|d_{(|B|)}.
\end{align*}
\end{enumerate}
Then, (a) implies (c) and if $2c$ is not a negative integer, then (c) implies (a).

Assume that if $N$ is sufficiently large, then $\prod_{j=1}^N (1+2(j-1)\theta)\not=0$. Then, (b) implies (d) and if $2c$ is not a negative integer less than $-1$, then (d) implies (b).
\end{corollary}

\begin{proof}
See \Cref{thm:equivalence_bc2}, \Cref{cor:equivalence_bc2}, and \Cref{thm:equivalence_odd2}.
\end{proof}

\section{Asymptotics of the coefficients of Bessel functions}
\label{sec:bessel_coeff}

In this section, we determine the asymptotics of the coefficients of Bessel functions in each of the regimes mentioned in \Cref{subsec:regime}. Note that we determine the asymptotics for the coefficients of the terms with a fixed degree and the analyses for different degrees are separate. Therefore, additional assumptions are required to deduce the convergence of the Bessel functions themselves rather than only their coefficients. For more discussion about this direction, see \Cref{subsec:uniformconverge}.

The basic idea of the proofs is to determine the asymptotics of the inverse of the matrix $([p_\lambda,p_\nu]_{A^{N-1}(\theta)})_{\lambda,\nu\in\Gamma[k]}$ and the analogous matrices for the $BC^N$ and $D^N$ root systems so that we can apply \pref{eq:eigenval}.

\subsection{Coefficients for the type A root system}

First, we consider the regime $|\theta N|\rightarrow\infty$.

\begin{lemma}
\label{lemma:inv_matrix_a1}
Assume that $|\theta N| \rightarrow\infty$. Suppose $k\geq 1$ and define $M\triangleq ([p_\lambda,p_\nu]_{A^{N-1}(\theta)}\\)_{\lambda,\nu\in\Gamma[k]}$. Suppose $\lambda,\nu\in\Gamma[k]$ and $\ell(\lambda)\leq \ell(\nu)$. 
\begin{enumerate}
\item[(a)] The entries $M^{-1}_{\lambda\nu}$ and $M^{-1}_{\nu\lambda}$ equal
\[
(\theta N)^{-k+\ell(\lambda)}N^{-\ell(\nu)} (\ma[k]^{-1}_{\lambda\nu} + o_N(1)).
\]
\item[(b)] The diagonal entry $\ma[k]_{\lambda\lambda}^{-1}$ equals $(\ma[k]_{\lambda\lambda})^{-1}$.
\item[(c)] If $\ell(\lambda)=\ell(\nu)$ and $\lambda\not=\nu$, then $\ma[k]_{\lambda\nu}^{-1}=0$.
\end{enumerate}
\end{lemma}

\begin{proof}
(a): Define the diagonal matrices $D_1$ and $D_2$ with diagonal entry $(\gamma,\gamma)$ given by $N^{\ell(\gamma)}$ and $(\theta N)^{\ell(\gamma)}$, respectively, for all $\gamma\in\Gamma[k]$. Based on \Cref{thm:leadingorder_a}, it is evident that 
\[
M = (\theta N)^k D_1 \left(\ma[k]+o_N(1)\right) D_2^{-1}.
\]
To obtain that, consider
\[
\mathcal{M}'=(\theta N)^{-k} D_1^{-1} M D_2.
\]
For $\lambda',\nu'\in\Gamma$, entry $(\lambda',\nu')$ of this matrix is $(\theta N)^{-k} N^{-\ell(\lambda')}(\theta N)^{\ell(\nu')} M_{\lambda'\nu'}$. If $\ell(\lambda')\leq \ell(\nu')$, then by \Cref{thm:leadingorder_a}, the entry is $\ma_{\lambda'\nu'}+o_N(1)$. Otherwise, if $\ell(\lambda')>\ell(\nu')$, the entry is
\[
O_N(N^{2(\ell(\nu')-\ell(\lambda'))}\theta^{\ell(\nu')-\ell(\lambda')}) = O_N(N^{\ell(\nu')-\ell(\lambda')} (\theta N)^{\ell(\nu')-\ell(\lambda')}) = o_N(1).
\]
Hence, $\mathcal{M}'=\ma[k]+o_N(1)$. It follows that
\[
M^{-1} = (\theta N)^{-k}D_2\left(\ma[k]^{-1}+o_N(1)\right) D_1^{-1}.
\]
With this expression, we can approximate $M^{-1}_{\lambda\nu}$ and therefore $M_{\nu\lambda}^{-1}$ as well.

(b) and (c): Suppose the elements of $I\subset \Gamma[k]$ are consecutive based on the ordering of the rows and columns of $\ma$; note that the rows and columns are ordered increasingly by length. Then, we have that
\[
\ma[k][I, I]^{-1} = \ma[k]^{-1}[I, I],
\]
where $\ma[k][I,I]$ and $\ma[k]^{-1}[I, I]$ denote $\ma[k]$ and $\ma[k]^{-1}$ with their rows and columns restricted to $I$, respectively. It is straightforward to prove (b) and (c) using this result.
\end{proof}

\begin{corollary}
\label{cor:bessel_a1}
If $|\theta N|\rightarrow\infty$, then for $k\geq 1$ and $a,x\in\mathbb{C}^N$,
\begin{align*}
& J^{A^{N-1}(\theta)}_a[k](x) = \sum_{r\in \Gamma[k]} (\theta N)^{-k} \theta^{\ell(r)}\left(\pi(r)\prod_{l=1}^{\ell(r)}r_l\right)^{-1}(1+o_N(1))r(a)r(x)\\
& + \sum_{r,s\in\Gamma[k],\,r\not=s}(\theta N)^{-k}(\theta N)^{\min(\ell(r), \ell(s))} N^{-\max(\ell(r),\ell(s))}\times \\
& (\ma[k]^{-1}_{rs} + \ma[k]^{-1}_{sr}+ o_N(1))r(a)s(x).
\end{align*}
\end{corollary}

\begin{proof}
See \pref{eq:eigenval}, \Cref{lemma:matrix_a1}, and \Cref{lemma:inv_matrix_a1}.
\end{proof}

Next, we consider the regime $\theta N\rightarrow c\in\mathbb{C}$. Recall from \Cref{thm:equivalence_a2} that in order for the invertibility of $\mathcal{D}(A^{N-1}(\theta))$ to hold in an asymptotic sense, we must have that $c$ is not a negative integer. We require the invertibility of $\mathcal{D}(A^{N-1}(\theta))$ to compute the Bessel function, so we assume that $c$ is not a negative integer.

\begin{lemma}
\label{lemma:inv_matrix_a2}
Assume that $\theta N\rightarrow c\in\mathbb{C}$ and that $c$ is not a negative integer. Suppose $k\geq 1$ and define $M\triangleq ([p_\lambda,p_\nu]_{A^{N-1}(\theta)})_{\lambda,\nu\in\Gamma[k]}$. Suppose $\lambda,\nu\in\Gamma[k]$ and $\ell(\lambda)\leq \ell(\nu)$. 
\begin{enumerate}
\item[(a)] The entries $M^{-1}_{\lambda\nu}$ and $M^{-1}_{\nu\lambda}$ equal
\[
N^{-\ell(\nu)}(\Wa(c)[k]^{-1}_{\lambda\nu} + o_N(1)).
\]
\item[(b)] The diagonal entry $\Wa(c)[k]_{\lambda\lambda}^{-1}$ equals $(\Wa(c)[k]_{\lambda\lambda})^{-1}$.
\item[(c)] If $\ell(\lambda)=\ell(\nu)$ and $\lambda\not=\nu$, then $\Wa(c)[k]_{\lambda\nu}^{-1}=0$.
\end{enumerate}
\end{lemma}

\begin{proof}
(a): We can use the same argument as the proof of \Cref{lemma:inv_matrix_a1} and \Cref{thm:leadingorder_a2} to obtain that
\[
M=D (\mathcal{W}^A(c)[k] + o_N(1))
\]
where $D$ is a diagonal matrix with diagonal entry $(\gamma,\gamma)$ given by $N^{\ell(\gamma)}$ for $\gamma\in\Gamma[k]$. Hence,
\[
M^{-1} = D^{-1}(\mathcal{W}^A(c)[k]^{-1}+o_N(1)),
\]
which completes the proof since we can approximate $M^{-1}_{\lambda\nu}$ and therefore $M^{-1}_{\nu\lambda}$ as well.
\end{proof}

\begin{corollary}
\label{cor:bessel_a2}
If $\theta N\rightarrow c\in\mathbb{C}$ and $c$ is not a negative integer, then for $k\geq 1$ and $a,x\in\mathbb{C}^N$,
\begin{align*}
& J^{A^{N-1}(\theta)}_a[k](x) = \sum_{r\in \Gamma[k]} N^{-\ell(r)}\left(\pi(r)\prod_{l=1}^{\ell(r)}r_l W^A([r_l])(c)\right)^{-1}(1+o_N(1))r(a)r(x)\\
& + \sum_{r,s\in\Gamma[k],\, r\not=s}N^{-\max(\ell(r),\ell(s))}(\Wa(c)[k]^{-1}_{rs} + \Wa(c)[k]^{-1}_{sr}+ o_N(1))r(a)s(x).
\end{align*}
\end{corollary}

\begin{proof}
See \pref{eq:eigenval}, \Cref{lemma:matrix_a2}, and \Cref{lemma:inv_matrix_a2}.
\end{proof}

\subsection{Coefficients for the type BC root system}

The framework of this subsection is analogous to that of the $A^{N-1}$ root system.

\begin{lemma}
\label{lemma:inv_matrix_bc1}
Suppose $|\theta_0 N|\rightarrow\infty$ and $\frac{\theta_1}{\theta_0 N} \rightarrow c\in\mathbb{C}$ such that $c\not=-1$. Suppose $k\geq 1$ and define $M\triangleq ([p_\lambda,p_\nu]_{BC^N(\theta_0,\theta_1)})_{\lambda,\nu\in\even[k]}$. Suppose $\lambda,\nu\in\even[k]$ and $\ell(\lambda)\leq \ell(\nu)$. 
\begin{enumerate}
\item[(a)] The entries $M^{-1}_{\lambda\nu}$ and $M^{-1}_{\nu\lambda}$ equal
\[
(\theta_0 N)^{-k+\ell(\lambda)}N^{-\ell(\nu)} (\mbc(c)[k]^{-1}_{\lambda\nu} + o_N(1)).
\]
\item[(b)] The diagonal entry $\mbc(c)[k]_{\lambda\lambda}^{-1}$ equals $(\mbc(c)[k]_{\lambda\lambda})^{-1}$.
\item[(c)] If $\ell(\lambda)=\ell(\nu)$ and $\lambda\not=\nu$, then $\mbc(c)[k]_{\lambda\nu}^{-1}=0$.
\end{enumerate}
\end{lemma}

\begin{proof}
See \Cref{thm:leadingorder_bc}.
\end{proof}

\begin{corollary}
\label{cor:bessel_bc1}
If $|\theta_0 N|\rightarrow\infty$, $\frac{\theta_1}{\theta_0 N}\rightarrow c\in\mathbb{C}$, and $c\not=-1$, then for $k\geq 1$ and $a,x\in\mathbb{C}^N$,
\begin{align*}
& J^{BC^N(\theta_0,\theta_1)}_a[k](x) = \sum_{r\in \even[k]} (\theta_0 N)^{-k} \theta_0^{\ell(r)}\left(\pi(r)\prod_{l=1}^{\ell(r)}2^{r_l-1}r_l (1+c)^{o([r_l])}\right)^{-1} \times \\ 
&(1+o_N(1))r(a)r(x)+ \sum_{r,s\in\even[k],\, r\not=s}(\theta_0 N)^{-k}(\theta_0 N)^{\min(\ell(r), \ell(s))} N^{-\max(\ell(r),\ell(s))}\times \\
& (\mbc(c)[k]^{-1}_{rs} + \mbc(c)[k]^{-1}_{sr} + o_N(1))r(a)s(x).
\end{align*}
\end{corollary}

\begin{proof}
See \pref{eq:eigenval}, \Cref{lemma:matrix_bc1}, and \Cref{lemma:inv_matrix_bc1}.
\end{proof}

\begin{lemma}
\label{lemma:inv_matrix_bc2}
Suppose $\theta_0 N\rightarrow c_0\in\mathbb{C}$, $\theta_1\rightarrow c_1\in\mathbb{C}$, $c_0$ is not a negative integer, and $2c_0+2c_1$ is not a negative odd integer. Suppose $k\geq 1$ and define $M\triangleq ([p_\lambda,p_\nu]_{BC^N(\theta_0,\theta_1)}\\)_{\lambda,\nu\in\even[k]}$. Suppose $\lambda,\nu\in\even[k]$ and $\ell(\lambda)\leq \ell(\nu)$. 
\begin{enumerate}
\item[(a)] The entries $M^{-1}_{\lambda\nu}$ and $M^{-1}_{\nu\lambda}$ equal
\[
N^{-\ell(\nu)}(\Wbc(c_0,c_1)[k]^{-1}_{\lambda\nu} + o_N(1)).
\]
\item[(b)] The diagonal entry $\Wbc(c_0, c_1)[k]_{\lambda\lambda}^{-1}$ equals $(\Wbc(c_0, c_1)[k]_{\lambda\lambda})^{-1}$.
\item[(c)] If $\ell(\lambda)=\ell(\nu)$ and $\lambda\not=\nu$, then $\Wbc(c_0, c_1)[k]_{\lambda\nu}^{-1}=0$.
\end{enumerate}
\end{lemma}

\begin{proof}
See \Cref{thm:leadingorder_bc2}.
\end{proof}

\begin{corollary}
\label{cor:bessel_bc2}
If $\theta_0 N\rightarrow c_0\in\mathbb{C}$, $\theta_1\rightarrow c_1\in\mathbb{C}$, $c_0$ is not a negative integer, and $2c_0+2c_1$ is not a negative odd integer, then for $k\geq 1$ and $a,x\in\mathbb{C}^N$,
\begin{align*}
& J^{BC^N(\theta_0,\theta_1)}_a[k](x) = \sum_{r\in \even[k]} N^{-\ell(r)}\left(\pi(r)\prod_{l=1}^{\ell(r)}r_l W^{BC}([r_l])(c_0, c_1)\right)^{-1}\times \\
& (1+o_N(1))r(a)r(x)+ \sum_{r,s\in\even[k],\, r\not=s}N^{-\max(\ell(r),\ell(s))} \times \\
& (\Wbc(c_0,c_1)[k]^{-1}_{rs} + \Wbc(c_0,c_1)[k]^{-1}_{sr} + o_N(1))r(a)s(x).
\end{align*}
\end{corollary}

\begin{proof}
See \pref{eq:eigenval}, \Cref{lemma:matrix_bc2}, and \Cref{lemma:inv_matrix_bc2}.
\end{proof}

\subsection{Coefficients for the type D root system}

The framework of this subsection is analogous to that of the $A^{N-1}$ root system. By \cite{demni_type_d}, we have that 
\begin{equation}
\label{eq:type_d_simple}
J_a^{D^N(\theta)}(x) = J_a^{BC^N(\theta,0)}(x) + \frac{e(a)e(x)}{\prod_{i=1}^N (1+2(i-1)\theta)} J_a^{BC^N(\theta,1)}(x).
\end{equation}
Therefore, it suffices to consider the coefficients of the Bessel function for the $BC^N$ root system.

\section{Applications}
\label{sec:applications}

In this section, we discuss applications of the main results of this paper. Some settings that we consider include the weak convergence of the measures in \Cref{conjecture} to the free convolution and the uniform convergence of the Bessel functions in the $\theta N\rightarrow c\in\mathbb{C}$ regime and for Vershik-Kerov sequences.

\subsection{Weak convergence to the free convolution}
\label{subsec:conjecture_applications}

The results of this subsection are based on the assumption of \Cref{conjecture}. Afterwards, the goal is to show that if $a_1$ and $a_2$ converge weakly to some distributions, then $\mu_{a_1,a_2}^{\mathcal{R}(\theta)}$ also converges weakly to some distribution that can be described in terms of the free convolution, if we are working with the type A and D root systems. In particular, we prove \Cref{cor:freeconv}. For the type BC root system, we show weak converge to the rectangular free convolution, see \Cref{cor:freeconv_rectangular}.

The following result describes the support of $\mu_{a_1,a_2}^{\mathcal{R}(\theta)}$. As we mentioned earlier, \cite{trimeche} shows that there exists a solution for $\mu_{a_1,a_2}^{\mathcal{R}(\theta)}$ that is signed and is supported over $B(0, \norm{a_1}_2+\norm{a_2}_2)$. If we assume the conjecture, then we can reduce the domain of $\mu_{a_1,a_2}^{\mathcal{R}(\theta)}$ further, see \Cref{lemma:conjecture_support}. We conjecture that there always exists a signed solution for $\mu_{a_1,a_2}^{\mathcal{R}(\theta)}$ that has the same support as what is stated in the corollary.

\begin{lemma}
\label{lemma:conjecture_support}
Assume that $\theta\in\theta(\mathcal{R})$ is nonnegative. Suppose $a_1,a_2\in\mathbb{R}^N$ and $\mu_{a_1,a_2}^{\mathcal{R}(\theta)}$ satisfies the conditions of \Cref{conjecture}. Then, every element of $\text{supp}(\mu^{\mathcal{R}(\theta)}_{a_1,a_2})$ can be expressed as $u_1+u_2$, where $u_i$ is in the convex hull of $H(\mathcal{R})a_i$ for $i\in\{1,2\}$.
\end{lemma}

\begin{proof}
Consider the measures $\mu_a^{\text{sym}}$ for $a\in\mathbb{R}^N$ as they are defined in \Cref{thm:positivity}. Then, we have that
\[
\int_{\mathbb{R}^N}\int_{\mathbb{R}^N} e^{\sum_{i=1}^N x_i (\epsilon_i^1 + \epsilon_i^2)} d\mu_{a_1}^{\text{sym}}(\epsilon^1)d\mu_{a_2}^{\text{sym}}(\epsilon^2) = \int_{\mathbb{R}^N} \int_{\mathbb{R}^N}  e^{\sum_{i=1}^N x_i\epsilon_i}  d\mu_a^{\text{sym}}(\epsilon)d\mu^{\mathcal{R}(\theta)}_{a_1,a_2}(a).
\]
for all $x\in\mathbb{C}^N$, where we have used the fact that $\mu_a^{\text{sym}}$ is compactly supported for all $a\in\mathbb{R}^N$ to obtain the left hand side.

By considering $x\in \mathbf{i}\mathbb{R}^N$, we have the characteristic functions of the random variables $\epsilon^1+\epsilon^2$ where $\epsilon^1\sim \mu_{a_1}^{\text{sym}}$ and $\epsilon^2\sim\mu_{a_2}^{\text{sym}}$ and $\epsilon$ where $a\sim\mu^{\mathcal{R}(\theta)}_{a_1,a_2}$ and $\epsilon\sim\mu_a^{\text{sym}}$ are equal. Therefore, these two random variables have the same distribution.

Suppose $u\in\text{supp}(\mu^{\mathcal{R}(\theta)}_{a_1,a_2})$ and $r>0$. We show that the probability that $\epsilon$ is in $B(u,r)$ is positive to show that $u$ is in the support of $\epsilon$. First, observe that
\[
\Pr[\epsilon\in B(u,r)] \geq \Pr_{a\sim \mu_{a_1,a_2}^{\mathcal{R}(\theta)},\, \epsilon'\sim \mu_a^{\text{sym}}}\left[a\in B(u, \frac{r}{2}), \epsilon'\in B(a, \frac{r}{2})\right].
\]
For the sake of contradiction, assume that $\Pr_{a\sim \mu_{a_1,a_2}^{\mathcal{R}(\theta)},\, \epsilon'\sim \mu_a^{\text{sym}}}\left[a\in B(u, \frac{r}{2}), \epsilon'\in B(a, \frac{r}{2})\right]=0$. Then, there exists a sequence $\{z_i\}_{i\geq 1}$ of elements of $B(u,r)$ such that
\[
\lim_{i\rightarrow\infty} \mu_{z_i}^{\text{sym}}\left[B(z_i, \frac{r}{2})\right] = 0,
\]
since $\mu^{\mathcal{R}(\theta)}_{a_1,a_2}\left[B(u,\frac{r}{2})\right]>0$. Let $z$ be an element of the closed ball $\overline{B}(u,r)$ that is the limit of a subsequence $\{z_{i_j}\}_{j\geq 1}$ of $\{z_i\}_{i\geq 1}$.

Observe that for all $x\in\mathbb{R}^N$, $\lim_{j\rightarrow\infty} J_{z_{i_j}}^{\mathcal{R}(\theta)}(\mathbf{i}x) = J_z^{\mathcal{R}(\theta)}(\mathbf{i}x)$. Hence, it follows that the characteristic functions of $\mu^{\text{sym}}_{z_{i_j}}$ converge pointwise to the characteristic function of $\mu^{\text{sym}}_z$. By L\'{e}vy's continuity theorem, $\mu^{\text{sym}}_{z_{i_j}}$ converges weakly to $\mu^{\text{sym}}_z$. It then follows that 
\[
\mu^{\text{sym}}_z\left(B(z,\frac{r}{4})\right) = \lim_{j\rightarrow\infty} \mu^{\text{sym}}_{z_{i_j}}\left(B(z,\frac{r}{4})\right) \leq \lim_{j\rightarrow\infty} \mu^{\text{sym}}_{z_{i_j}}\left(B(z_{i_j}, \frac{r}{2})\right) = 0,
\]
which is a contradiction since $z\in\text{supp}(\mu_z^{\text{sym}})$ by \Cref{thm:positivity}. Thus, $u$ is in the support of $\epsilon$.

However, the distribution of $\epsilon$ is the same as that of $\epsilon^1+\epsilon^2$, so $u$ is in the support of $\epsilon^1+\epsilon^2$. Since the support of $\epsilon^i$ is a subset of the convex hull of $H(\mathcal{R})a_i$ for $i\in \{1,2\}$ by \Cref{thm:positivity}, the result follows.
\end{proof}

In the following corollary, we prove the convergence of the measures $\mu_{a,b}^{\mathcal{R}(\theta)}$ in terms of power sums after assuming that $a$ and $b$ converge in terms of power sums as $N\rightarrow\infty$. Afterwards, it is straightforward to deduce \Cref{cor:freeconv,cor:freeconv_rectangular}.

\begin{corollary}
\label{cor:product}
Suppose that for $N\geq 2$, $a(N),b(N)\in\mathbb{R}^N$. For all $N\geq 2$, assume that $\theta,\theta_0,\theta_1\geq 0$.

\begin{itemize}
\item[(A)] Suppose $\lim_{N\rightarrow\infty} \theta N=\infty$ and that $m_d^a$ and $m_d^b$ are real numbers such that for all $d\geq 1$,
\[
\lim_{N\rightarrow\infty} \frac{p_{(d)}(a(N))}{(\theta N)^{d}N}=m_d^a \text{ and } \lim_{N\rightarrow\infty} \frac{p_{(d)}(b(N))}{(\theta N)^{d}N^1}=m_d^b.
\]
Let $a_{(k)}$ and $b_{(k)}$ for $k\geq 1$ be real numbers that solve 
\[
m_d^a = \sum_{\pi\in NC(k)} \prod_{B\in \pi}|B| a_{(|B|)} \text{ and } m_d^b = \sum_{\pi\in NC(k)} \prod_{B\in \pi}|B| b_{(|B|)}
\]
for $d\geq 1$. Assume that for $N\geq 2$, $\mu_{a(N),\,b(N)}^{A^{N-1}(\theta)}$ satisfies the conditions of \Cref{conjecture}. Then, for all $\lambda\in\Gamma$, 
\[
\lim_{N\rightarrow\infty} \frac{\E_{c\sim\mu_{a(N),\,b(N)}^{A^{N-1}(\theta)}}\left[p_\lambda(c)\right]}{(\theta N)^{|\lambda|}N^{\ell(\lambda)}} = \prod_{i=1}^{\ell(\lambda)}\sum_{\pi\in NC(\lambda_i)} \prod_{B\in \pi} |B|(a_{(|B|)} + b_{(|B|)}).
\]
\item[(B)] Suppose $c\geq 0$, $\lim_{N\rightarrow\infty} \theta_0 N = \infty$, and $\lim_{N\rightarrow\infty}\frac{\theta_1}{\theta_0 N} = c$. Suppose $m_{2d}^a$ and $m_{2d}^b$ are real numbers such that 
\[
\lim_{N\rightarrow\infty} \frac{\sum_{i=1}^N a(N)_i^{2d}}{(\theta_0 N)^{2d}N}=m_{2d}^a \text{ and } \lim_{N\rightarrow\infty} \frac{\sum_{i=1}^N b(N)_i^{2d}}{(\theta_0 N)^{2d}N}=m_{2d}^b
\]
for $d\geq 1$. Let $a_{(2k)}$ and $b_{(2k)}$ for $k\geq 1$ be real numbers that solve 
\begin{align*}
& m_{2d}^a = \sum_{\pi\in \NCeven(2d)} (1+c)^{o(\pi)} \prod_{B\in \pi} 2^{|B|-1}|B|a_{(|B|)} \text{ and} \\
& m_{2d}^b = \sum_{\pi\in \NCeven(2d)} (1+c)^{o(\pi)}\prod_{B\in \pi} 2^{|B|-1}|B|b_{(|B|)}
\end{align*}
for $d\geq 1$. Assume that for $N\geq 2$, $\mu_{a(N),\,b(N)}^{BC^N(\theta_0,\theta_1)}$ satisfies the conditions of \Cref{conjecture}. Then, for all $\lambda\in\even$,
\begin{align*}
&\lim_{N\rightarrow\infty} \frac{\E_{c\sim\mu_{a(N),\,b(N)}^{BC^N(\theta_0, \theta_1)}}\left[p_\lambda(c)\right]}{(\theta_0 N)^{|\lambda|}N^{\ell(\lambda)}} \\
& = \prod_{i=1}^{\ell(\lambda)}\sum_{\pi\in \NCeven(\lambda_i)}(1+c)^{o(\pi)} \prod_{B\in \pi} 2^{|B|-1}|B|(a_{(|B|)} + b_{(|B|)}).
\end{align*}
\item[(C)] Suppose $\lim_{N\rightarrow\infty} \theta N = \infty$ and that $m_{2d}^a$ and $m_{2d}^b$ are real numbers such that 
\[
\lim_{N\rightarrow\infty} \frac{\sum_{i=1}^N a(N)_i^{2d}}{(\theta_0 N)^{2d}N}=m_{2d}^a \text{ and } \lim_{N\rightarrow\infty} \frac{\sum_{i=1}^N b(N)_i^{2d}}{(\theta_0 N)^{2d}N}=m_{2d}^b
\]
for $d\geq 1$. Let $a_{(2k)}$ and $b_{(2k)}$ for $k\geq 1$ be real numbers that solve 
\begin{align*}
& m_{2d}^a = \sum_{\pi\in \NCeven(2d)} \prod_{B\in \pi} 2^{|B|-1}|B| a_{(|B|)} \text{ and } \\ 
& m_{2d}^b = \sum_{\pi\in \NCeven(2d)} \prod_{B\in \pi} 2^{|B|-1}|B|b_{(|B|)}
\end{align*}
for $d\geq 1$. Assume that for $N\geq 2$, $\mu_{a(N),\,b(N)}^{D^N(\theta)}$ satisfies the conditions of \Cref{conjecture}. Then, for all $\lambda\in\even$,
\[
\lim_{N\rightarrow\infty} \frac{\E_{c\sim\mu_{a(N),\,b(N)}^{D^N(\theta)}}\left[p_\lambda(c)\right]}{(\theta N)^{|\lambda|}N^{\ell(\lambda)}} = \prod_{i=1}^{\ell(\lambda)}\sum_{\pi\in \NCeven(\lambda_i)} \prod_{B\in \pi} 2^{|B|-1}|B|(a_{(|B|)} + b_{(|B|)}).
\]
\end{itemize}
\end{corollary}

\begin{remark}
We note that the proof of \Cref{cor:product} is straightforward and that its framework has appeared previously, for example see \cite{matrix}*{Proof of Theorem 1.5}. The main contribution of the proof is the application of \Cref{cor:equivalence_a1}.
\end{remark}

\begin{proof}[Proof of \Cref{cor:product}]
(A): Suppose $N\geq 2$. Suppose $\Omega_N\triangleq B(0,r_N)\subset\mathbb{C}^N$ for $r_N$ sufficiently small such that $J_{a(N)}^{A^{N-1}(\theta)}$ and $J_{b(N)}^{A^{N-1}(\theta)}$ are nonzero over $\Omega_N$. Note that since 
\[
J_{a(N)}^{A^{N-1}(\theta)}(0) = J_{b(N)}^{A^{N-1}(\theta)}(0) = 1,
\]
$\Omega_N$ exists. Then, we can express $J_{a(N)}^{A^{N-1}(\theta)} = \exp\left(\sum_{\lambda\in\Gamma_N} a_\lambda(N)p_\lambda\right)$ and $J_{b(N)}^{A^{N-1}(\theta)} = \exp\left(\sum_{\lambda\in\Gamma_N} b_\lambda(N)p_\lambda\right)$ over $\Omega_N$.

For $N\geq 2$, $\mu_{a(N),\,b(N)}^{A^{N-1}(\theta)}$ is always compactly supported by \Cref{lemma:conjecture_support} and is therefore exponentially decaying, see \Cref{def:exponent_decay}. Thus, part (D) of \Cref{lemma:exponent_decay} implies that for all $\lambda\in\Gamma$,
\begin{equation}
\label{eq:sum_a1}
\begin{split}
\E_{c\sim\mu_{a(N),\,b(N)}^{A^{N-1}(\theta)}}[p_\lambda(c_1,\ldots,c_N)] &= \left[p_\lambda, J_{a(N)}^{A^{N-1}(\theta)}J_{b(N)}^{A^{N-1}(\theta)}\right]_{A^{N-1}(\theta)} \\
& = \left[p_\lambda, \exp\left(\sum_{\lambda\in\Gamma_N} (a_\lambda(N)+b_\lambda(N))p_\lambda\right)\right]_{A^{N-1}(\theta)};
\end{split}
\end{equation}
the second equality follows from the fact that $J_{a(N)}^{A^{N-1}(\theta)}J_{b(N)}^{A^{N-1}(\theta)}=\exp\big(\sum_{\lambda\in\Gamma_N} (a_\lambda(N)+\\b_\lambda(N))p_\lambda\big)$ over $\Omega_N$ and $0\in\Omega_N$.

By \Cref{cor:equivalence_a1}, we have that
\[
\lim_{N\rightarrow\infty} \frac{a_{(k)}(N)}{\theta N} = a_{(k)} \text{ and } \lim_{N\rightarrow\infty}\frac{b_{(k)}(N)}{\theta N} = b_{(k)}
\]
for $k\geq 1$ and
\[
\lim_{N\rightarrow\infty} \frac{a_\lambda(N)}{(\theta N)^{\ell(\lambda)}} = \lim_{N\rightarrow\infty} \frac{b_\lambda(N)}{(\theta N)^{\ell(\lambda)}} = 0
\]
for $\lambda\in\Gamma$ such that $\ell(\lambda)\geq 2$. Using \Cref{cor:equivalence_a1} again and \pref{eq:sum_a1} then gives that for all $\lambda\in\Gamma$, 
\[
\lim_{N\rightarrow\infty}\frac{\E_{c\sim\mu_{a(N),\,b(N)}^{A^{N-1}(\theta)}}[p_\lambda(c_1,\ldots,c_N)]}{N^{\ell(\lambda)}(\theta N)^{|\lambda|}} = \prod_{i=1}^{\ell(\lambda)} \sum_{\pi\in NC(\lambda_i)} \prod_{B\in \pi} (a_{(|B|)} + b_{(|B|)}). 
\]

(B) and (C): For (B) and (C), we can use \Cref{cor:equivalence_bc1,cor:equivalence_odd1}, respectively, and follow the same idea as the proof of (A).
\end{proof}

We also state the following results for the case where we are given a sequence $\{\mu_N\}_{N\geq 1}$ of measures whose even power sums converge and we need to analyze the convergence of $\{\frac{1}{2}\mu_N + \frac{1}{2}\mu_N(x\mapsto -x)\}_{N\geq 1}$ in terms of power sums. Observe that the convergence of the even power sums of $\{\mu_N\}_{N\geq 1}$ is equivalent to $\{\mu_N(x\mapsto x^2)\}_{N\geq 1}$ converging in terms of power sums. 

\begin{lemma}
\label{lemma:powersum_even}
Suppose $\{\mu_N\}_{N\geq 1}$ is a sequence of probability measures over $\mathbb{C}^N$. Assume that $\{\mu_N(x\mapsto x^2)\}_{N\geq 1}$ converges in terms of power sums. Then, $\{\frac{1}{2}\mu_N + \frac{1}{2}\mu_N(x\mapsto -x)\}_{N\geq 1}$ converges in terms of power sums.
\end{lemma}

\begin{proof}
It suffices to prove that for all $\nu\in\Gamma\backslash \even$, $\lim_{N\rightarrow\infty} \frac{\mathbb{E}_{a\sim \frac{1}{2}\mu_N + \frac{1}{2}\mu_N(x\mapsto -x)}[p_\nu(a)]}{N^{\ell(\nu)}} = 0$. Observe that the integral $\mathbb{E}_{a\sim \frac{1}{2}\mu_N + \frac{1}{2}\mu_N(x\mapsto -x)}[p_\nu(a)]$ converges if $N$ is significantly large, since this is the case for all $\nu\in\even$ by the assumption that $\{\mu_N(x\mapsto x^2)\}_{N\geq 1}$ converges in terms of power sums. Suppose we expand $p_\nu$ as a linear combination of monomials. Any summands that have an odd degree will not contribute to the expected value because $a$ is sampled from $\mu_N$ with probability $\frac{1}{2}$ and $\mu_N(x\mapsto -x)$ with probability $\frac{1}{2}$. Since $\nu$ has an odd part, it is clear that the sum of the summands with all even degrees can be expressed as a linear combination of $\{p_\lambda: \lambda\in \even[|\nu|], \ell(\lambda) < \ell(\nu)\}$. Due to the convergence of $\{\frac{1}{2}\mu_N + \frac{1}{2}\mu_N(x\mapsto -x)\}_{N\geq 1}$ in terms of power sums, the limit evaluates to zero.
\end{proof}

With this result, we are prepared to prove \Cref{cor:freeconv}.

\begin{proof}[Proof of \Cref{cor:freeconv}]
(A): By part (A) of \Cref{cor:product}, $\mu_{a(N),\,b(N)}^{A^{N-1}(\theta)}(x\mapsto\frac{x}{\theta N})$ converges in terms of power sums to $\mu$ as $N\rightarrow\infty$.

(B): The odd moments of $\frac{1}{2}(\mu_a+\mu_a^-)$ are always zero, so the odd free cumulants of the measure are also always zero, and similarly for $\frac{1}{2}(\mu_b+\mu_b^-)$. It follows that the odd free cumulants of $\tilde{\mu}$ are always zero and therefore that the odd moments of $\tilde{\mu}$ are always zero. Moreover, the even free cumulants of $\tilde{\mu}$ are given by adding the even free cumulants of $\frac{1}{2}(\mu_a+\mu_a^-)$ and $\frac{1}{2}(\mu_b+\mu_b^-)$. 

By part (C) of \Cref{cor:product}, the $2k$th free cumulants of $\frac{1}{2}(\mu_a+\mu_a^-)$ and $\frac{1}{2}(\mu_b+\mu_b^-)$ are $k2^{2k}a_{(2k)}$ and $k2^{2k}b_{(2k)}$, respectively, for $k\geq 1$. Therefore, the $2k$th free cumulant of $\tilde{\mu}$ is $k2^{2k}(a_{(2k)}+b_{(2k)})$ for $k\geq 1$. Using part (C) of \Cref{cor:product} and \Cref{lemma:powersum_even} allows for the deduction that $\frac{1}{2}\mu_{a(N),\,b(N)}^{D^N(\theta)}(x\mapsto \frac{x}{\theta N}) + \frac{1}{2}\mu_{a(N),\,b(N)}^{D^N(\theta)}(x\mapsto -\frac{x}{\theta N})$ converges in terms of power sums to $\tilde{\mu}$.
\end{proof}

We can use the same framework to prove \Cref{cor:freeconv_rectangular}.

\begin{proof}[Proof of \Cref{cor:freeconv_rectangular}]
We use part (B) of \Cref{cor:product}. Suppose $d\geq 1$ and $\pi\in \NCeven(2d)$. Recall that $o(\pi)$ is the number of $i\in [2d]$ such that $i$ is not the minimal element of a block of $\pi$ and $d(i;\,\pi)$ is odd. Suppose $\pi = B_1\sqcup \cdots \sqcup B_m$ such that $\min(B_j)<\min(B_{j+1})$ for $1\leq j\leq m-1$. Suppose $i\in B_j$ for $j\in [m]$. We have that 
\[
d(i;\,\pi) = \sum_{j'=1}^j |B_{j'}| - (i-1) \equiv i+1\pmod{2}.
\]
Hence, $d(i;\,\pi)$ is odd if and only if $i$ is even. It follows that $o(\pi)$ equals $d$ minus the number of even $i\in [2d]$ such that $b(i;\,\pi)=1$. If $E(\pi)$ denotes the number of even $i\in [2d]$ such that $b(i;\,\pi)=1$, then, 
\[
(1+c)^{o(\pi)}\prod_{B\in \pi} 2^{|B|-1}|B|a_{(|B|)} = (1+c)^{-E(\pi)}\prod_{B\in \pi} (1+c)^{\frac{|B|}{2}}2^{|B|-1}|B|a_{(|B|)};
\]
the equation also holds after replacing $a_{(|B|)}$ with $b_{(|B|)}$.

By the formula for the rectangular free convolution given in \cite{rectangular_free_convolution}*{Proposition 3.1}, the $2k$th rectangular free cumulants of $\frac{1}{2}(\mu_a+\mu_a^-)$ and $\frac{1}{2}(\mu_b+\mu_b^-)$ are $k(4(1+c))^ka_{(2k)}$ and $k(4(1+c))^kb_{(2k)}$, respectively, for $k\geq 1$. Hence, the $2k$th rectangular free cumulant of $\mu$ is $k(4(1+c))^k(a_{(2k)}+b_{(2k)})$ for $k\geq 1$. After applying \Cref{lemma:powersum_even}, it is evident that $\frac{1}{2}\mu_{a(N),\,b(N)}^{BC^N(\theta_0,\theta_1)}(x\mapsto \frac{x}{\theta_0 N}) + \frac{1}{2}\mu_{a(N),\,b(N)}^{BC^N(\theta_0,\theta_1)}(x\mapsto -\frac{x}{\theta_0 N})$ converges in terms of power sums to $\mu$, which satisfies the condition that $\mu(B)=\mu(-B)$ for all open subsets $B$ of $\mathbb{R}$.
\end{proof}

\subsection{Upper bounds of the magnitudes of Bessel functions}

In preparation for studying the uniform convergence of Bessel functions, we prove various upper bounds on them.

\begin{lemma}
\label{lemma:besselupper_1}
Suppose $\theta\in\Theta(\mathcal{R})$ and $\text{Re}(\theta(r))\geq 0$ for all $r\in\mathcal{R}$. Then,
\[
|J_a^{\mathcal{R}(\theta)}(x)| \leq \exp \left(\max_{h\in H(\mathcal{R})} \text{Re}(\product{ha, x})\right).
\]
\end{lemma}
\begin{proof}
By \cite{dunklbound}*{Theorem 3.1},
\begin{align*}
\left(\frac{1}{|H(\mathcal{R})|} \sum_{h\in H(\mathcal{R})} E_{ha}^{\mathcal{R}(\theta)}(x) \right)^2 \leq \frac{1}{|H(\mathcal{R})|} \sum_{h\in H(\mathcal{R})} E_{ha}^{\mathcal{R}(\theta)}(x)^2 \leq \exp\left(2\max_{h\in H(\mathcal{R})} \text{Re}(\product{ha, x})\right),
\end{align*}
which finishes the proof.
\end{proof}

\begin{lemma}
\label{lemma:besselupper_2}
Suppose $\theta\in\theta(\mathcal{R})$ is nonnegative. Suppose $a\in\mathbb{R}^N$ and $x\in\mathbb{C}^N$. Then,
\[
\norm{J^{\mathcal{R}(\theta)}_a(x)}_2 \leq \frac{1}{|H|} \sum_{h\in H} e^{\sum_{i=1}^N \text{Re}(x_i)ha_i}.
\]
\end{lemma}

\begin{proof}
First, observe that from \Cref{thm:positivity},
\[
J^{\mathcal{R}(\theta)}_a(x) = \int_{\mathbb{R}^N} e^{\sum_{i=1}^N x_i\epsilon_i}d\nu_a^{\text{sym; $\mathcal{R}(\theta)$}}(\epsilon)
\]
for a measure $\nu_a^{\text{sym; $\mathcal{R}(\theta)$}}$ over $\mathbb{R}^N$ which is supported over the convex hull of $H(\mathcal{R})a$ and is invariant with respect to the action of $H(\mathcal{R})$. In particular, 
\begin{align*}
\norm{J^{\mathcal{R}(\theta)}_a(x)}_2 & = \norm{\int_{\mathbb{R}^N} \sum_{h'\in H} \frac{1}{|H|}e^{\sum_{i=1}^N x_i h'\epsilon_i}d\nu_a^{\text{sym; $\mathcal{R}(\theta)$}}(\epsilon)}_2 \\ 
& \leq \int_{\mathbb{R}^N} \sum_{h'\in H} \frac{1}{|H|}e^{\sum_{i=1}^N \text{Re}(x_i) h'\epsilon_i}d\nu_a^{\text{sym; $\mathcal{R}(\theta)$}}(\epsilon).
\end{align*}
Suppose $\epsilon=\sum_{h\in H}c_h ha$, where $c_h\in [0,1]$ for $h\in H$ and $\sum_{h\in H} c_h = 1$. Then, for $h'\in H$,
\[
e^{\sum_{i=1}^N \text{Re}(x_i) h'\epsilon_i} = e^{\sum_{i=1}^N \text{Re}(x_i) \sum_{h\in H} c_h h'ha_i} \leq \sum_{h\in H} c_h e^{\sum_{i=1}^N \text{Re}(x_i)h'ha_i}
\]
by Jensen's inequality. It follows that 
\[
\sum_{h'\in H} \frac{1}{|H|}e^{\sum_{i=1}^N \text{Re}(x_i) h'\epsilon_i} \leq \sum_{h\in H} \frac{c_h}{|H|} \sum_{h'\in H} e^{\sum_{i=1}^N \text{Re}(x_i)h'ha_i} = \frac{1}{|H|} \sum_{h\in H} e^{\sum_{i=1}^N \text{Re}(x_i)ha_i}.
\]
This completes the proof.
\end{proof}

\begin{lemma}
\label{lemma:besselupper_3}
Suppose $N\geq 1$, $a\in\mathbb{R}^N$, and $x\in\mathbb{C}^N$. Suppose $R$ is a nonnegative real number. Assume that $H(\mathcal{R})$ contains the reflections which permute the entries of $x$. Furthermore, assume that $\theta\in\theta(\mathcal{R})$ is nonnegative.

\begin{enumerate}
\item[(A)]
Assume that for all $d\in\mathbb{N}$, $\frac{1}{N}\sum_{i=1}^N |a_i|^d \leq R^d$. Then,
\[
\norm{J_a^{\mathcal{R}(\theta)}(x)}_2 \leq e^{R\sum_{i=1}^N |\text{Re}(x_i)|}.
\]

\item[(B)] Assume that for all $d\in\mathbb{N}$, $\frac{1}{N}\sum_{i=1}^N |a_i|^d \leq d!R^d$. Then,
\[
\norm{J_a^{\mathcal{R}(\theta)}(x)}_2 \leq \sum_{d=0}^\infty \left(R\sum_{i=1}^N |\text{Re}(x_i)|\right)^d.
\]

\end{enumerate}
\end{lemma}

\begin{proof}
(A): We use \Cref{thm:positivity}. Note that for $z\in\mathbb{C}^N$, we let $\text{Re}(z)$ denote $(\text{Re}(z_i))_{i\in[N]}$ and for $r\in\mathbb{R}^N$, we let $|r|$ denote $(|r_i|)_{i\in[N]}$. First, observe that 
\begin{align*}
\norm{J_a^{\mathcal{R}(\theta)}(x)}_2 &= \norm{\int_{\mathbb{R}^N} e^{\product{x,\epsilon}}d\nu_a^{\text{sym; $\mathcal{R}(\theta)$}}(\epsilon)}_2 \leq \int_{\mathbb{R}^N} e^{\product{\text{Re}(x),\epsilon}}d\nu_a^{\text{sym; $\mathcal{R}(\theta)$}}(\epsilon) \\
& \leq \int_{\mathbb{R}^N} e^{\product{|\text{Re}(x)|,|\epsilon|}}d\nu_a^{\text{sym; $\mathcal{R}(\theta)$}}(\epsilon).
\end{align*}

Next, note that since $\nu_a^{\text{sym; $\mathcal{R}(\theta)$}}$ is permutation invariant because $H(\mathcal{R})$ contains the symmetric group,
\[
\int_{\mathbb{R}^N} e^{\product{|\text{Re}(x)|,|\epsilon|}}d\nu_a^{\text{sym; $\mathcal{R}(\theta)$}}(\epsilon) = \int_{\mathbb{R}^N}1+ \sum_{\nu\in\Gamma} \frac{M_\nu(|\text{Re}(x)|)M_\nu(|\epsilon|)}{M_\nu(1,\ldots,1)\nu!}d\nu_a^{\text{sym; $\mathcal{R}(\theta)$}}(\epsilon).
\]
Suppose $\epsilon=\sum_{h\in H} c_h ha$, where $c_h\in [0,1]$ for $h\in H$ satisfy $\sum_{h\in H} c_h=1$. Furthermore, suppose $\nu\in\Gamma$. By Muirhead's inequality and the triangle inequality,
\[
\frac{M_\nu(|\epsilon|)}{M_\nu(1,\ldots,1)} \leq \frac{\sum_{i=1}^N |\epsilon_i|^{|\nu|}}{N} \leq \frac{\sum_{i=1}^N \left(\sum_{h\in H} c_h|ha_i|\right)^{|\nu|}}{N}.
\]
By Jensen's inequality, we have that 
\[
\frac{\sum_{i=1}^N \left(\sum_{h\in H} c_h|ha_i|\right)^{|\nu|}}{N} \leq \frac{1}{N}\sum_{h\in H} c_h \sum_{i=1}^N |ha_i|^{|\nu|} = \frac{\sum_{i=1}^N |a_i|^{|\nu|}}{N} \leq R^{|\nu|}.
\]
Hence,
\begin{equation}
\label{eq:upperbound1}
\int_{\mathbb{R}^N} e^{\product{|\text{Re}(x)|,|\epsilon|}}d\nu_a^{\text{sym; $\mathcal{R}(\theta)$}}(\epsilon) \leq 1+\sum_{\nu\in\Gamma} R^{|\nu|}\frac{M_\nu(|\text{Re}(x)|)}{\nu!} = e^{R\sum_{i=1}^N |\text{Re}(x)_i|}.
\end{equation}

(B): The equation \pref{eq:upperbound1} becomes 
\[
\int_{\mathbb{R}^N} e^{\product{|\text{Re}(x)|,|\epsilon|}}d\nu_a^{\text{sym; $\mathcal{R}(\theta)$}}(\epsilon) \leq 1+\sum_{\nu\in\Gamma} |\nu|!R^{|\nu|}\frac{M_\nu(|\text{Re}(x)|)}{\pi(\nu)} = \sum_{d=0}^\infty \left(R\sum_{i=1}^N |\text{Re}(x_i)|\right)^d,
\]
which completes the proof.
\end{proof}

Parts (A) and (B) of the following lemma are from \cite{rank_infinity}, where they are proved using results from \cites{baker2,baker,okounkov_olshanki,bessel_convolution,forrester}. We state a proof of part (A) which does not involve Jack polynomials. Furthermore, we use a continuity argument to remove the requirement that the multiplicity function is nonnegative, which extends upon previous results.

\begin{lemma}
\label{lemma:onevar}
Suppose $N\geq 2$, $a\in\mathbb{C}^N$, and $z\in\mathbb{C}$. For $k\geq 1$, suppose $c_k$ is the homogeneous degree $k$ polynomial such that $c_k(a)$ is the coefficient of $z^k$ in $\exp\\\left(\theta\sum_{m=1}^\infty p_{(m)}(a)\frac{z^m}{m}\right)$.
\begin{enumerate}
\item[(A)] Assume that $\theta\in\mathbb{C}$ such that $\mathcal{D}(A^{N-1}(\theta))$ is invertible. The value of $J_a^{A^{N-1}(\theta)}(z,0,\\\ldots,0)$ is $1+\sum_{k=1}^\infty\frac{c_k(a)}{(\theta N)_k}z^k$.
\item[(B)] Assume that $\theta_0,\theta_1\in\mathbb{C}$ such that $\mathcal{D}(BC^N(\theta_0,\theta_1))$ is invertible. The value of $J_a^{BC^N(\theta_0,\theta_1)}(z,0,\ldots,0)$ is $1+\sum_{k=1}^\infty\frac{c_k(a^2)}{4^k(\theta_1+(N-1)\theta_0+\frac{1}{2})_k(\theta_0 N)_k}z^{2k}$.
\item[(C)] Assume that $\theta\in\mathbb{C}$ such that $\mathcal{D}(D^N(\theta))$ is invertible. The value of $J_a^{D^N(\theta)}(z,0,\ldots,\\0)$ is $1+\sum_{k=1}^\infty\frac{c_k(a^2)}{4^k((N-1)\theta+\frac{1}{2})_k(\theta N)_k}z^{2k}$.
\end{enumerate}
\end{lemma}

\begin{proof}
(A): Suppose $k\geq 1$. Let $r_k$ be the homogeneous degree $k$ polynomial such that $r_k(a)$ is the coefficient of $z^k$ in $J_a^{A^{N-1}(\theta)}(z,0,\ldots,0)$. Then,
\[
J_a^{A^{N-1}(\theta)}(z,0,\ldots,0) = 1+\sum_{k=1}^\infty r_k(a)z^k.
\]
The goal is to show that $r_k=\frac{c_k}{(\theta N)_k}$.

Observe that $r_k$ is the unique solution to $[r_k, m_\epsilon]_{A^{N-1}(\theta)}=0$ for $\epsilon\in\Gamma_N[k]$ such that $\ell(\epsilon)\geq 2$ and $[r_k,p_{(k)}]_{A^{N-1}(\theta)}=1$. Hence, it suffices to show that $r_k=\frac{c_k}{(\theta N)_k}$ satisfies these properties.

First, observe that $\mathcal{D}_2\mathcal{D}_1\exp\left(\theta\sum_{m=1}^\infty p_{(m)}(a)\frac{z^m}{m}\right)=0$; note that the Dunkl operators are applied to the variable $a$. This is because
\[
\mathcal{D}_1\exp\left(\theta\sum_{m=1}^\infty p_{(m)}(a)\frac{z^m}{m}\right)= \theta\sum_{m=1}^\infty a_1^{m-1}z^m\exp\left(\theta\sum_{m=1}^\infty p_{(m)}(a)\frac{z^m}{m}\right).
\]
Applying $\mathcal{D}_2$ to this expression gives
\[
\theta^2\left[\left(\sum_{m=1}^\infty a_1^{m-1}z^m\right)\left(\sum_{m=1}^\infty a_2^{m-1}z^m\right) - \sum_{m=1}^\infty z^m\sum_{i=0}^{m-2} a_1^{m-2-i}a_2^i\right]\exp\left(\theta\sum_{m=1}^\infty p_{(m)}(a)\frac{z^m}{m}\right) = 0.
\]
Hence, $[c_k,m_\epsilon]_{A^{N-1}(\theta)}=0$ for $\epsilon\in\Gamma_N[k]$ such that $\ell(\epsilon)\geq 2$.

To finish the proof, we must show that $[c_k, p_{(k)}]_{A^{N-1}(\theta)}=(\theta N)_k$. Equivalently, we must show that 
\[
\left[p_{(k)}(a), \exp\left(\theta\sum_{m=1}^\infty p_{(m)}(a)\frac{z^m}{m}\right)\right]_{A^{N-1}(\theta)}=(\theta N)_kz^k,
\]
since the left hand side equals $[c_k, p_{(k)}]_{A^{N-1}(\theta)}z^k$. However, note that this expression equals
\begin{align*}
\mathcal{D}(p_{(k)})\exp\left(\theta\sum_{m=1}^\infty p_{(m)}(a)\frac{z^m}{m}\right) &= \left(\sum_{i=1}^N \mathcal{D}_i\right)^k \exp\left(\theta\sum_{m=1}^\infty p_{(m)}(a)\frac{z^m}{m}\right) \\ 
& = \left(\sum_{i=1}^N \partial_i\right)^k \exp\left(\theta\sum_{m=1}^\infty p_{(m)}(a)\frac{z^m}{m}\right)
\end{align*}
Hence, $[c_k,p_{(k)}]_{A^{N-1}(\theta)}z^k = [1] \left(\sum_{i=1}^N \partial_i\right)^k \exp\left(\theta\sum_{m=1}^\infty p_{(m)}(a)\frac{z^m}{m}\right)$.

First, note that 
\[
[1] e^{\sum_{i=1}^N\partial_i} \exp\left(\theta\sum_{m=1}^\infty p_{(m)}(a)\frac{z^m}{m}\right) = \exp\left(N\theta\sum_{m=1}^\infty \frac{z^m}{m}\right).
\]
Hence, the value of $[1] \left(\sum_{i=1}^N \partial_i\right)^k \exp\left(\theta\sum_{m=1}^\infty p_{(m)}(a)\frac{z^m}{m}\right)$ is $k!$ times the term of degree $k$ in $z$ of $\exp\left(N\theta\sum_{m=1}^\infty \frac{z^m}{m}\right)$. That is, $[c_k,p_{(k)}]_{A^{N-1}(\theta)}=\partial_z^k \exp\left(N\theta\sum_{m=1}^\infty \frac{z^m}{m}\right)|_{z=0}$.

It is straightforward to deduce that this quantity is $(\theta N)_k$. Initiate $c=1$. At each of $k$ steps, we can decide to either multiply $c$ by $N\theta\sum_{m=1}^\infty z^m$ or differentiate $c$. Label the terms that we multiply $c$ by starting from $1$. We must multiply $c$ by $N\theta\sum_{m=1}^\infty z^m$ at the first step, and we label this term as $1$. If we do not multiply $c$ by $N\theta\sum_{m=1}^\infty z^m$ at a step, then we must replace one of the previously labeled terms with its derivative. Therefore, we always have that $c$ is the product of the labeled terms, which may have been differentiated.

Suppose we are currently at step $i\in [k]$ such that $i\geq 2$. Suppose $L\geq 1$ terms have been labeled. For $j\in [L]$, let $d_j$ be the number of times that term $j$ has been differentiated; note that $d_j$ does not include the first derivative applied to $N\theta\sum_{m=1}^\infty \frac{z^m}{m}$. Then, we have that
\[
\sum_{j\in L} d_j + L = i-1.
\]
Furthermore, the constant term of term $j$ is $d_j!$ for all $j\in [L]$ so the constant term of $c$ is $\prod_{j=1}^L d_j!$. For step $j$, we can decide to either differentiate term $j$ for some $j\in [L]$ or multiply $c$ by $N\theta\sum_{m=1}^\infty z^m$. If we differentiate term $j$ for some $j\in [L]$, we will multiply the constant term of $c$ by $d_j+1$. Otherwise, if we multiply $c$ by $N\theta\sum_{m=1}^\infty z^m$, we will multiply the constant term by $N\theta$. By adding these choices, we will multiply the constant term of $c$ by 
\[
\sum_{j=1}^L (d_j+1) + N\theta = N\theta + i-1.
\]
When $i=1$, we multiply the constant term by $N\theta$. Therefore,
\[
[c_k,p_{(k)}]_{A^{N-1}(\theta)}=\partial_z^k \exp\left(N\theta\sum_{m=1}^\infty \frac{z^m}{m}\right)\bigg|_{z=0}=\prod_{i=1}^k (N\theta + i - 1) = (N\theta)_k.
\]

(B): We follow the same method as (A). Suppose $k\geq 1$. Let $r_{2k}$ be the homogeneous degree $k$ polynomial such that $r_{2k}(a)$ is the coefficient of $z^{2k}$ in $J_a^{BC^N(\theta_0,\theta_1)}(z,0,\ldots,0)$. It suffices to show that $r_{2k}(a)=\frac{c_{k}(a^2)}{4^k(\theta_1+(N-1)\theta_0+\frac{1}{2})_k(\theta_0 N)_k}$. 

It suffices to show that $[c_k(a^2), m_\epsilon(a^2)]_{BC^N(\theta_0,\theta_1)}=0$ for all $\epsilon\in\Gamma_N[k]$ such that $\ell(\epsilon)\geq 2$ and that $[c_k(a^2), p_{(2k)}(a)]_{BC^N(\theta_0,\theta_1)}=4^k(\theta_1+(N-1)\theta_0+\frac{1}{2})_k(\theta_0 N)_k$. For the first statement, we similarly have that $\mathcal{D}_2\mathcal{D}_1 \exp\left(\theta_0 \sum_{m=1}^{\infty}p_{(2m)}(a)\frac{z^m}{m}\right)=0$. Observe that
\[
\mathcal{D}_1  \exp\left(\theta_0 \sum_{m=1}^{\infty}p_{(2m)}(a)\frac{z^m}{m}\right)= 2\theta_0\left(\sum_{m=1}^\infty a_1^{2m-1}z^m\right)\exp\left(\theta_0 \sum_{m=1}^{\infty}p_{(2m)}(a)\frac{z^m}{m}\right).
\]
Applying $\mathcal{D}_2$ to this expression gives 
\begin{align*}
& 4\theta_0^2 \left[\left(\sum_{m=1}^\infty a_1^{2m-1}z^m\right)\left(\sum_{m=1}^\infty a_2^{2m-1}z^m\right)-\sum_{m=1}^\infty z^m\sum_{i=0}^{m-2} a_1^{2m-3-2i}a_2^{2i+1}\right] \\ 
& \exp\left(\theta_0 \sum_{m=1}^{\infty}p_{(2m)}(a)\frac{z^m}{m}\right)=0.
\end{align*}

Hence, it suffices to show that $[c_k(a^2), p_{(2k)}(a)]_{BC^N(\theta_0,\theta_1)}=4^k(\theta_1+(N-1)\theta_0+\frac{1}{2})_k(\theta_0 N)_k$. It should be possible to do so using a similar argument as in (A). However, we use a different argument that proves the identity for all $\theta_0,\theta_1\in\mathbb{C}$ to avoid the computational details. First, we note that $[c_k(a^2), p_{(2k)}(a)]_{BC^N(\theta_0,\theta_1)}$ is a polynomial in $\theta_0$ and $\theta_1$ with real coefficients, where we assume that $N$ is fixed. Since the expression is true whenever $\theta_0,\theta_1\geq 0$ from \cite{baker2} or \cite{rank_infinity}, it must be true for all $\theta_0,\theta_1\in\mathbb{C}$. 

(C): This follows from (B) with $\theta_1$ set to be zero and \pref{eq:type_d_simple}.
\end{proof}

\subsection[Bessel functions in the high temperature regime]{The uniform convergence of Bessel functions in the high temperature regime}

\label{subsec:uniformconverge}

First, we introduce the following well-known result; we include the proof, which is also well-known, for completeness. A similar argument is used in \cite{rank_infinity}.

\begin{lemma}
\label{lemma:uniformconverge}
Suppose $r\geq 1$ and that $\Omega\subset\mathbb{C}^r$ is open and simply connected. Let $\{f_N\}_{N\geq 1}$ be a sequence of holomorphic functions over $\Omega$.
\begin{enumerate}
    \item[(A)] Assume that $f:\Omega\rightarrow\mathbb{C}$ and $\{f_N\}_{N\geq 1}$ converges to $f$ uniformly over compact subsets of $\Omega$. Then, $f$ is holomorphic over $\Omega$ and the partial derivatives of $\{f_N\}_{N\geq 1}$ converge to the partial derivatives of $f$ uniformly over compact subsets of $\Omega$.
    \item[(B)] Assume that $\{f_N\}_{N\geq 1}$ is uniformly bounded over compact subsets of $\Omega$. Suppose $z\in\Omega$, $\lim_{N\rightarrow\infty} f_N(z)=f(z)$, and the limits of the partial derivatives of $\{f_N\}_{N\geq 1}$ evaluated at $z$ equal the partial derivatives of $f$ evaluated at $z$. Then, $\{f_N\}_{N\geq 1}$ converges to $f$ uniformly over compact subsets of $\Omega$.
\end{enumerate}
\end{lemma}

\begin{proof}
(A): To show that $f$ is holomorphic, note that because $\Omega$ is simply connected, we may apply Cauchy's integral theorem and Morera's theorem. Afterwards, we can apply Cauchy's integral formula to deduce that the partial derivatives of $\{f_N\}_{N\geq 1}$ converge to the partial derivatives of $f$ uniformly over compact subsets of $\Omega$. 

Suppose the closed ball $\overline{B}(a, b)$ is a subset of $\Omega$. For the sake of contradiction, assume that for all $\epsilon>0$, the closed ball $\overline{B}(a,b+\epsilon)$ is not a subset of $\Omega$. Then, for $\epsilon>0$, let $x_\epsilon$ be an element of $\overline{B}(a,b)$ such that $\overline{B}(x_\epsilon, \epsilon)\not\subset \Omega$. By the compactness of $\overline{B}(a,b)$, $\{x_\epsilon\}_{\epsilon>0}$ has a convergent subsequence as $\epsilon\rightarrow 0$. Assume that the limit is $x$, which is an element of $\overline{B}(a,b)$ because the set is closed. Then, for all $\delta>0$, $B(x,\delta)\not\subset\Omega$, because there exists sufficiently small $\epsilon$ such that $\overline{B}(x_\epsilon,\epsilon)\subset B(x,\delta)$. This is a contradiction to $\Omega$ being open.

Assume that $\epsilon>0$ and $\overline{B}(a,b+\epsilon)\subset\Omega$. Since $\{f_N\}_{N\geq 1}$ is uniformly converging over $\overline{B}(a,b+\epsilon)$, by applying the Cauchy integral formula with the contour set as the boundary of $\overline{B}(a,b+\epsilon)$, we have that $\{\partial_i f_N\}_{N\geq 1}$ is uniformly converging over $\overline{B}(a,b)$ to $\partial_i f$ for all $i\in [r]$. Then, we can proceed with induction to obtain that all partial derivatives of $\{f_N\}_{N\geq 1}$ are uniformly converging over $\overline{B}(a,b)$.

Next, let $K\subset\Omega$ be a compact set. By the same argument that we mentioned earlier, there exists $\epsilon>0$ such that $\overline{B}(x,\epsilon)\subset\Omega_r$ for all $x\in K$. Then, $K\subset \bigcup_{x\in K} B(x,\epsilon)$; note that the union is over open balls rather than closed balls. Let $F$ be a finite subcover so that $K\subset\bigcup_{x\in F} B(x,\epsilon)$. Then, the partial derivatives of $\{f_N\}_{N\geq 1}$ are uniformly converging over $B(x,\epsilon)$ for each $x\in F$, so the partial derivatives are uniformly converging over $K$ to the partial derivatives of $f$.

(B): Suppose a subsequence of $\{f_N\}_{N\geq 1}$ converges uniformly to $g$ over compact subsets of $\Omega$. By (A), $g$ is holomorphic and the partial derivatives of the subsequence converge to the partial derivatives of $g$ uniformly over compact subsets of $\Omega$. Therefore, the partial derivatives of $f$ and $g$ evaluated at $z$ are equal. Furthermore, since $\lim_{N\rightarrow\infty} f_N(z)=f(z)$, we have that $f(z)=g(z)$. This implies that $f=g$ over $\Omega$ by the identity theorem. 

Afterwards, applying Montel's theorem implies that $\{f_N\}_{N\geq 1}$ converges to $f$ uniformly over compact subsets of $\Omega$. For the sake of contradiction, assume that $K$ is a compact subset of $\Omega$ such that $\{f_N\}_{N\geq 1}$ does not converge uniformly to $f$ over $K$. Then, suppose $\epsilon>0$ and $\{N_j\}_{j\geq 1}$ is an increasing sequence of positive integers such that for all $j\geq 1$, $\sup_{z\in K} |f_{N_j}(z)-f(z)|\geq \epsilon$. By Montel's theorem, since $\{f_{N_j}\}_{j\geq 1}$ is uniformly bounded over compact subsets of $\Omega$, it has a uniformly converging subsequence over compact subsets of $\Omega$. However, this subsequence must uniformly converge to $f$ over $K$, which is a contradiction to $\sup_{z\in K} |f_{N_j}(z)-f(z)|\geq \epsilon$ for all $j\geq 1$.
\end{proof}

We use the previous result to justify the uniform convergence of the Bessel functions $J_{a(N)}^{\mathcal{R}(\theta)}(z_1,\ldots,z_r,0,\ldots,0)$ as $N\rightarrow\infty$, similarly to what is proved in the papers \cites{assiotis,rank_infinity}. While the papers consider when $\theta\in\mathbb{R}_{\geq 0}$ is fixed, we consider when $\theta N \rightarrow c\in\mathbb{R}_{\geq 0}$. Furthermore, we focus on when we know some uniform bounds on the moments of $\{a(N)\}_{N\geq 1}$. In particular, if the bounds are strong enough, then we can deduce uniform convergence over compact subsets of $\mathbb{C}^r$.

\begin{theorem}
\label{thm:converge1}
Suppose $m_d\in\mathbb{R}$ for $d\geq 1$. Assume that for $N\geq 2$, $a(N)\in\mathbb{R}^N$ such that $\lim_{N\rightarrow\infty} \frac{1}{N}\sum_{i=1}^N a(N)_i^d=m_d$ for all $d\geq 1$. Suppose $M$ is a positive real number. Consider the following two conditions.
\begin{enumerate}
\item[(C1)] For sufficiently large $N\in\mathbb{N}$, it is the case that for all $d\in\mathbb{N}$, $\frac{1}{N}\sum_{i=1}^N |a(N)_i|^d \leq M^d$.
\item[(C2)] For sufficiently large $N\in\mathbb{N}$, it is the case that for all $d\in\mathbb{N}$, $\frac{1}{N}\sum_{i=1}^N |a(N)_i|^d \leq d!M^d$.
\end{enumerate}
Assume that $\theta, \theta_0, \theta_1\geq 0$ for all $N\geq 2$. Suppose $r\geq 1$ and define $\Omega_r\triangleq\{z\in\mathbb{C}: \sum_{i=1}^r \norm{z_i}_2<\frac{1}{M}\}$.
\begin{itemize} 
\item[(A)] Suppose that as $N\rightarrow\infty$, $\theta N\rightarrow c\in\mathbb{R}_{\geq 0}$. Also, suppose $c_{(k)}$ for $k\geq 1$ solve
\[
m_d=\sum_{\pi\in NC(d)} W^A(\pi)(c)\prod_{B\in \pi} |B|c_{(|B|)}
\]
for $d\geq 1$.

As $N\rightarrow\infty$, $J^{A^{N-1}(\theta)}_{a(N)}(z_1,\ldots,z_r,0,\ldots,0)$ converges to $\exp\left(\sum_{k\geq 1}c_{(k)} p_{(k)}\right)$ uniformly over compact subsets of $\Omega$, where $\Omega=\mathbb{C}^r$ if (C1) is satisfied and $\Omega=\Omega_r$ if (C2) is satisfied.

\item[(B)] Suppose that as $N\rightarrow\infty$, $\theta_0 N \rightarrow c_0\in\mathbb{R}_{\geq 0}$ and $\theta_1\rightarrow c_1\in\mathbb{R}_{\geq 0}$. Also, suppose $c_{(2k)}$ for $k\geq 1$ solve
\[
m_{2d}=\sum_{\pi\in \NCeven(2d)} W^{BC}(\pi)(c_0, c_1)\prod_{B\in \pi} |B|c_{(|B|)}
\]
for $d\geq 1$.

As $N\rightarrow\infty$, $J^{BC^N(\theta_0,\theta_1)}_{a(N)}(z_1,\ldots,z_r,0,\ldots,0)$ converges to $\exp\left(\sum_{k\geq 1}c_{(2k)} p_{(2k)}\right)$ uniformly over compact subsets of $\Omega$, where $\Omega=\mathbb{C}^r$ if (C1) is satisfied and $\Omega=\Omega_r$ if (C2) is satisfied.

\item[(C)] Suppose that as $N\rightarrow\infty$, $\theta N\rightarrow c\in\mathbb{R}_{\geq 0}$. Also, suppose $c_{(2k)}$ for $k\geq 1$ solve
\[
m_{2d}=\sum_{\pi\in \NCeven(2d)} W^A(\pi)(2c)\prod_{B\in \pi} |B|c_{(|B|)}
\]
for $d\geq 1$.

As $N\rightarrow\infty$, $J^{D^N(\theta)}_{a(N)}(z_1,\ldots,z_r,0,\ldots,0)$ converges to $\exp\left(\sum_{k\geq 1}c_{(2k)} p_{(2k)}\right)$ uniformly over compact subsets of $\Omega$, where $\Omega=\mathbb{C}^r$ if (C1) is satisfied and $\Omega=\Omega_r$ if (C2) is satisfied.
\end{itemize}
\end{theorem}

\begin{proof} 
(A): By \Cref{lemma:besselupper_3}, $\zeta\triangleq \{J^{A^{N-1}(\theta)}_{a(N)}(z_1,\ldots,z_r,0,\ldots,0)\}_{N\geq r}$ is uniformly bounded over compact subsets of $\Omega$, where $\Omega=\mathbb{C}^r$ if (C1) is satisfied and $\Omega=\Omega_r$ if (C2) is satisfied. 

Suppose $k\geq 1$ and $r\in\Gamma[k]$. We compute the coefficient of $r(x)$ in $J_{a(N)}^{A^{N-1}(\theta)}$ as $N\rightarrow\infty$. First, note that for $s\in\Gamma$,
\[
\lim_{N\rightarrow\infty} N^{-\ell(s)} s(a(N)) = \prod_{i=1}^{\ell(s)} m_{s_i}.
\]
Then, using \Cref{cor:bessel_a2} gives that it is evident that the coefficient of $r(x)$ as $N\rightarrow\infty$ is 
\[
\sum_{s\in\Gamma[k]} \mathcal{W}^A(c)[k]^{-1}_{rs} \prod_{i=1}^{\ell(s)}m_{s_i}.
\]
To show that this quantity is $\frac{\prod_{i=1}^{\ell(r)} c_{(r_i)}}{\pi(r)}$, it suffices to show that 
\[
\mathcal{W}^A(c)[k]\left[\frac{\prod_{i=1}^{\ell(r)} c_{(r_i)}}{\pi(r)}\right]_{r\in\Gamma[k]}^T = \left[\prod_{i=1}^{\ell(r)}m_{r_i}\right]_{r\in\Gamma[k]}^T,
\]
since $\mathcal{W}^A(c)[k]$ is invertible. However, this is evident based on the formula for $\mathcal{W}^A(c)$ and the definition of the $c_{(k)}$; for example, see the proof of \Cref{cor:equivalence_a1}. Thus, the $N\rightarrow\infty$ limit of the coefficient of $r(x)$ in $J_{a(N)}^{A^{N-1}(\theta)}$ is the coefficient of $r(x)$ in $\exp\left(\sum_{k\geq 1} c_{(k)}p_{(k)}\right)$.

Afterwards, by part (B) of \Cref{lemma:uniformconverge}, we have that $\zeta$ converges to $\exp\left(\sum_{k\geq 1} c_{(k)}p_{(k)}\right)$ uniformly over compact subsets of $\Omega$.

(B) and (C): We can use the same method as (A); to replace \Cref{cor:bessel_a2}, we use \Cref{cor:bessel_bc2} for (B) and \pref{eq:type_d_simple} and \Cref{cor:bessel_bc2} for (C). For (C), observe that the terms with all odd degrees are eliminated over $\Omega$.
\end{proof}

\subsection[Bessel functions for Vershik-Kerov sequences]{The uniform convergence of Bessel functions for Vershik-Kerov\texorpdfstring{\\}{}sequences}
\label{subsec:vk}

In this subsection, we reprove the results of the papers \cites{assiotis,rank_infinity} and extend these results to the $D^N$ root system. The setting that we consider is equivalent to the setting of Vershik-Kerov sequences that the papers consider, see \Cref{remark:vk}. First, we prove the analogues of \Cref{thm:equivalence_a1,thm:equivalence_bc1,thm:equivalence_odd1}.

\begin{theorem} 
\label{thm:vk_a}
Suppose $F_N(x_1,\ldots,x_N) = \exp\left(\sum_{\lambda\in\Gamma_N} c_{\lambda}(N)p_\lambda\right)$ for all $N\geq 2$. Assume that $\theta\in\mathbb{C}$ is fixed. Consider the following statements.
\begin{enumerate}
\item[(a)] For all $\lambda\in\Gamma$, $\lim_{N\rightarrow\infty} c_\lambda(N)=c_\lambda\in\mathbb{C}$. 
\item[(b)] For all $\nu\in\Gamma$,
\[
\lim_{N\rightarrow\infty} \frac{1}{N^{|\nu|}} [p_\nu, F_N]_{A^{N-1}(\theta)} = \theta^{|\nu|-\ell(\nu)} \pi(\nu)\prod_{l=1}^{\ell(\nu)}\nu_l[p_\nu] \exp\left(\sum_{\gamma\in\Gamma}c_\gamma p_\gamma\right).
\]
\end{enumerate}
Then, (a) implies (b) and if $\theta\not=0$, then (b) implies (a).
\end{theorem}

\begin{proof}
See \Cref{thm:leadingorder_a}. The main idea is that for $\lambda,\nu\in\Gamma$, the leading order term of the expansion of $[p_\lambda,p_\nu]_{A^{N-1}(\theta)}$ given in the theorem is nonzero if and only if $\lambda=\nu$.
\end{proof}

\begin{theorem} 
\label{thm:vk_bc}
Suppose $F_N(x_1,\ldots,x_N) = \exp\left(\sum_{\lambda\in\evenN} c_{\lambda}(N)p_\lambda\right)$ for all $N\geq 2$. Assume that $\theta_0\in\mathbb{C}$ is fixed and $\lim_{N\rightarrow\infty} \frac{\theta_1}{\theta_0N}=c\in\mathbb{C}$. Consider the following statements.
\begin{enumerate}
\item[(a)] For all $\lambda\in\even$, $\lim_{N\rightarrow\infty} c_\lambda(N)=c_\lambda\in\mathbb{C}$. 
\item[(b)] For all $\nu\in\even$,
\[
\lim_{N\rightarrow\infty} \frac{1}{N^{|\nu|}} [p_\nu, F_N]_{BC^N(\theta_0,\theta_1)} = (2\theta_0)^{|\nu|-\ell(\nu)} \pi(\nu) \prod_{l=1}^{\ell(\nu)}\nu_l(1+c)^{o([\nu_l])}[p_\nu] \exp\left(\sum_{\gamma\in\even}c_\gamma p_\gamma\right).
\]
\end{enumerate}
Then, (a) implies (b) and if $\theta_0\not=0$ and $c\not=-1$, then (b) implies (a).
\end{theorem}

\begin{proof}
See \Cref{thm:leadingorder_bc}.
\end{proof}

\begin{theorem} 
\label{thm:vk_d}
Suppose $F_N(x_1,\ldots,x_N) = \exp\left(\sum_{\lambda\in\evenN} c_{\lambda}(N)p_\lambda\right)+ed_0(N)+e\exp\\\left(\sum_{\lambda\in\evenN}d_\lambda(N)p_\lambda\right)$ for all $N\geq 2$. Assume that $\theta\in\mathbb{C}$ is fixed. Consider the following statements.
\begin{enumerate}
\item[(a)] For all $\lambda\in\even$, $\lim_{N\rightarrow\infty} c_\lambda(N)=c_\lambda\in\mathbb{C}$.
\item[(b)] It is the case that $\lim_{N\rightarrow\infty}d_0(N)=d_0\in\mathbb{C}$ and for all $\lambda\in\even$, $\lim_{N\rightarrow\infty} d_\lambda(N)=d_\lambda\in\mathbb{C}$.
\item[(c)] For all $\nu\in\even$,
\[
\lim_{N\rightarrow\infty} \frac{1}{N^{|\nu|}} [p_\nu, F_N]_{D^N(\theta)} = (2\theta)^{|\nu|-\ell(\nu)}  \pi(\nu)\prod_{l=1}^{\ell(\nu)}\nu_l[p_\nu] \exp\left(\sum_{\gamma\in\even}c_\gamma p_\gamma\right).
\]
\item[(d)] It is the case that 
\[
\lim_{N\rightarrow\infty} \frac{1}{\prod_{j=1}^N(1+2(j-1)\theta)} [e, F_N]_{D^N(\theta)} = d_0+1
\]
and for all $\nu\in\even$
\begin{align*}
&\lim_{N\rightarrow\infty} \frac{1}{\prod_{j=1}^N(1+2(j-1)\theta)N^{|\nu|}} [ep_\nu, F_N]_{D^N(\theta)} \\
& = (2\theta)^{|\nu|-\ell(\nu)}  \pi(\nu)\prod_{l=1}^{\ell(\nu)}\nu_l[p_\nu] \exp\left(\sum_{\gamma\in\even}d_\gamma p_\gamma\right).
\end{align*}
\end{enumerate}
Then, (a) implies (c) and if $\theta\not=0$, then (c) implies (a).

Assume that $0\notin \{1+2j\theta: j\in\mathbb{N}\}$. Then, (b) implies (d) and if $\theta\not=0$, then (d) implies (b).
\end{theorem}

\begin{proof}
See \Cref{thm:leadingorder_bc} with $c=0$ and \Cref{lemma:leading_order_odd}.
\end{proof}

\Cref{thm:converge2} is an example of an application of the results of this paper. Part (A) of the theorem appears in \cite{assiotis} while parts (A) and (B) appear in \cite{rank_infinity}. Part (C) has not appeared previously, although it is straightforward to deduce. See \Cref{remark:vk} for discussion on how the setting of the theorem is equivalent to the setting of Vershik-Kerov sequences.

\begin{theorem}
\label{thm:converge2}
Suppose $m_d\in\mathbb{R}$ for $d\geq 1$. Assume that for $N\geq 2$, $a(N)\in\mathbb{R}^N$ such that $\lim_{N\rightarrow\infty} \frac{\sum_{i=1}^N a(N)_i^d}{N^d}=m_d$ for all $d\geq 1$. Assume that $M$ is a positive real number such that for sufficiently large $N\in\mathbb{N}$, it is the case that for all $d\in\mathbb{N}$, $\frac{1}{N^d}\left|\sum_{i=1}^N a(N)_i^d\right|\leq M^d$. Also, assume that $\theta, \theta_0>0$ are fixed and that $\theta_1\geq 0$ for all $N\geq 2$. Suppose $r\geq 1$.

\begin{itemize} 
\item[(A)] As $N\rightarrow\infty$, $J^{A^{N-1}(\theta)}_{a(N)}(z_1,\ldots,z_r,0,\ldots,0)$ converges to $\exp\left(\sum_{k\geq 1}\frac{m_kp_{(k)}\theta^{-k+1}}{k}\right)$ uniformly over compact subsets of $\left\{z\in\mathbb{C}^r: |\text{Re}(z_i)| < \frac{\theta}{rM}\,\forall i\in [r]\right\}$.
\item[(B)] Suppose $\lim_{N\rightarrow\infty}\frac{\theta_1}{\theta_0 N}=c\geq 0$. As $N\rightarrow\infty$, $J^{BC^N(\theta_0,\theta_1)}_{a(N)}(z_1,\ldots,z_r,0,\ldots,0)$ converges to $\exp\left(\sum_{k\geq 1}\frac{m_{2k}p_{(2k)}(2\theta_0)^{-2k+1}(1+c)^{-k}}{2k}\right)$ uniformly over compact subsets of $\left\{z\in\mathbb{C}^r: |\text{Re}(z_i)| < \frac{2\theta_0\sqrt{1+c}}{rM}\,\,\forall i\in [r]\right\}$. \item[(C)] As $N\rightarrow\infty$, $J^{D^N(\theta)}_{a(N)}(z_1,\ldots,z_r,0,\ldots,0)$ converges to $\exp\left(\sum_{k\geq 1}\frac{m_{2k}p_{(2k)}(2\theta)^{-2k+1}}{2k}\right)$ uniformly over compact subsets of $\left\{z\in\mathbb{C}^r: |\text{Re}(z_i)| < \frac{2\theta}{rM}\,\,\forall i\in [r]\right\}$.
\end{itemize}
\end{theorem}

\begin{proof}
(A): We include the proof of this result from \cite{rank_infinity} for completeness. Using \Cref{thm:positivity} and H\"older's inequality gives that
\begin{align*}
\left|J^{A^{N-1}(\theta)}_{a(N)}(x_1,\ldots,x_r,0,\ldots,0)\right| & = \left|\int_{\mathbb{R}^N} e^{\sum_{i=1}^r x_i\epsilon_i} d\mu_a(\epsilon)\right| \\ & \leq \prod_{i=1}^r \left|\int_{\mathbb{R}^N} e^{r\text{Re}(x_i\epsilon_i)} d\mu_a(\epsilon)\right|^{\frac{1}{r}} \\
&= \prod_{i=1}^r J^{A^{N-1}(\theta)}_{a(N)}(r\text{Re}(x_i),0,\ldots,0)^{\frac{1}{r}}.
\end{align*}
Afterwards, using \Cref{lemma:onevar} gives that 
\begin{align*}
J^{A^{N-1}(\theta)}_{a(N)}(r\text{Re}(x_i),0,\ldots,0) &\leq 1+\sum_{k=1}^\infty \frac{|c_k(a)|}{(\theta N)^k} |r\text{Re}(x_i)|^k \\
& = 1+\sum_{k=1}^\infty \left|c_k\left(\frac{a}{\theta N}\right)r\text{Re}(x_i)\right|^k \\
& \leq \exp\left(\theta\sum_{m=1}^\infty \left|p_{(m)}\left(\frac{a}{N}\right)\right|\frac{|r\text{Re}(x_i)|^m}{\theta^mm}\right) \\
& \leq \exp\left(\theta\sum_{m=1}^\infty M^m \frac{|r\text{Re}(x_i)|^m}{\theta^mm}\right).
\end{align*}

It is then is clear that $\zeta\triangleq \{J_{a(N)}^{A^{N-1}(\theta)}(x_1,\ldots,x_r,0,\ldots,0)\}_{N\geq r}$ is uniformly bounded over compact subsets of $\{z\in\mathbb{C}^r:|\text{Re}(z_i)|<\frac{\theta}{rM}\,\forall i\in[r]\}$. Following this, we can use \Cref{cor:bessel_a1} to deduce that the limit of the coefficient of $r(x)$ in $\zeta$ equals its coefficient in $\exp\left(\sum_{k\geq 1}\frac{m_kp_{(k)}\theta^{-k+1}}{k}\right)$; the application of \Cref{cor:bessel_a1} is the difference between this proof and that of \cite{rank_infinity}. We can then apply part (B) of \Cref{lemma:uniformconverge} to conclude the proof.

(B) and (C): We follow the same method as the proof of (A). For (B), we use part (B) of \Cref{lemma:onevar} and \Cref{cor:bessel_bc1} and for (C), we use part (C) of \Cref{lemma:onevar}, \pref{eq:type_d_simple}, and \Cref{cor:bessel_bc1}.
\end{proof}

\begin{remark}
\label{remark:vk}
From \cite{assiotis}*{Proposition 2.3}, the condition that $\lim_{N\rightarrow\infty} \frac{\sum_{i=1}^N a(N)_i^d}{N^d}=m_d$ for all $d\geq 1$ implies that $\{a(N)\}_{N\geq 1}$ is a Vershik-Kerov sequence after it is reordered. Then, from \cite{rank_infinity}, we can set $M=\alpha+\epsilon$ in \Cref{thm:converge2} for any $\epsilon>0$, where $\alpha=\lim_{N\rightarrow\infty} \max_{i\in[N]} \left|\frac{a(N)_i}{N}\right|$. This will imply uniform convergence over compact subsets of $\{z\in\mathbb{C}^r: |\text{Re}(z_i)|<\frac{\theta}{r\alpha}\}$, which is the version of the result that appears in \cites{assiotis,rank_infinity}.
\end{remark}

\subsection[Bessel generating functions]{Bessel generating functions for exponentially decaying probability measures}
\label{subsec:expdecay}
Recall that the Bessel generating function is defined in \Cref{def:bgf}. We must place restrictions on $\mu$ so that the expected value converges. In particular, we consider the class of exponentially decaying probability measures studied in \cite{limits_prob_measures}, which is a modification of the class of measures studied in \cite{matrix}.

\begin{definition}[Definition 1.2 of \cite{limits_prob_measures}]
\label{def:exponent_decay}
A Borel probability measure $\mu$ over $\mathbb{C}^N$ \textit{exponentially decaying} at rate $R>0$ if $\int_{\mathbb{C}^N} e^{R\norm{a}_2} d\mu(a)$ is finite.
\end{definition}

If $\mu$ is exponentially decaying, then we have that $G_\mu^{\mathcal{R}(\theta)}$ is holomorphic in a neighborhood of the origin, which is required to apply the results of this paper. This implication as well as other essential implications are included in the following lemma.

\begin{lemma}[(A) is \cite{limits_prob_measures}*{Lemma 1.4} and (D) is \cite{matrix}*{Proposition 2.11}]
\label{lemma:exponent_decay}
Suppose the Borel probability measure $\mu$ over $\mathbb{C}^N$ is exponentially decaying at rate $R>0$. Furthermore, assume that $\theta\in\theta(\mathcal{R})$ satisfies $\text{Re}(\theta(r))\geq 0$ for all $r\in\mathcal{R}$.
\begin{enumerate}
\item[(A)] The Bessel generating function $G_\mu^{\mathcal{R}(\theta)}(x)$ converges over the closed ball of radius $R$ centered at the origin and is holomorphic over the interior of this domain.
\item[(B)] Let $\{1\}\cup\mathcal{B}$ be a basis of $\mathbb{C}^{H(\mathcal{R})}[x_1,\ldots,x_N]$. There exist unique constants $c_b^{\mathcal{R}(\theta)}(\mu)\in\mathbb{C}$ for $b\in\mathcal{B}$ such that 
\[
G_\mu^{\mathcal{R}(\theta)}(x) = \exp\left(\sum_{b\in\mathcal{B}} c_b^{\mathcal{R}(\theta)}(\mu)b\right)
\]
in a neighborhood of the origin.
\item[(C)] For all $p\in\mathbb{C}^{H(\mathcal{R})}[x_1,\ldots,x_N]$, the function $\mathcal{D}(\mathcal{R}(\theta))(p) G_\mu^{\mathcal{R}(\theta)}(x)$ is holomorphic and equals
\[
\int_{\mathbb{C}^N} p(a) J_a^{\mathcal{R}(\theta)}(x) d\mu(a)
\]
over the open ball of radius $R$ centered at the origin.
\item[(D)] For all $p\in\mathbb{C}^{H(\mathcal{R})}[x_1,\ldots,x_N]$,
\[
\mathcal{D}(\mathcal{R}(\theta))(p)G_\mu^{\mathcal{R}(\theta)}(0)  = \E_{a\sim\mu}[p(a)].
\]
\end{enumerate}
\end{lemma}

\begin{proof}
For (B), we note that $G_\mu^{\mathcal{R}(\theta)}(x)$ is holomorphic in a neighborhood of the origin and $G_\mu^{\mathcal{R}(\theta)}(0)=1$.

(C): We include a detailed proof of this statement for completeness. We prove this result for the nonsymmetric eigenfunction $E_a^{\mathcal{R}(\theta)}(x)$. It is straightforward to deduce the result for the Bessel function afterwards using the expression relating the two functions in \Cref{thm:dunkl_eigenfunction_symmetry}. Note that some of the ideas of this proof have previously appeared in \cite{limits_prob_measures}*{Remark 4.6}.

Suppose $x\in \overline{B}_R(0)$, which is the closed ball of radius $R$ centered at the origin. By part (e) of \Cref{lemma:dunkl_properties}, the integral
\[
\int_{\mathbb{C}^N} E_a^{\mathcal{R}(\theta)}(x) d\mu(a)
\]
converges because $|E_a^{\mathcal{R}(\theta)}(x)| \leq \sqrt{|H(\mathcal{R})|} \exp(\norm{a}_2\norm{x}_2)$ and $\mu$ exponentially decays at rate $R$. Furthermore, we can show that the integral is holomorphic over $B_R(0)$ by following \cite{matrix}*{Proof of Lemma 2.9}. Note that we have proved the nonsymmetric version of (A), which can be used to prove (A).

The goal is to prove using induction that for any sequence $\{i_j\}_{1\leq j\leq m}$ of elements of $[N]$, we have that for $x\in B_R(0)$,
\[
\prod_{j=1}^m \mathcal{D}_{i_j} \int_{\mathbb{C}^N} E_a^{\mathcal{R}(\theta)}(x) d\mu(a) = \int_{\mathbb{C}^N} \prod_{j=1}^m a_{i_j}  E_a^{\mathcal{R}(\theta)}(x) d\mu(a)
\]
and both sides of the equation converge to a holomorphic function over $B_R(0)$, which is the open ball of radius $R$ centered at the origin. We have already shown that this statement is true when the sequence is empty. Assume that the statement is true for the sequence $\{i_j\}_{1\leq j\leq m}$. We show that the statement is true for the sequence $\{i_j\}_{1\leq j\leq m+1}$.

It suffices to prove that for $x\in B_R(0)$,
\begin{equation}
\label{eq:exp_decay_induct}
\mathcal{D}_{i_{j+1}} \int_{\mathbb{C}^N} \prod_{j=1}^m a_{i_j}  E_a^{\mathcal{R}(\theta)}(x) d\mu(a) = \int_{\mathbb{C}^N} \prod_{j=1}^{m+1} a_{i_{j+1}} E_a^{\mathcal{R}(\theta)}(x) d\mu(a)
\end{equation}
and that the right hand side is holomorphic over $B_R(0)$. Note that $\norm{x}_2<R$, so applying part (e) of \Cref{lemma:dunkl_properties} gives that 
\[
\left|\prod_{j=1}^{m+1} a_{i_{j+1}} E_a^{\mathcal{R}(\theta)}(x)\right| \leq \sqrt{|H(\mathcal{R})}|\prod_{j=1}^{m+1} a_{i_{j+1}} \exp(\norm{a}_2\norm{x}_2),
\]
which implies that the right hand side converges. Afterwards, we can use the method used in \cite{matrix}*{Proof of Lemma 2.9} to show that the right hand side is holomorphic. To complete the proof, it suffices to show that \pref{eq:exp_decay_induct} is true. Without loss of generality, assume that $i_{m+1}=1$.

Suppose $x\in B_R(0)$. Observe that
\begin{equation}
\label{eq:exp_decay_integral}
\partial_1 \int_{\mathbb{C}^N} \prod_{j=1}^m a_{i_j}  E_a^{\mathcal{R}(\theta)}(x) d\mu(a)= \lim_{\epsilon\rightarrow 0} \int_{\mathbb{C}^N} \frac{\prod_{j=1}^m a_{i_j} (E_a(x_1+\epsilon, x_2,\ldots, x_N) - E_a(x))}{\epsilon} d\mu(a).
\end{equation}
By \cite{dunklbound}*{Lemma 3.5},
\[
\frac{|E_a(x_1+\epsilon, x_2,\ldots, x_N) - E_a(x)|}{\epsilon} \leq e\sqrt{|H(\mathcal{R})|}\norm{a}_2\exp(\norm{a}_2(\norm{x}_2 + \epsilon)).
\]
If we assume that $\epsilon < R - \norm{x}_2$, then the absolute value of the integrand of the right hand side of \pref{eq:exp_decay_integral} is upper bounded by an integrable function, since $\mu$ is exponentially decaying at rate $R$. Then, we can apply the dominated convergence theorem to deduce that 
\[
\partial_1 \int_{\mathbb{C}^N} \prod_{j=1}^m a_{i_j}  E_a^{\mathcal{R}(\theta)}(x) d\mu(a)=\int_{\mathbb{C}^N} \partial_1 \prod_{j=1}^m a_{i_j}E_a(x) d\mu(a).
\]

Suppose $\alpha\in\mathcal{R}^+$. Then, it is clear that 
\begin{align*}
&\int_{\mathbb{C}^N} \prod_{j=1}^m a_{i_j}  E_a^{\mathcal{R}(\theta)}(x) d\mu(a) - \int_{\mathbb{C}^N} \prod_{j=1}^m a_{i_j}  E_a^{\mathcal{R}(\theta)}(r_\alpha x) d\mu(a) \\ 
& = \int_{\mathbb{C}^N} \prod_{j=1}^m a_{i_j}  (E_a^{\mathcal{R}(\theta)}(x) -   E_a^{\mathcal{R}(\theta)}(r_\alpha x)) d\mu(a)
\end{align*}
since $r_\alpha x \in B_R(0)$. After using the expression given in $\mathcal{D}_1$ in \Cref{subsec:dunkl}, we deduce that 
\[
\mathcal{D}_1 \int_{\mathbb{C}^N} \prod_{j=1}^m a_{i_j} E_a^{\mathcal{R}(\theta)}(x) d\mu(a) = \int_{\mathbb{C}^N} \mathcal{D}_1 \prod_{j=1}^m a_{i_j} E_a^{\mathcal{R}(\theta)}(x) d\mu(a) = \int_{\mathbb{C}^N} a_1\prod_{j=1}^m a_{i_j} E_a^{\mathcal{R}(\theta)}(x) d\mu(a),
\]
which completes the induction.

(D): This is straightforward to deduce from (C).
\end{proof}

\begin{remark}
The papers \cites{matrix,limits_prob_measures} only consider the $A^{N-1}$ root system. However, (A) and (D) are generalizable to any finite root system after using the method discussed in \cite{limits_prob_measures} that involves applying the results of \cite{dunklbound}. Furthermore, observe that (A) implies (B). However, we require the exponentially decaying condition to prove (C) and (D).
\end{remark}

Using \Cref{lemma:exponent_decay} and part (A) of \Cref{thm:main}, we are able to deduce the following corollary for the $A^{N-1}$ root system. It is also straightforward to deduce the analogous results for the $BC^N$ and $D^N$ root systems. The corollary generalizes the results of \cite{limits_prob_measures}.

\begin{corollary}
\label{cor:lln_satisfaction_a}
Suppose $\mu_N$ is an exponentially decaying Borel probability measure over $\mathbb{C}^N$ for all $N\geq 2$. Assume that $\theta\in\mathbb{C}$ has nonnegative real part for all $N\geq 2$ and $\lim_{N\rightarrow\infty} |\theta N| = \infty$.

For $\mathcal{B}=\{p_\lambda:\lambda\in\Gamma_N\}$, define the coefficients $c_{p_\lambda}^{A^{N-1}(\theta)}(\mu_N)\in\mathbb{C}$ as they are defined in part (B) of \Cref{lemma:exponent_decay}. Then, the following are equivalent.

\begin{enumerate}
\item[(a)] For all $\lambda\in\Gamma$, $\lim_{N\rightarrow\infty} \frac{c_{p_\lambda}^{A^{N-1}(\theta)}(\mu_N)}{\theta N} = c_\lambda \in \mathbb{C}$ if $\ell(\lambda)=1$ and $\lim_{N\rightarrow\infty} \frac{c_{p_\lambda}^{A^{N-1}(\theta)}(\mu_N)}{(\theta N)^{\ell(\lambda)}}=0$ if $\ell(\lambda)\geq 2$. 
\item[(b)] For all $\nu\in\Gamma$,
\[
\lim_{N\rightarrow\infty} \frac{1}{N^{\ell(\nu)}}\E_{a\sim\mu_N}\left[p_\nu\left(\frac{a}{\theta N}\right)\right]  = \prod_{i=1}^{\ell(\nu)} \sum_{\pi\in NC(\nu_i)} \prod_{B\in \pi} |B|c_{(|B|)}.
\]
\end{enumerate}
\end{corollary}

Condition (b) of \Cref{cor:lln_satisfaction_a} has been studied in previous works such as \cites{llnclt,matrix,limits_prob_measures,rectangularmatrix,jackgenfunc}. When the condition is satisfied, the sequence of probability measures \textit{satisfies a law of large numbers}. Next, we state the analogous result for the $BC^N$ root system.

\begin{corollary}
\label{cor:lln_satisfaction_bc}
Suppose $\mu_N$ is an exponentially decaying Borel probability measure over $\mathbb{C}^N$ for all $N\geq 2$. Assume that $\theta_0,\theta_1\in\mathbb{C}$ have nonnegative real part for all $N\geq 2$, $\lim_{N\rightarrow\infty} |\theta_0 N| = \infty$, and $\lim_{N\rightarrow\infty} \frac{\theta_1}{\theta_0 N} = c\in\mathbb{C}\backslash\{-1\}$.

For $\mathcal{B}=\{p_\lambda:\lambda\in\evenN\}$, define the coefficients $c_{p_\lambda}^{BC^N(\theta_0,\theta_1)}(\mu_N)\in\mathbb{C}$ for $\lambda\in\evenN$ as they are defined in part (B) of \Cref{lemma:exponent_decay}. Then, the following are equivalent.

\begin{enumerate}
\item[(a)] For all $\lambda\in\even$, $\lim_{N\rightarrow\infty} \frac{c_{p_\lambda}^{BC^N(\theta_0,\theta_1)}(\mu_N)}{\theta N} = c_\lambda \in \mathbb{C}$ if $\ell(\lambda)=1$ and $\lim_{N\rightarrow\infty}\\\frac{c_{p_\lambda}^{BC^N(\theta_0,\theta_1)}(\mu_N)}{(\theta N)^{\ell(\lambda)}}=0$ if $\ell(\lambda)\geq 2$. 
\item[(b)] For all $\nu\in\even$,
\[
\lim_{N\rightarrow\infty} \frac{1}{N^{\ell(\nu)}}\E_{a\sim\mu_N}\left[p_\nu\left(\frac{a}{\theta_0 N}\right)\right]  = \prod_{i=1}^{\ell(\nu)} \sum_{\pi\in \NCeven(\nu_i)} (1+c)^{o(\pi)}\prod_{B\in \pi}2^{|B|-1} |B|c_{(|B|)}.
\]
\end{enumerate}
\end{corollary}

For the type D analogue of \Cref{cor:lln_satisfaction_a}, we adjust the notation used to match the notation of part (C) of \Cref{thm:main}. Therefore, we first prove the following lemma, which is analogous to part (B) of \Cref{lemma:exponent_decay}.

\begin{lemma}
\label{lemma:exponent_decay_type_D}
Suppose the Borel probability measure $\mu$ over $\mathbb{C}^N$ is exponentially decaying at rate $R>0$. Furthermore, assume that $\theta\in\mathbb{C}$ has nonnegative real part. 

Let $\{1\}\cup\{\mathcal{B}\}$ be a basis of $\mathbb{C}^{H(BC^N)}[x_1,\ldots,x_N]$. There exist unique constants $d_0^{D^N(\theta)}(\mu)\\\in\mathbb{C}$ and $c_b^{D^N(\theta)}(\mu),d_b^{D^N(\theta)}(\mu)\in\mathbb{C}$ for $b\in\mathcal{B}$ such that
\[
G_\mu^{D^N(\theta)} = \exp\left(\sum_{b\in\mathcal{B}} c_b^{D^N(\theta)}(\mu)b\right) + ed_0^{D^N(\theta)}(\mu) + e\exp\left(\sum_{b\in\mathcal{B}} d_b^{D^N(\theta)}(\mu)b\right)
\]
over a neighborhood of the origin.
\end{lemma}

\begin{proof}
We have that $\frac{1}{2}(G_\mu^{D^N(\theta)}(x) + G_\mu^{D^N(\theta)}(-x_1,x_2,\ldots,x_N))$ is even and evaluates to one at the origin, from which the uniqueness of $c_b^{D^N(\theta)}(\mu)$ for $b\in\mathcal{B}$ follows. Furthermore, we have that $\frac{1}{2}(G_\mu^{D^N(\theta)}(x)-G_\mu^{D^N(\theta)}(-x_1,x_2,\ldots,x_N))$ is odd and $\frac{G_\mu^{D^N(\theta)}(x) - G_\mu^{D^N(\theta)}(-x_1,x_2,\ldots,x_N)}{2e(x)}$ is even. It is clear that $d_0^{D^N(\theta)}(\mu) = \frac{G_\mu^{D^N(\theta)}(x) - G_\mu^{D^N(\theta)}(-x_1,x_2,\ldots,x_N)}{2e(x)} - 1$. Next, we have that $\frac{G_\mu^{D^N(\theta)}(x) - G_\mu^{D^N(\theta)}(-x_1,x_2,\ldots,x_N)}{2e(x)} - d_0^{D^N(\theta)}(\mu)$ evaluates to one at the origin, from which the uniqueness of $d_b^{D^N(\theta)}(\mu)$ for $b\in\mathcal{B}$ follows.
\end{proof}

\begin{corollary}
\label{cor:lln_satisfaction_d}
Suppose $\mu_N$ is an exponentially decaying Borel probability measure over $\mathbb{C}^N$ for all $N\geq 2$. Assume that $\theta\in\mathbb{C}$ has nonnegative real part for all $N\geq 2$ and $\lim_{N\rightarrow\infty} |\theta N| = \infty$.

Define the coefficients $d_0^{D^N(\theta)}(\mu)$ and $c_{p_\lambda}^{D^N(\theta)}(\mu_N), d_{p_\lambda}^{D^N(\theta)}(\mu_N)\in\mathbb{C}$ for $\lambda\in\evenN$ as they are defined in \Cref{lemma:exponent_decay_type_D} for $\mathcal{B}=\{p_\lambda:\lambda\in\evenN\}$. Consider the following statements.

\begin{enumerate}
\item[(a)] For all $\lambda\in\even$, $\lim_{N\rightarrow\infty} \frac{c_{p_\lambda}^{D^N(\theta)}(\mu_N)}{\theta N} = c_\lambda \in \mathbb{C}$ if $\ell(\lambda)=1$ and $\lim_{N\rightarrow\infty} \frac{c_{p_\lambda}^{D^N(\theta)}(\mu_N)}{(\theta N)^{\ell(\lambda)}}\\=0$ if $\ell(\lambda)\geq 2$. 
\item[(b)] It is the case that $\lim_{N\rightarrow\infty}d_0^{D^N(\theta)}(\mu_N)=d_0\in\mathbb{C}$ and for all $\lambda\in\evenN$, $\lim_{N\rightarrow\infty} \\\frac{d_{p_\lambda}^{D^N(\theta)}(\mu_N)}{\theta N} = d_\lambda \in \mathbb{C}$ if $\ell(\lambda)=1$ and $\lim_{N\rightarrow\infty} \frac{d_{p_\lambda}^{D^N(\theta)}(\mu_N)}{(\theta N)^{\ell(\lambda)}}=0$ if $\ell(\lambda)\geq 2$. 
\item[(c)] For all $\nu\in\even$,
\[
\lim_{N\rightarrow\infty} \frac{1}{N^{\ell(\nu)}}\E_{a\sim\mu_N}\left[p_\nu\left(\frac{a}{\theta N}\right)\right]  = \prod_{i=1}^{\ell(\nu)} \sum_{\pi\in \NCeven(\nu_i)} \prod_{B\in \pi}2^{|B|-1} |B|c_{(|B|)}.
\]
\item[(d)] It is the case that 
\[
\lim_{N\rightarrow\infty} \frac{\E_{a\sim\mu_N}\left[e(a)\right]}{\prod_{j=1}^N(1+2(j-1)\theta)} = d_0+1
\]
and for all $\nu\in\even$,
\[
\lim_{N\rightarrow\infty} \frac{\E_{a\sim\mu_N}\left[e(a)p_\nu\left(\frac{a}{\theta N}\right)\right]}{N^{\ell(\nu)}\prod_{j=1}^N(1+2(j-1)\theta)}  = \prod_{i=1}^{\ell(\nu)} \sum_{\pi\in \NCeven(\nu_i)} \prod_{B\in \pi}2^{|B|-1} |B|d_{(|B|)}.
\]
\end{enumerate}
Then, (a) and (c) are equivalent and (b) and (d) are equivalent.
\end{corollary}

\begin{remark}
It is straightforward to establish analogues of \Cref{cor:lln_satisfaction_a,cor:lln_satisfaction_bc,cor:lln_satisfaction_d} in the $\theta N\rightarrow c\in\mathbb{C}$ regime.
\end{remark}

Furthermore, we can generalize \Cref{cor:freeconv} and \Cref{cor:freeconv_rectangular} to when $a(N)$ and $b(N)$ are randomly distributed. We first state the following generalization of \Cref{lemma:conjecture_support}.

\begin{lemma}
\label{lemma:conjecture_support_2}
Suppose $\theta\in\theta(\mathcal{R})$ is nonnegative. Suppose $\mu_1$, $\mu_2$, and $\mu$ are exponentially decaying Borel probability measures over $\mathbb{R}^N$ such that $G_\mu^{\mathcal{R}(\theta)} = G_{\mu_1}^{\mathcal{R}(\theta)}G_{\mu_2}^{\mathcal{R}(\theta)}$ over an open subset of $\mathbf{i}\mathbb{R}^N$. Then, every element of $\text{supp}(\mu)$ can be expressed as $u_1+u_2$, where $u_i$ is in the convex hull of $H(\mathcal{R})a_i$ for some $a_i\in \text{supp}(\mu_i)$ for $i\in \{1,2\}$.
\end{lemma}

\begin{proof}
We follow the method of the proof of \Cref{lemma:conjecture_support}. We have that for $x$ in an open subset of $\mathbf{i}\mathbb{R}^N$, the characteristic function of the random variables $\epsilon_1+\epsilon_2$ where $a_i\sim \mu_i$ and $\epsilon_i\sim\mu_{a_i}^{\text{sym}}$ for $i\in \{1,2\}$ and $\epsilon$ where $a\sim \mu$ and $\epsilon\sim\mu_a^{\text{sym}}$. Then, since the characteristic function is defined based on its values in an open set, the characteristic functions of the two random variables are equal everywhere, so the two random variables have the same distribution.

We can similarly deduce that for all $u\in\text{supp}(\mu)$, $u$ is in the support of $\epsilon$. Additionally, the support of $\epsilon$ is the same as the support of $\epsilon^1+\epsilon^2$.

Suppose $i\in \{1,2\}$, $u\in \text{supp}(\epsilon_i)$, and $r>0$. Observe that
\begin{equation}
\label{eq:support_computation}
\int_{\mathbb{R}^N} \Pr_{\epsilon_i\sim \mu_{a_i}^{\text{sym}}}[\epsilon_i \in B(u,r)] d\mu_i(a_i) = \Pr[\epsilon_i\in B(u,r)]> 0.
\end{equation}
For the sake of contradiction, assume that the distance between $u$ and $\text{supp}(\mu_{a_i}^{\text{sym}})$ for any element $a_i\in\text{supp}(\mu_i)$ is greater than $2r$. Then, for all $a_i\in\text{supp}(\mu_i)$, $\Pr_{\epsilon_i\sim \mu_{a_i}^{\text{sym}}}[\epsilon_i \in B(u,r)]=0$, which is a contradiction to \pref{eq:support_computation}. Recall from \Cref{thm:positivity} that for $a\in\mathbb{R}^N$, the support of $\mu_a^{\text{sym}}$ is a subset of the convex hull of $H(\mathcal{R})a$. By considering when $r\rightarrow 0$, we have that $u$ must be in the convex hull of $H(\mathcal{R}) a_i$ for some $a_i\in\text{supp}(\mu_i)$.
\end{proof}

\begin{corollary}
\label{cor:lr}

\begin{enumerate}
\item[(A)] Suppose $\theta\in\mathbb{C}\backslash\{0\}$ has nonnegative real part and $|\theta N|\rightarrow\infty$. Suppose $\mu_a^N$ and $\mu_b^N$ are exponentially decaying Borel probability measures over $\mathbb{C}^N$ for $N\geq 2$ such that:
\begin{enumerate}
\item[1.] The measure $\mu_a^N(x\mapsto \frac{x}{\theta N})$ converges to the Borel probability measure $\mu_a$ over $\mathbb{C}$ with finite moments in terms of power sums.
\item[2.] The measure $\mu_b^N(x\mapsto \frac{x}{\theta N})$ converges to the Borel probability measure $\mu_b$ over $\mathbb{C}$ with finite moments in terms of power sums.
\end{enumerate}
For $N\geq 2$, suppose $\mu^N$ is an exponentially decaying Borel probability measure such that $G_{\mu^N}^{A^{N-1}(\theta)} = G_{\mu_a^N}^{A^{N-1}(\theta)}G_{\mu_b^N}^{A^{N-1}(\theta)}$ in a neighborhood of the origin. Then, the measure $\mu_N(x\mapsto \frac{x}{\theta N})$ converges to the free convolution of $\mu_a$ and $\mu_b$ with finite moments in terms of power sums.
\item[(B)] Suppose $\theta\in\mathbb{C}$ is positive. Suppose $\mu_a^N$ and $\mu_b^N$ are compactly supported Borel probability measures over $\mathbb{R}^N$ for $N\geq 2$ and $\mu_a$ and $\mu_b$ are Borel probability measures over $\mathbb{R}$ with finite moments such that conditions 1 and 2 from (A) are true. For $N\geq 2$, suppose $\mu^N$ is a Borel probability measure over $\mathbb{R}^N$ such that $G_{\mu^N}^{A^{N-1}(\theta)} = G_{\mu_a^N}^{A^{N-1}(\theta)}G_{\mu_b^N}^{A^{N-1}(\theta)}$ in a neighborhood of the origin. Then, the measure $\mu_N(x\mapsto \frac{x}{\theta N})$ converges to the free convolution of $\mu_a$ and $\mu_b$ in terms of power sums.
\end{enumerate}
\end{corollary}

\begin{proof}
(A): We can follow the proof of \Cref{cor:freeconv}. Instead of using \Cref{thm:main}, we use \Cref{cor:lln_satisfaction_a}.

(B): We use \Cref{lemma:conjecture_support_2} to deduce that for $N\geq 2$, $\mu^N$ is compactly supported and therefore exponentially decaying. Afterwards, we follow the proof of (A).
\end{proof}

In the following two corollaries, we use notation defined in \Cref{cor:freeconv_rectangular}. The proofs of these corollaries are straightforward and use the same techniques as those of \Cref{cor:freeconv,cor:freeconv_rectangular,cor:lr}.

\begin{corollary}
\label{cor:lr_bc}
Suppose $\theta_0\in\mathbb{C}\backslash\{0\}$ and $\theta_1\in\mathbb{C}$ have nonnegative real part such that $|\theta_0 N|\rightarrow\infty$ and $\lim_{N\rightarrow\infty}\frac{\theta_1}{\theta_0 N} = c \in\mathbb{C}$. 

\begin{enumerate}
\item[(A)] Suppose $\mu_a^N$ and $\mu_b^N$ are exponentially decaying Borel probability measures over $\mathbb{C}^N$ for $N\geq 2$ such that:
\begin{enumerate}
\item[1.] The measure $\mu_a^N(x\mapsto \frac{x}{\theta_0 N})$ converges to the Borel probability measure $\mu_a$ over $\mathbb{C}$ with finite moments in terms of power sums.
\item[2.] The measure $\mu_b^N(x\mapsto \frac{x}{\theta_0 N})$ converges to the Borel probability measure $\mu_b$ over $\mathbb{C}$ with finite moments in terms of power sums.
\end{enumerate}
For $N\geq 2$, suppose $\mu^N$ is an exponentially decaying Borel probability measure such that $G_{\mu^N}^{BC^N(\theta_0,\theta_1)} = G_{\mu_a^N}^{BC^N(\theta_0,\theta_1)}G_{\mu_b^N}^{BC^N(\theta_0,\theta_1)}$ in a neighborhood of the origin.

Then, the measure $\frac{1}{2}\mu^N(x\mapsto \frac{x}{\theta_0 N}) + \frac{1}{2}\mu^N(x\mapsto -\frac{x}{\theta_0 N})$ converges to the rectangular free convolution of $\frac{1}{2}(\mu_a+\mu_a^-)$ and $\frac{1}{2}(\mu_b+\mu_b^-)$ with $\lambda$ set as $\frac{1}{1+c}$ in terms of power sums.
\item[(B)] Suppose $\theta_0$ and $\theta_1$ are nonnegative. Suppose $\mu_a^N$ and $\mu_b^N$ are Borel probability measures over $\mathbb{R}^N$ for $N\geq 2$ and $\mu_a$ and $\mu_b$ are Borel probability measures over $\mathbb{R}$ with finite moments such that conditions 1 and 2 from (A) are true. For $N\geq 2$, suppose $\mu^N$ is a Borel probability measure such that $G_{\mu^N}^{BC^N(\theta_0,\theta_1)} = G_{\mu_a^N}^{BC^N(\theta_0,\theta_1)}G_{\mu_b^N}^{BC^N(\theta_0,\theta_1)}$ in a neighborhood of the origin. 

Then, the measure $\frac{1}{2}\mu^N(x\mapsto \frac{x}{\theta_0 N}) + \frac{1}{2}\mu^N(x\mapsto -\frac{x}{\theta_0 N})$ converges to the rectangular free convolution of $\frac{1}{2}(\mu_a+\mu_a^-)$ and $\frac{1}{2}(\mu_b+\mu_b^-)$ with $\lambda$ set as $\frac{1}{1+c}$ in terms of power sums.
\end{enumerate}
\end{corollary}

\begin{remark}
Note that in part (A) of the previous corollary, the rectangular free convolution is not defined in \cite{rectangular_free_convolution} when $c$ is complex. In this case, we define it using part (B) of \Cref{cor:product}.
\end{remark}

\begin{corollary}
\label{cor:lr_d}

\begin{enumerate}
\item[(A)] Suppose $\theta\in\mathbb{C}\backslash\{0\}$ has nonnegative real part. Suppose $\mu_a^N$ and $\mu_b^N$ are exponentially decaying Borel probability measures over $\mathbb{C}^N$ for $N\geq 2$ such that:
\begin{enumerate}
\item[1.] The measure $\mu_a^N(x\mapsto \frac{x}{\theta N})$ converges to the Borel probability measure $\mu_a$ over $\mathbb{C}$ with finite moments in terms of power sums.
\item[2.] The measure $\mu_b^N(x\mapsto \frac{x}{\theta N})$ converges to the Borel probability measure $\mu_b$ over $\mathbb{C}$ with finite moments in terms of power sums.
\end{enumerate}
For $N\geq 2$, suppose $\mu^N$ is an exponentially decaying Borel probability measure such that $G_{\mu^N}^{D^N(\theta)} = G_{\mu_a^N}^{D^N(\theta)}G_{\mu_b^N}^{D^N(\theta)}$ in a neighborhood of the origin. Then, the measure $\frac{1}{2}\mu^N(x\mapsto \frac{x}{\theta N}) + \frac{1}{2}\mu^N(x\mapsto -\frac{x}{\theta N})$ converges to the free convolution of $\frac{1}{2}(\mu_a+\mu_a^-)$ and $\frac{1}{2}(\mu_b+\mu_b^-)$ in terms of power sums.
\item[(B)] Suppose $\theta\in\mathbb{C}$ is positive. Suppose $\mu_a^N$ and $\mu_b^N$ are compactly supported Borel probability measures over $\mathbb{R}^N$ for $N\geq 2$ and $\mu_a$ and $\mu_b$ are Borel probability measures over $\mathbb{R}$ with finite moments such that conditions 1 and 2 from (A) are true. For $N\geq 2$, suppose $\mu^N$ is a Borel probability measure over $\mathbb{R}^N$ such that $G_{\mu^N}^{D^N(\theta)} = G_{\mu_a^N}^{D^N(\theta)}G_{\mu_b^N}^{D^N(\theta)}$ in a neighborhood of the origin. Then, the measure $\frac{1}{2}\mu^N(x\mapsto \frac{x}{\theta N}) + \frac{1}{2}\mu^N(x\mapsto -\frac{x}{\theta N})$ converges to the free convolution of $\frac{1}{2}(\mu_a+\mu_a^-)$ and $\frac{1}{2}(\mu_b+\mu_b^-)$ in terms of power sums.
\end{enumerate}
\end{corollary}

\begin{example}
\label{example:dbm}
As an example, we consider the \textit{Dyson Brownian motion} with a random initial condition. For $\theta>0$, consider the stochastic differential equation given by initial values $(a_1(0),\ldots,a_N(0))$ and
\begin{equation}
\label{eq:dbm}
da_i(t) = \theta\sum_{j\in [N]\backslash\{i\}} \frac{1}{a_i(t) - a_j(t)}dt + dB_i(t), 
\end{equation}
where $B_i$, $1\leq i\leq N$ are independent standard Brownian motions. Let $\mathcal{W}(A^{N-1})$ denote the Weyl chamber $\{x\in\mathbb{R}^N: x_1> \cdots > x_N\}$. We add the assumption that $(a_1(t),\ldots,a_N(t))\in\overline{\mathcal{W}(A^{N-1})}$ for all $t\geq 0$ and denote the resulting SDE by $DBM(A^{N-1}(\theta))$. 

The solution to $DBM(A^{N-1}(\theta))$ is unique if we assume that it is a continuous semimartingale, although we do not require this assumption for uniqueness if $\theta\geq \frac{1}{2}$, see \cites{skew-product,AGZ}. Moreover, by the results of \cites{cepa}, a continuous semimartingale solution exists for all $\theta>0$.

For $a\in\overline{\mathcal{W}(A^{N-1})}$, let $\left\{X_t^{A^{N-1}(\theta)}(a)\right\}_{t\geq 0}$ denote the continuous Markov process that takes values in $\overline{\mathcal{W}(A^{N-1})}$ such that $X_0^{A^{N-1}}(a)=a$ and the transition probability from an observation $x\in\overline{\mathcal{W}(A^{N-1})}$ at time $t_1\geq 0$ to an observation $y\in\overline{\mathcal{W}(A^{N-1})}$ at time $t_2> t_1$ is proportional to
\begin{equation}
\label{eq:dbm_transition1}
\frac{1}{(t_2-t_1)^{\gamma_{A^{N-1}(\theta)}+\frac{N}{2}}}e^{-\frac{\norm{x}_2^2+\norm{y}_2^2}{2(t_2-t_1)}}J_x^{A^{N-1}(\theta)}\left(\frac{y}{t_2-t_1}\right) h_{A^{N-1}(\theta)}(y)^2,
\end{equation}
where $\gamma_{A^{N-1}(\theta)}\triangleq \sum_{r\in (A^{N-1})^+} \theta(r) = \frac{N(N-1)\theta}{2}$. For the definition of $h_{A^{N-1}(\theta)}$, see \Cref{thm:product}.

The Markov process is studied in \cites{rosler_markov,gallardo_yor,skew-product,airy_beta} and \cites{skew-product} shows that it is the unique solution to $DBM(A^{N-1}(\theta))$ initiated at $a\in \overline{\mathcal{W}(A^{N-1})}$ if we assume that $\theta\geq \frac{1}{2}$.

For $t\geq 0$, suppose $\mu_t^{A^{N-1}(\theta)}(a)$ is the probability measure over $\overline{\mathcal{W}(A^{N-1})}$ induced by $X_t^{A^{N-1}(\theta)}(a)$. From \cite{airy_beta}*{Lemma 3.8},
\begin{equation}
\label{eq:dbm_lln}
G_{\mu_t^{A^{N-1}(\theta)}(a)}^{A^{N-1}(\theta)}(x) = J_a^{A^{N-1}(\theta)}(x)\exp\left(\frac{t}{2}\sum_{i=1}^N x_i^2\right)
\end{equation}
for all $x\in\mathbb{C}^N$. 

Suppose $\eta$ is a probability measure over $\overline{\mathcal{W}(A^{N-1})}$. Then, for $t\geq 0$, define 
\[
\mu^{A^{N-1}(\theta)}_t(\eta) \triangleq \E_{a\sim \eta}\left[\mu^{A^{N-1}(\theta)}_t(a)\right].
\]
We prove the following generalization of \pref{eq:dbm_lln}.

\begin{lemma}
\label{lemma:bgf_dbm}
Assume that $\theta>0$ and $\eta$ is a Borel probability measure over $\mathbb{R}^N$ that exponentially decays at rate $R>0$. Then, for all $t\geq 0$, $\mu^{A^{N-1}(\theta)}_t(\eta)$ exponentially decays at any rate less than $R$ and
\[\
G_{\mu^{A^{N-1}(\theta)}_t(\eta)}^{A^{N-1}(\theta)}(x) = \exp\left(\frac{t}{2}\sum_{i=1}^N x_i^2\right)G_\eta^{A^{N-1}(\theta)}(x)
\]
over $B_R(0)$.
\end{lemma}
\begin{proof}
First, we show that for all $t\geq 0$, $\mu^\theta_t(\eta)$ exponentially decays at any rate less than $R$. This is obviously true if $t=0$, so assume that $t>0$.

Suppose $y\in\mathbb{R}^N$. By \pref{eq:dbm_transition1}, we have that the density of $y$ for $DBM(A^{N-1}(\theta))$ at time $t$ conditioned on the initial value $a$ is at most $\exp\left(-\frac{(\norm{a}_2-\norm{y}_2)^2}{2t}\right)y^{O(1)}$, where we have applied the upper bound $J_a^{A^{N-1}(\theta)}\left(\frac{y}{\alpha N}\right) \leq \sqrt{N!}\exp\left(\frac{\norm{a}_2\norm{y}_2}{\alpha N}\right)$ that we deduce from part (e) of \Cref{lemma:dunkl_properties}. Then, we have that for all $\epsilon>0$,
\begin{align*}
& \E_{\substack{a\sim \eta \\ y\sim \mu_t^{A^{N-1}(\theta)}(a)}}\left[\exp\left(\frac{R}{1+\epsilon}\norm{y}_2\right)\mathbf{1}\{\norm{y}_2 \geq (1+\epsilon)\norm{a}_2\}\right] \\
& \leq \E_{\substack{a\sim \eta, \\ y\sim \mu_t^{A^{N-1}(\theta)}(a)}}\left[\exp\left(-\frac{((1-\frac{1}{1+\epsilon})\norm{y}_2)^2}{2t} + \frac{R}{1+\epsilon}\norm{y}_2\right)y^{O(1)}\right] < \infty
\end{align*}
and
\begin{align*}
& \E_{\substack{a\sim \eta, \\ y\sim \mu_t^{A^{N-1}(\theta)}(a)}}\left[\exp\left(\frac{R}{1+\epsilon}\norm{y}_2\right)\mathbf{1}\{\norm{y}_2 < (1+\epsilon)\norm{a}_2\}\right] \\
& < \E_{\substack{a\sim \eta, \\ y\sim \mu_t^{A^{N-1}(\theta)}(a)}}\left[\exp(R\norm{a}_2)\mathbf{1}\{\norm{y}_2 < (1+\epsilon)\norm{a}_2\}\right] < \infty.
\end{align*}
This implies that $\mu_t^{A^{N-1}(\theta)}(\eta)$ exponentially decays at any rate less than $R$.

For the computation of the Bessel generating function $G_{\mu_t^{A^{N-1}(\theta)}(\eta)}^{A^{N-1}(\theta)}$, we have that
\[
G_{\mu_t^{A^{N-1}(\theta)}(\eta)}^{A^{N-1}(\theta)}(x) = \int_{\mathbb{R}^{2N}} J_{a}^{A^{N-1}(\theta)}(x) d\mu_t^\theta(\eta)(a).
\]
for $x\in\mathbb{C}^N$. Because $\mu_t^{A^{N-1}(\theta)}(\eta)$ exponentially decays at any rate less than $R$, we can use part (e) of \Cref{lemma:dunkl_properties} and Fubini's theorem to deduce that for $x\in B_R(0)$,
\[
G_{\mu_t^{A^{N-1}(\theta)}(\eta)}^{A^{N-1}(\theta)}(x) = \int_{\mathbb{R}^N}\left(\int_{\mathbb{R}^N} J_{a(t)}^{A^{N-1}(\theta)}(x)d\mu_t^{A^{N-1}(\theta)}(a)(a(t))\right) d\eta(a).
\]
Afterwards, applying \pref{eq:dbm_lln} and part (A) of \Cref{lemma:exponent_decay} implies that this expression equals
\[
\int_{\mathbb{R}^N} J_a^{A^{N-1}(
\theta)}(x)\exp\left(\frac{t}{2}\sum_{i=1}^Nx_i^2\right) d\eta(a) = \exp\left(\frac{t}{2}\sum_{i=1}^Nx_i^2\right)G_\eta^{A^{N-1}(\theta)}(x).
\]
\end{proof}

Furthermore, for $\theta>0$, define the \textit{Hermite $\theta$-ensemble} to be the random variable $H_\theta$ over $\mathbb{R}^N$ with density at $a$ proportional to $\prod_{1\leq i<j\leq N} |a_i-a_j|^{2\theta} \prod_{i=1}^N e^{-\frac{\theta a_i^2}{2}}$. By \cite{betaprocess}*{Proposition 4.2}, 
\[
G_{H_\theta}^{A^{N-1}(\theta)}(x) = \exp\left(\frac{1}{2\theta}\sum_{i=1}^N x_i^2\right)
\]
over $\mathbb{C}^N$; this formula also follows from \Cref{thm:product}.

For $t>0$, the pushforward $H_\theta(x\mapsto tx)$ of $H_\theta$ with respect to multiplication by $t$ is a random variable over $\mathbb{R}^N$ with density at $a$ proportional to $\prod_{1\leq i<j\leq N} |a_i-a_j|^{2\theta}\prod_{i=1}^N e^{-\frac{\theta a_i^2}{2t^2}}$. This random variable is studied in the paper \cite{airy_beta} and Lemma 3.7 of the paper computes that
\[
G_{H_\theta(x\mapsto tx)}^{A^{N-1}(\theta)}(x) = \exp\left(\frac{t^2}{2\theta} \sum_{i=1}^N x_i^2\right)
\]
over $\mathbb{C}^N$. 

Observe that $H_\theta(x\mapsto tx)$ clearly exponentially decays at any rate. Correspondingly, the measure's Bessel generating function converges over $\mathbb{C}^N$. Furthermore, \Cref{lemma:bgf_dbm} implies that 
\begin{equation}
\label{eq:random_product}
G_{\mu_t^{A^{N-1}(\theta)}(\eta)}^{A^{N-1}(\theta)}(x) = G_\eta^{A^{N-1}(\theta)}(x)G^{A^{N-1}(\theta)}_{H_\theta(x\mapsto \sqrt{\theta t} x)}(x)
\end{equation}
over a neighborhood of the origin for exponentially decaying Borel probability measures $\eta$. This example satisfies \Cref{conjecture} for randomly distributed $a_1$ and $a_2$.

\begin{corollary}
\label{cor:lln_dbm1}
Suppose $\theta N\rightarrow\infty$ and $\eta_N$ is an exponentially decaying Borel probability measure over $\mathbb{R}^N$ for $N\geq 2$ such that $\eta_N(x\mapsto\frac{x}{\theta N})$ converges in terms of power sums to the distribution $\eta$ over $\mathbb{R}$ with finite moments. Then, for $\alpha \geq 0$, $\mu_{\alpha \theta N}^{A^{N-1}(\theta)}(\eta_N)(x\mapsto \frac{x}{\theta N})$ converges in terms of power sums to the free convolution of $\eta$ and the semicircle law multiplied by $\sqrt{\alpha}$ as $N\rightarrow\infty$.
\end{corollary}

\begin{proof}
Assume that $\alpha>0$, since the result is trivial if $\alpha=0$. We can follow the method of the proof of \Cref{cor:lr}. First, since condition (b) of \Cref{cor:lln_satisfaction_a} is satisfied when we set $\mu_N$ to be $\eta_N$ and $c_{(k)}$ to be the $k$th free cumulant of $\eta$ for all $k\geq 1$, we have that condition (a) is satisfied as well. The next step is to deduce that $\mu_{\alpha N}^\theta(\eta_N)$ is exponentially decaying for all $N\geq 2$. This does not follow from \pref{eq:random_product} and \Cref{lemma:conjecture_support_2} because $H_\theta(x\mapsto \sqrt{\alpha\theta N}x)$ is not compactly supported.

Assume that $\eta_N$ exponentially decays at rate $R$. By \Cref{lemma:bgf_dbm}, $\mu^{A^{N-1}(\theta)}_{\alpha \theta N}(\eta_N)(x\mapsto\frac{x}{\theta N})$ exponentially decays at any rate less than $R$. Then, because 
\[
G_{\mu_{\alpha \theta N}^{A^{N-1}(\theta)}(\eta_N)}^{A^{N-1}(\theta)}(x) = \exp\left(\frac{\alpha\theta N}{2}\sum_{i=1}^N x_i^2\right)G_{\eta_N}^{A^{N-1}(\theta)}(x)
\]
over $B_R(0)$ by \Cref{lemma:bgf_dbm}, applying \Cref{cor:lln_satisfaction_a} again gives that the $k$th free cumulant of $\mu^{A^{N-1}(\theta)}_{\alpha \theta N}(\eta_N)(x\mapsto\frac{x}{\theta N})$ converges to $c_{(k)}+\mathbf{1}\{k=2\} \alpha$ as $N\rightarrow\infty$. Hence, the probability measures $\mu^{A^{N-1}(\theta)}_{\alpha \theta N}(\eta_N)(x\mapsto\frac{x}{\theta N})$ converges to the free convolution of $\eta$ and the semicircle law multiplied by $\alpha$ in terms of power sums as $N\rightarrow\infty$.
\end{proof}

\begin{corollary}
\label{cor:lln_dbm2}
Suppose $\theta>0$, $\theta N\rightarrow c \geq 0$, and $\eta_N$ is an exponentially decaying Borel probability measure over $\mathbb{R}^N$ for $N\geq 2$ such that $\eta_N$ converges in terms of power sums as $N\rightarrow\infty$. Then, for $\alpha \geq 0$, $ \mu_\alpha^{A^{N-1}(\theta)}(\eta_N)$ converges in terms of power sums as $N\rightarrow\infty$.
\end{corollary}

\begin{proof}
Follow the proof of \Cref{cor:lln_dbm1} and replace \Cref{cor:lln_satisfaction_a} with its high temperature regime version, which can be proved using \Cref{cor:equivalence_a2} and \Cref{lemma:exponent_decay}.
\end{proof}
\end{example}

\begin{example}
We consider the type D Dyson Brownian motion. For $\theta>0$, consider the stochastic differential equation given by initial values $(a_1(0),\ldots,a_N(0))$ and
\begin{equation}
\label{eq:dbm2}
da_i(t) = \theta\sum_{j\in [N]\backslash\{i\}} \frac{1}{a_i(t) - a_j(t)} + \frac{1}{a_i(t)+a_j(t)} dt + dB_i(t), 
\end{equation}
where $B_i$, $1\leq i\leq N$ are independent standard Brownian motions. Let $\mathcal{W}(D^N)$ denote the Weyl chamber $\{x\in\mathbb{R}^N: x_1> \cdots > x_N, x_{N-1}+x_N> 0\}$. We add the assumption that $(a_1(t),\ldots,a_N(t))\in\overline{\mathcal{W}(D^N)}$ for all $t\geq 0$ and denote the resulting SDE as $DBM(D^N(\theta))$. 

The solution to $DBM(D^N(\theta))$ is unique if we assume that it is a continuous semimartingale, although we do not require this assumption for uniqueness if $\theta\geq \frac{1}{2}$. Moreover, by the results of \cites{demni_sde}, a continuous semimartingale solution exists for all $\theta>0$.

We define the Markov process $\left\{X^{D^N(\theta)}_t(a)\right\}_{t\geq 0}$ similarly to how $\left\{X^{A^{N-1}(\theta)}_t(a)\right\}_{t\geq 0}$ is defined, where its transition probability is given by replacing $A^{N-1}(\theta)$ with $D^N(\theta)$ in \pref{eq:dbm_transition1}. Observe that $\gamma_{D^N(\theta)}$ is defined to be $\sum_{r\in (D^N)^+} \theta(r) = N(N-1)\theta$. 

The Markov process $\left\{X^{D^N(\theta)}_t(a)\right\}_{t\geq 0}$ is studied in \cites{rosler_markov,gallardo_yor,skew-product} and \cite{skew-product} shows that it is the unique solution to $DBM(D^N(\theta))$ initiated at $a\in\overline{\mathcal{W}(D^N)}$ if we assume that $\theta \geq \frac{1}{2}$. For $t\geq 0$ and a probability measure $\eta$ over $\overline{\mathcal{W}(D^N)}$, we define $\mu_t^{D^N(\theta)}(\eta)$ to be the probability measure over $\overline{\mathcal{W}(D^N)}$ that is induced by $X^{D^N(\theta)}_t(a)$ if $a$ is independently sampled from $\eta$.

It is straightforward to deduce that \Cref{lemma:bgf_dbm} with $A^{N-1}(\theta)$ replaced with $D^N(\theta)$ is true. We generalize \Cref{def:converge_powersum} to signed Borel measures as follows. 

\begin{definition}
\label{def:converge_powersum2}
For $N\geq 1$, suppose $\eta_N$ is a signed Borel measure over $\mathbb{C}^N$. Then, $\{\eta_N\}_{N\geq 1}$ converges in terms of power sums if $\eta_N(\mathbb{C}^N)$ converges and there exists $m_k\in \mathbb{C}$ for $k\geq 1$ such that for all $\nu\in\Gamma$,
\[
\lim_{N\rightarrow\infty}\frac{1}{N^{\ell(\nu)}}\int_{\mathbb{C}^N} p_\nu(a)d\eta_N(a) = \prod_{i=1}^{\ell(\nu)} m_{\nu_i}.
\]
\end{definition}

\begin{remark}
If $\{\eta_N\}_{N\geq 1}$ converges in terms of power sums, since $\eta_N(\mathbb{C}^N)$ converges as $N\rightarrow\infty$, we must have that $\eta_N$ has finite total variation for sufficiently large $N$. Similarly, we require that for all $\nu\in\Gamma$, $p_\nu(a)d\eta_N(a)$ is absolutely integrable for sufficiently large $N$.
\end{remark}

Using the notion of convergence in terms of power sums for signed Borel measures, we obtain the following two corollaries.

\begin{corollary}
\label{cor:lln_dbm_d1}
Suppose $\theta N\rightarrow\infty$ and $\eta_N$ is an exponentially decaying Borel probability measure over $\mathbb{R}^N$ for $N\geq 2$. Define $\hat{\eta}_N$ to be the signed measure $\frac{e(a)}{\prod_{i=1}^N (1+2(i-1)\theta)}d\eta_N(a)$. Assume that $\hat{\eta}_N\left(x\mapsto \frac{x}{\theta N}\right)$ converges in terms of power sums. Suppose $\alpha\geq 0$. Define $\hat{\mu}_N$ to be the signed measure $\frac{e(a)}{\prod_{i=1}^N (1+2(i-1)\theta)} d\mu_{\alpha \theta N}^{D^N(\theta)}(\eta_N)(a)$. Then, $\hat{\mu}_N\left(x\mapsto\frac{x}{\theta N}\right)$ converges in terms of power sums.
\end{corollary}

\begin{proof}
Follow the proof of \Cref{cor:lln_dbm1}. We use the modification of \Cref{lemma:bgf_dbm} that applies to the type D root system to deduce that $\mu_{\alpha\theta N}^{D^N(\theta)}(\eta_N)$ exponentially decays. Afterwards, we use the equivalence of conditions (b) and (d) of \Cref{cor:lln_satisfaction_d} rather than \Cref{cor:lln_satisfaction_a}.
\end{proof}

\begin{corollary}
\label{cor:lln_dbm_d2}
Suppose $\theta>0$, $\theta N\rightarrow c \geq 0$, and $\eta_N$ is an exponentially decaying Borel probability measure over $\mathbb{R}^N$ for $N\geq 2$. Define $\hat{\eta}_N$ to be the signed measure $\frac{e(a)}{\prod_{i=1}^N (1+2(i-1)\theta)}d\eta_N(a)$. Assume that $\hat{\eta}_N$ converges in terms of power sums. Suppose $\alpha\geq 0$. Define $\hat{\mu}_N$ to be the signed measure $\frac{e(a)}{\prod_{i=1}^N (1+2(i-1)\theta)} d\mu_\alpha^{D^N(\theta)}(\eta_N)(a)$. Then, $\hat{\mu}_N$ converges in terms of power sums.
\end{corollary}

\begin{proof}
Follow the proof of \Cref{cor:lln_dbm_d1} and replace \Cref{cor:lln_satisfaction_d} with its high temperature regime version, which can be proved using \Cref{cor:equivalence_odd2} and \Cref{lemma:exponent_decay}.
\end{proof}

\end{example}

\begin{example}
Next, we study the law of large numbers for the $\beta$-Laguerre ensemble. First, we compute the type BC Bessel generating function of the Chiral ensemble, which has the same distribution as the square roots of the $\beta$-Laguerre ensemble, where $\beta=2\theta$. We discuss its construction for $\theta\in\{\frac{1}{2}, 1, 2\}$ and its definition for general values of $\theta\geq 0$ from \cite{forrester}*{Section 3.1}. Suppose $M\geq N$ such that $M-N+1-\frac{1}{2\theta}\geq 0$. Let $X$ denote a random $M\times N$ matrix whose independent entries are real, complex, or real quaternion numbers with densities: \begin{itemize}
\item $\frac{1}{\sqrt{2\pi}}e^{-\frac{x^2}{2}}$ if $\theta=\frac{1}{2}$, 
\item $\frac{1}{\pi} e^{-|x|^2}$ if $\theta=1$,
\item $\frac{4}{\pi^2}e^{-2|x|^2-2|y|^2}$ if $\theta=2$. Note that we parameterize a quaternion as $(x,y)\in\mathbb{C}^2$. 
\end{itemize}
The Chiral ensemble is defined as the positive eigenvalues of
\[
H = \begin{bmatrix} 0_{M\times M} & X \\ X^H & 0_{N\times N}\end{bmatrix}.
\]
This matrix almost surely has $N$ negative eigenvalues, $M-N$ zero eigenvalues, and $N$ positive eigenvalues.

Define the probability distribution $C_{\theta,M,t}^N$ over $\mathbb{R}_{\geq 0}^N$ to have the density of $a$ proportional to 
\[
\prod_{i=1}^N a_i^{2\theta(M-N+1) - 1} e^{-\frac{a_i^2}{2t}} \prod_{1\leq i<j\leq N} |a_i^2-a_j^2|^{2\theta}.
\]
Then, $C_{\theta,M,t}^N$ is the density function of the positive eigenvalues of $H$ after they are rescaled when $\theta\in\{\frac{1}{2}, 1, 2\}$. The following lemma is a straightforward generalization of \cite{rectangularmatrix}*{Proposition 5.21}.

\begin{lemma}
\label{lemma:bgf_chiral}
Suppose $\theta \geq 0$, $M\geq N$ such that $M-N+1-\frac{1}{2\theta}\geq 0$, and $t>0$. Then,
\[
G_{C^N_{\theta, M, t}}^{BC^N(\theta, \theta(M-N+1)-\frac{1}{2})}(x) = \exp\left(\frac{t}{2}\sum_{i=1}^N x_i^2\right).
\]
\end{lemma}

Next, we discuss the $\beta$-Laguerre ensemble, which has the same distribution as the squares of the eigenvalues of the Chiral ensemble. Suppose $M\geq N$ such that $M-N+1-\frac{1}{\theta}\geq 0$. Define the probability distribution $L_{\theta,M,t}^N$ over $\mathbb{R}^N_{\geq 0}$ to have the density of $a$ proportional to
\[
\prod_{i=1}^N a_i^{\theta (M-N+1)-1}e^{-\frac{a_i}{2t}}\prod_{1\leq i<j\leq N} |a_i-a_j|^{2\theta}.
\]
Although we cannot generalize the random matrix that we previously defined for $\theta\in\{\frac{1}{2},1,2\}$, \cite{betaensembles} computes a tridiagonal matrix whose eigenvalue distribution is given by $L_{\theta,M,t}^N$ for all $\theta>0$. Moreover, the distribution is the image of $C_{\theta,M,t}^N$ under the pushforward map $x\mapsto x^2$. Despite this relation, it remains challenging to compute the Bessel generating function for the $\beta$-Laguerre ensemble. 

Afterwards, we obtain \Cref{cor:laguerre_moments}, which has been previously proved in Theorem 1.1 of the paper \cite{betalaguerre}. The paper proves that the $\beta$-Laguerre ensemble weakly converges to the Marchenko-Pastur law. It is straightforward to verify that the moments appearing in the corollary are the same as those of the Marchenko-Pastur law. 

For $k\geq 1$, let $\mathfrak{D}_k$ denote the set of Dyck paths of length $2k$ and for $p\in\mathfrak{D}_k$, let $E(p)$ denote the number of even $i\in [2k]$ such that a descent of $p$ is located at its $i$th position. 

\begin{corollary}
\label{cor:laguerre_moments}
Assume that $\lim_{N\rightarrow\infty} \theta N=\infty$, $M\geq N$, $M-N+1-\frac{1}{2\theta}\geq 0$, $\lim_{N\rightarrow\infty} \frac{M}{N} = c \geq 1$, and $\lim_{N\rightarrow\infty} t = \alpha\geq 0$. Then, for all $\nu\in\Gamma$,
\[
\lim_{N\rightarrow\infty} \frac{\mathbb{E}_{a\sim L^N_{\theta,M, t}}[p_\nu(a)]}{(\theta N)^{|\nu|}N^{\ell(\nu)}} = (2\alpha)^{|\nu|}\prod_{i=1}^{\ell(\nu)}\sum_{p\in\mathfrak{D}_{\nu_i}}c^{E(p)}.
\]
\end{corollary}

\begin{proof}
Use \Cref{cor:lln_satisfaction_bc}, \Cref{lemma:bgf_chiral}, and the fact that $L_{\theta,M,t}^N$ is the image of $C_{\theta,M,\theta N t}^N$ under the pushforward map $x\mapsto \frac{x^2}{\theta N}$.
\end{proof}
\end{example}

\section{Computing the Dunkl bilinear form with combinatorics}
\label{sec:combinatorics}

For any $\theta,\theta_0,\theta_1\in\mathbb{C}$ and $\lambda,\nu\in\Gamma$, we set 
\[
[p_\lambda,p_\nu]_{A^0(\theta)},\,[p_\lambda,p_\nu]_{BC^1(\theta_0,\theta_1)},\,[p_\lambda,p_\nu]_{D^1(\theta)}\triangleq [1]\partial_1^{\lambda}x_1^{\nu}=\mathbf{1}\{|\lambda|=|\nu|\} |\lambda|!.
\]
Note that $A^0$, $B^1$, $C^1$, and $D^1$ are not actually root systems; we use this notation for simplicity.

The following result computes combinatorial expressions for $[p_\lambda,p_\nu]_{A^{N-1}(\theta)}$, $[p_\lambda,p_\nu\\]_{BC^N(\theta_0,\theta_1)}$, and $[p_\lambda,p_\nu]_{D^N(\theta)}$ in terms of the values of the quantities for small values of $N$, including $N=1$.

\begin{theorem}
\label{thm:summation}
Suppose $\lambda,\nu\in\Gamma$ and $|\lambda|=|\nu|$. For (B) and (C), assume that $\lambda,\nu\in\even$. Let $k=|\lambda|$ and suppose $N\geq 2$. Furthermore, assume that $k+\ell(\lambda)-\ell(\nu)<N$. For all $\theta,\theta_0,\theta_1\in\mathbb{C}$, the following expressions are true:
\begin{align*}
&(A)\,\,\,[p_\lambda, p_\nu]_{A^{N-1}(\theta)} && = \sum_{i=1}^{k+\ell(\lambda)-\ell(\nu)} (-1)^{k+\ell(\lambda)-\ell(\nu)-i}[p_\lambda,p_\nu]_{A^{i-1}(\theta)} \times\\
& && \binom{N}{i}\binom{N-i-1}{k+\ell(\lambda)-\ell(\nu)-i} \\
&(B)\,\,\,[p_\lambda, p_\nu]_{BC^N(\theta_0,\theta_1)} && = \sum_{i=1}^{k+\ell(\lambda)-\ell(\nu)} (-1)^{k+\ell(\lambda)-\ell(\nu)-i}[p_\lambda,p_\nu]_{BC^i(\theta_0,\theta_1)}\times \\ & &&\binom{N}{i}\binom{N-i-1}{k+\ell(\lambda)-\ell(\nu)-i} \\
&(C)\,\,\,[p_\lambda, p_\nu]_{D^N(\theta)} && = \sum_{i=1}^{k+\ell(\lambda)-\ell(\nu)} (-1)^{k+\ell(\lambda)-\ell(\nu)-i}[p_\lambda,p_\nu]_{D^i(\theta)} \times \\
& && \binom{N}{i}\binom{N-i-1}{k+\ell(\lambda)-\ell(\nu)-i}
\end{align*}
\end{theorem}

\begin{proof}
We confirm expression (A). Expressions (B) and (C) follow similarly. Let $\Delta = k+\ell(\lambda)-\ell(\nu)$ and $k=|\lambda|$. Suppose $\mathcal{S}$ is the set of sequences of operators appearing in $\mathcal{D}(A^{N-1}(\theta))(p_\lambda)$, so that
\[
[p_\lambda, p_\nu]_{A^{N-1}(\theta)} = \sum_{s\in\mathcal{S}} sp_\nu,
\]
where for $s\in\mathcal{S}$, $sp_\nu\triangleq s_k\circ\cdots\circ s_1 p_\nu$.

For $s\in\mathcal{S}$, let $i(s)$ be the number of distinct indices in $s$. Observe that if $s\in\mathcal{S}$ and $s_k \circ\cdots\circ s_1 p_\nu\not=0$, then we must have that $i(s)\leq k+\ell(\lambda)-\ell(\nu)$, by the argument that $s$ must have at least $\ell(\nu)$ derivatives. Thus, let $\mathcal{S}'$ be the set of $s\in\mathcal{S}$ such that $i(s)\leq k+\ell(\lambda)-\ell(\nu)$, so that 
\[
[p_\lambda, p_\nu]_{A^{N-1}(\theta)} = \sum_{s\in\mathcal{S}'} sp_\nu.
\]

Note that for $i\geq 1$, $[p_\lambda,p_\nu]_{A^{i-1}(\theta)}$ corresponds to picking a sequence of operators in $\mathcal{D}(A^{N-1}(\theta))(p_\lambda)$ with all indices in $[i]$. Then, $[p_\lambda,p_\nu]_{A^{i-1}(\theta)}\binom{N}{i}$ corresponds to picking this sequence as well as replacing $[i]$ by any subset of $[N]$ with size $i$.

Suppose $s\in\mathcal{S}'$. Assume that the set of distinct indices in $s$ is $[i(s)]$. Then, it is clear that
\[
s p_\nu(x_1,\ldots,x_N) = sp_\nu(x_1,\ldots,x_{i(s)},0,\ldots,0).
\]
For $i\geq i(s)$, the number of times that $sp_\nu$ is counted in $[p_\lambda,p_\nu]_{A^{i-1}(\theta)}\binom{N}{i}$ is $\binom{N-i(s)}{i-i(s)}$. Hence, the total number of times that $sp_\nu$ is counted is 
\[
\sum_{i=i(s)}^{k+\ell(\lambda)-\ell(\nu)} (-1)^{k+\ell(\lambda)-\ell(\nu)-i}\binom{N-i(s)}{i-i(s)}\binom{N-i-1}{k+\ell(\lambda)-\ell(\nu)-i}.
\]
It suffices to show that this quantity equals one. Equivalently, it suffices to show that
\begin{equation}
\label{eq:combinatorics_1}
\sum_{i=0}^{\Delta-z} (-1)^{i}\binom{N-z}{\Delta-z-i}\binom{N-\Delta-1+i}{i}=1
\end{equation}
for $\Delta\geq 1$, $z\in [\Delta]$, and $N\in\mathbb{N}$.

We prove \pref{eq:combinatorics_1} using induction on $z$ from $z=\Delta$ to $1$. The base case $z=\Delta$ is clear. Assume that the inductive hypothesis is true for $z=m$, where $m\in \{2,\ldots,\Delta\}$. We prove that it is true for $z=m-1$. We have that 
\begin{align*}
& \sum_{i=0}^{\Delta-m+1} (-1)^{i}\binom{N-m+1}{\Delta-m+1-i}\binom{N-\Delta-1+i}{i} \\
&= \sum_{i=0}^{\Delta-m+1}(-1)^i \left(\binom{N-m}{\Delta-m-i} + \binom{N-m}{\Delta-m+1-i}\right)\binom{N-\Delta-1+i}{i}.
\end{align*}
Next, observe that 
\begin{align*}
& \sum_{i=0}^{\Delta-m+1}(-1)^i \binom{N-m}{\Delta-m-i} \binom{N-\Delta-1+i}{i} \\
& = \sum_{i=0}^{\Delta-m}(-1)^i \binom{N-m}{\Delta-m-i} \binom{N-\Delta-1+i}{i} = 1
\end{align*}
by the inductive hypothesis for $z=m$. Hence, it suffices to show that 
\[
\sum_{i=0}^{\Delta-m+1}(-1)^i \binom{N-m}{\Delta-m+1-i}\binom{N-\Delta-1+i}{i} = 0.
\]
This expression evaluates to
\[
\frac{(N-m)!}{(N-\Delta-1)!}\sum_{i=0}^{\Delta-m+1}\frac{(-1)^i}{(\Delta-m+1-i)!i!}= 0.
\]
\end{proof}

\begin{remark}
The previous result is not applicable when $k+\ell(\lambda)-\ell(\nu)\geq N$, but if this is the case, then we must consider when the number of distinct indices is $N$. Then, to obtain a summation formula for $[p_\lambda,p_\nu]_{A^{N-1}(\theta)}$ in (A), we must set $[p_\lambda,p_\nu]_{A^{N-1}(\theta)}$ to be one of the summands. Similarly, we cannot obtain analogous formulas for (B) and (C). 

Furthermore, for (B) and (C) we require that $\lambda,\nu\in\even$ so that $i(s)\leq k + \ell(\lambda)-\ell(\nu)$. If this is not the case, then $s$ does not necessarily need to have at least $\ell(\nu)$ derivatives and the operators at the locations $1+\lambda_1+\cdots + \lambda_i$ for $0\leq i\leq \ell(\lambda)-1$ do not necessarily have to be derivatives, so we can have that $i(s)>k+\ell(\lambda)-\ell(\nu)$. Afterwards, we cannot ensure that the order of the Dunkl bilinear form is $N^{k+\ell(\lambda)-\ell(\nu)}$, if we assume that $\lambda$ and $\nu$ are fixed.
\end{remark}

Assuming that $\theta$ is fixed, the order of the expressions in \Cref{thm:summation} matches the order of the expressions in \Cref{thm:leadingorder_a,thm:leadingorder_bc}, which is $N^{k+\ell(\lambda)-\ell(\nu)}$. However, the leading order coefficients are not apparent from the formulas.

\bibliography{mybib.bib}

\end{document}